\documentclass[11pt]{amsart}

\newif\iffinalrun

\usepackage{enumitem}
\usepackage{amsmath}
\usepackage{latexsym}
\usepackage{amsfonts}
\usepackage{amssymb}
\usepackage{graphicx}
\usepackage{mathrsfs}
\usepackage{graphpap}
\usepackage{color}
\usepackage[margin=1.3in]{geometry}

\usepackage{multirow}
\usepackage{hyperref}

\theoremstyle{definition}

\numberwithin{equation}{subsection}

\usepackage[all]{xy}

 \theoremstyle{plain}
\newtheorem{theo}[equation]{Theorem}
 \theoremstyle{plain}

\theoremstyle{plain}

\theoremstyle{plain}

\theoremstyle{plain}

\theoremstyle{plain}
\newtheorem{conj}[equation]{Conjecture}
\theoremstyle{plain}

\theoremstyle{plain}
  \newtheorem{prop}[equation]{Proposition}
\theoremstyle{plain}
 \newtheorem{lemm}[equation]{Lemma}
\theoremstyle{plain}

\theoremstyle{plain}
\newtheorem{coro}[equation]{Corollary}
\theoremstyle{plain}

\theoremstyle{definition}
  \newtheorem{defi}[equation]{Definition}
\theoremstyle{definition}
  
 \theoremstyle{definition}
  
  \newtheorem{exam}[equation]{Example}
\theoremstyle{remark}
\newtheorem{rema}[equation]{Remark}

\iffinalrun
  \newcommand{\need}[1]{}
  \newcommand{\mar}[1]{}
\else
  \newcommand{\need}[1]{{\tiny *** #1}}
  \newcommand{\mar}[1]{\marginpar{\raggedright\tiny #1}}
\fi

\DeclareMathOperator{\Gal}{Gal}
\DeclareMathOperator{\Mat}{Mat}
\DeclareMathOperator{\Hom}{Hom}

\newcommand{\M}{\mathrm{M}}
\newcommand{\N}{\mathbf{N}}
\newcommand{\Q}{\mathbf{Q}}
\newcommand{\Z}{\mathbf{Z}}
\newcommand{\R}{\mathbf{R}}
\newcommand{\Qp}{\mathbf{Q}_p}
\newcommand{\F}{\mathbf{F}}

\newcommand{\Bst}{\mathbf{B}_{\mathrm{st}}}
\newcommand{\Bdr}{\mathbf{B}_{\mathrm{dR}}}
\newcommand{\barS}{\overline{S}}

\newcommand{\Dst}{\mathrm{D}_{\mathrm{st}}}

\newcommand{\Tst}{\mathrm{T}_{\mathrm{st}}}

\newcommand{\SBr}{S_{\mathcal{O}_E}}
\newcommand{\rhobar}{\overline{\rho}}

\newcommand{\cM}{\mathcal{M}}

\newcommand{\fM}{\mathfrak{M}}
\newcommand{\fL}{\mathfrak{L}}

\newcommand{\fm}{\mathfrak{m}}
\newcommand{\Fil}{\mathrm{Fil}}

\newcommand{\GL}{\mathrm{GL}}

\newcommand{\cD}{\mathcal{D}}

\newcommand{\cO}{\mathcal{O}}
\newcommand{\cJ}{\mathcal{J}}

\newcommand{\barE}{\overline{E}}
\newcommand{\barF}{\overline{F}}

\newcommand{\adj}{\mathrm{adj}}
\newcommand{\Tr}{\mathrm{Tr}}

\newcommand{\vx}{\vec{x}}
\newcommand{\vr}{\vec{r}}

\newcommand{\vl}{\vec{l}}
\newcommand{\vLambda}{\vec{\Lambda}}
\newcommand{\vL}{\vec{\fL}}

\newcommand{\vk}{\vec{k}}

\newcommand{\cK}{\mathcal{K}}

\newcommand{\vth}{\vec{\theta}}
\newcommand{\vrho}{\vec{\varrho}}
\newcommand{\vdel}{\vec{\delta}}
\newcommand{\vDel}{\vec{\Delta}}
\newcommand{\vLam}{\vec{\Lambda}}
\newcommand{\vTh}{\vec{\Theta}}

\DeclareMathOperator{\Supp}{Supp}
\newcommand{\IQptwo}{{I_{\Q_{p^2}}}}

\newcommand{\baseE}{\underline{E}}
\newcommand{\baseF}{\underline{F}}

\author{Seongjae Han}
\author{Chol Park}
\address{Department of Mathematical Sciences, Ulsan National Institute of Science and Technology, Unist-gil 50, Ulsan 44919, Republic of Korea}
\email{seongjaehan@unist.ac.kr,\,\,\,cholpark@unist.ac.kr}

\begin{document}
\title{On families of strongly divisible modules of rank~$2$}

\begin{abstract}
Let $p$ be an odd prime, and $\Q_{p^f}$ the unramified extension of $\Q_p$ of degree~$f$. In this paper, we reduce the problem of constructing strongly divisible modules for $2$-dimensional semi-stable non-crystalline representations of $\Gal(\overline{\Q}_p/\Q_{p^f})$ with Hodge--Tate weights in the Fontaine--Laffaille range to solving systems of linear equations and inequalities. We also determine the Breuil modules corresponding to the mod-$p$ reduction of the strongly divisible modules. We expect our method to produce at least one Galois-stable lattice in each such representation for general $f$. Moreover, when the mod-$p$ reduction is an extension of distinct characters, we further expect our method to provide the two non-homothetic lattices. As applications, we show that our approach recovers previously known results for $f=1$ and determine the mod-$p$ reduction of the semi-stable representations with some small Hodge--Tate weights when $f=2$.
\end{abstract}

\maketitle

\tableofcontents

\section{Introduction}
The significance of the mod-$p$ reduction of geometric Galois representations has been demonstrated in various areas of Number Theory. In particular, one can use the mod-$p$ reduction to construct Galois deformation rings. The non-existence of certain potentially crystalline lifts (equivalently, the vanishing of certain potentially crystalline deformation rings) can be used to eliminate Sere weights in the weight part of Serre's conjecture \cite{GLS14,GLS15,MP,LMP,LLL}. Additionally, understanding the special fibers of certain deformation rings can be used to show the modularity of specific Serre weights \cite{LLHLM18,LLL,LLMPQb}. Furthermore, it plays a key role in understanding Breuil--M\'ezard conjecture \cite{BM,Kis09,Pas15,CEGS,EG,GK,LLHLM23}, Gelfand--Kirillov dimension \cite{BHHMS}, mod-$p$ local-global compatibility \cite{PQ,LLMPQa,EL}, etc.

There are a few known results on mod-$p$ reduction of geometric Galois representations. Breuil--M\'ezard computed the mod-$p$ reduction of $2$-dimensional semi-stable representations of $\Gal(\overline{\Q}_p/\Q_p)$ with Hodge--Tate weights $(0,r)$ when $r$ is odd and $0<r<p-1$ by constructing the strongly divisible modules \cite{BM}. This strategy was extended in \cite{GP,LP} to the case where $r$ is even and $0<r<p-1$. Savitt used the same approach to determine the mod-$p$ reduction of $2$-dimensional potentially crystalline representations of $\Gal(\overline{\Q}_p/\Q_p)$ with Hodge--Tate weights $(0,1)$ \cite{Savitt}. Bergdall--Levin--Liu computed the mod-$p$ reduction of $2$-dimensional semi-stable representations of $\Gal(\overline{\Q}_p/\Q_p)$ with large $\fL$-invariant, by using Breuil--Kisin modules \cite{BLL}. Buzzard--Gee employed the compatibility between $p$-adic Langlands correspondence and mod-$p$ Langlands correspondence to compute the mod-$p$ reduction of $2$-dimensional crystalline representations of $\Gal(\overline{\Q}_p/\Q_p)$ of slopes in $(0,1)$ \cite{BG09,BG13}, and Ghate and his collaborators applied similar strategies for the cases of slopes in $[1,2)$ \cite{BhG,BGR,GaG,GRa}. Chitrao--Ghate extended the results of \cite{BM,GP,LP} to the Hodge--Tate weights $(0,r)$ for $r=p-1,p$, by leveraging the Iwahori mod-$p$ local Langlands correspondence \cite{CG}. There are some notable results as well for $3$-dimensional representations of $\Gal(\overline{\Q}_p/\Q_p)$. The mod-$p$ reduction of $3$-dimensional semi-stable representations of $\Gal(\overline{\Q}_p/\Q_p)$ with Hodge--Tate weights $(0,1,2)$ was studied in \cite{Par} constructing strongly divisible modules. Additionally, Liu treated $3$-dimensional crystalline representations of $\Gal(\overline{\Q}_p/\Q_p)$ with Hodge--Tate weights $(0,r,s)$ for $2\leq r\leq p-2$ and $p+2\leq s\leq r+p-2$ using Breuil--Kisin modules \cite{Liu}. Despite these advances, little is known about the mod-$p$ reduction of semi-stable representations of $\Gal(\overline{\Q}_p/K)$ for finite extensions $K$ beyond $\Q_p$.

Let $\Q_{p^f}$ be the unramified extension of $\Q_p$ of degree~$f$. In this paper, we study the mod-$p$ reduction of $2$-dimensional semi-stable non-crystalline representations of $\Gal(\overline{\Q}_p/\Q_{p^f})$ with Hodge–Tate weights in the Fontaine–Laffaille range. We use strongly divisible modules as in \cite{BM}, but our approach is slightly different. We first introduce the notion of pseudo-strongly divisible modules of rank $2$, motivated by the anticipated structure of strongly divisible modules. The isotypic components of strongly divisible modules generated by elements of a particular form serve as canonical examples of pseudo-strongly divisible modules; however, the converse is far from true. These pseudo-strongly divisible modules are parameterized by quadruples $(\Lambda, \Theta, \Omega, x) \in (\overline{\Q}_p^\times)^3 \times \overline{\Q}_p$, which naturally appear as coefficients in their construction. We then formulate explicit conditions—given by a system of simple equations and inequalities—under which such a pseudo-strongly divisible module arises as an isotypic component of a genuine strongly divisible module (see \textbf{A} and \textbf{B} in \S\ref{subsec: main results, intro} or Proposition~\ref{prop: system of equations and inequalities sharpened}).

We expect that solving these equations and inequalities yields at least one strongly divisible module of rank $2$ corresponding to a Galois-stable lattice in each $2$-dimensional semi-stable non-crystalline representation of $\Gal(\overline{\Q}_p/\Q_{p^f})$ with Hodge–Tate weights in the Fontaine–Laffaille range, for arbitrary $f$. More optimistically, when the mod-$p$ reduction is an extension of two distinct characters, we further expect our method to produce both non-homothetic lattices. To support this framework, we verify that it reproduces known results for $f = 1$ \cite{BM, GP, LP}, and we present examples for $f = 2$ by explicitly solving the associated systems (see \S\ref{sec: examples}, in particular Remarks~\ref{rema: 2,2} and~\ref{rema: 1,5}).

\subsection{Main results}\label{subsec: main results, intro}
In this subsection, we introduce our main results. We start by introducing some notation. Let $E$ be a finite extension of $\Q_{p^f}$ with ring of integers $\cO$, uniformizer $\varpi$, and residue field $\F$. We will always consider $\Q_{p^f}$ as a subfield of $E$, and so if $a\in\Q_{p^f}$ then $a\in E$. Denote by $\cJ$ the cyclic group $\Z/f\Z$ which we identify with the set $\Hom_{\Q_p}(\Q_{p^f},E)$ by associating $1$ with the geometric Frobenius. Fix an isomorphism $$\Z_{p^f}\otimes_{\Z_p}\cO\cong \bigoplus_{j\in\cJ}\cO: a\otimes b\mapsto (\varphi^{-j}(a)b)_{j\in\cJ},$$ where $\varphi$ is the arithmetic Frobenius on $\Q_{p^f}$ and $\Z_{p^f}$ is the ring of integers of $\Q_{p^f}$. For each $j\in\cJ$, let $e_j$ be the idempotent element in $\Z_{p^f}\otimes_{\Z_p}\cO$ corresponding to the element $(0,\cdots,0,1,0,\cdots,0)\in \bigoplus_{j\in\cJ}\cO$, under the isomorphism above, where the non-zero entry appears at the $j$-th position. Additionally, we fix the isomorphism $\Q_{p^f}\otimes_{\Q_p} E\cong \bigoplus_{j\in\cJ}E$ and $\F_{p^f}\otimes_{\F_p}\F\cong \bigoplus_{j\in\cJ}\F$, both of which are compatible with the isomorphism above. By abuse of the notation, we will also write $e_j$ for the image of $e_j$ under the natural quotient map $\Z_{p^f}\otimes_{\Z_p}\cO\rightarrow \F_{p^f}\otimes_{\F_p}\F$. Note that we have $\varphi(e_j)=e_{j+1}$ and so $\varphi((d_j)_{j\in\cJ})=(d_{j-1})_{j\in\cJ}$. We also note that we will write $v_p$ for the $p$-adic valuation on $E$ normalized as $v_p(p)=1$, and write $v_p(\vec{y})$ for $(v_p(y_j))_{j\in\cJ}$ if $\vec{y}=(y_j)_{j\in\cJ}\in \bigoplus_{j\in\cJ}E(=E^f)$.

Let $\vr:=(r_j)_{j\in\cJ}$ be an $f$-tuple of positive integers with $r_j<p-1$ for all $j\in\cJ$, and fix $r\in\Z$ with $r_j\leq r<p-1$ for all $j\in\cJ$. Let $\cJ_0$ be a subset of $\cJ$ satisfying
\begin{equation}\label{eq: admissibility, intro}
\frac{1}{2f}\sum_{j\in\cJ}(r_j-1)\geq\frac{1}{f}\sum_{j\in \cJ_0}r_j,
\end{equation}
and set $\hat{\Q}:=\Q\cup\{\infty\}$ with the usual order. (Note that the condition~\eqref{eq: admissibility, intro} occurs from the admissibility of filtered $(\phi,N)$-modules.) For $\vk'=(k'_j)_{j\in\cJ}$ with 
\begin{equation}\label{eq: condition for k', intro}
k'_j\in
\begin{cases}
   [\frac{1}{2},r_j+\frac{1}{2}]\cap\frac{1}{2}\Z & \mbox{if } j\not\in\cJ_0; \\
   \{\infty\} & \mbox{if } j\in\cJ_0,
\end{cases}
\end{equation}
we further define a sign function $s_j:=(-1)^{2k'_j}$ if $j\not\in \cJ_0$ and $s_j:=-1$ if $j\in\cJ_0$.

We are now ready to introduce the equations and inequalities: for a subset $\cJ_0\subset\cJ$ with \eqref{eq: admissibility, intro} and for a tuple $\vk'=(k'_j)_{j\in\cJ}$ with \eqref{eq: condition for k', intro}, we consider the following equations and inequalities in $\vec{T}:=(T_0,T_1,\cdots,T_{f-1})\in \hat{\Q}^f$ and $\vec{t}:=(t_0,t_1,\cdots,t_{f-1})\in\hat{\Q}^f$:
\begin{enumerate}[leftmargin=*]
\item[\textbf{A}.] for each $j\in\cJ$
        \begin{equation*}
            T_j+s_jT_{j+1}
            =
            \begin{cases}
                2k'_j-r_j-\frac{1+s_j}{2} & \mbox{if } \frac{1}{2}\leq k'_j\leq \frac{r_j}{2}\,\,\mbox{or}\,\,k'_j=r_j+\frac{1}{2};\\
                2k'_j-r_j-\frac{1+s_j}{2}+2t_j  & \mbox{if } \frac{r_j+1}{2}\leq k'_j\leq r_j;\\
                r_j+1 & \mbox{if } k'_j=\infty,
            \end{cases}
        \end{equation*}
\item[\textbf{B}.] for each $j\in\cJ$
        \begin{equation*}
           T_j\leq 0\quad\mbox{ and }\quad
            \begin{cases}
                T_j\leq t_j  & \mbox{if }k'_j=\frac{1}{2};\\
                T_j\leq t_j,\, -1\leq T_{j+1} & \mbox{if } 1\leq k'_j\leq \frac{r_j}{2};\\
                T_j\geq t_j<0,\, -1\leq T_{j+1} &\mbox{if } \frac{r_j+1}{2}\leq k'_j\leq r_j-\frac{1}{2};\\
                t_j\leq T_j\leq t_j+r_j &\mbox{if }k'_j=r_j;\\
                T_j\geq t_j+r_j &\mbox{if }k'_j=r_j+\frac{1}{2};\\
                \mbox{none} &\mbox{if }k'_j=\infty.
            \end{cases}
        \end{equation*}
\end{enumerate}

For $(\vTh,\vx)=((\Theta_j)_{j\in\cJ},(x_j)_{j\in\cJ})\in (E^f)^{2}$ with $(\vec{T},\vec{t})=(v_p(\vTh),v_p(\vx))$ satisfying the equations in \textbf{A} and the inequalities in \textbf{B} for some $\cJ_0\subset\cJ$ with \eqref{eq: admissibility, intro} and for some $\vk'$ with \eqref{eq: condition for k', intro}, one can further choose $\vLam=(\Lambda_j)_{j\in\cJ}\in E^f$ satisfying
\begin{equation}\label{eq: definition of Lambda, intro}
v_p(\Lambda_j)=\frac{r_j-1-v_p(\Theta_j)+v_p(\Theta_{j+1})}{2}
\end{equation}
for all $j\in\cJ$. Then we give explicit recipes to construct a strongly divisible module and an admissible filtered $(\phi,N)$-module, determined by these data (see \S\ref{subsec: main results} as well as \S\ref{subsec: set up} for the precise construction):
\begin{itemize}[leftmargin=*]
\item the data $(\vLam,\vx,\vr)$ determines an admissible filtered $(\phi,N)$-module $D$ corresponding to a $2$-dimensional semi-stable non-crystalline representation of $\Gal(\overline{\Q}_p/\Q_{p^f})$ with $j$-labeled Hodge--Tate weights $(0,r_j)$ for all $j\in\cJ$;
\item the data $(\vLam,\vTh,\vx,\vr,r,\vk')$ determines a strongly divisible module $\fM$ of weight $r$.
\end{itemize}

\begin{theo}[Theorem~\ref{theo: main 2}]
$\fM$ corresponds to a Galois stable lattice in the semi-stable representation corresponding to $D$.
\end{theo}

An advantage of this theorem is that it allows the construction of strongly divisible modules corresponding to Galois stable lattices in $2$-dimensional semi-stable non-crystalline representations of $\Gal(\overline{\Q}_p/\Q_{p^f})$ with Hodge--Tate weights in the Fontaine--Laffaille range by solving the system of linear equations in \textbf{A} and the inequalities in \textbf{B}.

We expect that this method is quite general. More precisely, we propose the following:

\begin{conj}\label{conj: exhaust} 
\begin{enumerate}[leftmargin=*] 
\item This method provides at least one Galois stable lattice in each $2$-dimensional semi-stable non-crystalline representation of $\Gal(\overline{\Q}_p/\Q_{p^f})$ with Hodge--Tate weights in the Fontaine--Laffaille range. 
\item If the mod-$p$ reduction is an extension of two distinct characters, then this method exhausts the two non-homothetic lattices.
\end{enumerate}
\end{conj}

To support our conjecture, we verify that our method recovers the known results when $f=1$ (see \cite{BM, GP, LP}) and provide illustrative examples for $f=2$ (see \S\ref{sec: examples}; in particular, see Remarks~\ref{rema: 2,2} and~\ref{rema: 1,5}).

We further point out that we also compute the mod-$p$ reductions of pseudo-strongly divisible modules (see \S\ref{subsec: mod p reduction}). Once $(\vTh, \vx)$ satisfies the equations in \textbf{A} and the inequalities in \textbf{B} for some $\vk'$, these mod-$p$ reductions determine the Breuil module corresponding to the mod-$p$ reduction of the strongly divisible module $\fM$.

In the following example, we use our main results to compute the mod-$p$ reductions of the $2$-dimensional semi-stable non-crystalline representations of $\Gal(\overline{\Q}_p/\Q_{p^f})$ with parallel Hodge--Tate weights $(0,1)$.

\begin{exam}
Let $\vr\in\Z^f$ with $r_j=1$ for all $j\in\cJ$, and set $r=1$. It is immediate that $\cJ_0=\emptyset$ is the only choice satisfying the condition~\eqref{eq: admissibility, intro}. Thus $k'_j\neq \infty$ for any $j\in\cJ$, and it turns out that it is enough to consider the case $k'_j=\frac{1}{2}$ for all $j\in\cJ$.

For this $\vk':=(k'_j)_{j\in\cJ}$, the system of the equations in \textbf{A} and the inequalities in \textbf{B} is given by
\begin{itemize}[leftmargin=*]
  \item $T_j-T_{j+1}=0$ for all $j\in\cJ$;
  \item $T_j\leq t_j$ and $T_j\leq 0$ for all $j\in\cJ$,
\end{itemize}
which are summarized as
\begin{equation}\label{eq: example, intro}
v_p(\Theta_0)=v_p(\Theta_1)=\cdots=v_p(\Theta_{f-1})\leq \min\{0,v_p(x_0),\cdots,v_p(x_{f-1})\}
\end{equation}
for $\vTh\in E^f$ and $\vx\in E^f$. Moreover, we need to choose $\vLam\in (\cO^\times)^f$ by \eqref{eq: definition of Lambda, intro}. Note that for any $\vx\in E^f$ there exists $\vTh\in E^f$ satisfying \eqref{eq: example, intro}.

It turns out that $\vx$ determines the $\fL$-invariant $\vL\in [\mathbb{P}^1(E)]^f$ (see \eqref{eq: definition of varrho}, \eqref{eq: delta(-1)}, and \eqref{eq: definition of x}), and in this case we have $$x_j=\frac{p\fL_j-\fL_{j-1}}{p}\in E$$
for all $j\in\cJ$. Moreover, one can choose $\lambda\in\cO^\times$ satisfying $\lambda^f=\prod_{j\in\cJ}\Lambda_j$ (see \eqref{eq: From the setting 2}), and these data $(\lambda,\vL)$ together with $\vr$ determine an admissible filtered $(\phi,N)$-module, denoted by $D(\lambda,\vL)$, in Example~\ref{exam: admissible filtered modules}, which does not depend on the choice of $\lambda$ (see Lemma~\ref{lemm: isomorphism class of filtered phi N modules}).

The data $(\vLam,\vTh,\vx,\vr,r,\vk')$ also determine a strongly divisible module, but in this example we only describe the Breuil module, denoted by $\cM$, corresponding to the mod-$p$ reduction of the strongly divisible module, as the structure of the strongly divisible module is too complicated for the purpose of the introduction. For each $j\in\cJ$, we describe the isotypic component $$(\cM e_j,\Fil^1\cM e_j,\phi_1:\Fil^1\cM e_j\rightarrow \cM e_{j+1},N:\cM e_j\rightarrow \cM e_j)$$ following the matrix notation in \S\ref{subsubsec: Miscellaneous}: by Proposition~\ref{prop: mod p reduction, r_j=1} there exists a basis $\underline{E}:=(\bar{E}_1,\bar{E}_2)$ for $\cM$ and a system of generators $\underline{F}:=(\bar{F}_1,\bar{F}_2)$ for $\Fil^1\cM$ such that for each $j\in\cJ$
\begin{itemize}[leftmargin=*]
\item $\mathrm{Mat}_{\underline{E}e_j,\underline{F}e_j}(\Fil^1\cM e_j)=
\begin{bmatrix}
1 & 0 \\ \frac{x_j}{\Theta_j} & u 
\end{bmatrix}$;
\item $\mathrm{Mat}_{\underline{E}e_{j+1},\underline{F}e_j}(\phi_1)=
\begin{bmatrix}
\Lambda_j & 0 \\ 0 & -\tfrac{\Lambda_j\Theta_j}{\Theta_{j+1}}
\end{bmatrix}$;
\item $\mathrm{Mat}_{\underline{E}e_j,\underline{E}e_j}(N)=
\begin{bmatrix}
0 & 0 \\ \frac{1}{\Theta_j} & 0
\end{bmatrix}$.
\end{itemize}
Hence, one can conclude the following results by Caruso~\cite{Caruso}: if we let $\rhobar:=\Tst^\ast(\cM)$, where $\Tst^\ast$ is defined at \eqref{eq: definition of Tst}, then 
$$\rhobar\simeq
\begin{bmatrix}
  \omega & \ast \\
  0 & 1 
\end{bmatrix}\otimes U_{\lambda^f}$$
where $U_{\lambda^f}$ is the unramified character sending $\varphi^{-f}$ to $\lambda^f\in\F^{\times}$ and $\omega$ is the fundamental character of neveau $1$. Moreover, one can describe the extension class $\ast$ by observing the Breuil module above:
\begin{itemize}[leftmargin=*]
\item Assume that $v_p(x_j)\geq 0$ for all $j\in\cJ$. If we choose $\vTh$ with $v_p(\Theta_0)=\cdots=v_p(\Theta_{f-1})=0$ then $\rhobar_0$ is tr\`es ramifi\'ee, and otherwise it is split;
\item Assume that $v_p(x_j)<0$ for some $j\in\cJ$. Then $\rhobar$ is peu ramifi\'ee, and it is split if and only if the inequality in \eqref{eq: example, intro} is strict.
\end{itemize}
We note that this example recovers the main result of Cheon in his master's thesis~\cite{Cheon}.
\end{exam}

In the following subsection, we outline our approach for obtaining the equations and inequalities. We also briefly explain how a common solution $(\vLam,\vTh,\vx)$ to the equations in \textbf{A} and the inequalities in \textbf{B} for some $\vk'$ determines an admissible filtered $(\phi,N)$-module and a strongly divisible module.

\subsection{Strategy and overview}
Our basic strategy to compute mod-$p$ reduction follows \cite{BM}. The first step is to classify admissible filtered $(\phi, N)$-modules of rank $2$ with coefficients in $\Q_{p^f} \otimes_{\Q_p} E$, as there is an equivalence between the category of semi-stable representations and the category of admissible filtered $(\phi, N)$-modules \cite{CF}. The isomorphism classes of admissible filtered $(\phi, N)$-modules of rank $2$ have already been classified in \cite{Dous}. These classes are parameterized by the Frobenius eigenvalue $\lambda \in \cO$ and the $\fL$-invariants $\vL := (\fL_j)_{j \in \cJ} \in [\mathbb{P}^1(E)]^f$, and are denoted by $D(\lambda, \vL)$, when we fix the $j$-labeled Hodge--Tate weights $(0,r_j)$ for $r_j>0$ for all $j\in\cJ$ (see Example~\ref{exam: admissible filtered modules}). Since the absolute Galois group is compact, every semi-stable representation admits Galois stable lattices. Thanks to \cite{Bre99, Liu}, the corresponding admissible filtered $(\phi, N)$-module also has lattice structures, called strongly divisible modules, which correspond to these Galois stable lattices, if we further assume $r_j<p-1$ for all $j\in\cJ$. As the mod-$p$ reduction of Galois stable lattices is naturally defined, there exist Breuil modules associated with the mod-$p$ reduction of strongly divisible modules. By studying the structure of these Breuil modules, one can determine the mod-$p$ reduction of the semi-stable representations, and thanks to \cite{Caruso} our understanding for simple Breuil modules is relatively good.

The most difficult part in \cite{BM,GP} is constructing the strongly divisible modules, and it is the same here. But our approach to construct strongly divisible modules differs slightly from that of \cite{BM}. To explain our approach, we introduce the Breuil ring $S$. Let $S$ be the $p$-adic completion of the ring $\Z_{p^f}[\frac{u^i}{i!}\mid i\in\N]$, and fix $v:=u-p$. The ring $S$ can be explicitly described as follows:
$$S=\left\{\sum_{k=0}^{\infty}a_k\frac{v^k}{k!}\mid a_k\in\Z_{p^f}\mbox{ for all }k\in\N\cup\{0\}\mbox{ and }a_k\mapsto 0\mbox{ $p$-adically} \right\}.$$
This ring $S$ has additional structures: the filtration of ideals $\Fil^iS$ of $S$ for all $i\in\N\cup\{0\}$, the Frobenius map $\phi$ on $S$, and the monodromy operator $N$ on $S$. We further define $S_\cO:=S\otimes_{\Z_p}\cO$, and extend the definitions of $\phi$, $N$, and $\Fil^iS_{\cO}$ $\cO$-linearly. Note that one can decompose $S_\cO\cong\bigoplus_{j\in\cJ}S_\cO e_j$, and we often consider an $S_\cO$-module $\fM$ as an $\bigoplus_{j\in\cJ}S_\cO e_j$-module $\fM=\bigoplus_{j\in\cJ}\fM e_j$ via the isomorphism above. Each isotypic component $\fM e_j$ is an $S_\cO e_j$-module. To define a strongly divisible module, we first need to extend the coefficient of the admissible filtered $(\phi,N)$-module $D(\lambda,\vL)$ to $S_E:=S\otimes_{\Z_p}E$, by letting $\cD(\lambda,\vL):=S\otimes_{\Z_{p^f}}D(\lambda,\vL)$. Note that $\cD(\lambda,\vL)$ is naturally endowed with additional structures: the filtration $\{\Fil^i\cD\}_{i\in\N\cup\{0\}}$ (resp. $\phi$, $N$) induced from the filtration of $S$ (resp. $\phi$ on $S$, $N$ on $S$) and the filtration of $D$ (resp. $\phi$ on $D$, $N$ on $D$) (see \S\ref{subsec: strongly divisible modules, preliminary} for the precise definition). Then we fix $r\in\N$ with $r_j\leq r<p-1$ for all $j\in\cJ$, and a strongly divisible module of weight $r$ is, roughly speaking, a lattice structure in $\cD(\lambda,\vL)$ (see Definition~\ref{defi: sdm in D} for its definition).

We first introduce the notion of pseudo-strongly divisible modules. Let $S'_\cO$ (resp. $S'_E$) denote the ring $S_\cO e_j$ (resp. $S_E e_j$) for a fixed (and thus every) $j\in\cJ$. Let $\cD'$ be a free $S'_E$-module of rank $2$, with basis denoted by $E_1'$ and $E_2'$. Let $(r,r')$ be a pair of positive integers satisfying $r' \leq r < p-1$. For each $k\in [0,r']\cup\{\infty\}$, the pseudo-strongly divisible modules in $\mathbf{Case}~(k)$ consist of a quadruple
$$
(\fM',\, \Fil^{r;r'}_k \fM',\, \phi': \fM' \to \cD',\, N': \fM' \to \cD'),
$$
where: 
\begin{itemize}[leftmargin=*] 
\item $\fM'$ is the free $S'_\cO$-submodule of $\cD'$ generated by $E_1'$ and $E_2'$; 
\item $\Fil^{r;r'}_k\fM'$ is an $S'_\cO$-submodule of $\cD'$; 
\item $\phi'$ is a continuous $\phi$-semilinear map on $\fM'$; 
\item $N'$ is a continuous $\cO$-linear derivation. 
\end{itemize} 
Our definition of pseudo-strongly divisible modules is explicit (see Definition~\ref{defi: pseudo-sdm} for the precise statement) and is inspired by the expected structure of strongly divisible modules. The positive integer $r$ indicates the weight of the strongly divisible module, and $r'$ corresponds to a Hodge--Tate weight. The filtration $\Fil^{r;r'}_k\fM'$ is generated by a set of explicit elements in $\cD'$ depending on $k$. In particular, when $k=\infty$, we have 
\begin{equation}\label{eq: filtration of pseudo when k=infty} 
\Fil^{r;r'}_\infty\fM' = S'_\cO(v^r E_1', v^{r-r'} E_2') + \Fil^p S'_\cO\fM', 
\end{equation} 
which is intended to arise as the isotypic component at $j$ of a strongly divisible module when $\fL_j = \infty$. If the triple $(r,r',k)$ is fixed, then these pseudo-strongly divisible modules are parameterized by the data $(\Lambda,\Theta,\Omega,x) \in (E^\times)^3 \times E$ and $\Delta(T),\tilde{\Delta}(T) \in (E(T)\cap (E+\cO[[T+1]]))$. These data $(\Lambda,\Theta,\Omega,x)$ and $\Delta(\gamma-1),\tilde{\Delta}(\gamma-1) \in S'_\cO$, where $\gamma:=\frac{(u-p)^p}{p}\in S'_\cO$, naturally appears in the coefficients of the definition.

We wish that the pseudo-strongly divisible modules become the isotypic components of the strongly divisible modules. To achieve this, it suffices to impose the following conditions: 
\begin{enumerate}[leftmargin=*]
\item $\phi'(\fM')\subseteq\fM'$ and $N'(\fM')\subseteq\fM'$;
\item $\Fil^{r;r'}_k\fM'$ is a $S'_\cO$-submodule of $\fM'$;
\item $\phi'_r(\Fil^{r;r'}_k\fM')$ is contained in $\fM'$ and generates $\fM'$ over $S'_\cO$, where $\phi'_r:=\tfrac{1}{p^r}\phi'$.
\end{enumerate}
Condition~(i) can be readily described in terms of inequalities involving the quadruple $(v_p(\Lambda), v_p(\Theta), v_p(\Omega), v_p(x))$ (see Lemma~\ref{lemm: pseudo, condition for phi,N stable}). We also derive inequalities on the quadruple that are nearly equivalent to condition~(ii) for each $k$ (see Theorem~\ref{theo: filtration of pseudo}). Moreover, we show that for each $k$ $\Fil^{r;r'}_k\fM'$ can be written as
$$\Fil^{r;r'}_k\fM'=S'_\cO(G'_1,G'_2)+\Fil^pS'_\cO\fM',$$ 
where $G'_1,G'_2$ are explicitly defined for each $k$ (see \eqref{eq: filtration when r' leq r}).  For example, the filtration described in~\eqref{eq: filtration of pseudo when k=infty} corresponds to the case $k=\infty$. To check the condition (iii),  it is necessary for $\phi'$ to satisfy
\begin{equation}\label{eq: strong divisibility, intro}
\fM'=S'_\cO\left(\phi'_r(G'_1),\phi'_r(G'_2)\right)
\end{equation} 
for each $\mathbf{Case}~(k)$, and so we need to check the square matrix appeared in the following matrix multiplication is invertible:
$$(\phi'_r(G'_1),\phi'_r(G'_2))=(E'_1,E'_2)\cdot
\begin{bmatrix}
  a_k & b_k \\
  c_k & d_k 
\end{bmatrix}.$$
We will consider the following two possibilities $a_k\in(S_\cO')^\times$ or $b_k\in(S_\cO')^\times$, and it turns out that these two possibilities are mutually exclusive (depending the choice of $\Delta(T)$). Under the setting of $\mathbf{Case}~(k)$, we say that the pseudo-strongly divisible module $\fM'$ \emph{satisfies the strong divisibility of type $\mathbf{Case}_\phi(k)$} (resp. \emph{of type $\mathbf{Case}_\phi(k+\tfrac{1}{2})$}) if \eqref{eq: strong divisibility, intro} holds and $b_k\in(S_\cO')^\times$ (resp. $a_k\in(S_\cO')^\times$). Hence, we consider $\mathbf{Case}_\phi(k')$ for $k'\in([0,r'+\tfrac{1}{2}]\cap\tfrac{1}{2}\Z)\cup\{\infty,\infty+\tfrac{1}{2}\}$. But it turns out that $\mathbf{Case}_\phi(0)$ and $\mathbf{Case}_\phi(\infty)$ are empty, and so we consider $\mathbf{Case}_\phi(k')$ only for $k'\in([\tfrac{1}{2},r'+\tfrac{1}{2}]\cap\tfrac{1}{2}\Z)\cup\{\infty\}$, after renaming $\mathbf{Case}_\phi(\infty+\tfrac{1}{2})$ to $\mathbf{Case}_\phi(\infty)$. For $k'\in([\tfrac{1}{2},r'+\tfrac{1}{2}]\cap\tfrac{1}{2}\Z)\cup\{\infty\}$, we give enough conditions on the quadruple $(v_p(\Lambda),v_p(\Theta),v_p(\Omega),v_p(x))$ to guarantee that $\fM'$ satisfies the strong divisibility of type $\mathbf{Case}_\phi(k')$, and these conditions involve the equations as well as the inequalities on the quadruple. All the equations and inequalities obtained for (i), (ii), and (iii) are summarized in Table~\ref{tab:my_label}. Moreover, for each $k'$ we determine $\Delta(T)$ to ensure the strong divisibility of type $\mathbf{Case}_\phi(k')$, and it is worth noting that $\tilde{\Delta}(T)$ is not relevant to the strong divisibility (see Lemma~\ref{lemma formula for phi_r(F_kk) and phi_r(F_k0)}). Note that \S\ref{sec: pseudo strongly divisible modules} is devoted to deriving these conditions, and that we heavily use the frameworks developed in \S\ref{sec: Frame works} for the proofs.

The next step is applying the conditions obtained from the step in the previous paragraph to construct the strongly divisible modules. We begin by determining the form of the generators of strongly divisible modules in $\cD(\lambda,\vL)$, which we expect to be as follow: if we write $\eta_1,\eta_2$ for the basis of $D(\lambda,\vL)$ as in Example~\ref{exam: admissible filtered modules} then
\begin{equation}\label{eq: intro, generators}
    E_1=\vth_1\bigg(\eta_1+\frac{\varphi(\vrho)+\vDel^\varphi(\gamma-1)}{p}\eta_2\bigg)\qquad\&\qquad E_2=\vth_2\eta_2
\end{equation}
for some $\vth_1,\vth_2\in (E^\times)^f$, $\vrho \in E^f$, and $\vDel(T)\in (E(T)\cap(E+\cO[[T+1]]))^f$, where $\gamma:=\frac{(u-p)^p}{p}\in S_\cO$. (Here, by $\vDel^\varphi(T)$ we mean $(\Delta^{(j-1)}(T))_{j\in\cJ}$ if we write $\vDel(T)=(\Delta^{(j)}(T))_{j\in\cJ}$.) We further require that $\vrho$ satisfies
$\varrho_{j}=\fL_j$ if $\fL_j\neq \infty$, a condition that also arises from the definition of strongly divisible modules (see Remark~\ref{rema: why fL_j=varho_j}). We will write $\fM^\diamond$ for the $S_\cO$-submodule of $\cD(\lambda,\vL)$ generated by $E_1,E_2$. Moreover, we further have $\Fil^r\fM^\diamond:=\fM^\diamond\cap \Fil^r\cD$ as well as $\phi^\diamond,N^\diamond$ on $\fM^\diamond$ induced from $\phi,N$ on $\cD(\lambda,\vL)$, respectively, by restriction. This leads to the quadruple
\begin{equation}\label{eq: pseudo quadruple, intro}
(\fM^\diamond,\,\Fil^r\fM^{\diamond},\,\phi^\diamond:\fM^\diamond\rightarrow\cD(\lambda,\vL),\,N^\diamond:\fM^\diamond\rightarrow\cD(\lambda,\vL)),
\end{equation} which is parameterized by the data $(\vLam,\vTh,\vx)\in ((E^\times)^{f})^{2}\times E^f$ along with $\vDel(T)$. These parameters satisfy the relations
\begin{equation}\label{eq: connection between x and L, intro}
\vLam=\frac{\lambda\vth_1}{\varphi^{-1}(\vth_1)},\quad \vTh=\frac{\vth_2}{\vth_1},\quad \mbox{and}\quad
\vx=\frac{p\vrho-\varphi(\vrho)-\vDel^\varphi(-1)}{p}.
\end{equation}
For each $j \in \cJ$, we consider the isotypic component of $\fM^\diamond$ at $j\in\cJ$
$$
\big(\fM^\diamond e_j, \Fil^r \fM^\diamond e_j, \phi^{\diamond,(j)}: \fM^\diamond e_j \to \cD(\lambda,\vL) e_{j+1}, N^{\diamond,(j)}: \fM^\diamond e_j \to \cD(\lambda,\vL) e_j\big),
$$
which are parameterized by $(\Lambda_j,\Theta_j,\Theta_{j+1},x_j)\in (E^{\times})^3\times E$ and $(\Delta^{(j)}(T),\Delta^{(j-1)}(T))$. If we define the natural isomorphism $$\psi^{(j)}:\cD'\rightarrow \cD(\lambda,\vL)e_j\,:\,\,E'_1\mapsto E_1e_j\,\,\&\,\,E'_2\mapsto E_2e_j,$$ then by identifying $(r',\Lambda,\Theta,\Omega,x,\Delta,\tilde{\Delta})$ with $(r_j,\Lambda_j,\Theta_j,\Theta_{j+1},x_j,\Delta^{(j)},\Delta^{(j-1)})$, we have the following commutative diagram
$$
\xymatrix{
\cD'\ar@{->}[d]_{\psi^{(j)}}^{\simeq}&&\fM'\ar@{->}[ll]_{N'}\ar@{->}[rr]^{\phi'}\ar@{->}[d]_{\psi^{(j)}}^{\simeq} &&\cD'\ar@{->}[d]_{\psi^{(j+1)}}^{\simeq} &&\Fil^{r;r'}_k\fM'\ar@{->}[ll]_{\phi_r'}\ar@{->}[d]_{\psi^{(j)}}^{\simeq}\\
\cD(\lambda,\vL) e_{j}&&\fM^{\diamond} e_j\ar@{->}[ll]^{N^{\diamond,(j)}}\ar@{->}[rr]_{\phi^{\diamond,(j)}}&&\cD(\lambda,\vL) e_{j+1}&&\Fil^r\fM^{\diamond}e_{j}\ar@{->}[ll]^{\phi_r^{\diamond,(j)}}
}
$$
where $\fM'$ is a pseudo-strongly divisible module in $\mathbf{Case}~(k)$, under the assumption that $\Fil^{r;r'}_k\fM'$ is a $S'_\cO$-submodule of $\fM$ together with a parallel assumption on $\Fil^r\fM^{\diamond}e_{j}$ (see Proposition~\ref{prop: typical example of pseudo-SDM}).

We apply the properties of the first row of the diagram above to the second row, identifying $(r',\Lambda,\Theta,\Omega,x,\Delta,\tilde\Delta,k')$ with $(r_j,\Lambda_j,\Theta_j,\Theta_{j+1},x_j,\Delta^{(j)},\Delta^{(j-1)},k'_j)$ for each $j\in\cJ$. The resulting conditions on $(v_p(\vLam),v_p(\vTh),v_p(\vx))$ are summarized in \textbf{A}, \textbf{B}, and \eqref{eq: definition of Lambda, intro}. In other words, if $(v_p(\vTh),v_p(\vx))$ satisfies the corresponding equations and inequalities for some $\vk' = (k'_j)_{j\in\cJ}$, and if we choose $\vLam\in E^f$ according to \eqref{eq: definition of Lambda, intro}, then the quadruple in \eqref{eq: pseudo quadruple, intro} induced by \eqref{eq: intro, generators}, where $\vth_1,\vth_2,\vrho$ are determined by \eqref{eq: connection between x and L, intro} and $\vDel(T)$ is determined by $\vk'$, forms a strongly divisible module. Note that \S\ref{sec: main results} is devoted to deriving the main results.

As stated in Conjecture~\ref{conj: exhaust}, we believe that our method is quite general. To support this conjecture, we verify that it recovers known results in the case $f=1$ (see \cite{BM, GP, LP}) and provide illustrative examples for $\vr=(2,2)$ and $\vr=(1,5)$ when $f=2$ (see \S\ref{sec: examples}).

\subsection{Notation}
Much of the notation introduced in this subsection will also be introduced in the text, but we try to collect various definitions here for ease of reading.

As usual, we write $\N$ (resp. $\Z$, resp. $\Q$, resp. $\R$) for the set of natural numbers (resp. the ring of integers, resp. the field of rational numbers, resp. the field of real numbers). We further set $\N_0:=\N\cup\{0\}$.

\subsubsection{Galois groups}
Let $p$ be a prime, and $f$ be a positive integer. We let $\Q_{p^f}$ be the finite unramified extension of $\Q_p$ of degree $f$ with ring of integers $\Z_{p^f}$(=$W(\F_{p^f})$, the ring of Witt vectors over $\F_{p^f}$). We write $G_{\Q_{p^f}}$ for the absolute Galois group of $\Q_{p^f}$, and $I_{\Q_{p^f}}$ for the inertial subgroup of $G_{\Q_{p^f}}$. The cyclotomic character is denoted by $\varepsilon:G_{\Q_p}\rightarrow \Z_p^{\times}$, and the fundamental character of neveau $f$ is denoted by $\omega_f:G_{\Q_{p^f}}\rightarrow \F_{p^f}^\times$. We will naturally identify $\Gal(\Q_{p^f}/\Q_p)$ with $\Gal(\F_{p^f}/\F_p)$ via the natural surjection $\Z_{p^f}\twoheadrightarrow \F_{p^f}$, and will consider the arithmetic Frobenius $\varphi\in\Gal(\F_{p^f}/\F_p)$ as an element in $\Gal(\Q_{p^f}/\Q_p)$ via the identification. Since $\varphi$ is the generator of the Galois group, the association $\varphi^{-1}\leftrightarrow 1$ gives rise to an isomorphism $\Gal(\Q_{p^f}/\Q_p)\cong \Z/f\Z$, and we will fix this isomorphism throughout this paper.

\subsubsection{Coefficients}
Let $E$ be a finite extension of $\Q_{p^f}$ with ring of integers $\cO$, maximal ideal $\fm$, uniformizer $\varpi$, and residue field $\F$. We will always consider $\Q_{p^f}$ as a subfield of $E$, and so if $a\in\Q_{p^f}$ then $a\in E$. We may naturally consider the characters $\varepsilon$ and $\omega_f$ are defined over $\cO$ and $\F$, respectively. Let $\cJ:=\Hom_{\Q_p}(\Q_{p^f},\overline{\Q}_p)$, and we will naturally identify the following sets $$\cJ=\Hom_{\Q_p}(\Q_{p^f},E)=\Hom_{\Z_p}(\Z_{p^f},\cO)=\Hom_{\F_p}(\F_{p^f},\F)=\Gal(\F_{p^f}/\F_p)=\Z/f\Z.$$

Fix a ring isomorphism
\begin{equation}\label{eq: ring coefficient}
\Z_{p^f}\otimes_{\Z_p}\cO\cong \bigoplus_{j\in\cJ}\cO: a\otimes b \mapsto \left(\varphi^{-j}(a)\cdot b\right)_{j\in\cJ},
\end{equation}
and write $e_j\in\Z_{p^f}\otimes_{\Z_p}\cO$ for the idempotent element corresponding to the element $(b_{j'})_{j'\in\cJ}\in \bigoplus_{j\in\cJ}\cO$ with $b_{j'}=1$ if $j'=j$ and $b_{j'}=0$ otherwise. The action of $\varphi$ on $\Z_{p^f}\otimes_{\Z_p}\cO$ (acting naturally on $\Z_{p^f}$ and trivially on $\cO$) sends $e_j$ to $e_{j+1}$ due to our choice of the isomorphism above. We will often consider $\varphi$ as a map from $\cO^f$ onto itself sending $(a_j)_{j\in\cJ}$ to $(a_{j-1})_{j\in\cJ}$. Note that these idempotent elements satisfy
    $$e_j\cdot e_{j'}=
    \begin{cases}
    e_j & \mbox{if }j=j';\\
    0 & \mbox{otherwise}
    \end{cases}
    \quad\mbox{ and }\quad\sum_{j\in\cJ}e_j=1.
    $$
Moreover, they also satisfy that for each $j\in\cJ$ and for all $a\otimes b\in \Z_{p^f}\otimes_{\Z_p}\cO$
    \begin{equation}\label{eq: field coefficient}
    (a\otimes b)e_j=\left(1\otimes \varphi^{-j}(a)\cdot b\right)e_j.
    \end{equation}
    The fixed isomorphism above also induces isomorphisms $$\Q_{p^f}\otimes_{\Q_p}E\cong \bigoplus_{j\in\cJ}E\quad\mbox{ and }\quad\F_{p^f}\otimes_{\F_p}\F\cong \bigoplus_{j\in\cJ}\F,$$ and by abuse of the notation we also write $e_j$ for the corresponding idempotents in each case. We further note that these $e_j$ in $\Q_{p^f}\otimes_{\Q_p}E$ (resp. in $\F_{p^f}\otimes_{\F_p}\F$) also satisfy the same properties as described on $\Z_{p^f}\otimes_{\Z_p}\cO$.

Let $v_p$ be the $p$-adic valuation on $\overline{\Q}_p$ normalized as $v_p(p)=1$. If $\vec{y}=(y_j)_{j\in\cJ}\in\bigoplus_{j\in\cJ}E$, then we write $v_p(\vec{y})$ for $(v_p(y_j))_{j\in\cJ}$.

\subsubsection{Elements of the ring $S$}
Let $S$ be the $p$-adic completion of the ring $\Z_{p^f}[\frac{u^i}{i!}\mid i\in\N]$, and fix $v:=u-p\in S$. We also let $\phi:S\rightarrow S$ be a continuous $\varphi$-semilinear map with $\phi(u)=u^{p}$. We let $c:=\phi(v)/p\in S^\times$ and $\gamma:=\frac{(u-p)^p}{p}\in S$. It is easy to see that $c-(\gamma-1)\equiv pu^{p-1}\pmod{p^2S}$ and $\phi(\gamma)\in p^{p-1}S$. These two elements naturally appear in our strongly divisible modules. We further let $S_\cO:=S\otimes_{\Z_p}\cO$, and often identify $$S_\cO\cong\bigoplus_{j\in\cJ}S_{\cO} e_j.$$
If we write $S'$ for the $p$-adic completion of $\Z_p[\frac{u^i}{i!}\mid i\in\N]$, then we naturally identify $S'_\cO:=S'\otimes_{\Z_p}\cO$ with $S_\cO e_j$ for all $j\in\cJ$. If $\fM$ is an $S_\cO$-module, then for each $j\in\cJ$ we will naturally consider the isotypic component $\fM e_j$ at $j\in\cJ$ as a module of $S'_\cO$ via the identification. Moreover, if $\vDel(u)\in S_\cO$, then we often consider it as an element of $\bigoplus_{j\in\cJ}S'_{\cO}$ via the identification above, and write it as $\vDel(u)=(\Delta^{(j)}(u))_{j\in\cJ}$. By $\phi(\vDel(u))$ we mean the usual action of $\phi$ on $S_\cO$, and by $\vDel^\varphi(u)$ we mean the shifting, i.e., $\vDel^\varphi(u):=(\Delta^{(j-1)}(u))_{j\in\cJ}$. Similarly, we write $S'_E$ for $S'\otimes_{\Z_p}E$ and we have natural isomorphisms $S'_E\cong S_Ee_j$ for all $j\in\cJ$.

Let $R[[T]]$ be the power series ring over a commutative ring $R$ with one variable $T$, and let $n$ be positive integer. We write $R[T]^{(n)}$ for the subgroup of $R[T]$ consisting of the polynomials of degree less than equal to $n$. For $F(T)=\sum_{k\geq0}a_kT^k\in R[[T]]$, we set $[F(T)]_T^{(n)}:=\sum_{k=0}^{n}a_kT^k$. For instance, we set $f_n(T):=0$ if $n=1$, and
    \begin{equation}\label{eq: definition of f_r}
    f_{n}(T):=\bigg[\frac{1}{T}\log(1+T)\bigg]_{T}^{(n-2)}=\sum_{k=0}^{n-2}\frac{(-1)^k}{k+1}T^k\in\Z_{(p)}[T]
    \end{equation}
if $n\geq 2$. We also write $\dot{F}(T)$ for the formal derivative of $F(T)=\sum_{k\geq0}a_kT^k$, i.e., $\dot{F}(T)=\sum_{k\geq 1}ka_kT^{k-1}\in R[[T]]$. Note that if $F(T)\in\cO[[T+1]]$, then $F(\gamma-1)$ naturally belongs to $S'_\cO$.

\subsubsection{Miscellaneous}\label{subsubsec: Miscellaneous}
Let $R$ be a commutative ring. We write $\mathrm{M}_{m\times n}(R)$ for the ring of matrices of size $(m\times n)$ over $R$, and $I_n\in \mathrm{M}_{n\times n}(R)$ (resp. $0_{m,n}\in\mathrm{M}_{m\times n}(R)$) for the identity matrix (resp. for the zero matrix). If $M$ is a free module over $R$ with a basis $\underline{E}:=(E_1,E_2)$, and $A\in M$ then by $[A]_{E_i}$ we mean the coefficients of $E_i$ in $A$. If $N$ is an $R$-module generated by $\underline{F}:=(F_1,F_2)$, and $f:N\rightarrow M$ is a (semi)-linear map, then $f(F_j)=A_{1,j}E_1+A_{2,j}E_2$ for all $j=1,2$ if and only if
$$\mathrm{Mat}_{\underline{E},\underline{F}}(f)=
\begin{bmatrix}
  A_{1,1} & A_{1,2} \\
  A_{2,1} & A_{2,2}
\end{bmatrix}.$$
Furthermore, if $N$ is a submodule of $M$, then $[F_j]_{E_i}=B_{i,j}\in R$ for all $i,j\in\{1,2\}$ if and only if
$$\mathrm{Mat}_{\underline{E},\underline{F}}(N)=
\begin{bmatrix}
  B_{1,1} & B_{1,2} \\
  B_{2,1} & B_{2,2}
\end{bmatrix}.$$

If $M$ is a non-empty set, we often write $M^{\cJ}$ for $M^f$. If $\vec{A}\in M^\cJ$, by $A_j$ we often mean the entry of $\vec{R}$ at the $j$-th position, i.e., $\vec{A}=(A_0,\cdots,A_{f-1})$. If $\cJ_0\subset \cJ$ then we will naturally identify $M^\cJ$ with $M^{\cJ_0}\times M^{\cJ\setminus\cJ_0}$ respecting the coordinates. For instance, we write $\hat{\R}$ for $\R\cup\{\infty\}$ with order topology, and we will naturally identify $\hat{\R}^\cJ=\hat{\R}^{\cJ_0}\times \hat{\R}^{\cJ\setminus \cJ_0}$.

Finally, for a positive integer $n\in\N$, we write $H_n\in\Q$ for the harmonic number $\sum_{i=1}^n\frac{1}{i}$.

\subsection{Acknowledgment}
The authors sincerely thank Christophe Breuil and Florian Herzig for numerous helpful comments and suggestions. The authors also thank Eknath Ghate, Yongquan Hu, Wansu Kim, Stefano Morra, Zicheng Qian, and David Savitt for a plenty of helpful discussions.

This work was supported by Samsung Science and Technology Foundation under Project Number SSTF-BA2001-02.

\smallskip

\section{Preliminary}\label{sec: preliminary}
In this section, we quickly review (integral) $p$-adic Hodge theory and fix some notation. We note that the materials in this section are already well-known or easy generalization of known results. Moreover, the exposition in this section is quite close to \cite{GP,Par}, but we decided to have a quick review of integral $p$-adic Hodge theory for the completeness and for the ease of reading.

Throughout this section, we write $E$ for a finite extension of $\Q_p$ containing $\Q_{p^f}$ with the ring of integers $\cO$, the maximal ideal $\fm$, a uniformizer $\varpi$, and the residue field $\F$. We also let $K$ be a unramified extension of $\Q_p$ of degree $f$, i.e., $K=\Q_{p^f}$. All the results in this section still hold for arbitrary finite extension of $\Q_p$, but for simpler notation, we restrict our attention to the unramified extension of $\Q_p$, which is our context. We write $\cO_K$, $\fm_K$, and $k$ for the ring of integers, the maximal ideal, and the residue field, respectively, of $K$.

\subsection{Filtered $(\phi,N)$-modules}
In this subsection, we introduce weakly admissible filtered $(\phi,N)$-modules and their connection to semi-stable representations.

We fix a prime $p$, so that we fix the inclusion $\Bst\hookrightarrow\Bdr$. (For detail, see \cite{Fon94}.)
\begin{defi}
A \emph{filtered $(\phi,N,K,E)$-module} (or simply, a \emph{filtered $(\phi,N)$-module}) is a free $K\otimes_{\Q_p}E$-module $D$ of finite rank together with a triple $\left(\phi,N,\{\Fil^iD\}_{i\in\Z}\right)$ where
\begin{itemize}[leftmargin=*]
\item the \emph{Frobenius map} $\phi:D\rightarrow D$ is a $\varphi$-semilinear and $E$-linear automorphism;
\item the \emph{monodromy operator} $N:D\rightarrow D$ is a (nilpotent) $K\otimes_{\Q_p}E$-linear endomorphism such that $N\phi=p\phi N$;
\item the \emph{Hodge filtration} $\{\Fil^{i}D\}_{i\in\Z}$ is a decreasing filtration on $D$ such that a $K\otimes_{\Q_p}E$-submodule $\Fil^{i}D$ is
$D$ if $i\ll 0$ and $0$ if $i\gg 0$.
\end{itemize}
The morphisms of filtered $(\phi,N)$-modules are $K\otimes_{\Q_p} E$-module homomorphisms that commute with $\phi$ and $N$ and that preserve the filtration. We also say that a filtered $(\phi,N)$-module $D$ is \emph{weakly admissible} if it is in the sense of \cite{BM}.
\end{defi}

Let $V$ be a finite-dimensional $E$-vector space equipped with continuous $E$-linear action of $G_{K}$, and define $$\Dst(V):=(\Bst\otimes_{\Q_p}V)^{G_{K}}.$$  Then $\dim_{K}\Dst(V)\leq\dim_{\Q_p}V$.  If the equality holds, then we say that $V$ is \emph{semi-stable}; in that case $\Dst(V)$ inherits from $\Bst$ the structure of a weakly admissible filtered $(\phi,N)$-module. We say that $V$ is \emph{crystalline} if $V$ is semi-stable and the monodromy operator $N$ on $\Dst(V)$ is~$0$.
\begin{theo}[\cite{CF}]
The functor $\Dst$ provides an equivalence between the category of semi-stable $E$-representations of $G_{K}$ and the category of weakly admissible filtered $(\phi,N,K,E)$-modules.
\end{theo}

The functor $\Dst$ does depend on the choice of an element $\pi$ in $\overline{\Q}_p$ with $v_p(\pi)>0$. More precisely, the filtration of the filtered $(\phi,N)$-module $\Dst(V)$ depend on the choice of the embedding $\Bst\hookrightarrow\Bdr$ which is determined by $\pi$. To avoid this issue, we choose $p\in\Q_p$. Note that the functor $\Dst$ restricted to the category of crystalline representations does not depend on~$\pi$.

If we write $V^*$ for the dual representation of $V$, then $V$ is semi-stable (resp. crystalline) if and only if so is $V^*$. If we denote $\mathrm{D_{st}^{*}}(V):=\Dst(V^*)$, then the functor $\mathrm{D_{st}^{*}}$ gives rise to an anti-equivalence between the category of semi-stable $E$-representations of $G_{K}$ and the category of weakly admissible filtered $(\phi,N,K,E)$-modules. The quasi-inverse to $\mathrm{D^{*}_{st}}$ is given by $$\mathrm{V^{*}_{st}}(D):=\Hom_{\phi,N}(D,\Bst)\cap\Hom_{\Fil^{*}}(D, \Bst).$$ (that is, the homomorphisms of $K\otimes_{\Q_p} E$-modules that commute with $\phi$ and $N$ and that preserve the filtration.)
$\mathrm{V^{*}_{st}}(D)$ inherits an $E$-module structure from the $E$-module structure on $D$ and an action of $G_{K}$ from the action of $G_{K}$ on $\Bst$.

Let $D$ be a filtered $(\phi,N,K,E)$-module of rank $n$. Then the isomorphism $\Q_{p^f}\otimes_{\Q_p}E\cong \bigoplus_{j\in\cJ}E$ induces an isomorphism $$D\cong\bigoplus_{j\in\cJ}D^{(j)}:\, d\mapsto (de_j)_{j\in\cJ}$$ where $D^{(j)}:=De_j$ and $e_j$ is the idempotent elements described in the notation section. For each $j\in\cJ$ the isotypic component $D^{(j)}$ at $j$ is an $n$-dimensional vector space over $E$ by \eqref{eq: field coefficient}, and the Frobenius map $\phi$ (resp. the monodromy operator $N$) on $D$ inherits an $E$-linear map $$\phi^{(j)}:D^{(j)}\rightarrow D^{(j+1)}\qquad(\mbox{resp.}\quad N^{(j)}:D^{(j)}\rightarrow D^{(j)})$$ for all $j\in\cJ$. Moreover, for each $j\in\cJ$ $\Fil^iD^{(j)}:=(\Fil^iD)e_j$ is an $E$-linear subspace of $D^{(j)}$, and we may write $\Fil^iD=\bigoplus_{j\in\cJ}\Fil^iD^{(j)}$. The \emph{$j$-labeled Hodge--Tate weights} of the filtered $(\phi,N)$-module $D$ are defined to be the integers $h$ such that $\Fil^{h}D^{(j)} \neq \Fil^{h+1}D^{(j)}$, each counted with multiplicity $\dim_{E}(\Fil^{h}D^{(j)}/\Fil^{h+1}D^{(j)})$. We say that $h$ is a \emph{Hodge--Tate weight of $D$} if it is a $j$-labeled Hodge--Tate weight for some $j\in\cJ$. When a filtered $(\phi,N)$-module $D$ of rank $n$ has $n$ distinct $j$-labeled Hodge--Tate weights for all $j\in\cJ$, we say that $D$ is \emph{regular} (or that it has regular Hodge--Tate weights). We say that a filtered $(\phi,N)$-module is \emph{positive} if the lowest $j$-labeled Hodge--Tate weight is greater than or equal to $0$ for each $j\in\cJ$. We also say a filtered $(\phi,N)$-module is \emph{parallel} if the $j$-labeled Hodge--Tate weights does not depend on $j$.

If $V$ is semi-stable, then by \emph{$j$-labeled Hodge--Tate weights of $V$} (resp. \emph{Hodge--Tate weights of $V$}), we mean those of $\mathrm{D_{st}^{*}}(V)$, and we write $\mathrm{HT}(V)_j$ for the set of $j$-labeled Hodge--Tate weights of $V$.  Our normalization implies that the cyclotomic character $\varepsilon:G_{\Q_p}\rightarrow E^\times$ has Hodge–Tate weight $1$. After a twist of a suitable Lubin--Tate character, we are free to assume that the lowest $j$-labeled Hodge--Tate weight is $0$ for all $j\in\cJ$. Similarly, we say that $V$ is \emph{regular} (resp. \emph{positive}, resp. \emph{parallel}) if $\mathrm{D_{st}^{*}}(V)$ is. In particular, if $V$ is parallel then we have $\mathrm{HT}(V)_j=\mathrm{HT}(V)_{j'}$ for all $j,j'\in\cJ$.

The following example comes from \cite{Dous}.
\begin{exam}\label{exam: admissible filtered modules}
Let $D$ be a free $K\otimes_{\Q_p}E$-module of rank $2$ with basis $\underline{\eta}:=(\eta_1,\eta_2)$, i.e., $D=K\otimes_{\Q_p}E(\eta_1,\eta_2)$. We endow $D$ with a filtered $(\phi,N,K,E)$-module structure, denoted by $D(\lambda,\{[\xi_j:\nu_j]\}_{j\in\cJ})$:
\begin{itemize}[leftmargin=*]
\item $\Mat_{\underline{\eta}}(N)=
\begin{bmatrix}
0 &0\\
1 & 0
\end{bmatrix}$ and
$\Mat_{\underline{\eta}}(\phi)=
\begin{bmatrix}
p\lambda &0\\
0 & \lambda
\end{bmatrix}$ for $\lambda\in E$;
\item $\Fil^iD=\bigoplus_{j\in\cJ}\Fil^iD^{(j)}$ where for each $j\in\cJ$
$$\Fil^iD^{(j)}=
\begin{cases}
D^{(j)} & \mbox{if }i\leq 0;\\
E(\xi_j\eta_1+\nu_j\eta_2)e_j & \mbox{if }0<i\leq r_j;\\
0 & \mbox{if }r_j<i
\end{cases}
$$
for $[\xi_j:\nu_j]\in\mathbb{P}^1(E)$.
\end{itemize}
It turns out that $D(\lambda,\{[\xi_j:\nu_j]\}_{j\in\cJ})$ is weakly admissible if and only if
\begin{equation}\label{eq: condition for lambda, weakly admissible}
v_p(\lambda)=\frac{1}{2f}\sum_{j\in\cJ}(r_j-1)\quad\mbox{ and }\quad v_p(\lambda)\geq\frac{1}{f}\sum_{j\in I_D}r_j,
\end{equation}
where we set $$I_D:=\{j\in\cJ\mid\xi_j=0\}.$$ Moreover, it is also known that $D(\lambda,\{[\xi_j:\nu_j]\}_{j\in\cJ})$ exhaust all the $2$-dimensional semi-stable non-crystalline representations $V$ of $G_{K}$ with $\mathrm{HT}(V)_j=\{0,r_j\}$ for all $j\in\cJ$.

We often write $\fL_j:=[\xi_j:\nu_j]=\frac{\nu_j}{\xi_j}\in\mathbb{P}^1(E)$ and $\vL:=(\fL_j)_{j\in\cJ}$. In particular, we write $\fL_j=\infty$ if $\xi_j=0$. Note that the data $\left(\lambda,(\fL_j)_{j\in\cJ}\right)$ determines a weakly admissible filtered $(\phi,N)$-module $D(\lambda,(\fL_j)_{j\in\cJ})$ if $\lambda$ satisfies \eqref{eq: condition for lambda, weakly admissible}.
\end{exam}

By easy linear algebra, one can describe the isomorphism classes of $D(\lambda,(\fL_j)_{j\in\cJ})$.

\begin{lemm}\label{lemm: isomorphism class of filtered phi N modules}
$D(\lambda,(\fL_j)_{j\in\cJ})$ is isomorphic to $D(\lambda',(\fL'_j)_{j\in\cJ})$ if and only if $\lambda^f=(\lambda')^f$ and $\fL_j=\fL'_j$ for all $j\in\cJ$.
\end{lemm}

\begin{proof}
Let $f:D(\lambda,(\fL_j)_{j\in\cJ})\rightarrow D(\lambda',(\fL'_j)_{j\in\cJ})$ be an isomorphism of filtered $(\phi,N)$-modules. From the commutativity with $\phi,N$, for $i=1,2$ and $j\in\cJ$ we have $f(\eta_ie_j)=a_{j}\eta_ie_j$ for some $a_{j}\in E^\times$ with $a_{j+1}\lambda=a_{j}\lambda'$. Since $\cJ$ is cyclic, we conclude $\lambda^f=(\lambda')^f$. Moreover, we further conclude that $\fL_j=\fL_j'$ for all $j\in\cJ$, as it should preserve the filtration. For the converse, first observe $\lambda'=\zeta\cdot\lambda$ for some $f$-th root of unity $\zeta\in E$. If we let $f:D(\lambda,(\fL_j)_{j\in\cJ})\rightarrow D(\lambda',(\fL'_j)_{j\in\cJ})$ be a $K\otimes_{\Q_p}E$-linear map given by $f(\eta_i e_j)=\zeta^{j}\eta_ie_j$ for all $j\in\cJ$ and $i=1,2$, then it is tedious to see that it is an isomorphism between the two filtered $(\phi,N)$-modules.
\end{proof}

\subsection{Strongly divisible modules}\label{subsec: strongly divisible modules, preliminary}
In this subsection, we introduce strongly divisible modules and their connection to Galois stable lattices of semi-stable representations.

Let $W(k)$ be the ring of Witt vectors over $k$ so that $K=W(k)[\frac{1}{p}]$. Let $v:=u-p\in W(k)[u]$, and let $S$ be the $p$-adic completion of $W(k)[\frac{u^{i}}{i!}]_{i\in\N}$. We endow $S$ with the following structure:
\begin{itemize}[leftmargin=*]
\item a continuous $\varphi$-semilinear map $\phi:S\rightarrow S$ with $\phi(u)=u^{p}$, called the \emph{Frobenius map} on $S$;
\item a continuous $W(k)$-linear derivation $N:S\rightarrow S$ with $N(u^{i}/i!)=-iu^{i}/i!$, called the \emph{monodromy map} on $S$;
\item a decreasing filtration $\{\Fil^{i}S\}_{i\in\N_{0}}$ where $\Fil^{i}S$ is the $p$-adic completion of $\sum_{j\geq i}\frac{v^{j}}{j!}S$.
\end{itemize}
Note that $N\phi=p\phi N$ and $\phi(\Fil^{i}S)\subset p^{i}S$ for $0\leq i\leq p-1$.

Let $S_{\cO}:=S\otimes_{\Z_p}\cO$ and $S_{E}:=S_{\cO}\otimes_{\Z_p}\Q_p$, and extend the definitions of $\Fil^*$, $\phi$, and $N$ to $S_{\cO}$ and $S_{E}$ $\cO$-linearly and $E$-linearly, respectively. Let $\mathcal{MF}(\phi,N,K,E)$ be the category whose objects are finite free $S_{E}$-modules $\cD$ with
\begin{itemize}[leftmargin=*]
\item a $\phi$-semilinear and $E$-linear morphism $\phi:\cD\rightarrow\cD$ such that the determinant of $\phi$ with respect to some choice of $S_{\Q_p}$-basis is invertible in $S_{\Q_p}$ (which does not depend on the choice of basis);
\item a decreasing filtration of $\cD$ by $S_{E}$-submodules $\Fil^{i}\cD$, $i\in\Z$, with $\Fil^{i}\cD=\cD$ for $i\leq 0$ and $\Fil^{i}S_{E}\cdot\Fil^{j}\cD\subset\Fil^{i+j}\cD
    $ for all $j$ and all $i\geq 0$;
\item a $K_{0}\otimes E$-linear map $N:\cD
\rightarrow \cD
$ such that
    \begin{itemize}[leftmargin=*]
    \item $N(sx)=N(s)x+sN(x)$ for all $s\in S_{E}$ and $x\in\cD$,
    \item $N\phi=p\phi N$,
    \item $N(\Fil^{i}\cD)\subset\Fil^{i-1}\cD$ for all $i$.
    \end{itemize}
\end{itemize}

For a filtered $(\phi,N)$-module $D$ with positive Hodge--Tate weights, one can associate an object $\cD\in\mathcal{MF}(\phi,N,K,E)$ by the following:
\begin{itemize}[leftmargin=*]
\item $\cD:=S\otimes_{W(k)}D$;
\item $\phi:=\phi\otimes\phi:\cD\rightarrow\cD$;
\item $N:=N\otimes\mathrm{Id}+\mathrm{Id}\otimes N:\cD\rightarrow\cD$;
\item $\Fil^{0}\cD:=\cD$ and, by induction, $$\Fil^{i+1}\cD:=\{x\in\cD\mid N(x)\in\Fil^{i}\cD\mbox{ and }f_{p}(x)\in\Fil^{i+1}D\}$$ where $f_{p}:\cD\twoheadrightarrow D$ is defined by $s(u)\otimes x\mapsto s(p)x$.
\end{itemize}
By \cite{Bre97}, we see that the functor $\cD:D\mapsto S\otimes_{W(k)}D$ gives rise to an equivalence between the category of positive filtered $(\phi,N,K,E)$-modules and the category $\mathcal{MF}(\phi,N,K,E)$.

\begin{exam}\label{exam: filtration of cD}
Let $D$ be the weakly admissible filtered $(\phi,N)$-module in Example~\ref{exam: admissible filtered modules}, and fix a positive integer $r$ with $r_j\leq r <p-1$. If we let $\cD:=\cD(D)$ then we may write $$\cD=\bigoplus_{j\in\cJ}\cD^{(j)}\supset \Fil^r\cD=\bigoplus_{j\in\cJ}\Fil^r\cD^{(j)}$$
where $\cD^{(j)}:=\cD e_j$, $\Fil^r\cD^{(j)}:=(\Fil^r\cD) e_j$, and $e_j$ is the idempotents defined in \eqref{eq: ring coefficient}. Then, it is easy to see that for each $j\in\cJ$
$$\Fil^r\cD^{(j)}=\Fil^{r-r_j}S_E \left(\xi_j\eta_1+\nu_j\eta_2+\xi_j\frac{v}{p}f_{r_j}\big(\frac{v}{p}\big)\eta_2\right)e_j+(\Fil^rS_E\cD) e_j$$
where $f_{r_j}$ is defined in \eqref{eq: definition of f_r}, as a $S_Ee_j$-submodule of the isotypic component $\cD^{(j)}$.
\end{exam}

\begin{defi}
Fix a positive integer $r< p-1$. A \emph{strongly divisible module of weight~$r$} is defined to be a free $S_{\cO}$-modules $\fM$ of finite rank with an $S_{\cO}$-submodule $\Fil^{r}\fM$ and additive maps $\phi,N:\fM\rightarrow\fM$ such that the following properties hold:
\begin{itemize}[leftmargin=*]
\item $\Fil^{r}S_{\cO}\cdot\fM\subset\Fil^{r}\fM$;
\item $\Fil^{r}\fM\cap I\fM=I\Fil^{r}\fM$ for all ideals $I$ in $\cO$;
\item $\phi(sx)=\phi(s)\phi(x)$ for all $s\in S_{\cO}$ and for all $x\in\fM$;
\item $\phi_r(\Fil^{r}\fM)$ is contained in $\fM$ and generates it over $S_{\cO}$, where $\phi_r:=\frac{1}{p^r}\phi$;
\item $N(sx)=N(s)x+sN(x)$ for all $s\in S_{\cO}$ and for all $x\in\fM$;
\item $N\phi=p\phi N$;
\item $v N(\Fil^{r}\fM)\subset\Fil^{r}\fM$.
\end{itemize}
The morphisms are $S_{\cO}$-linear maps that preserve $\Fil^r$ and commute with $\phi$ and $N$. The category of strongly divisible modules of weight $r$ is denoted by $\mathfrak{MD}^{r}_{\cO}$.
\end{defi}

For a strongly divisible module $\fM$ of weight $r$, there exists a unique weakly admissible filtered $(\phi,N)$-module $D$ with Hodge--Tate weights lying in $[0,r]$ such that $\fM[\frac{1}{p}]\simeq \cD(D)$. More precisely, one can construct a free $S_E$-module $\cD:=\fM\otimes_{\Z_p}\Q_p$, and extend $\phi$ and $N$ on $\cD$. The filtration on $\cD$ is described as follows: $\Fil^r\cD=\Fil^r\fM[\frac{1}{p}]$ and
\begin{equation*}
\Fil^i\cD:=
\begin{cases}
\cD &\mbox{if }i\leq 0;\\
\{x\in\cD\mid v^{r-i}x\in\Fil^r\cD\}&\mbox{if }0\leq i\leq r;\\
\sum_{j=0}^{i-1}(\Fil^{i-j}S_{\Q_p})\Fil^j\cD & \mbox{if }i>r, \mbox{ inductively.}
\end{cases}
\end{equation*}
Let $s_0:S_{\Q_p}\rightarrow K$ and $s_p:S_{\Q_p}\rightarrow K$ be defined by $u\mapsto 0$ and $u\mapsto p$, respectively, and let $D:=\cD\otimes_{S_{\Q_p},s_0}K$. The map $s_0$ induces $\phi$ and $N$ on $D$, and the map $s_p$ induces the filtration on $D$ by taking $\Fil^iD:=s_p(\Fil^i\cD)$. Then it turns out that $D$ is a weakly admissible filtered $(\phi,N,K,E)$-module with $\Fil^0D=D$ and $\Fil^{r+1}D=0$.

Hence, one has the following equivalent definition of strongly divisible modules.
\begin{defi}\label{defi: sdm in D}
For a weakly admissible filtered $(\phi,N)$-module $D$ such that $\Fil^{0}D=D$ and $\Fil^{r+1}D=0$, a \emph{strongly divisible module} in $\cD:=\cD(D)$ is defined as an $S_{\cO}$-submodule $\fM$ of $\cD$ such that
\begin{itemize}[leftmargin=*]
\item $\fM$ is a free $S_{\cO}$-module of finite rank such that $\fM[\frac{1}{p}]\simeq\cD$;
\item $\fM$ is stable under $\phi$ and $N$;
\item $\phi(\Fil^r\fM)\subset p^{r}\fM$ where $\Fil^r\fM:=\fM\cap\Fil^r\cD$.
\end{itemize}
\end{defi}

For a strongly divisible module $\fM$, an $\cO[G_{K}]$-module $\Tst^\ast(\fM)$ is defined as follows:
$$\Tst^\ast(\fM):=\Hom_{S,\Fil^r,\phi,N}(\fM,\widehat{\mathbf{A}}_{\mathrm{st}}).$$  (see \cite{Bre99} for detail.) $\Tst^\ast(\fM)$ inherits an $\cO$-module structure from the $\cO$-module structure on $\fM$ and an action of $G_{K}$ from the action of $G_{K}$ on $\widehat{\mathbf{A}}_{\mathrm{st}}$.
\begin{theo}[\cite{Bre99,Liu08,EGH}]
Assume that $0< r< p-1$. The functor $\Tst^\ast$ provides an anti-equivalence of categories between the category $\mathfrak{MD}^{r}_{\cO}$ of strongly divisible modules of weight $r$ and the category of $G_{K}$-stable $\cO$-lattices in semi-stable $E$-representations of $G_{K}$ with Hodge--Tate weights lying in $[0,r]$, .
\end{theo}
There is also a compatibility: if $\fM$ is a strongly divisible module in $\cD:=\cD(D)$ for a weakly admissible filtered $(\phi,N)$-module $D$, then $\Tst^\ast(\fM)$ is a Galois stable $\cO$-lattice in the semi-stable representation $\mathrm{V^{*}_{st}}(D)$.

For a strongly divisible module $\fM$ of weight $r$, we often write
$$\fM=\bigoplus_{j\in\cJ}\fM^{(j)}\supset \Fil^r\fM=\bigoplus_{j\in\cJ} \Fil^r\fM^{(j)}$$
where $\Fil^r\fM^{(j)}:=(\Fil^r\fM) e_j$ and $\fM^{(j)}:=\fM e_j$. Note that $\Fil^r\fM^{(j)}$ is a $S_\cO e_j$-submodule of the isotypic component $\fM^{(j)}$. Moreover, the map $\phi$ (resp. $N$) on $\fM$ inherits the maps
$$\phi^{(j)}:\fM^{(j)}\rightarrow \fM^{(j+1)}\qquad(\mbox{resp.}\quad N^{(j)}:\fM^{(j)}\rightarrow \fM^{(j)})$$
for all $j\in\cJ$. If we write $S'$ for the $p$-adic completion of $\Z_p[\frac{u^i}{i!}]_{i\in\N}$ and $\phi'$ is the Frobenius map on $S'$, then by identifying $S'_\cO:=S'\otimes_{\Z_p}\cO$ with $S_\cO e_j$ we may regard $\phi^{(j)}$ as a $\phi'$-semilinear and $\cO$-linear map. The map $\phi^{(j)}$ also induces $$\phi_r^{(j)}:=\frac{1}{p^r}\phi^{(j)}:\Fil^r\fM^{(j)}\rightarrow\fM^{(j+1)}.$$
Moreover, if we write $N'$ for the monodromy map on $S'$, then $N^{(j)}$ satisfies $N^{(j)}(sx)=N'(s)x+sN^{(j)}(x)$ for all $s\in S'_\cO$ and $x\in\fM^{(j)}$.

\subsection{Breuil modules}
In this subsection, we introduce Breuil modules and their connection to the mod-$p$ reduction of semi-stable representations. Fix again a positive integer $r<p-1$, and let $\barS_\F:=S_\cO/(\varpi,\Fil^pS_\cO)=(k\otimes_{\F_p}\F)[u]/u^{p}$.
\begin{defi}
A \emph{Breuil modules of weight $r$} consists of quadruples $(\cM,\Fil^r\cM,\phi_{r},N)$ where
\begin{itemize}[leftmargin=*]
\item $\cM$ is a finitely generated $\barS_\F$-module, free over $k[u]/u^{p}$, (which implies that $\cM$ is in fact a free $\barS_\F$-module of finite rank);
\item $\Fil^r\cM$ is a $\barS_\F$-submodule of $\cM$ containing $u^{r}\cM$;
\item $\phi_{r}:\Fil^r\cM\rightarrow\cM$ is $\F$-linear and $\phi$-semilinear (where $\phi:k[u]/u^{p}\rightarrow k[u]/u^{p}$ is the $p$-th power map) with image generating $\cM$ as $\barS_\F$-module;
\item $N:\cM\rightarrow\cM$ is $k\otimes_{\F_p}\F$-linear and satisfies
     \begin{itemize}[leftmargin=*]
     \item $N(ux)=uN(x)-ux$ for all $x\in\cM$,
     \item $uN(\Fil^r\cM)\subset \Fil^r\cM$, and
     \item $\phi_{r}(uN(x))=cN(\phi_{r}(x))$ for all $x\in\Fil^r\cM$, where $c\in (k[u]/u^{p})^{\times}$ is the image of $\frac{1}{p}\phi(v)$ under the natural map $S\rightarrow k[u]/u^{p}$.
     \end{itemize}
\end{itemize}
The morphisms are $\barS_\F$-module homomorphisms that preserve $\Fil^r\cM$ and commute with $\phi_{r}$ and $N$. The category of Breuil modules of weight $r$ is denoted by $\mathrm{BrMod}^{r}_{\F}$.
\end{defi}

If $\fM$ is an object of $\mathfrak{MD}^{r}_{\cO}$, then $$\cM:=\fM/(\varpi,\Fil^{p}S_\cO)\fM$$ is naturally an object of $\mathrm{BrMod}^{r}_{\F}$. More precisely,
$\phi$, $N$, and $\Fil^r\cM$ are induced as follow:
\begin{itemize}[leftmargin=*]
\item $\Fil^r\cM$ is the image of $\Fil^r\fM$ in $\cM$;
\item the map $\phi_{r}$ is induced by $\frac{1}{p^{r}}\phi|_{\Fil^r\fM}$;
\item $N$ is induced by the one on $\fM$.
\end{itemize}
Note that this association gives rise to a functor from the category $\mathfrak{MD}^{r}_{\cO}$ to $\mathrm{BrMod}^{r}_{\F}$.

For a Breuil module $\cM$, we define a $\F[G_K]$-module $\Tst^\ast(\cM)$ as follows: 
\begin{equation}\label{eq: definition of Tst}
\Tst^\ast(\cM):=\Hom_{k[u]/u^{p},\Fil^r,\phi_{r},N}(\cM,\widehat{\mathbf{A}}).
\end{equation} 
(See~\cite{EGH} for detail.) $\Tst^\ast(\cM)$ inherits an $\F$-module structure from the $\F$-module structure on $\cM$ and an action of $G_{K}$ from the action of $G_{K}$ on $\widehat{\mathbf{A}}$.
\begin{theo}[\cite{Caruso}]
The functor $\Tst^\ast$ gives rise to a fully faithful contravariant functor from the category $\mathrm{BrMod}^{r}_{\F}$ to the category of finite-dimensional $\F$-representations of $G_{K}$ with $\dim_{\F}\Tst^\ast(\cM)=\mathrm{rank}_{(k\otimes_{\F_p}\F)[u]/u^{p}}\cM$.
\end{theo}

In general, the functor $\Tst^\ast$ is only faithful. More precisely, if $K$ is an arbitrary finite extension with the absolute ramification index $e$, it is full when $er<p-1$. In our case, it is fully faithful as we restrict our attention to a finite unramified extension $K$ and assume $r\leq p-2$.

There is also a compatibility: if $\fM\in\mathfrak{MD}^{r}_{\cO}$ and $\cM:=\fM/(\varpi,\Fil^{p}S)\fM$ denotes the Breuil module corresponding to the reduction of $\fM$, then $\Tst^\ast(\fM)\otimes_{\cO}\F\simeq\Tst^\ast(\cM)$.

For a Breuil module $\cM$ of weight $r$, we often write
$$\cM=\bigoplus_{j\in\cJ}\cM^{(j)}\supset \Fil^r\cM=\bigoplus_{j\in\cJ} \Fil^r\cM^{(j)}$$
where $\Fil^r\cM^{(j)}:=(\Fil^r\cM) e_j$, $\cM^{(j)}:=\cM e_j$, and $e_j$ is the idempotents defined in \eqref{eq: ring coefficient}. Note that $\Fil^r\cM^{(j)}$ is a $e_j\barS_\F (\cong \barS'_\F)$-submodule of the isotypic component $\cM^{(j)}$, where $\barS'_\F:=S'/(\varpi,\Fil^pS'_\cO)=\F[u]/u^p$. Moreover, the map $\phi_r$ (resp. $N$) on $\cM$ inherits the maps
$$\phi_r^{(j)}:\Fil^r\cM^{(j)}\rightarrow \cM^{(j+1)}\qquad(\mbox{resp.}\quad N^{(j)}:\cM^{(j)}\rightarrow \cM^{(j)})$$
for all $j\in\cJ$. The map $\phi_r^{(j)}$ is a $\F$-linear and $\phi'$-semilinear map where $\phi':\F_p[u]/u^p\rightarrow \F_p[u]/u^p$ is the $p$-th power map, and $N^{(j)}$ is a $\F$-linear map satisfying $N^{(j)}(ux)=uN^{(j)}(x)-ux$ for all $x\in\fM^{(j)}$. Finally, we note that if $\cM=\fM/(\varpi,\Fil^pS_\cO)$ for a strongly divisible module $\fM$ of weight $r$, then for each $j\in\cJ$, the maps $\phi_r^{(j)}$, $N^{(j)}$ and the filtration $\Fil^{r}\cM^{(j)}$ on $\cM$ are compatible with the mod-$p$ reduction of  the maps $\phi_r^{(j)}$, $N^{(j)}$, and the filtration $\Fil^{r}\fM^{(j)}$ on $\fM$, respectively.

We close this subsection by giving some examples of Breuil modules for $f=2$, which will be used later. We first introduce a Breuil module of rank $1$.
\begin{exam}\label{exam: rank 1 BrMod}
Let $a,b\in[0,r]\cap\Z$ and $\alpha,\beta\in\F^\times$. Denote by $\cM(a,b;\alpha,\beta)$ the free $\barS_\F$-module $\cM$ of rank $1$ with additional structures:
\begin{itemize}[leftmargin=*]
\item $\cM$ is a free $\barS_\F$-module of rank $1$ generated by $\barE$, i.e., $\cM=\barS_{\F}(\barE)$;
\item $\Fil^r\cM=\barS_{\F}(u^a\barE^{(0)} +u^b \barE^{(1)})$, where $\barE^{(j)}:=\barE e_j$ for all $j\in\cJ=\Z/2\Z$;
\item $\phi_r(u^a\barE^{(0)})=\alpha\barE^{(1)}$ and $\phi_r(u^b\barE^{(1)})=\beta\barE^{(0)}$;
\item $N(\barE^{(0)})=N(\barE^{(1)})=0$.
\end{itemize}
Then $\cM(a,b;\alpha,\beta)$ forms a rank $1$ Breuil module of weight $r$.
\end{exam}

We now introduce a Breuil module of rank $2$.
\begin{exam}\label{exam: rank 2 BrMod}
Fix $i\in\cJ=\Z/2\Z$, and let
    \begin{itemize}[leftmargin=*]
      \item $\vec{a}=(a_0,a_1)$ and $\vec{b}=(b_0,b_1)$ with $a_j,b_j\in[0,r]\cap\Z$ for all $j\in\cJ$ and $\vec{a}\neq\vec{b}$;
      \item $\vec{\alpha}=(\alpha_0,\alpha_1)$ and $\vec{\beta}=(\beta_0,\beta_1)$ with $\alpha_j,\beta_j\in\F^\times$ for all $j\in\cJ$.
    \end{itemize}
Denote by $\cM_i(\vec{a},\vec{b};\vec{\alpha},\vec{\beta})$ the free $\barS_\F$-module $\cM$ of rank $2$ with additional structures:
\begin{itemize}[leftmargin=*]
\item $\cM$ is a free $\barS_\F$-module of rank $2$ generated by $(\barE_1,\barE_2)$, i.e., $\cM=\barS_{\F}(\barE_1,\barE_2)$;
\item $\Fil^r\cM$ is a $\barS_\F$-submodule of $\cM$ generated by $(\barF_1,\barF_2)$;
\item $\Mat_{\baseE^{(j)},\baseE^{(j)}}(N^{(j)})=0_{2,2}$ for $j\in\cJ$, where $\baseE^{(j)}:=(\barE_1e_j,\barE_2e_j)$;
\item $\Mat_{\baseE^{(j)},\baseF^{(j)}}(\Fil^r\cM^{(j)})=
      \begin{bmatrix}
        u^{a_j} & 0 \\
        0 & u^{b_j}
      \end{bmatrix}$ for $j\in\cJ$, where $\baseF^{(j)}:=(\barF_1e_j,\barF_2e_j)$;
\item $\Mat_{\baseE^{(i+1)},\baseF^{(i)}}(\phi_r^{(i)})=
      \begin{bmatrix}
        0 & \alpha_i \\
        \beta_i & 0
      \end{bmatrix}$ and 
      $\Mat_{\baseE^{(i)},\baseF^{(i+1)}}(\phi_r^{(i+1)})=
      \begin{bmatrix}
        \alpha_{i+1} & 0 \\
        0 & \beta_{i+1}
      \end{bmatrix}$.
    \end{itemize}
Then $\cM_i(\vec{a},\vec{b};\vec{\alpha},\vec{\beta})$ forms a rank $2$ Breuil module of weight $r$.
\end{exam}

\begin{rema}
There are isomorphisms between $\cM_i(\vec{a},\vec{b};\vec{\alpha},\vec{\beta})$ for different choices of data $(\vec{a},\vec{b};\vec{\alpha},\vec{\beta})$. Let $\cM=\cM_i(\vec{a},\vec{b};\vec{\alpha},\vec{\beta})$, and let $\vec{a}'=(a_0,b_1)$, $\vec{b}'=(b_0,a_1)$, $\vec{\alpha}'=(\alpha_0,\beta_1)$, and $\vec{\beta}'=(\beta_0,\alpha_1)$. Then we have
\begin{enumerate}[leftmargin=*]
\item $\cM\cong\cM_i(\vec{a},\vec{b};(1,1),(1,\alpha_0\alpha_1\beta_0\beta_1))$, after suitable scaling on each bases;
\item $\cM\cong\cM_{i+1}(\vec{a}',\vec{b}';\vec{\beta}',\vec{\alpha}')$, via swapping the order in $\baseE^{(1)}$ and $\baseF^{(1)}$;
\item $\cM\cong\cM_{i+1}(\vec{b}',\vec{a}';\vec{\alpha}',\vec{\beta}')$, via swapping the order in $\baseE^{(0)}$ and $\baseF^{(0)}$.
\end{enumerate}
\end{rema}

We now compute the mod-$p$ representations corresponding the Breuil modules above.
\begin{lemm}\label{lemm: Tst^r for irred BrMod}
\begin{enumerate}[leftmargin=*]
\item If we write $\cM$ for the Breuil module $\cM(a,b;\alpha,\beta)$ in Example~\ref{exam: rank 1 BrMod}, then
    $$\Tst^\ast(\cM)|_{I_{\Q_{p^2}}} \cong \omega_2^{(r-a)+(r-b)p}.$$
\item If we write $\cM_i$ for the Breuil module $\cM_i(\vec{a},\vec{b};\vec{\alpha},\vec{\beta})$ in Example~\ref{exam: rank 2 BrMod}, then
$$\Tst^\ast(\cM_0)|_{I_{\Q_{p^2}}}=\omega_4^{(r-a_0)+(r-a_1)p+(r-b_0)p^2+(r-b_1)p^3} \oplus \omega_4^{(r-b_0)+(r-b_1)p+(r-a_0)p^2+(r-a_1)p^3}$$
and
$$\Tst^\ast(\cM_1)|_{I_{\Q_{p^2}}}=\omega_4^{(r-a_0)+(r-b_1)p+(r-b_0)p^2+(r-a_1)p^3} \oplus \omega_4^{(r-b_0)+(r-a_1)p+(r-a_0)p^2+(r-b_1)p^3}.$$
\end{enumerate}
\end{lemm}

\begin{proof}
The proofs for both (i) and (ii) are similar to \cite{Par} Proposition~2.5, which highly rely on \cite{Caruso}, Theorem~5.2.2.
\end{proof}

\smallskip

\section{Frameworks}\label{sec: Frame works}
Fix a positive integer $r'$ with $1\leq r'<p-1$, and let $m:=\lfloor r'/2 \rfloor$. We also fix $x\in E$ and $\Theta\in E^\times$. In this section, we deduce various properties of the polynomials $B(T)$ and $C(T)$ in $\cO[T]^{(r'-1)}$ satisfying the following equation
\begin{equation}\label{eq: B and C}
B(T)=\bigg[C(T)\bigg(x+\frac{T}{p}f_{r'}\Big(\frac{T}{p}\Big) \bigg)\bigg]_{T}^{(r'-1)}
\end{equation}
together with $\Theta C(T)\in \cO[T]^{(r'-1)}$. This equation naturally appears in the filtration of our pseudo-strongly divisible modules, as we will see later.

\subsection{Linear expressions of $B(pT)$ and $C(pT)$}\label{subsec: linear expressions of B and C}
Let $\mathbb{V}$ be a $r'$-dimensional vector space over $E$ generated by the free variables $C_0, C_1,\cdots,C_{r'-1}$, and if we let $B(T):=\sum_{i=0}^{r'-1}B_iT^i$ and $C(T):=\sum_{i=0}^{r'-1}C_iT^i$ are polynomials in $T$ satisfying \eqref{eq: B and C} then it is clear that $B_0,B_1,\cdots,B_{r'-1}$ belong to $\mathbb{V}$. Let $\mathbb{M}$ be the $\cO$-submodule of $\mathbb{V}$ generated by $$\Theta C_0,\cdots,\Theta C_{r'-1}, B_0,\cdots,B_{r'-1}.$$ As $\cO$ is a principal ideal domain, we see that $\mathbb{M}$ is a free $\cO$-module of rank $r'$, and set $D_0,\cdots,D_{r'-1}$ to be the generators of $\mathbb{M}$, depending on the $p$-adic valuations of $x$ and $\Theta$.

In this subsection, for a given choice of the generators $D_0,D_1,\cdots,D_{r'-1}$ of $\mathbb{M}$, we describe $C(pT)$ and $B(pT)$ as linear combinations of $D_i$'s and introduce various properties of the coefficients of the linear combinations. We consider $(r'+1)$ choices of the generators $D_0,D_1,\cdots,D_{r'-1}$ of $\mathbb{M}$ described as follow: for a given integer $k$ with $0\leq k\leq r'$, $D_0,\cdots,D_{r'-1}$ are determined by
    \begin{equation}\label{eq: definition of cases}
        \begin{cases}
            D_0=B_{r'-k}\\
            \vdots\\
            D_{k-1}=B_{r'-1}\\
            D_k=\Theta C_k\\
            \vdots\\
            D_{r'-1}=\Theta C_{r'-1}.
        \end{cases}
    \end{equation}
By $\mathbf{Case}~(k)$, we mean that we choose the generators $D_0,D_1,\cdots,D_{r'-1}$ as displayed in \eqref{eq: definition of cases}. For instances, by $\mathbf{Case}~(0)$ we mean $D_i=\Theta C_i$ for all $0\leq i\leq r'-1$, and by $\mathbf{Case}~(r')$ we mean that $D_i=B_i$ for all $0\leq i\leq r'-1$. This choice of the generators will be illustrated in terms of $v_p(x)$ and $v_p(\Theta)$ (see Theorem~\ref{theo: filtration of pseudo}). The reason why we consider these $(r'+1)$ cases is illustrated in the following.

\begin{rema}
We choose the generators $D_0,\cdots,D_{r'-1}$ which have the highest possible $p$-powers, which can be readily seen in the matrix equation \eqref{eq: B and C explicit} induced from \eqref{eq: B and C}. Hence, it is very likely to imply the rest of the variables to be integral under weaker conditions on $v_p(x)$ and $v_p(\Theta)$. In \S\ref{subsec: Generators of the filtration} for each choice of generators $D_0,\cdots,D_{r'-1}$ we will determine the generators of our pseudo-strongly divisible modules and give the precise conditions on $v_p(x)$ and $v_p(\Theta)$ in each case.
\end{rema}

For the rest of this subsection, we will consider $x$ as well as $C_0,C_1,\cdots,C_{r'-1}$ as free variables, unless otherwise stated. As it is easier to handle, we consider $$B(pT)=[C(pT)(x+Tf_{r'}(T))]_{T}^{(r'-1)},$$ which is obtained by plugging in $pT$ to the equation in \eqref{eq: B and C}, and this identity can be described by the following identity of column vectors:
\begin{equation}\label{eq: B and C explicit}
        \begin{bmatrix}
       p^{r'-1} B_{r'-1}\\
       p^{r'-2} B_{r'-2}\\
        \vdots\\
        pB_1\\
        B_0
    \end{bmatrix}
    =
    M_{r',x}
    \begin{bmatrix}
        p^{r'-1} C_{r'-1}\\
        p^{r'-2} C_{r'-2}\\
        \vdots\\
        pC_1\\
        C_0
    \end{bmatrix}
\end{equation}
where $M_{r',x}$ is the $(r'\times r')$-matrix described by
\begin{align*}
    M_{r',x}
    :=\begin{bmatrix}
        x & 1 & -\frac{1}{2} & \cdots & \frac{(-1)^{r'-3}}{r'-2} & \frac{(-1)^{r'-2}}{r'-1}\\
        0 & x & 1 &  \cdots & \frac{(-1)^{r'-4}}{r'-3} & \frac{(-1)^{r'-3}}{r'-2}\\
        0 & 0 & x & \cdots & \frac{(-1)^{r'-5}}{r'-4} & \frac{(-1)^{r'-4}}{r'-3}\\
        \vdots & \vdots & \vdots & \ddots & \vdots & \vdots \\
        0 & 0 & 0 & \cdots & x & 1 \\
        0 & 0 & 0 & \cdots & 0 & x
    \end{bmatrix}
    \in \M_{r'\times r'}(\Z_{(p)}[x]).
\end{align*}
For each integer $k$ with $0\leq k\leq r'$, we write $M_{r',x}$ as block matrices
$$M_{r',x}=
\begin{bmatrix}
A_{k,x} & B_{k,x}\\
C_{k,x} & D_{k,x}
\end{bmatrix}
$$
with $A_{k,x}\in\M_{k\times(r'-k)}(\Z_{(p)}[x])$, $B_{k,x}\in\M_{k\times k}(\Z_{(p)}[x])$, $C_{k,x}\in\M_{(r'-k)\times (r'-k)}(\Z_{(p)}[x])$, and $D_{k,x}\in\M_{(r'-k)\times k}(\Z_{(p)}[x])$. In particular, we have $M_{r',x}=C_{0,x}=B_{r',x}$.

The matrix $B_{k,x}$ as well as $A_{k,x}$ will play important role on describing the generators of the filtration of pseudo-strongly divisible modules, as we will see later. We observe that if $0\leq k\leq m$ then
\begin{equation}\label{eq: B_k,x when k leq m}
    B_{k,x}
    =
    \begin{bmatrix}
        \frac{(-1)^{r'-k-1}}{r'-k} & \frac{(-1)^{r'-k}}{r'-k+1} & \cdots & \frac{(-1)^{r'-2}}{r'-1} \\
        \frac{(-1)^{r'-k-2}}{r'-k-1} & \frac{(-1)^{r'-k}}{r'-k+1} & \cdots & \frac{(-1)^{r'-3}}{r'-2} \\
        \vdots & \vdots & \ddots & \vdots \\
        \frac{(-1)^{r'-2k}}{r'-2k+1} & \frac{(-1)^{r'-2k+1}}{r'-2k+2} & \cdots & \frac{(-1)^{r'-k-1}}{r'-k}
    \end{bmatrix}
    \in \M_{k\times k}(\Z_{(p)})
\end{equation}
which is independent of $x$, and so $B_{k,x}=B_{k,0}$, and if $m+1\leq k\leq r'$ then
\begin{equation}\label{eq: B_k,x when k geq m+1}
    B_{k,x}
    =\begin{bmatrix}
        \frac{(-1)^{r'-k-1}}{r'-k} & \cdots & \frac{(-1)^{k-2}}{k-1} & \frac{(-1)^{k-1}}{k} & \cdots & \frac{(-1)^{r'-2}}{r'-1} \\
        \vdots & \ddots & \vdots & \vdots & \cdots & \vdots \\
        1 & \cdots & \frac{(-1)^{2k-r'-1}}{2k-r'} & \frac{(-1)^{2k-r'}}{2k-r'+1} & \cdots & \frac{(-1)^{k-1}}{k} \\
        x & \cdots & \frac{(-1)^{2k-r'-2}}{2k-r'-1} & \frac{(-1)^{2k-r'-1}}{2k-r'} & \cdots & \frac{(-1)^{k-2}}{k-1} \\
        \vdots & \ddots & \vdots & \vdots & \ddots & \vdots \\
        0 & \cdots & x & 1 & \cdots & \frac{(-1)^{r'-k-1}}{r'-k}
    \end{bmatrix}\in \M_{k\times k}(\Z_{(p)}[x]),
\end{equation}
where the number of $x$'s on the matrix is $2k-r'$.

We now assume $\mathbf{Case}~(k)$ and consider the following
\begin{equation}\label{eq: D and C matrix}
    \begin{bmatrix}
        \frac{p^{r'-1}}{\Theta}D_{r'-1}\\
        \vdots\\
        \frac{p^k}{\Theta}D_k\\
        p^{r'-1}D_{k-1}\\
        \vdots\\
        p^{r'-k}D_0
    \end{bmatrix}
    =
    \begin{bmatrix}
        p^{r'-1}C_{r'-1}\\
        \vdots\\
        p^kC_k\\
        p^{r'-1}B_{r'-1}\\
        \vdots\\
        p^{r'-k}B_{r'-k}
    \end{bmatrix}
    =
    \begin{bmatrix}
        I_{r'-k} & 0_{r'-k,k}\\
        A_{k,x} & B_{k,x}
    \end{bmatrix}
    \begin{bmatrix}
        p^{r'-1}C_{r'-1}\\
        p^{r'-2}C_{r'-2}\\
        \vdots\\
        pC_1\\
        C_0
    \end{bmatrix}
\end{equation}
where the first equality is due to $\mathbf{Case}~(k)$ and the second equality follows from \eqref{eq: B and C}, which together with \eqref{eq: B and C explicit} implies
\begin{equation}\label{eq: D and B matrix}
        \begin{bmatrix}
        p^{r'-1}B_{r'-1}\\
        p^{r'-2}B_{r'-2}\\
        \vdots\\
        pB_1\\
        B_0
    \end{bmatrix}
    =
    M_{r',x}
    \begin{bmatrix}
        p^{r'-1}C_{r'-1}\\
        p^{r'-2}C_{r'-2}\\
        \vdots\\
        pC_1\\
        C_0
    \end{bmatrix}
    =
    M_{r',x}
    \begin{bmatrix}
        I_{r'-k} & 0_{r'-k,k}\\
        A_{k,x} & B_{k,x}
    \end{bmatrix}^{-1}
    \begin{bmatrix}
        \frac{p^{r'-1}}{\Theta}D_{r'-1}\\
        \vdots\\
        \frac{p^k}{\Theta}D_k\\
        p^{r'-1}D_{k-1}\\
        \vdots\\
        p^{r'-k}D_0
    \end{bmatrix}.
\end{equation}

For each integer $i$ with $0\leq i\leq r'-1$, we define $P_{k,i}(T,x)$ and $Q_{k,i}(T,x)$ in $\Q(x)[T]$ as follow:
\begin{equation*}
    P_{k,i}(T,x):=
    \begin{bmatrix}
        T^{r'-1} & T^{r'-2} & \cdots & 1
    \end{bmatrix}
    \begin{bmatrix}
        I_{r'-k} & 0_{r'-k,k}\\
        A_{k,x} & B_{k,x}
    \end{bmatrix}^{-1}
    \mathbf{e}_{r'-i}
\end{equation*}
and
\begin{equation*}
    Q_{k,i}(T,x):=
    \begin{bmatrix}
        T^{r'-1} & T^{r'-2} & \cdots & 1
    \end{bmatrix}
    M_{r',x}
    \begin{bmatrix}
        I_{r'-k} & 0_{r'-k,k}\\
        A_{k,x} & B_{k,x}
    \end{bmatrix}^{-1}
    \mathbf{e}_{r'-i}
\end{equation*}
where $\mathbf{e}_1,\cdots,\mathbf{e}_{r'}$ are the standard Euclidean basis of column vectors.

For a later use, we point out that
\begin{equation}\label{eq: mathrix I,O,A,B}
\begin{bmatrix}
        I_{r'-k} & 0_{r'-k,k}\\
        A_{k,x} & B_{k,x}
 \end{bmatrix}
=
\begin{cases}
I_{r'} & \mbox{if }k=0;\\
B_{r',x}=M_{r',x}&\mbox{if }k=r'.
\end{cases}
\end{equation}
Moreover, if $1\leq k\leq r'-1$ then it is easy to see that
\begin{equation}\label{eq: inverse of M'}
    \begin{bmatrix}
        I_{r'-k} & 0_{r'-k,k}\\
        A_{k,x} & B_{k,x}
    \end{bmatrix}^{-1}
    =\begin{bmatrix}
        I_{r'-k} & 0_{r'-k,k}\\
        -B_{k,x}^{-1}A_{k,x} & B_{k,x}^{-1}
    \end{bmatrix}
\end{equation}
and
\begin{equation}\label{eq: inverse of M''}
    M_{r',x}
    \begin{bmatrix}
        I_{r'-k} & 0_{r'-k,k}\\
        A_{k,x} & B_{k,x}
    \end{bmatrix}^{-1}
   =
    \begin{bmatrix}
        0_{k,r'-k} & I_k\\
        C_{k,x}-D_{k,x}B_{k,x}^{-1}A_{k,x} & D_{k,x}B_{k,x}^{-1}
    \end{bmatrix}.
\end{equation}

\begin{lemm}\label{lemma: C and B in terms of P and Q}
For each integer $k$ with $0\leq k\leq r'$, $\mathbf{Case}~(k)$ implies that
\begin{equation*}
C(pT)=
    \sum_{i=0}^{k-1}p^{i+r'-k}D_i P_{k,i}(T,x)
    +\sum_{i=k}^{r'-1}\frac{p^i}{\Theta}D_i P_{k,i}(T,x)
\end{equation*}
and
\begin{equation*}
B(pT)=
    \sum_{i=0}^{k-1}p^{i+r'-k}D_i Q_{k,i}(T,x)
    +\sum_{i=k}^{r'-1}\frac{p^i}{\Theta}D_i Q_{k,i}(T,x).
\end{equation*}
\end{lemm}

\begin{proof}
The identity for $C(pT)$ follows immediately from \eqref{eq: D and C matrix}. Similarly, the identity for $B(pT)$ follows immediately from \eqref{eq: D and B matrix}.
\end{proof}

When $k=0$ and $k=r'$, $P_{k,i}(T,x)$ and $Q_{k,i}(T,x)$ can be written explicitly.
\begin{lemm}\label{lemma: k=0,k=r}
For $0\leq i\leq r'-1$ we have
\begin{equation*}
   Q_{0,i}(T,x)=[T^i(x+Tf_{r'}(T))]_{T}^{(r'-1)},\quad  P_{0,i}(T,x)=T^i=Q_{r',i}(T,x),
\end{equation*}
and
\begin{equation*}
    P_{r',i}(T,x) =  x^{-r'}\sum_{s=i}^{r'-1}x^s (-1)^{r'-s-1}\Big[T^{i+r'-s-1}f_{r'}(T)^{r'-s-1}\Big]_{T}^{(r'-1)}.
\end{equation*}
In particular, $P_{k,i}(T,x)$ and $Q_{k,i}(T,x)$ belong to $\Z_{(p)}[x,\frac{1}{x},T]$ if $k=0$ or $k=r'$, and
$$Q_{k,i}(T,x)=[(x+Tf_{r'}(T))P_{k,i}(T,x)]_{T}^{(r'-1)}$$ for $k\in\{0,r'\}$ and for all $0\leq i\leq r'-1$.
\end{lemm}

For $r'=1$ the above lemma yields $Q_{0,0}(T,x)=x$, $P_{0,0}(T,x)=1=Q_{1,0}(T,x)$, and $P_{1,0}(T,x)=x^{-1}$, where we use the convention $f_1(T)=0$ and $0^0=1$ for the last one.

\begin{proof}
From \eqref{eq: mathrix I,O,A,B}, it is immediate that $P_{0,i}(T,x)=T^i=Q_{r',i}(T,x)$ for all $0\leq i\leq r'-1$.
Moreover, it is also easy to see that $$Q_{0,i}(T,x)=
    \begin{bmatrix}
        T^{r'-1} & T^{r'-2} & \cdots & 1
    \end{bmatrix}
    M_{r',x}\mathbf{e}_{r'-i}=[T^i(x+Tf_{r'}(T))]_{T}^{(r'-1)}.$$

We now check the identity for $ P_{r',i}(T,x)$. It is easy to see that
\begin{align}\label{eq: elementry computation}
\begin{split}
\big(x+Tf_{r'}(T)\big)\sum_{s=0}^{r'-1}x^s\big(-Tf_{r'}(T)\big)^{r'-s-1}=x^{r'}-(-1)^{r'}\big(Tf_{r'}(T)\big)^{r'}\equiv x^{r'}\pmod{T^{r'}},
\end{split}
\end{align}
which implies that
\begin{align*}
    C(pT)
    &\equiv
    C(pT)(x+Tf_{r'}(T))\sum_{s=0}^{r'-1}x^s\big(-Tf_{r'}(T)\big)^{r'-s-1}\cdot x^{-r'}\pmod{T^{r'}}
    \\
    &\equiv
    B(pT)\sum_{s=0}^{r'-1}x^{s-r'}\big(-Tf_{r'}(T)\big)^{r'-s-1}    \pmod{T^{r'}}.
\end{align*}
Hence, the polynomial $\big[\sum_{s=0}^{r'-1}x^{s-r'}\big(-Tf_{r'}(T)\big)^{r'-s-1}\big]^{(r'-1)}_T$ determines the matrix $M_{r',x}^{-1}$ by \eqref{eq: B and C explicit}, so that we have
\begin{align*}
P_{r',i}(T,x)=
\begin{bmatrix}
        T^{r'-1} & T^{r'-2} & \cdots & 1
    \end{bmatrix}
    M_{r',x}^{-1}\mathbf{e}_{r'-i}=\bigg[T^i\sum_{s=0}^{r'-1}x^{s-r'}\big(-Tf_{r'}(T)\big)^{r'-s-1}\bigg]^{(r'-1)}_T.
\end{align*}

Finally, $P_{k,i}(T,x)$ and $Q_{k,i}(T,x)$ obviously belong to $\Z_{(p)}[x,\frac{1}{x},T]$ as $r'<p-1$, and the last part is also clear from the statement.
\end{proof}

When $0<k<r'$ and $x$ is specialized in $E$, we can also describe the degrees of $P_{k,i}(T,x)$ and $Q_{k,i}(T,x)$ in $T$.
\begin{lemm}\label{lemma: PQdegree}
Let $1\leq k\leq r'-1$, and assume that $x\in E$ with $\det (B_{k,x})\neq 0$. Then for each $0\leq i\leq r'-1$, $$\big(P_{k,i}(T,x),Q_{k,i}(T,x)\big)$$ is the unique pair of polynomials in $T$ satisfying the following properties:
    \begin{enumerate}[leftmargin=*]
        \item $Q_{k,i}(T,x)=[(x+Tf_{r'}(T))P_{k,i}(T,x)]_{T}^{(r'-1)}$;
        \item if $0\leq i\leq k-1$, then $\deg_T P_{k,i}(T,x)\leq k-1$ and $\deg_T\big( Q_{k,i}(T,x)-T^{i+r'-k} \big)\leq r'-k-1$;
        \item if $k\leq i\leq r'-1$, then $\deg_T\big( P_{k,i}(T,x)-T^i \big)\leq k-1$ and $\deg_T Q_{k,i}(T,x)\leq r'-k-1$.
    \end{enumerate}
\end{lemm}

For $1\leq k\leq r'$ we write $a_k(x)$ for the coefficient of $T^{k-1}$ in the polynomial $P_{k,0}(T,x)$, where $k-1$ is the maximal possible degree of $P_{k,0}(T,x)$ by Lemma~\ref{lemma: PQdegree}. We often write $a_k$ for $a_k(0)$ for simplicity. Note that $a_{r'}(0)$ is not defined, because $a_{r'}(x)$ is a non-constant polynomial of $x^{-1}$(see Lemma~\ref{lemma: k=0,k=r}).

\begin{proof}
We first check that the pair $(P_{k,i}(T,x),Q_{k,i}(T,x))$ satisfies those properties. For~(i), it is immediate from the definition of $P_{k,i}(T,x)$ that
\begin{align*}
\big[(x+Tf_{r'}(T))P_{k,i}(T,x)\big]_{T}^{(r'-1)} &=
            \begin{bmatrix}
                T^{r'-1} & T^{r'-2} & \cdots & 1
            \end{bmatrix}
            M_{r',x}
            \begin{bmatrix}
                I_{r'-k} & 0_{r'-k,k} \\
                A_{k,x} & B_{k,x}
            \end{bmatrix}^{-1}
            \mathbf{e}_{r'-i}
            \\
            &=Q_{k,i}(T,x).
        \end{align*}
For~(ii), let $0\leq i\leq k-1$ and let $\mathbf{e}'_1,\cdots,\mathbf{e}'_{k}$ be the standard basis of $\Q^k$ of collum vectors. Then by \eqref{eq: inverse of M'} we have
\begin{align*}
P_{k,i}(T,x)&=
            \begin{bmatrix}
                T^{r'-1} & T^{r'-2} &\cdots & 1
            \end{bmatrix}
            \begin{bmatrix}
                I_{r'-k} & 0_{r'-k,k} \\
                -B_{k,x}^{-1}A_{k,x} & B_{k,x}^{-1}
            \end{bmatrix}\mathbf{e}_{r'-i}\\
            &=
            \begin{bmatrix}
                T^{r'-1} & T^{r'-2} &\cdots & 1
            \end{bmatrix}
            \begin{bmatrix}
                0_{r'-k,k} \\
                B_{k,x}^{-1}
            \end{bmatrix} \mathbf{e}'_{k-i} \\
            &=
            \begin{bmatrix}
                T^{k-1} & T^{k-2} &\cdots & 1
            \end{bmatrix} B_{k,x}^{-1} \mathbf{e}'_{k-i},
\end{align*}
and by \eqref{eq: inverse of M''} we have
\begin{align*}
Q_{k,i}(T,x)&=
            \begin{bmatrix}
                T^{r'-1}& T^{r'-2} & \cdots & 1
            \end{bmatrix}
            \begin{bmatrix}
                0_{k,r'-k} & I_k\\
                C_{k,x}-D_{k,x}B_{k,x}^{-1}A_{k,x} & D_{k,x}B_{k,x}^{-1}
            \end{bmatrix}\mathbf{e}_{r'-i} \\
            &=
            \begin{bmatrix}
                T^{r'-1}& T^{r'-2} & \cdots & 1
            \end{bmatrix}
            \begin{bmatrix}
                I_k\\
                D_{k,x}B_{k,x}^{-1}
            \end{bmatrix} \mathbf{e}'_{k-i} \\
            &=
            T^{i+r'-k}+
            \begin{bmatrix}
                T^{k-1}& T^{k-2} & \cdots & 1
            \end{bmatrix}
            D_{k,x}B_{k,x}^{-1}
            \mathbf{e}'_{k-i}.
        \end{align*}
For~(iii), let $k\leq i\leq r'-1$ and let $\mathbf{e}''_1,\cdots,\mathbf{e}''_{r'-k}$ be the standard basis of $\Q^{r'-k}$ of collum vectors. Then by the same idea as (ii), we get $$P_{k,i}(T,x)=T^i-
            \begin{bmatrix}
                T^{k-1} & T^{k-2} & \cdots & 1
            \end{bmatrix}
            B_{k,x}^{-1}A_{k,x}
            \mathbf{e}''_{r'-i}$$
and
$$Q_{k,i}(T,x)=\begin{bmatrix}
                T^{r'-k-1} &T^{r'-k-2} & \cdots & 1
            \end{bmatrix}
            (C_{k,x}-D_{k,x}B_{k,x}^{-1}A_{k,x})
            \mathbf{e}''_{r'-i}.$$
Hence, we proved that the pair satisfies those properties.

The rest of the proof is devoted to the uniqueness. For given $k$ and $i$, let $(p(T),q(T))$ be a pair of polynomials in $E[T]$ satisfying those properties. If $0\leq i\leq k-1$ then by the property (ii) we may write
$$p(T)=p_0+p_1T+\cdots+p_{k-1}T^{k-1}\,\,\mbox{and}\,\, q(T)=q_0+q_1T+\cdots+q_{r'-k-1}T^{r'-k-1}+T^{i+r'-k},$$
and if $k\leq i\leq r'-1$ then by the property (iii) we may write
$$p(T)=p_0+p_1T+\cdots+p_{k-1}T^{k-1}+T^i\,\,\mbox{and}\,\,q(T)=q_0+q_1T+\cdots+q_{r'-k-1}T^{r'-k-1}.$$
Then from the property (i) we have
\begin{align*}
\left[\sum_{s=0}^{k-1}p_sT^s(x+Tf_{r'}(T))\right]_{T}^{(r'-1)}-\sum_{s=0}^{r'-k-1}q_sT^s=
 \begin{cases}
 T^{i+r'-k}&\mbox{if }0\leq i<k;\\
 [T^i(x+Tf_{r'}(T))]_{T}^{(r'-1)}&\mbox{if }k\leq i<r'.
 \end{cases}
\end{align*}
We now write the left hand side of the identity as matrix multiplication.
\begin{align*}
[&(p_0+\cdots+p_{k-1}T^{k-1})(x+Tf_{r'}(T))]_{T}^{(r'-1)}-(q_0+\cdots+q_{r'-k-1}T^{r'-k-1})\\
&=
        \begin{bmatrix}
            T^{r'-1} & \cdots & 1
        \end{bmatrix}
        \left(
        M_{r',x}
        \begin{bmatrix}
            0_{r'-k,1}\\
            \mathbf{p}
        \end{bmatrix}
        -
        \begin{bmatrix}
            0_{k,1}\\
            \mathbf{q}
        \end{bmatrix}
        \right)\\
&=
        \begin{bmatrix}
            T^{r'-1} & \cdots & 1
        \end{bmatrix}
        \left(
        \begin{bmatrix}
            A_{k,x} & B_{k,x} \\
            C_{k,x} & D_{k,x}
        \end{bmatrix}
        \begin{bmatrix}
            0_{r'-k,k} & 0_{r'-k,r'-k} \\
            I_k & 0_{k,r'-k}
        \end{bmatrix}
        \begin{bmatrix}
            \mathbf{p} \\
            \mathbf{q}
        \end{bmatrix}
        -
        \begin{bmatrix}
            0_{k,k} & 0_{k,r'-k} \\
            0_{r'-k,k} & I_{r'-k}
        \end{bmatrix}
        \begin{bmatrix}
            \mathbf{p}\\
            \mathbf{q}
        \end{bmatrix}
        \right)\\
&=
        \begin{bmatrix}
            T^{r'-1} & \cdots & 1
        \end{bmatrix}
        \begin{bmatrix}
            B_{k,x} & 0_{k,r'-k} \\
            D_{k,x} & -I_{r'-k}
        \end{bmatrix}
        \begin{bmatrix}
            \mathbf{p} \\
            \mathbf{q}
        \end{bmatrix}
    \end{align*}
where $\mathbf{p}=
        \begin{bmatrix}
            p_{k-1} &\cdots&  p_0
        \end{bmatrix}^t$ and
        $\mathbf{q}=
        \begin{bmatrix}
            q_{r'-k-1} & \cdots & p_0
        \end{bmatrix}^t$. Therefore, the uniqueness follows from the non-singularity
        $$
        \det\left(\begin{bmatrix}
            B_{k,x} & 0_{k,r'-k} \\
            D_{k,x} & -I_{r'-k}
        \end{bmatrix}\right)=(-1)^{r'-k}\det (B_{k,x})\neq0,$$
where the last inequality is due to our assumption.
\end{proof}

We further study the quotient $Q_{k,i}(T,x)/P_{k,i}(T,x)$ for $i\in\{0,k\}$ which will be used to define the coefficients of our strongly divisible modules when we specialize $x$ in $E$.

We recall the definition of Pad\'e approximation.
\begin{defi}[Pad\'e approximation]
Let $R[[T]]$ be the power series ring over a commutative ring $R$. For $f(T)\in R[[T]]$ and for a pair of integers $(m,n)$ with $m\geq0$ and $n\geq0$, the \emph{$[m/n]$-Pad\'e approximant of $f(T)$} is a rational function
\begin{equation*}
   \frac{Q(T)}{P(T)}\equiv f(T) \pmod{T^{m+n+1}}
\end{equation*}
where $P(T)\in R[T]$ (resp. $Q(T)\in R[T]$) is a polynomial with $\deg P(T)\leq n$ (resp. $\deg Q(T)\leq m$).
\end{defi}

Note that if such a pair $(P(T),Q(T))$ exists, then the quotient $P(T)/Q(T)$ is unique, although the pair $(P(T),Q(T))$ may not be unique up to scalar multiple in general.

\begin{lemm}\label{lemma: PQFad'e}
Let $0\leq k\leq r'$, and assume that $x\in E$ with $\det (B_{k,x})\neq 0$.
    \begin{enumerate}[leftmargin=*]
        \item For each $1\leq k\leq r'$, $Q_{k,0}(T,x)$ is monic and $Q_{k,0}(T,x)/P_{k,0}(T,x)$ is the $[(r'-k)/(k-1)]$-Pad\'e approximation of $x+\log(1+T)$.
        \item For each $0\leq k\leq r'-1$, $P_{k,k}(T,x)$ is monic and $Q_{k,k}(T,x)/P_{k,k}(T,x)$ is the $[(r'-k-1)/k]$-Pad\'e approximation of $x+\log(1+T)$.
        \item For $1\leq k\leq r'$, we have
        $$ P_{k-1,k-1}(T,x)=\frac{1}{a_k(x)}P_{k,0}(T,x)\quad\mbox{and}\quad Q_{k-1,k-1}(T,x)=\frac{1}{a_k(x)}Q_{k,0}(T,x).$$
    \end{enumerate}
\end{lemm}
Note that $a_k(x)$ is defined right after Lemma~\ref{lemma: PQdegree}.

\begin{proof}
(i) and (ii) follows immediately from Lemma~\ref{lemma: PQdegree} together with Lemma~\ref{lemma: k=0,k=r}. (iii) follows from the uniqueness in Lemma~\ref{lemma: PQdegree} if $k<r'$ and from the explicit formula in Lemma~\ref{lemma: k=0,k=r} if $k=r'$.
\end{proof}

\subsection{The polynomials $P_{k,i}(T,x)$ and $Q_{k,i}(T,x)$}
In this subsection, we study the necessary properties of $P_{k,i}(T,x)$ and $Q_{k,i}(T,x)$ as polynomials in $x$. Later, we will specialize $x$ in $E$ and so it is necessary to investigate the form of those polynomials as polynomials in $x$, to keep track the $p$-adic valuation of the coefficients of those polynomials. We keep the notation of \S\ref{subsec: linear expressions of B and C}.

\subsubsection{\textbf{When $0\leq k\leq m$}}
In this section of paragraph, we treat the case $0\leq k\leq m$. We first show that the coefficients of $P_{k,i}(T,x)$ and $Q_{k,i}(T,x)$ are $p$-adically integral if take $x=0$.
\begin{lemm}\label{lemma: B_k,0 k leq m}
Let $1\leq k\leq m$. Then
\begin{enumerate}[leftmargin=*]
    \item $B_{k,x}=B_{k,0}$ and $v_p(\det(B_{k,0}))=0$;
    \item $P_{k,i}(T,0),Q_{k,i}(T,0)\in \Z_{(p)}[T]$ for all $0\leq i\leq r'-1$;
    \item $a_k\in\Z_{(p)}^{\times}$ and $P_{k,0}(T,x)=P_{k,0}(T,0)$.
\end{enumerate}
\end{lemm}

\begin{proof}
We first point out that if $1\leq k\leq m$ then none of the entries of $B_{k,x}$ is $x$ (see \eqref{eq: B_k,x when k leq m}), which proves the first identity of (i). The second identity of (i) follows from the proof of \cite[Lemma~4.1.1]{GP}.

For (ii), we have $B_{k,0}^{-1}\in \M_{k\times k}(\Z_{(p)})$ and $D_{k,0}B_{k,0}^{-1}\in\M_{(r'-k)\times k}(\Z_{(p)})$, as $\det(B_{k,0})\in\Z_{(p)}^\times$ and $D_{k,x}\in \M_{(r'-k)\times k}(\Z_{(p)})$. Now, \eqref{eq: inverse of M'} and \eqref{eq: inverse of M''} completes the proof of (ii).

For (iii), as $B_{k,x}^{-1}=B_{k,0}^{-1}\in \M_{k\times k}(\Z_{(p)})$, \eqref{eq: inverse of M'} implies that $P_{k,0}(T,x)=P_{k,0}(T,0)$. Moreover, we see that the $(1,k)$-entry of $B_{k,x}^{-1}$ is $a_k(x)=a_k$ by definition. On the other hand, the $(1,k)$-entry of $B_{k,x}^{-1}$ can be also written as $\frac{(-1)^{1+k}\det(B_{k-1,x})}{\det(B_{k,x})}$, where we let $\det(B_{0,x})=1$ by convention. Hence, by (i) we conclude that $v_p(a_k)=0$, which completes the proof.
\end{proof}

For $m+1\leq k\leq r'-1$, this statement does not hold in general. For example, if $r'=6$, $k=4$ and $p=19$, then
    \begin{equation*}
        \det B_{4,0}=-\frac{19}{720},
    \end{equation*}
and so $v_{19}(\det B_{4,0})\neq 0$.

We further study the necessary properties of $P_{k,i}(T,x)$ and $Q_{k,i}(T,x)$ as polynomials in $x$. Recall that if $0\leq k\leq m$, then the matrix $B_{k,x}$ is of the form \eqref{eq: B_k,x when k leq m}, so that we have $B_{k,x}=B_{k,0}$.

\begin{lemm}\label{lemma-PQ-k leq m}
Let $0\leq k\leq m$.
\begin{enumerate}[leftmargin=*]
\item if $0\leq i\leq r'-k-1$ then
$$P_{k,i}(T,x)=P_{k,i}(T,0)\quad\mbox{ and }\quad Q_{k,i}(T,x)=Q_{k,i}(T,0)+xP_{k,i}(T,0);$$
\item if $r'-k\leq i\leq r'-1$ then
$$P_{k,i}(T,x)=P_{k,i}(T,0)-xP_{k,i-r'+k}(T,0)$$
and
$$Q_{k,i}(T,x)=Q_{k,i}(T,0)+xP_{k,i}(T,0)-x\big(Q_{k,i-r'+k}(T,0)+xP_{k,i-r'+k}(T,0)\big).$$
\end{enumerate}
\end{lemm}

\begin{proof}
As $0\leq k\leq m$, we may write
$  A_{k,x}
    =
    A_{k,0}
    +
    x\begin{bmatrix}
        I_k & 0_{k,r'-2k}
    \end{bmatrix}$
and so by \eqref{eq: inverse of M'}
\begin{equation*}
        \begin{bmatrix}
        I_{r'-k} & 0_{r'-k,k}\\
        A_{k,x} & B_{k,x}
    \end{bmatrix}^{-1}
    =
    \begin{bmatrix}
        I_{r'-k} & 0_{r'-k,k} \\
        -B_{k,0}^{-1}A_{k,0} & B_{k,0}^{-1}
    \end{bmatrix}
    -x
    \begin{bmatrix}
        0_{r'-k,k} & 0_{r'-k,r'-k} \\
        B_{k,0}^{-1}& 0_{k,r'-k}
    \end{bmatrix}.
\end{equation*}
Hence, from the definition of $P_{k,i}(T,x)$ we have
    \begin{align*}
        P_{k,i}(T,x)
        &=
        \begin{bmatrix}
          T^{r'-1} & \cdots & 1
        \end{bmatrix}\left(
        \begin{bmatrix}
            I_{r'-k} & 0_{r'-k,k} \\
            -B_{k,0}^{-1}A_{k,0} & B_{k,0}^{-1}
        \end{bmatrix}
        -x
    \begin{bmatrix}
        0_{r'-k,k} & 0_{r'-k,r'-k} \\
        B_{k,0}^{-1}& 0_{k,r'-k}
    \end{bmatrix}\right)
        \mathbf{e}_{r'-i}
        \\&
        =P_{k,i}(T,0)
        -x\begin{bmatrix}
          T^{r'-1} & \cdots & 1
        \end{bmatrix}
        \begin{bmatrix}
            0_{r'-k,k} & 0_{r'-k,r'-k} \\
            B_{k,0}^{-1}& 0_{k,r'-k}
        \end{bmatrix}
        \mathbf{e}_{r'-i}
    \end{align*}
If $0\leq i\leq r'-k-1$, then the $(r'-i)$-th column of the matrix in the last term is $0$, and so we have $P_{k,i}(T,x)=P_{k,i}(T,0)$. If $r'-k\leq i\leq r'$, then we observe that
    \begin{align*}
        \begin{bmatrix}
          T^{r'-1} & \cdots & 1
        \end{bmatrix}
        \begin{bmatrix}
            0_{r'-k,k} & 0_{r'-k,r'-k} \\
            B_{k,x}^{-1}& 0_{k,r'-k}
        \end{bmatrix}
        \mathbf{e}_{r'-i}
        &=
        \begin{bmatrix}
          T^{k-1} & \cdots & 1
        \end{bmatrix}
        B_{k,x}^{-1}
        \mathbf{e}'_{k-(i-r'+k)}\\
        &=P_{k,i-r'+k}(T,0),
    \end{align*}
providing $P_{k,i}(T,x)=P_{k,i}(T,0)-xP_{k,i-r'+k}(T,0)$.

We now compute for $Q_{k,i}(T,x)$. Recall from Lemma~\ref{lemma: k=0,k=r} and Lemma~\ref{lemma: PQdegree} that
$$Q_{k,i}(T,x)=\big[(x+Tf_{r'}(T))P_{k,i}(T,x)\big]_{T}^{(r'-1)}=\big[Tf_{r'}(T)P_{k,i}(T,x)\big]_T^{(r'-1)}+xP_{k,i}(T,x).$$
If $0\leq i\leq r'-k-1$, then we have
    \begin{align*}
        Q_{k,i}(T,x) &=\big[Tf_{r'}(T)P_{k,i}(T,0)\big]_{T}^{(r'-1)}+xP_{k,i}(T,0)
        \\&=Q_{k,i}(T,0)+xP_{k,i}(T,0).
    \end{align*}
If $r'-k\leq i\leq r'-1$, then we have
    \begin{align*}
        Q_{k,i}(T,x)&=\big[Tf_{r'}(T)\big(P_{k,i}(T,0)-xP_{k,i-r'+k}(T,0)\big)\big]_{T}^{(r'-1)}
        \\&\hspace{7em}
        +x\big(P_{k,i}(T,0)-xP_{k,i-r'+k}(T,0)\big)
        \\&=\big[Tf_{r'}(T)P_{k,i}(T,0)\big]_{T}^{(r'-1)}-x\big[Tf_{r'}(T)P_{k,i-r'+k}(T,0)\big]_{T}^{(r'-1)}
        \\&\hspace{8em}
        +xP_{k,i}(T,0)-x^2P_{k,i-r'+k}(T,0)
        \\&=Q_{k,i}(T,0)-xQ_{k,i-r'+k}(T,0)+xP_{k,i}(T,0)-x^2P_{k,i-r'+k}(T,0).
    \end{align*}
This completes the proof.
\end{proof}

\subsubsection{\textbf{When $m+1\leq k\leq r'$}}
In this section of paragraph, we treat the case $m+1\leq k\leq r'$. In this case, the matrix $B_{k,x}$ is of the form \eqref{eq: B_k,x when k geq m+1}, so that it is easy to see that
\begin{equation*}
    \adj(B_{k,x})=\det(B_{k,x})B_{k,x}^{-1}
    \in \M_{k\times k}(\Z_{(p)}[x]).
\end{equation*}
Moreover, we see that
$$
            \begin{bmatrix}
        I_{r'-k} & 0_{r'-k,k}\\
        A_{k,x} & B_{k,x}
    \end{bmatrix}^{-1}
=
    \begin{bmatrix}
        I_{r'-k} & 0_{r'-k,k} \\
        -B_{k,x}^{-1}A_{k,x} & B_{k,x}^{-1}
    \end{bmatrix}\in\M_{k\times k}(\Q_p(x)).$$

To normalize $P_{k,i}(T,x)$ and $Q_{k,i}(T,x)$, we set
\begin{equation*}
    d_k(x):=(-1)^{r'(k-1)}\frac{\det(B_{k,x})}{\det(B_{r'-k,0})}.
\end{equation*}
Here we set $\det(B_{0,0})=1$ for $k=r'$. Note that $B_{r'-k,x}=B_{r'-k,0}$ and $\det(B_{r'-k,0})$ is a constant in $\Z_{(p)}^\times$. We also note that $d_{k}(x)$ is a monic polynomial in $x$ of degree $2k-r'$ (see Lemma~\ref{lemma-QAinv}).
We now define
\begin{equation*}
   \begin{cases}
    \tilde P_{k,i}(T,x):=d_k(x)P_{k,i}(T,x);\\
    \tilde Q_{k,i}(T,x):=d_k(x)Q_{k,i}(T,x),
    \end{cases}
\end{equation*}
both of which belong to $\Z_{(p)}[x,T]$.

We also set
\begin{equation*}
    s_{k,i}:=\deg_x \tilde P_{k,i}(T,x)\quad\mbox{ and }\quad
    t_{k,i}:=\deg_x \tilde Q_{k,i}(T,x),
\end{equation*}
and write
\begin{equation*}
\tilde P_{k,i}(T,x)=\sum_{s=0}^{s_{k,i}} x^s P_{k,i,s}(T)\quad\mbox{ and }\quad\tilde Q_{k,i}(T,x)=\sum_{s=0}^{t_{k,i}} x^s Q_{k,i,s}(T),
\end{equation*}
with $P_{k,i,s}(T),\, Q_{k,i,s}(T)\in \Z_{(p)}[T].$ It will be useful to use $\tilde P_{k,i}(T,x)$ and $\tilde Q_{k,i}(T,x)$ to describe the valuations of the coefficients in practice.

\begin{lemm}\label{lemma: m+1 to r}
    Let $m+1\leq k\leq r'$. Then
    \begin{equation*}
       t_{k,i}-1= s_{k,i}=
        \begin{cases}
            2k-r'-1 & \mbox{if } 0\leq i\leq 2k-r'-1; \\
            2k-r' & \mbox{if } 2k-r'\leq i\leq k-1;\\
            2k-r'+1 & \mbox{if } k\leq i\leq r'-1,
        \end{cases}
    \end{equation*}
    and
    \begin{equation*}
       Q_{k,i,t_{k,i}}(T)= P_{k,i,s_{k,i}}(T)
        =
        \begin{cases}
            P_{r'-k,i+r'-k}(T,0) & \mbox{if } 0\leq i\leq 2k-r'-1; \\
            P_{r'-k, i+r'-2k}(T,0) & \mbox{if } 2k-r'\leq i\leq k-1;\\
            -P_{r'-k,i-k}(T,0) & \mbox{if } k\leq i\leq r'-1.
        \end{cases}
    \end{equation*}
\end{lemm}

We need a little preparation to prove this Lemma. It is easy to see that the matrix $B_{k,x}$ is a low-rank update of $B_{k,0}$ if $m+1\leq k<r'$. More precisely, if we set $t=2k-r'$ and define
\begin{equation*}
    U_t
    =
    \begin{bmatrix}
       0_{k-t,t} \\
       I_t
    \end{bmatrix}
    \in\M_{k\times t}(\Q)\quad\mbox{ and }\quad
    V_t=
    \begin{bmatrix}
        I_t & 0_{t,k-t}
    \end{bmatrix}
    \in\M_{t\times k}(\Q),
\end{equation*}
then we can write
\begin{equation*}
    B_{k,x}=B_{k,0}+xU_tV_t,
\end{equation*}
and we have
\begin{equation}\label{eq: inverse of B_k,x}
    B_{k,x}^{-1}=B_{k,0}^{-1}-xB_{k,0}^{-1}U_t(I_t+xV_tB_{k,0}^{-1}U_t)^{-1}V_tB_{k,0}^{-1},
\end{equation}
which is the Sherman-Morrison-Woodbury Formula.

We write the matrix $B_{k,0}^{-1}$ in block matrices as
\begin{equation*}
    B_{k,0}^{-1}=
    \begin{bmatrix}
        A_k & A'_k \\
        A''_k & A'''_k
    \end{bmatrix},
\end{equation*}
where $A_k\in\M_{t\times (k-t)}(\Q)$, $A'_k\in\M_{t\times t}(\Q)$, $A''_k\in\M_{(k-t)\times (k-t)}(\Q)$, and $A'''_k\in\M_{(k-t)\times t}(\Q)$.

\begin{lemm}\label{lemma-QAinv}
Let $m+1\leq k<r'$. Then
    \begin{enumerate}[leftmargin=*]
      \item $A'_k$ is invertible and $(A'_k)^{-1}\in\M_{t\times t}(\Z_{(p)})$;
      \item $B_{r'-k,0}^{-1}=A''_k-A'''_k(A'_k)^{-1}A_k$;
      \item $d_k(x)=\det(xI_t+(A'_k)^{-1})\in \Z_{(p)}[x]$, which is monic of degree $t$.
    \end{enumerate}
\end{lemm}

\begin{proof}
Write
    \begin{align*}
        B_{k,0}=
        \begin{bmatrix}
            B & B' \\
            B'' & B'''
        \end{bmatrix}
    \end{align*}
where $B\in\M_{(k-t)\times t}(\Q)$, $B'\in\M_{(k-t)\times(k-t)}(\Q)$, $B''\in\M_{t\times t}(\Q)$, and $B'''\in\M_{t\times (k-t)}(\Q)$. Note that $B'=B_{k-t,0}=B_{r'-k,0}$ as $k-t=r'-k$.

Computing $B_{k,0}B_{k,0}^{-1}=I_k$ blockwisely, we get 4 equations:
\begin{equation*}
    \begin{bmatrix}
        BA_k+B'A''_k & BA'_k+B'A'''_k\\
        B''A_k+B'''A''_k & B''A'_k+B'''A'''_k
    \end{bmatrix}
    =
          \begin{bmatrix}
        I_{k-t} & 0_{k-t,t}\\
        0_{t,k-t} & I_t
    \end{bmatrix}.
\end{equation*}
As $B'=B_{r'-k,0}$ is invertible, we can solve the equation in the (1,2)-entry for $A'''_k$ to get
    $$A'''_k=-(B')^{-1}B A'_k.$$
Plugging it into the equation in the (2,2)-entry, we get
$$I_t=B''A'_k-B'''(B')^{-1}BA'_k=(B''-B'''(B')^{-1}B)A'_k.$$
Therefore, $A'_k$ is invertible, and
$$
(A'_k)^{-1}=B''-B'''(B')^{-1}B \in\M_{t\times t}(\Z_{(p)}),$$
which proves (i).

For (ii), we solve the equation in the (1,2)-entry for $B$ to get
$$B=-B'A'''_k(A'_k)^{-1}.$$
Plugging it into the equation in the (1,1)-entry, we get
$$I_{k-t}=-B'A'''_k(A'_k)^{-1}A_k+B'A''_k=B'(A''_k-A'''_k(A'_k)^{-1}A_k),$$
which proves (ii).

For (iii), we observe that
    \begin{align*}
        B_{k,x}
        =
        \begin{bmatrix}
            B & B' \\
            xI_t+B'' & B'''
        \end{bmatrix}.
    \end{align*}
    From the identity
    \begin{align*}
        \begin{bmatrix}
            (B')^{-1} & 0_{k-t,t} \\
            -B'''(B')^{-1} & I_t
        \end{bmatrix}
        \begin{bmatrix}
            B & B' \\
            xI_t+B'' & B'''
        \end{bmatrix}
        =
        \begin{bmatrix}
            (B')^{-1}B & I_{k-t} \\
            xI_t+B''-B'''(B')^{-1}B & 0_{t,k-t}
        \end{bmatrix}
    \end{align*}
    we get
$$\det (B')^{-1}\det B_{k,x}=(-1)^{r'(k-1)}\det(xI_t+B''-B'''(B')^{-1}B)$$
where the sign of right-hand side comes from consecutive cofactor expansions at the rightmost rows. Also, from $B'=B_{r'-k,0}$ and $B''-B'''(B')^{-1}B=(A'_k)^{-1}$, one has
    \begin{align*}
        d_k(x)=(-1)^{r'(k-1)}\frac{\det(B_{k,x})}{\det(B_{r'-k,0})}=\det(xI_t+(A'_k)^{-1}).
    \end{align*}
This is a monic polynomial of $x$ with degree $t$ defined over $\Z_{(p)}$.
\end{proof}

\begin{proof}[Proof of Lemma~\ref{lemma: m+1 to r}]
Let $t=2k-r'$, and we first consider the case $k=r'$. But this case immediately follows from Lemma~\ref{lemma: k=0,k=r}.

From now on, we assume that $m+1\leq k\leq r'-1$. We have the following identities:
$$B_{k,0}^{-1}U_t
=
        \begin{bmatrix}
          A'_k \\
          A'''_k
        \end{bmatrix},
\qquad
    V_tB_{k,0}^{-1}
        =
        \begin{bmatrix}
          A_k & A'_k
        \end{bmatrix},
\qquad\mbox{and}\qquad
    V_tB_{k,0}^{-1}U_t
        =A'_k.$$
We consider $y=x^{-1}$ to compute the terms with highest power of $x$. As $A'_k$ is invertible by Lemma~\ref{lemma-QAinv}~(i), we observe that
    \begin{align*}
        (I_t+xV_tB_{k,0}^{-1}U_t)^{-1}
        =
        (I_t+xA'_k)^{-1}
        =y(A'_k)^{-1}(y(A'_k)^{-1}+I_t)^{-1}.
    \end{align*}
    On $\Q[[y]]$, we have
    \begin{align*}
        (y(A'_k)^{-1}+I_t)^{-1}
        =\sum_{n=0}^\infty (y(A'_k)^{-1})^n
        \in \M_{t\times t}\big( \Q[[y]] \big),
    \end{align*}
    which gives
    \begin{align*}
        (I_t+xA'_k)^{-1}
        \equiv
        y(A'_k)^{-1}+y^2((A'_k)^{-1})^2 \pmod{y^3}.
    \end{align*}
Combining this to \eqref{eq: inverse of B_k,x} gives
    \begin{align*}
        B_{k,x}^{-1}
        &=
        B_{k,0}^{-1}
        -x
        \begin{bmatrix}
            A'_k \\
            A'''_k
        \end{bmatrix}
        (I_t+xA'_k)^{-1}
        \begin{bmatrix}
            A_k & A'_k
        \end{bmatrix}
        \\
        &\equiv
        B_{k,0}^{-1}
        -
        \begin{bmatrix}
            A'_k \\
            A'''_k
        \end{bmatrix}
        \big((A'_k)^{-1}+y((A'_k)^{-1})^2\big)
        \begin{bmatrix}
            A_k & A'_k
        \end{bmatrix}
        \pmod{y^2}.
    \end{align*}

We set
    \begin{align*}
        \tilde A
        :=
        \begin{bmatrix}
            A'_k \\
            A'''_k
        \end{bmatrix}
        (A'_k)^{-1}
        \begin{bmatrix}
            A_k & A'_k
        \end{bmatrix}
        \qquad\mbox{and}\qquad
        \tilde A'
        :=
        \begin{bmatrix}
            A'_k \\
            A'''_k
        \end{bmatrix}
        ((A'_k)^{-1})^2
        \begin{bmatrix}
            A_k & A'_k
        \end{bmatrix}
    \end{align*}
so that we have $B_{k,x}^{-1}\equiv B_{k,0}^{-1}-\tilde A-y\tilde A' \pmod{y^2}$. Then we have
    \begin{align*}
        \tilde A
        =
        \begin{bmatrix}
            A_k & A'_k \\
            A'''_k(A'_k)^{-1}A_k & A'''_k
        \end{bmatrix}.
    \end{align*}
By Lemma~\ref{lemma-QAinv}~(ii),
    \begin{equation}\label{eq: inverse of B_k,x-tilde A}
        B_{k,0}^{-1}-\tilde A
        =
        \begin{bmatrix}
            0_{t,k-t} & 0_{t,t} \\
            A''_k-A'''_k(A'_k)^{-1}A_k & 0_{k-t,t}
        \end{bmatrix}
        =
        \begin{bmatrix}
            0_{t,k-t} & 0_{t,t} \\
            B_{k-t,0}^{-1} & 0_{k-t,t}
        \end{bmatrix}.
    \end{equation}
For $\tilde A'$, we observe that
    \begin{equation}\label{eq: tilde A prime}
        \tilde A'
        =B_{k,0}^{-1}U_t((A'_k)^{-1})^2V_tB_{k,0}^{-1}=
        \begin{bmatrix}
            (A'_k)^{-1}A_k & I_t \\
            A'''_k((A'_k)^{-1})^2A_k & A'''_k(A'_k)^{-1}
        \end{bmatrix}.
    \end{equation}
Then it is easy to see that
    \begin{align*}
        M_{r',0}
        \begin{bmatrix}
            0_{r'-k,t} \\
            \tilde A'
        \end{bmatrix}
        =
        \begin{bmatrix}
            A_{k,0} & B_{k,0} \\
            C_{k,0} & D_{k,0}
        \end{bmatrix}
        \begin{bmatrix}
            0_{r'-k,t} \\
            \tilde A'
        \end{bmatrix}
        =
        \begin{bmatrix}
          B_{k,0}\tilde A' \\
          D_{k,0}\tilde A'
        \end{bmatrix}.
    \end{align*}
On the upper block, from the first equality of \eqref{eq: tilde A prime} we observe
    \begin{align*}
        B_{k,0}\tilde{A}'
        =
        \begin{bmatrix}
            0_{k-t,t} \\
            I_t
        \end{bmatrix}
        ((A'_k)^{-1})^2V_tB_{k,0}^{-1}
        =
        \begin{bmatrix}
            0_{k-t,k} \\
            ((A'_k)^{-1})^2V_tB_{k,0}^{-1}
        \end{bmatrix}.
    \end{align*}

Now we compute $P_{k,i,s_{k,i}}(T)$ for $0\leq i\leq k-1$. Recall that $\mathbf{e}'_1,\cdots,\mathbf{e}'_k$ be the standard basis of $\Q^k$, and that we have
    \begin{align*}
        P_{k,i}(T,x)
        =
        \begin{bmatrix}
            T^{k-1} & \cdots & 1
        \end{bmatrix}
        B_{k,x}^{-1}\mathbf{e}'_{k-i}.
    \end{align*}
From Lemma~\ref{lemma-QAinv}~(iii), we have
    \begin{align*}
        d_k(x)
       =\det(y^{-1}I_t+(A'_k)^{-1})
       \equiv
        y^{-t}+y^{-t+1}\Tr((A'_k)^{-1})
        \pmod{y^{-t+2}},
    \end{align*}
and so
    \begin{align*}
        \tilde P_{k,i}(T,x)
        &=
        \sum_{s=0}^{s_{k,i}}y^{-s} P_{k,i,s}(T)
        =d_k(x)P_{k,i}(T,x)
        \\&\equiv
        \big(y^{-t}+y^{-t+1}\Tr((A'_k)^{-1})\big)
        \begin{bmatrix}
            T^{k-1} & \cdots & 1
        \end{bmatrix}
        (B_{k,0}^{-1}-\tilde A-y\tilde A')
        \mathbf{e}'_{k-i}\pmod{y^{-t+2}}
    \end{align*}
Hence, we observe that $s_{k,i}\leq t=2k-r'$.

We let
    \begin{align*}
        g_{k,i}(T)
        &:=
        \begin{bmatrix}
            T^{k-1} & \cdots & 1
        \end{bmatrix}
        (B_{k,0}^{-1}-\tilde A)
        \mathbf{e}'_{k-i};
        \\
        h_{k,i}(T)
        &:=
        \begin{bmatrix}
            T^{k-1} & \cdots & 1
        \end{bmatrix}
        \tilde A'
        \mathbf{e}'_{k-i},
    \end{align*}
so that we have
    \begin{align}\label{eq: tilde P_k,i}
    \begin{split}
        \tilde P_{k,i}(T,x)
        &\equiv
        y^{-t}\big(1+y\Tr((A'_k)^{-1})\big)
        \big(g_{k,i}(T)-yh_{k,i}(T)\big)
        \\
        &\equiv
        y^{-t}
        g_{k,i}(T)
        +
        y^{-t+1}
        \big(
        \Tr((A'_k)^{-1})g_{k,i}(T)
        -h_{k,i}(T)
        \big)\pmod{y^{-t+2}}.
    \end{split}
    \end{align}

Consider first the case $0\leq i\leq t-1$. We see that $g_{k,i}(T)=0$ as the $(k-i)$-th column of $B_{k,0}^{-1}-\tilde A$ is zero (see \eqref{eq: inverse of B_k,x-tilde A}). As $(k-i)$-th column of $\tilde A'$ is not zero, $h_{k,i}(T)\neq 0$. So, we get $s_{k,i}=t-1$ (from \eqref{eq: tilde P_k,i}) and
    \begin{align*}
        P_{k,i,t-1}(T)
        =h_{k,i}(T).
    \end{align*}
We further characterize $h_{k,i}(T)$. Observe that, from the definition of $h_{k,i}(T)$ together with \eqref{eq: tilde A prime}, we can easily see that
    \begin{align*}
        \deg_T\big(h_{k,i}(T)-T^{i+k-t}\big)
        \leq
        k-t-1.
    \end{align*}
Also, we see that
    \begin{align*}
        [Tf_{r'}(T)h_{k,i}(T)]_{T}^{(r'-1)}
        &=
        \begin{bmatrix}
            T^{k-1} & \cdots & 1
        \end{bmatrix}
        M_{r',0}
        \begin{bmatrix}
            0_{r'-k,t} \\
            \tilde A'
        \end{bmatrix}
        \mathbf{e}'_{k-i}
        \\&=
        \begin{bmatrix}
            T^{k-1} & \cdots & 1
        \end{bmatrix}
        \begin{bmatrix}
          B_{k,0}\tilde A' \\
          D_{k,0}\tilde A'
        \end{bmatrix}
        \mathbf{e}'_{k-i}
        \\&=
        \begin{bmatrix}
            T^{k-1} & \cdots & 1
        \end{bmatrix}
        \begin{bmatrix}
            0_{k-t,k} \\
            ((A'_k)^{-1})^2V_tB_{k,0}^{-1}\\
            D_{k,0}\tilde A'
        \end{bmatrix}
        \mathbf{e}'_{k-i}
    \end{align*}
and so
    \begin{align*}
        \deg_T [Tf_{r'}(T)h_{k,j}(T)]_{T}^{(r'-1)} \leq (r'-1)-(k-t).
    \end{align*}
Noting that $k-t=r'-k$, we observe that
    \begin{align*}
        \begin{cases}
            \deg_T (h_{k,i}(T)-T^{i+r'-k})\leq r'-k-1,\\
            \deg_T [Tf_{r'}(T)h_{k,i}(T)]_{T}^{(r'-1)} \leq k-1.
        \end{cases}
    \end{align*}
According to Lemma~\ref{lemma: PQdegree}, $P_{r'-k,i+r'-k}(T,0)$ is the unique polynomial satisfying such properties.
Therefore, we have
    \begin{align*}
        P_{k,i,t-1}(T)
        =h_{k,i}(T)
        =P_{r'-k,i+r'-k}(T,0).
    \end{align*}

We now consider the case $t\leq i\leq k-1$. From the definition of $g_{k,i}(T)$ together with \eqref{eq: inverse of B_k,x-tilde A}, we see that $g_{k,i}(T)\neq 0$ and $s_{k,i}=t$. Moreover, from \eqref{eq: tilde P_k,i} we have
    \begin{align*}
        P_{k,i,t}(T) =g_{k,i}(T).
    \end{align*}
Setting $\mathbf{e}''_1,\cdots,\mathbf{e}''_{k-t}$ be the standard basis of $\Q^{k-t}$, we have
    \begin{align*}
        g_{k,i}(T)
        =
        \begin{bmatrix}
            T^{k-t-1} & \cdots & 1
        \end{bmatrix}
        B_{k-t,0}^{-1}
        \mathbf{e}''_{k-i}=P_{k-t,i-t}(T,0).
    \end{align*}
Noting that $k-t=r'-k$, we get
    \begin{align*}
        P_{k,i,t}(T)=P_{r'-k,i+r'-2k}(T,0).
    \end{align*}

We finally consider the case $k\leq i\leq r'-1$. For $k\leq i\leq r'-1$, we have
    \begin{align*}
        P_{k,i}(T,x)
        &=
        T^i
        -
        \begin{bmatrix}
            T^{k-1} & \cdots & 1
        \end{bmatrix}
        B_{k,x}^{-1}A_{k,x}
        \mathbf{e}''_{r'-i}.
    \end{align*}
Set $y=x^{-1}$ as before, then
    \begin{align*}
        A_{k,x}\mathbf{e}''_{r'-i}
        &=
        x
        \begin{bmatrix}
           I_{k-t} \\
           0_{t,k-t}
        \end{bmatrix}
        \mathbf{e}''_{r'-i}
        +
        A_{k,0}\mathbf{e}''_{r'-i}
        =
        y^{-1}\mathbf{e}'_{r'-i}
        +A_{k,0}\mathbf{e}''_{r'-i}
        \\
        &\equiv
        y^{-1}\mathbf{e}'_{r'-i}
        \pmod{y^0}.
    \end{align*}
Using $B_{k,x}^{-1}\in \M_{k\times k}(\Q[[y]])$, one get $P_{k,i-r'+k}(T,x)\in (\Q[[y]])[T]$, and
    \begin{align*}
        P_{k,i}(T,x)
        &\equiv
        -
        \begin{bmatrix}
            T^{k-1} & \cdots & 1
        \end{bmatrix}
        B_{k,x}^{-1}(y^{-1}\mathbf{e}'_{r'-i})
        \\
        &\equiv
        -y^{-1}
        \begin{bmatrix}
            T^{k-1} & \cdots & 1
        \end{bmatrix}
        B_{k,x}^{-1}\mathbf{e}'_{k-(i-r'+k)}
        \\
        &\equiv
        -y^{-1}P_{k,i-r'+k}(T,x)
        \pmod{y^0}.
    \end{align*}
So this gives
    \begin{align*}
        \frac{\tilde P_{k,i}(T,x)}{d_k(x)}
        \equiv
        -y^{-1}\frac{\tilde P_{k,i-r'+k}(T,x)}{d_k(x)}
        \pmod{y^0}.
    \end{align*}
As $d_k(x)$ is monic of degree $t$, one has
    \begin{align*}
        d_k(x)
        \equiv
        y^{-t} \pmod{y^{-t+1}}.
    \end{align*}
Also, as $t\leq i-r'+k\leq k-1$, one has
    \begin{align*}
        \tilde P_{k,i-r'+k}(T,x)
        \equiv
        y^{-t}P_{r'-k,i-k}(T,0) \pmod{y^{-t+1}}.
    \end{align*}
Combining these gives
    \begin{align*}
        \tilde P_{k,i}(T,x)
        &\equiv -y^{-1}\tilde P_{k,i-r'+k}(T,x)
        \\&\equiv
        -y^{-t-1}P_{r'-k,i-k}(T,0) \pmod{y^{-t}}.
    \end{align*}
Therefore, we get $s_{k,i}=t+1$ and
    \begin{align*}
        P_{k,i,s_{k,i}}(T)=-P_{r'-k,i-k}(T,0)
    \end{align*}
for $k\leq i\leq r'-1$.

Lastly, from $\tilde Q_{k,i}(T,x)=\big[\tilde P_{k,i}(T,x)(x+Tf_{r'}(T))\big]_{T}^{(r'-1)}$ we have
    \begin{multline*}
        \tilde Q_{k,i}(T,x)=
        x^{s_{k,i}+1} P_{k,i,s_{k,i}}(T)\\ +\sum_{s=1}^{s_{k,i}}x^s\big[P_{k,i,s}(T)Tf_{r'}(T)+P_{k,i,s-1}(T)\big]_{T}^{(r'-1)}
        +\big[P_{k,i,0}(T)Tf_{r'}(T)\big]_{T}^{(r'-1)},
    \end{multline*}
and so we see that $t_{k,i}=s_{k,i}+1$ and $Q_{k,i,t_{k,i}}(T)=P_{k,i,s_{k,i}}(T)$.
\end{proof}

We close this subsection by proving the following lemma, which will be needed to compute the mod-$p$ reduction of our strongly divisible modules.
\begin{lemm}\label{lemma: a_k m+1 leq k}
    Let $m+1\leq k\leq r'$. Then
    \begin{enumerate}[leftmargin=*]
        \item if $k\leq i\leq r'-1$ then $P_{k,i,2k-r'}(T)$ is monic of degree $i$;
        \item if $2k-r'\leq i\leq k-1$ then $Q_{k,i,2k-r'}(T)$ is monic of degree $i+r'-k$;
        \item if $m+2\leq k\leq r'$ then
        \begin{equation*}
            P_{k,0,2k-r'-2}(T)=-\frac{1}{a_{r'-k+1}}P_{k-1,k-1,2k-r'-2}(T)
        \end{equation*}
        whose degree is $k-1$ and whose leading coefficient is $-1/a_{r'-k+1}\in\Z_{(p)}^{\times}$;
        \item if $k=m+1$ and $r'=2m$ then
        \begin{equation*}
            P_{m+1,0,0}(T)=-\frac{1}{a_m}P_{m,m}(T,0)
        \end{equation*}
        whose degree is $m$ and whose leading coefficient is $-1/a_m\in\Z_{(p)}^{\times}$;
        \item if $k=m+1$ and $r'=2m+1 >1$ then $$\frac{1}{a_{m+1}}=d_{m+1}(0)=2H_m\in\Z_{(p)}.$$
    \end{enumerate}
\end{lemm}

Note that we exclude $r'=1$ in (v) because $a_1(x)=x^{-1}$ and so $a_1(0)$ cannot be defined.

\begin{proof}
(i) For $k\leq i\leq r'-1$, consider
    \begin{align*}
        d_k(x)P_{k,i}(T,x)
        =\sum_{s=0}^{2k-r'+1}x^sP_{k,i,s}(T).
    \end{align*}
As $P_{k,i}(T,x)$ is monic of degree $i$ in $T$ by Lemma~\ref{lemma: PQdegree} and $d_k(x)$ is monic of degree $2k-r'$ in $x$, the coefficient of $x^{2k-r'}T^i$ on the left-hand side is $1$. On the right-hand side, the term $x^{2k-r'}T^i$ comes only from $x^{2k-r'}P_{k,i,2k-r'}(T)$. Therefore, $P_{k,i,2k-r'}(T)$ should be monic of degree $i$.

(ii) For $2k-r'-1\leq i\leq k-1$, we have $t_{k,i}-1=2k-r'$. Consider
    \begin{align*}
        d_k(x)Q_{k,i}(T,x)
        =\sum_{s=0}^{t_{k,i}}x^sQ_{k,i,s}(T).
    \end{align*}
As $Q_{k,i}(T,x)$ is monic of degree $i+r'-k$ in $T$ by Lemma~\ref{lemma: PQdegree} and $d_k(x)$ is monic of degree $2k-r'$ in $x$, the coefficient of $x^{2k-r'}T^{i+r'-k}$ on the left-hand side is $1$. On the right-hand side, the term $x^{2k-r'}T^{i+r'-k}$ comes only from $x^{2k-r'}Q_{k,i,2k-r'}(T)$. Therefore, $Q_{k,i,2k-r'}(T)$ should be monic of degree $i+r'-k$.

(iii) We compute $a_k(x)$ in Lemma~\ref{lemma: PQFad'e}~(iii). From
        $$P_{k,0}(T,x)=
        \begin{bmatrix}
            T^{k-1} & \cdots & 1
        \end{bmatrix}
        B_{k,x}^{-1}\mathbf{e}'_{k}$$
one gets
    \begin{align*}
        a_k(x)
        =(\mathbf{e}'_1)^T B_{k,x}^{-1}\mathbf{e}'_{k}
        =\frac{1}{\det(B_{k,x})}(\mathbf{e}'_1)^T \adj(B_{k,x}) \mathbf{e}'_{k}
        =(-1)^{1+k}\frac{\det(B_{k-1,x})}{\det(B_{k,x})}
    \end{align*}
because the $(k,1)$-minor of $B_{k,x}$ is $\det(B_{k-1,x})$. Hence, by Lemma~\ref{lemma: PQFad'e}~(iii) we have
    \begin{align*}
        \det(B_{k,x})P_{k,0}(T,x)=(-1)^{k+1}\det(B_{k-1,x})P_{k-1,k-1}(T,x).
    \end{align*}
Equivalently,
    \begin{align*}
        \det(B_{r'-k,0})\tilde P_{k,0}(T,x)
        =(-1)^{r'-k+1}\det(B_{r'-k+1,0})\tilde P_{k-1,k-1}(T,x).
    \end{align*}
Noting that $s_{k-1,k-1}=2(k-1)-r'+1=2k-r'-1$ and $s_{k,0}=2k-r'-1$, we can write
    \begin{align*}
        \sum_{s=0}^{2k-r'-1}x^sP_{k,0,s}(T)
        =\sum_{s=0}^{2k-r'-1}x^s
        \frac{(-1)^{r'-k+1}\det(B_{r'-k+1,0})}{\det(B_{r'-k,0})}P_{k-1,k-1,s}(T),
    \end{align*}
and so we have
    \begin{align}\label{eq: P_k,0,s and P_k-1,k-1,s}
        P_{k,0,s}(T)=\frac{(-1)^{r'-k+1}\det(B_{r'-k+1,0})}{\det(B_{r'-k,0})}P_{k-1,k-1,s}(T).
    \end{align}
If $s=2k-r'-1$, by Lemma~\ref{lemma: m+1 to r} and Lemma~\ref{lemma: PQFad'e}~(iii) one has
    \begin{align*}
        P_{k,0,2k-r'-1}(T)
        &=P_{r'-k,r'-k}(T,0)\\
        &=\frac{1}{a_{r'-k+1}}P_{r'-k+1,0}(T,0)
        =-\frac{1}{a_{r'-k+1}}P_{k-1,k-1,2k-r'-1}(T).
    \end{align*}
Therefore, we have
    \begin{align}\label{eq: a_r-k+1}
        \frac{(-1)^{r'-k+1}\det(B_{r'-k+1,0})}{\det(B_{r'-k,0})}=-\frac{1}{a_{r'-k+1}}
    \end{align}
    and
    \begin{align*}
        P_{k-1,k-1,2k-r'-2}(T)
        =-\frac{1}{a_{r'-k+1}}P_{k,0,2k-r'-2}(T).
    \end{align*}

(iv) Now we let $r'=2m$ and $k=m+1$. As before, we get
    \begin{align*}
        \det(B_{m+1,x})P_{m+1,0}(T,x)=(-1)^{m+2}\det(B_{m,x})P_{m,m}(T,x).
    \end{align*}
Applying Lemma~\ref{lemma-PQ-k leq m} and using $\det(B_{m,x})=\det(B_{m,0})$, one has
    \begin{align*}
        (-1)^{r'(m-1)}\det(B_{m-1,0})\tilde P_{m+1,0}(T,x)
        =(-1)^{m+2}\det(B_{m,0})\big(P_{m,m}(T,0)-xP_{m,0}(T,0)\big)
    \end{align*}
    and so
    \begin{align*}
        P_{m+1,0,0}(T)+xP_{m+1,0,1}(T)
        =\frac{(-1)^m\det(B_{m,0})}{\det(B_{m-1,0})}
        \big(P_{m,m}(T,0)-xP_{m,0}(T,0)\big)
    \end{align*}
As
    \begin{align*}
        P_{m+1,0,1}(T)
        =P_{m-1,m-1}(T,0)
        =\frac{1}{a_m}P_{m,0}(T,0),
    \end{align*}
we can rewrite the equation as
    \begin{align*}
        P_{m+1,0,0}(T)+xP_{m+1,0,1}(T)
        =-\frac{1}{a_m}\big(P_{m,m}(T,0)-xP_{m,0}(T,0)\big),
    \end{align*}
proving the assertion.

(v) Let $r'=2m+1 >1$ and $k=m+1$. By Lemma~\ref{lemma: PQFad'e}~(iii) and by Lemma~\ref{lemma: m+1 to r} we have $$P_{m+1,0,0}(T)=P_{m.m}(T,0)=\frac{1}{a_{m+1}}P_{m+1,0}(T,0).$$
On the other hand, by definition we have $$P_{m+1,0,0}(T)=d_{m+1}(0)P_{m+1,0}(T,0).$$
By Lemma~\ref{lemma: PQdegree} $P_{m,m}(T,0)$ is a monic polynomial of degree $m$, and so $a_{m+1}d_{m+1}(0)=1$ in $\Q$.

We now compute $d_{m+1}(0)$. We set the $(m+1)\times(m+1)$-matrix $M$ as follows:
\begin{align*}
        M
        =
        \begin{bmatrix}
          \frac{1}{m} & \frac{1}{m+1} & \cdots & \frac{1}{2m-1} & \frac{1}{2m} \\
          \frac{1}{m-1} & \frac{1}{m} & \cdots & \frac{1}{2m-2} & \frac{1}{2m-1} \\
          \vdots & \vdots & \ddots & \vdots & \vdots \\
          1 & \frac{1}{2} & \cdots & \frac{1}{m} & \frac{1}{m+1} \\
          0 & 1 & \cdots & \frac{1}{m-1} & \frac{1}{m}
        \end{bmatrix}.
\end{align*}
We also let $M'$ (resp. $M''$) be the matrix obtained from $M$ by replacing the $(m+1,1)$-entry of $M$ with $2H_m$ (resp. by deleting the first column and the last row of $M$). As we have
\begin{align*}
        B_{m+1,0}
        =\mathrm{Diag}((-1)^{m-1},\cdots,(-1)^{0}, (-1)^{-1})
        \cdot
        M
        \cdot
        \mathrm{Diag}((-1)^0,\cdots,(-1)^{m})
\end{align*}
and
\begin{align*}
B_{m,0}=\mathrm{Diag}((-1)^{m},\cdots,(-1)^{2}, (-1)^{1})
        \cdot
        M''
        \cdot
        \mathrm{Dig}((-1)^0,\cdots,(-1)^{m-1}),
\end{align*}
we have $$d_{m+1}(0)=(-1)^{rm}\frac{\det B_{m+1,0}}{\det B_{m,0}}=(-1)^{m+1}\frac{\det M}{\det M''}.$$
By definition of the determinant, it is easy to see that $$(-1)^{m+1}\det M''\left(d_{m+1}(0)-2H_m\right)=\det M'.$$

We now claim that $\det M'=0$, and we prove it by induction on $m$. By subtracting the first row of $M'$ from each other row, we get
    \begin{align*}
       D(m):= \det M'=
        \begin{vmatrix}
          \frac{1}{m} & \frac{1}{m+1} & \cdots & \frac{1}{2m-1} & \frac{1}{2m}\\
          \frac{1}{m-1}-\frac{1}{m} & \frac{1}{m}-\frac{1}{m+1} & \cdots
           & \frac{1}{2m-2}-\frac{1}{2m-1} & \frac{1}{2m-1}-\frac{1}{2m} \\
          \vdots & \vdots & \ddots & \vdots & \vdots \\
          1-\frac{1}{m} & \frac{1}{2}-\frac{1}{m+1} & \cdots
           & \frac{1}{m}-\frac{1}{2m-1} & \frac{1}{m+1}-\frac{1}{2m} \\
          2H_m-\frac{1}{m} & 1-\frac{1}{m+1} & \cdots
           & \frac{1}{m-1}-\frac{1}{2m-1} & \frac{1}{m}-\frac{1}{2m}
        \end{vmatrix}
    \end{align*}
which can be rewritten as
\begin{align*}
    D(m)=
        \begin{vmatrix}
          \frac{1}{m} & \frac{1}{m+1} & \cdots & \frac{1}{2m-1} & \frac{1}{2m} \\
          \frac{1}{(m-1)m} & \frac{1}{m(m+1)} & \cdots
           & \frac{1}{(2m-2)(2m-1)} & \frac{1}{(2m-1)2m} \\
          \vdots & \vdots & \ddots & \vdots & \vdots \\
          \frac{m-1}{1\cdot m} & \frac{m-1}{2(m+1)} & \cdots
           & \frac{m-1}{m(2m-1)} & \frac{m-1}{(m+1)2m} \\
          H_m+H_{m-1} & \frac{m}{1(m+1)} & \cdots
           & \frac{m}{(m-1)(2m-1)} & \frac{m}{m\cdot2m}
        \end{vmatrix}.
\end{align*}
Factoring out the common factors from each column and each row, we have
    \begin{align*}
        D(m)=
        \frac{m!}{m(m+1)\cdots 2m}
        \begin{vmatrix}
          1 & 1 & \cdots & 1 & 1 \\
          \frac{1}{m-1} & \frac{1}{m} & \cdots
           & \frac{1}{2m-2} & \frac{1}{2m-1} \\
          \vdots & \vdots & \ddots & \vdots & \vdots \\
          1 & \frac{1}{2} & \cdots
           & \frac{1}{m} & \frac{1}{m+1} \\
          H_m+H_{m-1} & 1 & \cdots
           & \frac{1}{m-1} & \frac{1}{m}
        \end{vmatrix}.
    \end{align*}
For the last determinant, we subtract the last column from each other column, to get
    \begin{align*}
        \frac{(2m)!}{m!(m-1)!} D(m)
        &=
        \begin{vmatrix}
          0 & 0 & \cdots & 0 & 1 \\
          \frac{1}{m-1}-\frac{1}{2m-1} & \frac{1}{m}-\frac{1}{2m-1} & \cdots
           & \frac{1}{2m-2}-\frac{1}{2m-1} & \frac{1}{2m-1} \\
          \vdots & \vdots & \ddots & \vdots & \vdots \\
          1-\frac{1}{m+1} & \frac{1}{2}-\frac{1}{m+1} & \cdots
           & \frac{1}{m}-\frac{1}{m+1} & \frac{1}{m+1} \\
          H_m+H_{m-1}-\frac{1}{m} & 1-\frac{1}{m} & \cdots
           & \frac{1}{m-1}-\frac{1}{m} & \frac{1}{m}
        \end{vmatrix}\\
        &=
        (-1)^m
        \begin{vmatrix}
          \frac{m}{(m-1)(2m-1)} & \frac{m-1}{m(2m-1)} & \cdots
           & \frac{1}{(2m-2)(2m-1)} \\
          \vdots & \vdots & \ddots & \vdots \\
          \frac{m}{1(m+1)} & \frac{m-1}{2(m+1)} & \cdots
           & \frac{1}{m(m+1)} \\
          2H_{m-1} & \frac{m-1}{1\cdot m} & \cdots
           & \frac{1}{(m-1)m}
        \end{vmatrix}\\
        &=
        \frac{(-1)^m m!}{m(m+1)\cdots(2m-1)}
        \begin{vmatrix}
          \frac{1}{m-1} & \frac{1}{m} & \cdots
           & \frac{1}{2m-2} \\
          \vdots & \vdots & \ddots & \vdots \\
          1 & \frac{1}{2} & \cdots
           & \frac{1}{m} \\
          2H_{m-1} & 1 & \cdots
           & \frac{1}{m-1}
        \end{vmatrix}\\
        &=(-1)^m\frac{m!(m-1)!}{(2m-1)!}D(m-1).
    \end{align*}
It is clear that $D(1)=0$ by direct computation, and so by induction on $m$ we conclude that $D(m)=0$ for all $m\geq 1$. Hence, we have $d_{m+1}(0)=2H_m$.
\end{proof}

\subsection{The elements $\delta_{k}\in S$}\label{subsec: the elements delta}
In this subsection, we define $\delta_{k}(T,x)$ as well as $\delta_k(T)$ for $1\leq k\leq m+1$, which will appear as the coefficients of our strongly divisible modules. Recall that by $r'$ we mean a positive integer with $1\leq r'<p-1$ and that $m=\lfloor r'/2\rfloor$.

\begin{defi}\label{defi: definition of delta}
Set $\delta_0(T,x):=-x$. For $1\leq k\leq m+1$, we set
\begin{equation*}
    \delta_k(T,x) :=\frac{Q_{k,0}(T,x)}{P_{k,0}(T,x)}-x\quad\mbox{ and }\quad \delta_k(T):=\delta_k(T,0).
\end{equation*}
\end{defi}

Note that by Lemma~\ref{lemma: PQFad'e} (iii) if $k\geq 1$ we have
    \begin{equation*}
        \delta_k(T,x)=\frac{Q_{k-1,k-1}(T,x)}{P_{k-1,k-1}(T,x)}-x,
    \end{equation*}
so in particular $\delta_1(T,x)=0$ if $r'=1$. Moreover, unless either $r'=2m$ and $k=m+1$ or $k=0$ we have $$\delta_{k}(T,x)=\delta_k(T)$$
by Lemma~\ref{lemma-PQ-k leq m} for $1\leq k\leq m$ and by Lemma~\ref{lemma: m+1 to r} for $k=m+1$ and $r'=2m+1$, and if $r'=2m$ and $k=m+1$ then
    \begin{equation*}
        \delta_{m+1}(T,x)=\frac{Q_{m,m}(T,0)-xQ_{m,0}(T,0)}{P_{m,m}(T,0)-xP_{m,0}(T,0)}=\frac{[\tilde P_{m+1,0}(T,x)Tf_{r'}(T)]_{T}^{(r'-1)}}{\tilde P_{m+1,0}(T,x)}
    \end{equation*}
where the first equality follows from Lemma~\ref{lemma-PQ-k leq m} and the second equality follows from Lemma~\ref{lemma: PQdegree}~(i) if $m+1<r'$ and from Lemma~\ref{lemma: k=0,k=r} if $m+1=r'$.

For the rest of this subsection, we assume $r'>1$ and investigate some necessary properties of the coefficients. In particular, we compute the values of $\delta_k(-1)$, $\dot{\delta}_k(-1)$ and their differences by series of lemmas.

\begin{lemm}\label{lemma-delta-denom-1}
Assume that $r'\geq 2$.
\begin{enumerate}[leftmargin=*]
  \item If $1\leq k\leq m$, then
    \begin{equation*}
        P_{k,0}(-1,0)=(-1)^{r'}\frac{(r'-k)!}{(r'-2k)!(k-1)!}\in\Z_{(p)}^\times.
    \end{equation*}
    In particular, $P_{k,0}(T,0)\in \Z_p[[T+1]]^\times$, and so
    $P_{k,0}(c,0)\in(S'_\cO)^\times$ and $$\delta_k(T)=\frac{Q_{k,0}(T,0)}{P_{k,0}(T,0)}\in\cO[[T+1]].$$
  \item If $2\leq k<r'/2+1$, then
    \begin{equation*}
        P_{k-1,k-1}(-1,0) =\frac{P_{k,0}(-1,0)}{a_k} =\frac{(r'-2k+1)!(k-1)!}{(r'-k)!} \in\Z_{(p)}^{\times}.
    \end{equation*}
    In particular, $P_{k-1,k-1}(T,0)\in \Z_p[[T+1]]^\times$ and so $P_{k-1,k-1}(c,0)\in(S'_\cO)^\times$.
  \item If $1\leq k\leq r'-1$, then we have
  \begin{align*}
    P_{k,k}(T,x)Q_{k,0}(T,x)-P_{k,0}(T,x)Q_{k,k}(T,x)=T^{r'}.
  \end{align*}
\end{enumerate}
\end{lemm}

\begin{proof}
Let $1\leq k< r'/2+1$, and consider the linear equation
    \begin{align*}
        \begin{bmatrix}
            y_{k-1} & y_{k-2} & \cdots & y_0
        \end{bmatrix}
        B_{k,0}
        =
        \begin{bmatrix}
            (-1)^{k-1} & (-1)^{k-2} & \cdots & 1
        \end{bmatrix}.
    \end{align*}
Solving this equation for $y_0$ gives
    \begin{align*}
        y_0
        =
        \begin{bmatrix}
            (-1)^{k-1} & (-1)^{k-2} & \cdots & 1
        \end{bmatrix}
        B_{k,0}^{-1}
        \mathbf{e}'_{k}
        =P_{k,0}(-1,0).
    \end{align*}
On the other hand, using Cram\'er's rule we have
$$y_0=\frac{1}{\det(B_{k,0})}\det(B_{k,0}')$$
where
    \begin{align*}
        B_{k,0}':=
        \begin{bmatrix}
          \frac{(-1)^{r'-k-1}}{r'-k} & \frac{(-1)^{r'-k}}{r'-k+1} & \cdots & \frac{(-1)^{r'-3}}{r'-2} & \frac{(-1)^{r'-2}}{r'-1} \\
          \frac{(-1)^{r'-k-2}}{r'-k-1} & \frac{(-1)^{r'-k-1}}{r'-k} & \cdots & \frac{(-1)^{r'-4}}{r'-3} & \frac{(-1)^{r'-3}}{r'-2} \\
          \vdots & \vdots & \ddots & \vdots & \vdots \\
          \frac{(-1)^{r'-2k+1}}{r'-2k+2} & \frac{(-1)^{r'-2k+2}}{r'-2k+3} & \cdots & \frac{(-1)^{r'-k-1}}{r'-k} & \frac{(-1)^{r'-k}}{r'-k+1} \\
          (-1)^{k-1} & (-1)^{k-2} & \cdots & (-1) & 1
        \end{bmatrix}.
    \end{align*}

Let
$$M=
        \begin{bmatrix}
          \frac{1}{r'-k} & \frac{1}{r'-k+1} & \cdots & \frac{1}{r'-2} & \frac{1}{r'-1} \\
          \frac{1}{r'-k-1} & \frac{1}{r'-k} & \cdots & \frac{1}{r'-3} & \frac{1}{r'-2} \\
          \vdots & \vdots & \ddots & \vdots & \vdots \\
          \frac{1}{r'-2k+2} & \frac{1}{r'-2k+3} & \cdots & \frac{1}{r'-k} & \frac{1}{r'-k+1} \\
          M_{k,1} & \frac{1}{r'-2k+2} & \cdots & \frac{1}{r'-k-1} & \frac{1}{r'-k}
        \end{bmatrix},$$
where $M_{k,1}$ is $\frac{1}{r'-2k+1}$ if $r'-2k+1>0$ and $0$ otherwise,
and let $M'$ be the square matrix of size $k\times k$ obtained from $M$ by replacing the last row with
$\begin{bmatrix}
  1 & 1 & \cdots & 1
\end{bmatrix}.$
Noting that
\begin{align*}
        B_{k,0}
        &=\mathrm{Diag}((-1)^{r'-k-1},\cdots,(-1)^{r'-2k+1}, (-1)^{r'-2k})
        \cdot
        M
        \cdot
        \mathrm{Diag}((-1)^0,\cdots,(-1)^{k-1})
\end{align*}
and
\begin{align*}
B_{k,0}'=\mathrm{Diag}((-1)^{r'-k-1},\cdots,(-1)^{r'-2k+1}, (-1)^{k-1})
        \cdot
        M'
        \cdot
        \mathrm{Dig}((-1)^0,\cdots,(-1)^{k-1}),
\end{align*}
we have
    \begin{equation}\label{eq: P_k,0(-1,0)}
        P_{k,0}(-1,0)=y_0
        =
        (-1)^{r'-k-1}
        \frac{\det M'}{\det M}.
    \end{equation}

(i) We now compute $\frac{\det M'}{\det M}$ for $1\leq k\leq m$ as follows: subtracting the $k$-th row from each row of $M$, we have
    \begin{align*}
        \det M=
        \begin{vmatrix}
          \frac{-(k-1)}{(r'-k)(r'-2k+1)} & \frac{-(k-1)}{(r'-k+1)(r'-2k+2)} & \cdots
           & \frac{-(k-1)}{(r'-2)(r'-k-1)} & \frac{-(k-1)}{(r'-1)(r'-k)} \\
          \frac{-(k-2)}{(r'-k-1)(r'-2k+1)} & \frac{-(k-2)}{(r'-k)(r'-2k+2)} & \cdots
           & \frac{-(k-2)}{(r'-3)(r'-k-1)} & \frac{-(k-2)}{(r'-2)(r'-k)} \\
          \vdots & \vdots & \ddots & \vdots & \vdots \\
          \frac{-1}{(r'-2k+2)(r'-2k+1)} & \frac{-1}{(r'-2k+3)(r'-2k+2)} & \cdots
           & \frac{-1}{(r'-k)(r'-k-1)} & \frac{-1}{(r'-k+1)(r'-k)} \\
          \frac{1}{r'-2k+1} & \frac{1}{r'-2k+2} & \cdots & \frac{1}{r'-k-1} & \frac{1}{r'-k}
        \end{vmatrix}.
    \end{align*}
Factoring out $-(k-i)$ from each $i$-th row and $\frac{1}{r'-2k+j}$ from each $j$-th column, we get
    \begin{align*}
    \det M
        &=
        \frac{(-1)^{k-1} (k-1)!}{(r'-2k+1)(r'-2k+2)\cdots(r'-k)}
        \begin{vmatrix}
          \frac{1}{r'-k} & \frac{1}{r'-k+1} & \cdots & \frac{1}{r'-2} & \frac{1}{r'-1} \\
          \frac{1}{r'-k-1} & \frac{1}{r'-k} & \cdots & \frac{1}{r'-3} & \frac{1}{r'-2} \\
          \vdots & \vdots & \ddots & \vdots & \vdots \\
          \frac{1}{r'-2k+2} & \frac{1}{r'-2k+3} & \cdots & \frac{1}{r'-k} & \frac{1}{r'-k+1} \\
          1 & 1 & \cdots & 1 & 1
        \end{vmatrix}\\
        &=
        \frac{(-1)^{k-1} (r'-2k)!(k-1)!}{(r'-k)!}\det M'.
    \end{align*}
Therefore, this together with \eqref{eq: P_k,0(-1,0)} gives the desired identity of (i).

(ii) From $P_{k,0}(T,0)=
        \begin{bmatrix}
            T^{k-1} & \cdots & 1
        \end{bmatrix}
        B_{k,0}^{-1}
        \mathbf{e}'_k$, we have
\begin{align*}
        a_k
        =(\mathbf{e}'_1)^T B_{k,0}^{-1} \mathbf{e}'_k
        =(-1)^{k+1}\frac{\det(B_{k-1,0})}{\det(B_{k,0})}.
\end{align*}
If we let $M''$ be the square matrix of size $(k-1)\times (k-1)$ obtained from $M$ by deleting the first column and the last row, then it is immediate that
\begin{align*}
        B_{k-1,0}
        &=\mathrm{Diag}((-1)^{r'-k},\cdots, (-1)^{r'-2k+2})
        \cdot
        M''
        \cdot
        \mathrm{Diag}(1,(-1),\cdots,(-1)^{k-2}),
\end{align*}
and so we have
\begin{align*}
        a_k=(-1)^{k+1}\frac{(-1)^{r'}\det M''}{\det M}
        =(-1)^{r'+k+1}\frac{\det M''}{\det M},
\end{align*}
which together with \eqref{eq: P_k,0(-1,0)} and Lemma~\ref{lemma: PQFad'e}~(iii) implies
\begin{align*}
        P_{k-1,k-1}(-1,0)
        =\frac{P_{k,0}(-1,0)}{a_k}
        =\frac{(-1)^{r'-k-1}\frac{\det M'}{\det M}}{(-1)^{r'+k+1}\frac{\det M''}{\det M}}
        =\frac{\det M'}{\det M''}.
\end{align*}

We now compute this quotient as follows: subtracting the first column from each column of $M'$, we have
\begin{align*}
        \det M'=
        \begin{vmatrix}
          \frac{1}{r'-k} & \frac{1}{r'-k+1}-\frac{1}{r'-k} & \cdots
           & \frac{1}{r'-2}-\frac{1}{r'-k} & \frac{1}{r'-1}-\frac{1}{r'-k} \\
          \frac{1}{r'-k-1} & \frac{1}{r'-k}-\frac{1}{r'-k-1} & \cdots
           & \frac{1}{r'-3}-\frac{1}{r'-k-1} & \frac{1}{r'-2}-\frac{1}{r'-k-1} \\
          \vdots & \vdots & \ddots & \vdots & \vdots \\
          \frac{1}{r'-2k+2} & \frac{1}{r'-2k+3}-\frac{1}{r'-2k+2} & \cdots
           & \frac{1}{r'-k}-\frac{1}{r'-2k+2} & \frac{1}{r'-k+1}-\frac{1}{r'-2k+2} \\
          1 & 0 & \cdots & 0 & 0
        \end{vmatrix}
\end{align*}
which can be rewritten as
\begin{align*}
\det M'= (-1)^{k+1}
        \begin{vmatrix}
          \frac{-1}{(r'-k+1)(r'-k)} & \cdots
           & \frac{-(k-2)}{(r'-2)(r'-k)} & \frac{-(k-1)}{(r'-1)(r'-k)} \\
          \frac{-1}{(r'-k)(r'-k-1)} & \cdots
           & \frac{-(k-2)}{(r'-3)(r'-k-1)} & \frac{-(k-1)}{(r'-2)(r'-k-1)} \\
          \vdots & \ddots & \vdots & \vdots \\
          \frac{-1}{(r'-2k+3)(r'-2k+2)} & \cdots
           & \frac{-(k-2)}{(r'-k)(r'-2k+2)} & \frac{-(k-1)}{(r'-k+1)(r'-2k+2)}
        \end{vmatrix}.
    \end{align*}
Factoring out $\frac{1}{r'-k+1-i}$ from each $i$-th row and $-j$ from each $j$-th column, we get
    \begin{align*}
    \det M'
        &=
        (-1)^{k+1}\frac{(-1)^{k-1}(k-1)!}{(r'-k)(r'-k-1)\cdots(r'-2k+2)}
        \begin{vmatrix}
          \frac{1}{r'-k+1} & \frac{1}{r'-k+2} & \cdots & \frac{1}{r'-1} \\
          \frac{1}{r'-k} & \frac{1}{r'-k+1} & \cdots & \frac{1}{r'-2} \\
          \vdots & \vdots & \ddots & \vdots \\
          \frac{1}{r'-2k+3} & \frac{1}{r'-2k+4} & \cdots & \frac{1}{r'-k+1}
        \end{vmatrix}\\
        &=
        \frac{(r'-2k+1)!(k-1)!}{(r'-k)!}\det M''.
    \end{align*}
Therefore, we have the desired identity of (ii).

(iii) By Lemma~\ref{lemma: PQdegree}, we have
    \begin{align*}
        \deg_T P_{k,k}(T,x)Q_{k,0}(T,x)=r'\quad\mbox{ and }\quad\deg_T P_{k,0}(T,x)Q_{k,k}(T,x)=r'-2.
    \end{align*}
As $P_{k,k}(T,x)$ and $Q_{k,0}(T,x)$ are monic, we have
    \begin{align*}
        &P_{k,k}(T,x)Q_{k,0}(T,x)-P_{k,0}(T,x)Q_{k,k}(T,x)
        \\&\hspace{3em}
        =
        T^{r'}+\big[P_{k,k}(T,x)Q_{k,0}(T,x)-P_{k,0}(T,x)Q_{k,k}(T,x)\big]_{T}^{(r'-1)}.
    \end{align*}
On the other hand, by Lemma~\ref{lemma: PQdegree}~(i) we have
    \begin{align*}
        &P_{k,k}(T,x)Q_{k,0}(T,x)-P_{k,0}(T,x)Q_{k,k}(T,x)
        \\&
        \equiv
        P_{k,k}(T,x)P_{k,0}(T,x)(x+\log(1+T))-P_{k,0}(T,x)P_{k,k}(T,x)(x+\log(1+T))
        \\&
        \equiv 0 \pmod{T^{r'}}.
    \end{align*}
Therefore, we get the desired identity of (iii).
\end{proof}

\begin{lemm}\label{lemma-delta}
Assume that $r'\geq 2$.
\begin{enumerate}[leftmargin=*]
  \item If $1\leq k<r'/2$, then we have
  \begin{align*}
    \delta_{k+1}(-1)-\delta_k(-1)=\frac{1}{r'-k}-\frac{1}{k}.
  \end{align*}
  In particular, $\delta_{k+1}(c)-\delta_k(c)\in(S'_\cO)^\times$.
  \item If $1\leq k<r'/2+1$, then we have
  \begin{align*}
    \delta_k(-1)=-H_{r'-k}-H_{k-1}.
  \end{align*}
 In particular, if $r'=2m+1$ then $Q_{m+1,0}(-1,0)=-1$.
  \item We have
    \begin{align*}
        \delta_{m+1}(T,x)-\delta_m(T) =-\frac{T^{r'}}{P_{m,m}(T,x)P_{m,0}(T,0)}.
    \end{align*}
    In particular, if $r'=2m$ and $x\in E$ with $v_p(x)<0$, then
    \begin{align*}
        \delta_{m+1}(T,x)-\delta_m(T)
        \equiv
        \frac{T^{r'}}{xP_{m,0}(T,0)^2}
        \pmod{x^{-1}\fm}
    \end{align*}
    and so $\delta_{m+1}(-1,x)-\delta_{m-1}(-1)\equiv\delta_m(-1)-\delta_{m-1}(-1)\pmod{\fm}$.
  \end{enumerate}
\end{lemm}

\begin{proof}
(i) By Lemma~\ref{lemma-delta-denom-1}~(iii), for $1\leq k\leq m$ we have
    \begin{align*}
        \delta_{k+1}(-1)-\delta_k(-1)
        &=
        \frac{Q_{k,k}(-1,0)}{P_{k,k}(-1,0)}-\frac{Q_{k,0}(-1,0)}{P_{k,0}(-1,0)}\\
        &=
        -\frac{
        P_{k,k}(-1,0)Q_{k,0}(-1,0)-P_{k,0}(-1,0)Q_{k,k}(-1,0)
        }{
        P_{k,0}(-1,0)P_{k,k}(-1,0)
        }\\
        &=
        \frac{(-1)^{r'+1}}{P_{k,0}(-1,0)P_{k,k}(-1,0)}.
    \end{align*}
Now, by Lemma~\ref{lemma-delta-denom-1}, if $1\leq k<r'/2$, then we have
    \begin{align*}
        \delta_{k+1}(-1)-\delta_k(-1)
        &=
        -\frac{(r'-2k)!(k-1)!}{(r'-k)!}\frac{(r'-k-1)!}{(r'-2k-1)!k!}\\
        &=
        -\frac{r'-2k}{(r'-k)k}=
        \frac{1}{r'-k}-\frac{1}{k}.
    \end{align*}
The last part of (i) follows immediately.

(ii) By Lemma~\ref{lemma: k=0,k=r}, we have
    \begin{align*}
        \delta_1(T)=\frac{Q_{0,0}(T,0)}{P_{0,0}(T,0)}=Q_{0,0}(T,0)
        =\sum_{n=1}^{r'-1}\frac{(-1)^{n-1}}{n}T^n
    \end{align*}
    and so
    \begin{align*}
        \delta_1(-1)
        =-H_{r'-1}.
    \end{align*}
Hence, for each $1\leq k<r'/2+1$, we have
    \begin{align*}
        \delta_k(-1)
        =-H_{r'-1}+\sum_{n=1}^{k-1}\bigg(\frac{1}{r'-n}-\frac{1}{n}\bigg)
        =-H_{r'-k}-H_{k-1}
    \end{align*}
by applying (i) inductively.

For the second part, from the definition together with Lemma~\ref{lemma: PQFad'e}~(iii), we have
$$Q_{m+1,0}(-1,0)=P_{m+1,0}(-1,0)\delta_{m+1}(-1)=P_{m,m}(-1,0)a_{m+1}\delta_{m+1}(-1).$$
By Lemma~\ref{lemma-delta-denom-1}~(ii) we have $P_{m,m}(-1,0)=1$, and by the first part of (ii) we have $\delta_{m+1}(-1)=-2H_m=-\frac{1}{a_{m+1}}$, where the last equality is due to Lemma~\ref{lemma: a_k m+1 leq k}~(v). Therefore, we have $Q_{m+1,0}(-1)=-1$.

(iii) By definition together with Lemma~\ref{lemma-PQ-k leq m}~(i), we have
\begin{align*}
        \delta_{m+1}(T,x)-\delta_m(T)
         &=\frac{Q_{m,m}(T,x)}{P_{m,m}(T,x)}-x-\frac{Q_{m,0}(T,0)}{P_{m,0}(T,0)}\\
         &=\frac{Q_{m,m}(T,x)}{P_{m,m}(T,x)}-x-\frac{Q_{m,0}(T,x)-xP_{m,0}(T,0)}{P_{m,0}(T,0)}.
\end{align*}
Now, by Lemma~\ref{lemma-delta-denom-1}~(iii), we conclude the first part
    \begin{align*}
        \delta_{m+1}(T,x)-\delta_m(T)
        &=-\frac{P_{m,m}(T,x)Q_{m,0}(T,x)-P_{m,0}(T,x)Q_{m,m}(T,x)}{P_{m,m}(T,x)P_{m,0}(T,x)}\\
        &=-\frac{T^{r'}}{P_{m,m}(T,x)P_{m,0}(T,x)}.
    \end{align*}
For the second part, we have
    \begin{align*}
        \delta_{m+1}(T,x)-\delta_m(T)
        =-\frac{T^{r'}}{(P_{m,m}(T,0)-xP_{m,0}(T,0))P_{m,0}(T,0)}
    \end{align*}
by Lemma~\ref{lemma-PQ-k leq m}. By Lemma~\ref{lemma-delta-denom-1}~(i), we have $P_{m,0}(T,0)\in\Z_p[[T+1]]^{\times}$, and so we conclude that
\begin{align*}
        \delta_{m+1}(T,x)-\delta_m(T)
        &=
        \frac{x^{-1}T^{r'}}{P_{m,0}(T,0)(P_{m,0}(T,0)-x^{-1}P_{m,m}(T,0))}\\
        &\equiv
        \frac{x^{-1}T^{r'}}{P_{m,0}(T,0)^2}
        \pmod{x^{-1}\fm},
    \end{align*}
which completes the proof.
\end{proof}

We need to compute $P_{m,m}(-1,0)$ when $r'=2m$, which is not covered by Lemma~\ref{lemma-delta-denom-1}.
\begin{lemm}\label{lemma-P-2m-m}
Let $r'=2m\geq 2$. Then we have
    \begin{align*}
        P_{m,m}(-1,0)=Q_{m,0}(-1,0)=-m(H_m+H_{m-1}).
    \end{align*}
\end{lemm}

\begin{proof}
During the proof, we write $P_{r';k,j}(T,x)$ and $Q_{r';k,j}(T,x)$ for $P_{k,j}(T,x)$ and $Q_{k,j}(T,x)$, respectively, as we need to indicate $r'$. We also write $\delta_{r';k}$ for $\delta_k$, to emphasize $r'$.

We first claim that the following identity:
\begin{multline*}
Q_{r'-1;m,0}(T,0)\big(P_{r';m,0}(T,0)+(1+T)\dot{P}_{r';m,m}(T,0)\big)-mT^{r'-1}\\
=P_{r'-1;m,0}(T,0)\big(Q_{r';m,0}(T,0)-P_{r';m,m}(T,0)+(1+T)\dot{Q}_{r';m,m}(T,0)\big).
\end{multline*}
By Lemma~\ref{lemma: PQdegree}, we have
    \begin{align*}
        P_{r';m,m}(T,0)\log(1+T)\equiv Q_{r';m,m}(T,0) \pmod{T^{r'}}.
    \end{align*}
Differentiating both sides with respect to $T$, we have
    \begin{align*}
        \dot{P}_{r';m,m}(T,0)\log(1+T) +P_{r';m,m}(T,0)(1+T)^{-1} \equiv \dot{Q}_{r';m,m}(T,0) \pmod{T^{r'-1}},
    \end{align*}
which can be rewritten as
    \begin{align*}
        (1+T)\dot{P}_{r';m,m}(T,0)\log(1+T) \equiv -P_{r';m,m}(T,0)+(1+T)\dot{Q}_{r';m,m}(T,0) \pmod{T^{r'-1}}.
    \end{align*}
Hence, we have
    \begin{align}\label{eq: r=2m, P_m,m, 2}
    \begin{split}
        &Q_{r'-1;m,0}(T,0)\big(P_{r';m,0}(T,0)+(1+T)\dot{P}_{r';m,m}(T,0)\big)\\
        &\equiv
        P_{r'-1;m,0}(T,0)\big(P_{r';m,0}(T,0)+(1+T)\dot{P}_{r';m,m}(T,0)\big)\log(1+T)\\
        &\equiv
        P_{r'-1;m,0}(T,0)\big(
        Q_{r';m,0}(T,0)-P_{r';m,m}(T,0)+(1+T)\dot{Q}_{r';m,m}(T,0)\big)
    \end{split}
    \end{align}
modulo $(T^{r'-1})$.

As $Q_{r'-1;m,0}(T,0)$ is a monic polynomial of degree $m-1$ by Lemma~\ref{lemma: PQdegree}, $P_{r';m,0}(T,0)$ has degree $m-1$ by Lemma~\ref{lemma: B_k,0 k leq m}, and $P_{r';m,m}(T,0)$ is a monic polynomial of degree $m$ by Lemma~\ref{lemma: PQdegree}, we conclude that the polynomial $(1+T)\dot{P}_{r';m,m}(T,0)$ has degree $m$ and the leading coefficient $m$, and so the degree of the most-left side of \eqref{eq: r=2m, P_m,m, 2} satisfies
    \begin{align*}
        \deg_T\Big(Q_{r'-1;m,0}(T,0)\big(P_{r';m,0}(T,0)+(1+T)\dot{P}_{r';m,m}(T,0)\big)-mT^{r'-1}\Big) \leq r'-2.
    \end{align*}
On the other hand, we also have the following:
    \begin{itemize}[leftmargin=*]
    \item $P_{r'-1;m,0}(T,0)$ has degree $m-1$ by Lemma~\ref{lemma: PQdegree} and Lemma~\ref{lemma: B_k,0 k leq m};
    \item both $Q_{r';m,0}(T,0)$ and $P_{r';m,m}(T,0)$ are monic polynomials of degree $m$ by Lemma~\ref{lemma: PQdegree}, so that we have
        $\deg \big(Q_{r';m,0}(T,0)-P_{r';m,m}(T,0)\big)\leq m-1$.
    \item $Q_{r';m,m}(T,0)$ has degree less than equal to $m-1$ by Lemma~\ref{lemma: PQdegree}, and so does $(1+T)\dot{Q}_{r';m,m}(T,0)$.
    \end{itemize}
Hence, we conclude that the degree of the most-right side of \eqref{eq: r=2m, P_m,m, 2} satisfies
    \begin{align*}
        \deg_T P_{r'-1;m,0}(T,0)\big(Q_{r';m,0}(T,0)-P_{r';m,m}(T,0)+(1+T)\dot{Q}_{r';m,m}(T,0)\big)\leq r'-2.
    \end{align*}
Therefore, the congruence \eqref{eq: r=2m, P_m,m, 2} together with these degree bounds gives rise to the identity in the claim.

We now evaluate the identity in the claim at $T=-1$. Then we have
    \begin{align*}
        Q_{r'-1;m,0}(-1,0)P_{r';m,0}(-1,0)+m =P_{r'-1;m,0}(-1,0)\big(Q_{r';m,0}(-1,0)-P_{r';m,m}(-1,0)\big).
    \end{align*}
As $Q_{r'-1;m,0}(-1,0)=-1$ by Lemma~\ref{lemma-delta}~(ii) and $P_{r';m,0}(-1,0)=m$ by Lemma~\ref{lemma-delta-denom-1}~(i), the left-hand side is $0$.
Moreover, from $$-1=Q_{r'-1;m,0}(-1,0) =P_{r'-1;m,0}(-1,0)\delta_{r'-1;m}(-1)$$
we see that $P_{r'-1;m,0}(-1,0)\not= 0$, so that by Lemma~\ref{lemma-delta}~(ii), we conclude that
    \begin{align*}
        P_{r';m,m}(-1,0) =Q_{r';m,0}(-1,0) =P_{r';m,0}(-1,0)\delta_{r';m}(-1) =-m(H_m+H_{m-1}).
    \end{align*}
This completes the proof.
\end{proof}

We further compute $\dot\delta_k(-1)$. By $\dot{P}_{k,i}(-1,0)$ we mean $\left.\frac{dP_{k,i}(T,0)}{dT}\right|_{T=-1}$.
\begin{lemm}\label{lemma-delta-derivative}
Assume that $r'\geq 2$.
\begin{enumerate}[leftmargin=*]
\item If $1\leq k< r'/2+1$, then we have
    \begin{align*}
        \dot{\delta}_k(-1)
        =
        2\frac{\dot{P}_{k,0}(-1,0)}{P_{k,0}(-1,0)}+(r'-1)
        =
        2\frac{\dot{P}_{k-1,k-1}(-1,0)}{P_{k-1,k-1}(-1,0)}+(r'-1).
    \end{align*}
\item If $1\leq k< r'/2$, then we have
    \begin{multline*}
        \frac{\dot{\delta}_{k+1}(-1)}{r'(2k+1)-2k(k+1)} - \frac{\dot{\delta}_k(-1)}{r'(2k-1)-2k(k-1)}\\
        = -\frac{1}{r'(2k+1)-2k(k+1)} +\frac{1}{r'(2k-1)-2k(k-1)}.
    \end{multline*}
\item If $1\leq k <r'/2+1$, then we have
    \begin{align*}
        \dot{\delta}_k(-1)=r'(2k-1)-2k(k-1)-1.
    \end{align*}
\item  If $x\in E$ with $v_p(x)<0$ then $\delta_{m+1}(c,x)\in S'_\cO$ and $$\dot{\delta}_{r'-m}(-1)\equiv\dot{\delta}_{m+1}(-1,x)\pmod{\fm}.$$

\end{enumerate}
\end{lemm}

\begin{proof}
(i) It suffices to prove the equation for $P_{k,0}(-1,0)$ by Lemma~\ref{lemma: PQFad'e}. For brevity, let $P(T)=P_{k,0}(T,0)$ and $Q(T)=Q_{k,0}(T,0)$. The condition
$Q(T)\equiv P(T)\log(1+T) \pmod{T^{r'}}$ in Lemma~\ref{lemma: PQdegree} yields
    \begin{align*}
    \dot{P}(T)&Q(T)-P(T)\dot{Q}(T)
        \\
        &\equiv
        \dot{P}(T)P(T)\log(1+T)
        -P(T)\bigg(\dot{P}(T)\log(1+T)+P(T)\frac{1}{1+T}\bigg)
        \\
        &\equiv
        -\frac{P(T)^2}{1+T}
        \pmod{T^{r'-1}},
    \end{align*}
    and so
    \begin{align*}
        (1+T)\big(\dot{P}(T)Q(T)-P(T)\dot{Q}(T)\big)
        \equiv
        -P(T)^2
        \pmod{T^{r'-1}}.
    \end{align*}
Noting that $\deg P=k-1$ and $\deg Q=r'-k$, we have
    \begin{align*}
       & \deg (1+T)\big(\dot{P}(T)Q(T)-P(T)\dot{Q}(T)\big)\leq r'-1;\\
       & \deg P(T)^2=2(k-1)\leq r'-1,
    \end{align*}
so we may write
    \begin{align*}
        (1+T)\big(\dot{P}(T)Q(T)-P(T)\dot{Q}(T)\big)
        =
        AT^{r'-1}-P(T)^2
    \end{align*}
    for some suitable constant $A$. Plugging $T=-1$ in, one has $A=(-1)^{r'-1}P(-1)^2$, and thus
    \begin{align*}
    \dot{\delta}_k(T)
        =
        -\frac{\dot{P}(T)Q(T)-P(T)\dot{Q}(T)}{P(T)^2}
        =
        \frac{P(T)^2-(-1)^{r'-1}P(-1)^2T^{r'-1}}{(1+T)P(T)^2}.
    \end{align*}
Hence, we have
    \begin{align*}
    \dot{\delta}_k(-1)
        &=
        \frac{1}{P(-1)^2}\bigg(
        \bigg.
        \frac{P(T)^2-(-1)^{r'-1}P(-1)^2T^{r'-1}}{1+T}
        \bigg|_{T=-1}
        \bigg)
        \\
        &=
        \frac{1}{P(-1)^2}\bigg(
        \bigg.\frac{d}{dT}\bigg|_{T=-1}
        \big(P(T)^2-(-1)^{r'-1}P(-1)^2T^{r'-1}\big)
        \bigg)
        \\
        &=
        \frac{1}{P(-1)^2}\big(
        2P(-1)\dot{P}(-1)+(r'-1)P(-1)^2
        \big)
        \\
        &=
        2\frac{\dot{P}(-1)}{P(-1)}+(r'-1),
    \end{align*}
    as desired.

(ii) By Lemma~\ref{lemma-delta-denom-1}~(iii) together with Lemma~\ref{lemma-PQ-k leq m}~(i), we have
\begin{equation}\label{eq: delta_k+1-delta_k}
\delta_{k+1}(T)-\delta_k(T)=-\frac{T^{r'}}{P_{k,0}(T,0)P_{k,k}(T,0)}.
\end{equation}
Then by a direct computation we have
    \begin{align*}
    &\dot{\delta}_{k+1}(T)-\dot{\delta}_k(T)\\
    &=
        \frac{d}{dT}\bigg(
        -\frac{T^{r'}}{P_{k,0}(T,0)P_{k,k}(T,0)}
        \bigg)
        \\
        &=
        -\frac{r'T^{r'-1}}{P_{k,0}(T,0)P_{k,k}(T,0)}
        +
        \frac{T^{r'}}{P_{k,0}(T,0)P_{k,k}(T,0)}\bigg(
        \frac{\dot{P}_{k,k}(T,0)}{P_{k,k}(T,0)}
        +\frac{\dot{P}_{k,0}(T,0)}{P_{k,0}(T,0)}
        \bigg).
    \end{align*}
Now, by (i) we have
    \begin{align*}
    &\dot{\delta}_{k+1}(-1)-\dot{\delta}_k(-1)
        \\
        &=
        \frac{(-1)^{r'}}{P_{k,0}(-1,0)P_{k,k}(-1,0)}
        \bigg(
        r'
        +\frac{1}{2}\big(\dot{\delta}_{k+1}(-1)-(r'-1)\big)
        +\frac{1}{2}\big(\dot{\delta}_k(-1)-(r'-1)\big)
        \bigg)
        \\
        &=
        \frac{(-1)^{r'}}{2P_{k,0}(-1,0)P_{k,k}(-1,0)}
        \big(
        \dot{\delta}_{k+1}(-1)+\dot{\delta}_k(-1)+2
        \big).
    \end{align*}
By Lemma~\ref{lemma-delta}~(i) together with \eqref{eq: delta_k+1-delta_k}, we have
    \begin{align*}
        \frac{(-1)^{r'}}{P_{k,0}(-1,0)P_{k,k}(-1,0)}
        =\frac{r'-2k}{k(r'-k)},
    \end{align*}
and so we get
    \begin{align*}
        \frac{r'(2k-1)-2k(k-1)}{2k(r'-k)}\dot{\delta}_{k+1}(-1)
        -\frac{r'(2k+1)-2k(k+1)}{2k(r'-k)}\dot{\delta}_k(-1)
        =\frac{r'-2k}{k(r'-k)}.
    \end{align*}
Noting
    \begin{align*}
        \big(r'(2k+1)-2k(k+1)\big)-\big(r'(2k-1)-2k(k-1)\big)
        =2(r'-2k)
    \end{align*}
completes the proof.

(iii) For brevity, we set $b_k=r'(2k-1)-2k(k-1)$. By (ii) we have
    \begin{align*}
        \frac{\dot{\delta}_{k+1}(-1)}{b_{k+1}}-\frac{\dot{\delta}_k(-1)}{b_k}
        =
        -\frac{1}{b_{k+1}}+\frac{1}{b_k}
    \end{align*}
for $1\leq k<r'/2$. Now, noting that $b_1=r'$ and $\delta'_1(-1)=r'-1$ by (i), we have $\frac{\delta'_1(-1)}{b_1}=1-\frac{1}{b_1}$. Hence, for $1\leq k<r'/2+1$, we have
    \begin{align*}
        \frac{\dot{\delta}_k(-1)}{b_k}
        =1-\frac{1}{b_1}+\sum_{n=1}^{k-1}\bigg(-\frac{1}{b_{n+1}}+\frac{1}{b_n}\bigg)
        =1-\frac{1}{b_k},
    \end{align*}
which gives the desired result.

(iv) By Lemma~\ref{lemma: m+1 to r}, we have
\begin{equation}\label{eq: congruence for delta_m+1}
\tilde{P}_{m+1,0}(T,x)\equiv x^{-1}P_{r'-m-1,r'-m-1}(T,0)\pmod{x^{-1}\fm}
\end{equation} as $v_p(x)<0$, and so we have
$$\delta_{m+1}(T,x)=\frac{\tilde{Q}_{m+1,0}(T,x)}{\tilde{P}_{m+1,0}(T,x)}-x\equiv \frac{Q_{m+1,0,t-1}(T)}{P_{r'-m-1,r'-m-1}(T,0)}\pmod{\fm}.$$
Hence, we have $\delta_{m+1}(c,x)\in S'_\cO$ by Lemma~\ref{lemma-delta-denom-1}~(ii).

By definition together with Lemma~\ref{lemma: PQdegree}, it is clear that $$[P_{m+1,0}(T,x)Tf_{r'}(T)]_T^{(r'-1)}-P_{m+1,0}(T,x)\delta_{m+1}(T,x)=0.$$ On the other hand, by \eqref{eq: congruence for delta_m+1} we have
$$[P_{r'-m-1,r'-m-1}(T,0)Tf_{r'}(T)]_T^{(r'-1)}-P_{r'-m-1,r'-m-1}(T,0)\delta_{m+1}(T,x)\equiv 0\pmod{\fm},$$
which immediately gives rise to
\begin{equation}\label{eq: congruence for delta_m+1, 2}
P_{r'-m-1,r'-m-1}(T-1,0)\big(\delta_{r'-m}(T-1)-\delta_{m+1}(T-1,x)\big)\equiv 0\pmod{\fm}
\end{equation}
by Lemma~\ref{lemma-delta-denom-1}~(ii). In particular, we have $\delta_{r'-m}(-1)\equiv\delta_{m+1}(-1,x)\pmod{\fm}$. Moreover, by differentiating \eqref{eq: congruence for delta_m+1, 2}, we conclude the desired result.
\end{proof}

We close this subsection by computing the following: for $r'=2m+1$, define
\begin{equation*}
b_{r',m+2}(T)
    :=\frac{[P_{m+2,0,1}(T)Tf_{r'}(T)]_{T}^{(r'-1)}-P_{m+2,0,1}(T)\delta_m(T)}{P_{m-1,m-1}(T)},
\end{equation*}
which is the same as
\begin{equation*}
    b_{r',m+2}(T)=-\frac{[P_{m+1,m+1,1}(T)Tf_{r'}(T)]_{T}^{(r'-1)}-P_{m+1,m+1,1}(T)\delta_m(T)}{P_{m,0}(T,0)}
\end{equation*}
by \eqref{eq: P_k,0,s and P_k-1,k-1,s}, \eqref{eq: a_r-k+1}, and Lemma~\ref{lemma: PQFad'e} (iii).
Note that $b_{2m+1,m+2}$ will appear in Proposition~\ref{prop-cM-Case-m+1} and Proposition~\ref{prop-cM-Case-m+2-r-1}.

\begin{lemm}\label{lemma-b-2m+1-m+2}
Let $r'=2m+1\geq 3$. Then the coefficient of $T^m$ in $Q_{m,0}(T,0)$ is $-m(m+1)$, i.e.
    \begin{align*}
        Q_{m,0}(T,0)=T^{m+1}-m(m+1)T^m+[Q_{m,0}(T,0)]_{T}^{(m-1)}.
    \end{align*}
Moreover, we have
    \begin{align*}
        b_{r',m+2}(T)
        =
        \frac{T^{r'+1}-2m(m+1)T^{r'}}{P_{m,0}(T,0)^2}.
    \end{align*}
    In particular, $b_{r',m+2}(-1)=-\frac{2m^2+2m+1}{m^2(m+1)^2}$.
\end{lemm}

\begin{proof}
We start to prove the first part of the lemma. From $Q_{m,0}(T,0)
        =
        T^{m+1}
        +
        \begin{bmatrix}
          T^m & T^{m-1} & \cdots & 1
        \end{bmatrix}
        D_{m,0}B_{m,0}^{-1}\mathbf{e}'_m$,
the coefficient of $T^m$ is given by
    \begin{align*}
        \begin{bmatrix}
          1 & 0 & \cdots & 0
        \end{bmatrix}
        D_{m,0}B_{m,0}^{-1}\mathbf{e}'_m
        =
        \begin{bmatrix}
          1 & -\frac{1}{2} & \cdots & \frac{(-1)^{m-1}}{m}
        \end{bmatrix}
        B_{m,0}^{-1}\mathbf{e}'_m.
    \end{align*}
Using Cram\'er's rule as in the proof of Lemma~\ref{lemma-delta-denom-1}~(i), this is equal to
$\frac{1}{\det(B_{m,0})}\det(\tilde{M}')$
where
    \begin{align*}
     \tilde{M}'=
        \begin{bmatrix}
          \frac{(-1)^{m}}{m+1} & \frac{(-1)^{m+1}}{m+2} & \cdots & \frac{(-1)^{r'-3}}{r'-2} & \frac{(-1)^{r'-2}}{r'-1} \\
          \frac{(-1)^{m-1}}{m} & \frac{(-1)^{m}}{m+1} & \cdots & \frac{(-1)^{r'-4}}{r'-3} & \frac{(-1)^{r'-3}}{r'-2} \\
          \vdots & \vdots & \ddots & \vdots & \vdots \\
          \frac{1}{3} & -\frac{1}{4} & \cdots & \frac{(-1)^m}{m+1} & \frac{(-1)^{m+1}}{m+2} \\
          1 & -\frac{1}{2} & \cdots & \frac{(-1)^{m-2}}{m-1} & \frac{(-1)^{m-1}}{m}
        \end{bmatrix}.
    \end{align*}

Let $M$, $M'$ be as in the proof of Lemma~\ref{lemma-delta-denom-1} (i) with $r'=2m+1$ and $k=m$, and let $\tilde M$ be the matrix given by
    \begin{align*}
        \tilde M
        =
        \begin{bmatrix}
          \frac{1}{m+1} & \frac{1}{m+2} & \cdots & \frac{1}{r'-2} & \frac{1}{r'-1} \\
          \frac{1}{m} & \frac{1}{m+1} & \cdots & \frac{1}{r'-3} & \frac{1}{r'-2} \\
          \vdots & \vdots & \ddots & \vdots & \vdots \\
          \frac{1}{3} & \frac{1}{4} & \cdots & \frac{1}{m+1} & \frac{1}{m+2} \\
          1 & \frac{1}{2} & \cdots & \frac{1}{m-1} & \frac{1}{m}
        \end{bmatrix}.
    \end{align*}
Then we have
    \begin{align*}
        \tilde{M}'=
        \mathrm{Diag}((-1)^m,\cdots,(-1)^2, 1)
        \cdot
        \tilde M
        \cdot
        \mathrm{Diag}(1,\cdots,(-1)^{m-2},(-1)^{m-1}),
    \end{align*}
    and we already know that
    \begin{align*}
        B_{m,0}
        =
        \mathrm{Diag}((-1)^m,\cdots,(-1)^2, (-1))
        \cdot
        M
        \cdot
        \mathrm{Diag}(1,\cdots,(-1)^{m-2},(-1)^{m-1}).
    \end{align*}
Hence, the coefficient of $T^m$ now becomes $-\frac{\det \tilde M}{\det M}$.

We now compute $\det \tilde M$: subtracting the $m$-th row from each row gives rise to
    \begin{align*}
       \det\tilde M= \begin{vmatrix}
          \frac{-m}{m+1} & \frac{-m}{(m+2)2} & \cdots
           & \frac{-m}{(r'-2)(m-1)} & \frac{-m}{(r'-1)m} \\
          \frac{-(m-1)}{m} & \frac{-(m-1)}{(m+1)2} & \cdots
           & \frac{-(m-1)}{(r'-3)(m-1)} & \frac{-(k-2)}{(r'-2)m} \\
          \vdots & \vdots & \ddots & \vdots & \vdots \\
          \frac{-2}{3} & \frac{-2}{4\cdot 2} & \cdots
           & \frac{-2}{(m+1)(m-1)} & \frac{-2}{(m+2)m} \\
          1 & \frac{1}{2} & \cdots & \frac{1}{m-1} & \frac{1}{m}
        \end{vmatrix}.
    \end{align*}
Factoring out $-(m-i+1)$ from each $i$-th row and $\frac{1}{j}$ from each $j$-th column, we get
    \begin{align*}
    \det \tilde M
        &=
        \frac{(-1)^{m-1} m!}{m!}
        \begin{vmatrix}
          \frac{1}{m+1} & \frac{1}{m+2} & \cdots & \frac{1}{r'-2} & \frac{1}{r'-1} \\
          \frac{1}{m} & \frac{1}{m+1} & \cdots & \frac{1}{r'-3} & \frac{1}{r'-2} \\
          \vdots & \vdots & \ddots & \vdots & \vdots \\
          \frac{1}{3} & \frac{1}{4} & \cdots & \frac{1}{m+1} & \frac{1}{m+2} \\
          1 & 1 & \cdots & 1 & 1
        \end{vmatrix}
        =
        (-1)^{m-1}\det M'
    \end{align*}

    On the other hand, we know from the proof of Lemma~\ref{lemma-delta-denom-1}~(i) that
    \begin{align*}
        \det M
        =\frac{(-1)^{m-1}(r'-2m)!(m-1)!}{(r'-m)!}\det M'
        =\frac{(-1)^{m-1}}{m(m+1)}\det M'.
    \end{align*}
Combining these two gives
    \begin{align*}
        -\frac{\det\tilde M}{\det M}
        =-(-1)^{m-1}\frac{\det M'}{\det M}
        =-m(m+1).
    \end{align*}

We now treat the second part. Note first that
    \begin{align}\label{eq: P_m-1,m-1 b_r,m+2}
    \begin{split}
    &P_{m-1,m-1}(T,0)^2b_{r',m+2}(T)\\
        &=
        P_{m-1,m-1}(T,0)[P_{m+2,0,1}(T)Tf_{r'}(T)]_{T}^{(r'-1)}
        -Q_{m-1,m-1}(T,0)P_{m+2,0,1}(T).
    \end{split}
    \end{align}
It is easy to see from Lemma~\ref{lemma: PQdegree} that we further have
    \begin{multline*}
    P_{m-1,m-1}(T,0)^2b_{r',m+2}(T)\equiv
        P_{m-1,m-1}(T,0)P_{m+2,0,1}(T)\log(1+T)\\
         -P_{m-1,m-1}(T,0)\log(1+T)P_{m+2,0,1}(T)\pmod{T^{r'}},
    \end{multline*}
and so we conclude that
\begin{equation}\label{eq: P_m-1,m-1 b_r,m+2 = 0}
[P_{m-1,m-1}(T,0)^2b_{r',m+2}(T)]_{T}^{(r'-1)}=0.
\end{equation}

We now claim that
 $$P_{m+2,0,1}(T)=Q_{m+2,0,2}(T)-Q_{m-1,m-1}(T,0).$$
Indeed, looking at the coefficients of $x^2$ in $\tilde Q_{m+2,0}(T,x)=\big[\tilde P_{m+2,0}(T,x)(x+Tf_{r'}(T))\big]_{T}^{(r'-1)}$ from Lemma~\ref{lemma: PQdegree}~(i), it easy to see that
$$Q_{m+2,0,2}(T)=[P_{m+2,0,2}(T)Tf_{r'}(T)]_T^{(r'-1)}+P_{m+2,0,1}(T).$$
By Lemma~\ref{lemma: m+1 to r}, $P_{m+2,0,2}(T)=P_{m-1,m-1}(T,0)$, and so we have $$[P_{m+2,0,2}(T)Tf_{r'}(T)]_T^{(r'-1)}=Q_{m-1,m-1}(T,0)$$ by Lemma~\ref{lemma: PQdegree}~(i), which completes the proof of the claim.

By the same argument as the claim above, we have $$[P_{m+2,0,1}(T)Tf_{r'}(T)]_{T}^{(r'-1)}=Q_{m+2,0,1}(T)-P_{m+2,0,0}(T),$$ which together with the claim immediately implies that from \eqref{eq: P_m-1,m-1 b_r,m+2} we have
\begin{multline}\label{eq: P_m-1,m-1 b_r,m+2 second}
    P_{m-1,m-1}(T,0)^2b_{r',m+2}(T)=
        P_{m-1,m-1}(T,0)\big(Q_{m+2,0,1}(T)-P_{m+2,0,0}(T)\big)  \\
        -Q_{m-1,m-1}(T,0)\big(Q_{m+2,0,2}(T)-Q_{m-1,m-1}(T,0)\big).
\end{multline}

On the other hand, it is easy to see from Lemma~\ref{lemma: PQdegree} that each term on the right hand side of \eqref{eq: P_m-1,m-1 b_r,m+2 second}, except the term $Q_{m-1,m-1}(T,0)^2$, has degree less than $r'$, so that we have the following identity
\begin{multline*}
    [P_{m-1,m-1}(T,0)^2b_{r',m+2}(T)]_{T}^{(r'-1)}= P_{m-1,m-1}(T,0)\big(Q_{m+2,0,1}(T)-P_{m+2,0,0}(T)\big)\\
     -Q_{m-1,m-1}(T,0)Q_{m+2,0,2}(T) +[Q_{m-1,m-1}(T,0)^2]_{T}^{(r'-1)},
\end{multline*}
which together with \eqref{eq: P_m-1,m-1 b_r,m+2 = 0} gives rise to
\begin{multline*}
P_{m-1,m-1}(T,0)\big(Q_{m+2,0,1}(T)-P_{m+2,0,0}(T)\big)
     -Q_{m-1,m-1}(T,0)Q_{m+2,0,2}(T)\\=-[Q_{m-1,m-1}(T,0)^2]_{T}^{(r'-1)}.
\end{multline*}

Hence, from \eqref{eq: P_m-1,m-1 b_r,m+2 second} we may write
\begin{align*}
    P_{m-1,m-1}(T,0)^2b_{r',m+2}(T)=Q_{m-1,m-1}(T,0)^2-[Q_{m-1,m-1}(T,0)^2]_{T}^{(r'-1)},
\end{align*}
which is the same as
    \begin{align*}
    P_{m-1,m-1}(T,0)^2b_{r',m+2}(T)=\frac{1}{a_m^2}\big(
        Q_{m,0}(T,0)^2-[Q_{m,0}(T,0)^2]_{T}^{(r'-1)}
        \big)
    \end{align*}
by Lemma~\ref{lemma: PQFad'e}~(iii). Finally, by (i) we have
    \begin{align*}
        b_{r',m+2}(T)
        =\frac{T^{r'+1}-2m(m+1)T^{r'}}{a_m^2P_{m-1,m-1}(T,0)^2}
        =\frac{T^{r'+1}-2m(m+1)T^{r'}}{P_{m,0}(T,0)^2},
    \end{align*}
as desired. Evaluating $b_{r',m+2}(T)$ at $T=-1$, we have $b_{r',m+2}(-1)=-\frac{2m^2+2m+1}{m^2(m+1)^2}$ because $P_{m,0}(-1,0)=-m(m+1)$ by Lemma~\ref{lemma-delta-denom-1} (i).
\end{proof}

\smallskip

\section{Pseudo-strongly divisible modules}\label{sec: pseudo strongly divisible modules}
In this section, we introduce the pseudo-strongly divisible modules, which are the main objects to be studied in this paper, and deduce the necessary conditions to become an isotypic component of the strongly divisible modules. 

We keep the notation of \S\ref{sec: Frame works}, but we specialize $x$ in $E$. In particular, we recall from \S\ref{subsec: linear expressions of B and C} the polynomials $P_{k,i}(T,x)$ and $Q_{k,i}(T,x)$ in $\mathbf{Case}~(k)$ for $0\leq k\leq r'$.

\subsection{Definitions and Summaries}\label{subsec: pseudo sdm}
In this subsection, we define the pseudo-strongly divisible modules of rank $2$ and introduce various terminology that will be used throughout the paper. Moreover, we also introduce the statement of the main results on pseudo-strongly divisible modules. Throughout this subsection, we fix a pair of integers $(r,r')$ with $0<r'\leq r<p-1$. We also define an $\cO$-submodule of $E((T+1))$
\begin{equation*}
    E(T)_\Delta:=E(T)\cap(E+\cO[[T+1]]).
\end{equation*}

Let $\cD':=S'_E(E'_1,E'_2)$ be the free $S'_E$-module of rank $2$ whose basis is denoted by $E_1',E_2'$, and let $k$ be an integer with $0\leq k\leq r'$ and $(x,\Theta)\in E\times E^{\times}$. For all $0\leq i\leq r'-1$, we define certain elements in $\cD'$
\begin{equation}\label{eq: definition of F'_k,i}
    F'_{k,i}:=
    \begin{cases}
        p^{i+r'-k}\Theta P_{k,i}\Big(\dfrac{v}{p},x\Big)E'_1
        +p^{i+r'-k}Q_{k,i}\Big(\dfrac{v}{p},x\Big)E'_2
        & \mbox{if } 0\leq i\leq k-1; \\
        p^i P_{k,i}\Big(\dfrac{v}{p},x\Big)E'_1
        +\dfrac{p^i}{\Theta} Q_{k,i}\Big(\dfrac{v}{p},x\Big)E'_2
        & \mbox{if } k\leq i\leq r'-1.
    \end{cases}
\end{equation}
In particular, if $r'=1$ then $F'_{0,0}=E'_1+\frac{x}{\Theta}E'_2$ and $F'_{1,0}=\frac{\Theta}{x}E'_1+E'_2$.

\begin{defi}\label{defi: pseudo-sdm}
Fix $k\in ([0,r']\cap\Z)\cup\{\infty\}$. For given $(\Lambda,\Theta,\Omega,x)\in (E^\times)^3\times E$ and $\Delta(T),\tilde{\Delta}(T)\in  E(T)_\Delta$, a \emph{pseudo-strongly divisible module in $\mathbf{Case}~(k)$} consists of quadruple $$(\fM',\,\Fil^{r;r'}_{k}\fM',\,\phi',\,N')$$
where
\begin{itemize}[leftmargin=*]
\item $\fM':=S'_\cO(E'_1,E'_2)$ is the $S'_\cO$-submodule of $\cD'$ generated by $E_1',E_2'$,
\item $\Fil^{r;r'}_{k}\fM'$ is an $S'_\cO$-submodule of $\cD'$ defined by $$\Fil^{r;r'}_{k}\fM':=[\Fil^{r:r'}_{k}\fM']^{(r-1)}+\Fil^r S'_\cO\cdot \fM'$$ where $[\Fil^{r:r'}_{k}\fM']^{(r-1)}$ is a $S'_\cO$-submodule of $\cD'$ generated by 
\begin{equation}\label{eq: definition of FilfM'}
\left\{
  \begin{array}{ll}
    \{v^{r-r'}F'_{k,i}\mid 0\leq i\leq r'-1\} & \hbox{if $k\in[0,r']\cap\Z$;} \\
    \{v^{r-r'}E'_2\} & \hbox{if $k=\infty$,}
  \end{array}
\right.
\end{equation}

\item $\phi':\fM'\rightarrow\cD'$ is a continuous $\phi$-semilinear map induced by
\begin{align}\label{eq: definition of phi'}
\begin{split}
\phi'(E'_1) &:=p\Lambda E'_1 -\frac{\Lambda}{\Omega}\big(x+\Delta(\gamma-1)\big)E'_2 +\frac{\Lambda}{p\Omega}\big(\tilde\Delta(\varphi(\gamma)-1)-\tilde\Delta(-1)\big)E'_2;\\
\phi'(E'_2) &:=\frac{\Lambda\Theta}{\Omega} E'_2,
\end{split}
\end{align}
\item $N':\fM'\rightarrow \cD'$ is a continuous $\cO$-linear derivation defined by
\begin{align}\label{eq: definition of N'}
\begin{split}
N'(E'_1)&:=\frac{1}{\Theta}\bigg(1+\frac{N(\gamma)}{p}\dot{\tilde\Delta}(\gamma-1)\bigg)E'_2;\\
N'(E'_2)&:=0.
\end{split}
\end{align}
\end{itemize}
\end{defi}
We write $\phi'_r:=\frac{1}{p^r}\phi':\Fil^{r;r'}_k\fM'\rightarrow \cD'$.

\begin{rema}
\begin{itemize}[leftmargin=*]
\item We point out that once $r,r',k$ are fixed, the pseudo-strongly divisible modules in $\mathbf{Case}~(k)$ are parameterized by $(\Lambda,\Theta,\Omega,x)\in (E^{\times})^3\times E$.
\item If $0\leq k\leq r'$ then the choice of the generators in the filtration $\Fil^{r;r'}_k\fM'$ in $\mathbf{Case}~(k)$ is induced from the choice of the generators for $\mathbb{M}$ in $\mathbf{Case}~(k)$ (see \eqref{eq: definition of cases}), as we will see in \S\ref{subsec: set up}. Hence, it is harmless to use the notation $\mathbf{Case}~(k)$ for both of them.
\item The choice of the generators in the filtration $\Fil^{r;r'}_k\fM'$ in $\mathbf{Case}~(k)$ for $k=\infty$ is induced from the case $\fL_j=\infty$, as we will see in \S\ref{subsec: set up}.
\end{itemize}
\end{rema}

We wish that a pseudo-strongly divisible module is an isotypic component of the strongly divisible modules, and so it is quite natural to require that the images of $\phi'$ and $N'$ are contained in $\fM'$.

\begin{lemm}\label{lemm: pseudo, condition for phi,N stable}
\begin{enumerate}[leftmargin=*]
\item $\phi'(\fM')$ and $N'(\fM')$ are contained in $\fM'$ if and only if 
\begin{equation}\label{eq: condition for phi,N stable}
\begin{cases}
v_p(\Lambda)\geq -1;\\
v_p(\Lambda)+v_p(x+\Delta(-1))\geq v_p(\Omega);\\
v_p(\Lambda)+v_p(\Theta)\geq v_p(\Omega);\\
v_p(\Theta)\leq 0.
\end{cases}
\end{equation}
\item $N'(E_1')\in \fm\fM'$ if and only if $v_p(\Theta)<0$.
\end{enumerate}
\end{lemm}

\begin{proof}
The proof is straightforward from the definition, as we have
\begin{equation}\label{eq: p^p-1 divides phi(gamma)}
p^{p-1}\mid \big(\tilde\Delta(\phi(\gamma)-1)-\tilde\Delta(-1)\big)
\end{equation}
together with $\Delta(T),\tilde{\Delta}(T)\in  E(T)_\Delta$.
\end{proof}

We will also determine the precise conditions under which $\Fil^{r;r'}_k\fM'$ becomes a $S'_\cO$-submodule of $\fM'$, which is the necessary property to become an isotypic component of strongly divisible modules. 
\begin{theo}\label{theo: filtration of pseudo}
Let $\fM'$ be a pseudo-strongly divisible module in $\mathbf{Case}~(k)$ for $k\in ([0,r']\cap\Z)\cup\{\infty\}$. Assume that $v_p(\Theta)\leq 0$, and further that $v_p(x)<0$ for $m+1\leq k\leq r'-1$. Then $\Fil_k^{r;r'}\fM'$ is a $S'_\cO$-submodule of $\fM'$ if and only if the inequalities
\begin{equation*} 
\begin{cases}
                v_p(\Theta)\leq v_p(x),\,\, v_p(\Theta)\leq -r'+1  & \mbox{if }k=0;\\
                v_p(\Theta)\leq v_p(x),\,\,-r'+2k-1\leq v_p(\Theta)\leq -r'+2k+1
 & \mbox{if } 1\leq k< \frac{r'}{2};\\
                v_p(\Theta)-2v_p(x)\leq 1,\,\,-1\leq v_p(\Theta) &\mbox{if } r'=2m=2k;\\
                v_p(x)\leq v_p(\Theta),\,\, 2k-r'-1\leq v_p(\Theta)-2v_p(x)\leq 2k-r'+1
 &\mbox{if }m+1\leq k\leq r'-1;\\
                v_p(x)\leq v_p(\Theta),\,\, v_p(\Theta)-2v_p(x)\geq r'-1 &\mbox{if }k=r';\\
                \mbox{none} &\mbox{if }k=\infty
\end{cases}
\end{equation*}
hold, in which case we have 
$$\Fil^{r;r'}_{k}\fM=S'_{\cO}(G'_1,G'_2)+\Fil^pS'_{\cO}\fM$$
where
\begin{equation}\label{eq: filtration when r' leq r}
    (G'_1,G'_2):=
    \begin{cases}
        \big(v^{r-r'}F'_{0,0},\, v^{r}E'_2\big)  & \mbox{if }k=0; \\
        \big(v^{r-r'}F'_{k,k},\, v^{r-r'}F'_{k,0}\big) & \mbox{if }0< k<r'; \\
        \big(v^{r}E'_1,\, v^{r-r'}F'_{r',0}\big)  & \mbox{if }k=r'; \\
        \big(v^{r}E'_1,\, v^{r-r'}E'_2\big)  & \mbox{if }k=\infty.
    \end{cases}
\end{equation}
\end{theo}

We prove Theorem~\ref{theo: filtration of pseudo} in \S\ref{subsec: Generators of the filtration} for the cases $r'>1$ and $k\neq\infty$ and in \S\ref{subsec: exceptional cases} for the cases $r'=1$ or $k=\infty$. The framework we developed in \S\ref{sec: Frame works} plays key roles. We also note that the condition $v_p(x)<0$ for $m+1\leq k\leq r'$ is the only artificial assumption we made in this paper, while it is harmless to assume $v_p(\Theta)\leq 0$ by Lemma~\ref{lemm: pseudo, condition for phi,N stable}.

We will further check the strong divisibility of $\fM'$. More precisely, it is necessary for $\phi'$ to satisfy
\begin{equation}\label{eq: strong divisibility, local}
\fM'=S'_\cO\left(\phi'_r(G'_1),\phi'_r(G'_2)\right)
\end{equation}
for each $\mathbf{Case}~(k)$. We will consider the following two possibilities $[\phi'_r(G'_1)]_{E'_1}\in(S_\cO')^\times$ or $[\phi'_r(G'_2)]_{E'_1}\in(S_\cO')^\times$, as it will turn out to be no overlaps between these two possibilities. 

\begin{defi}\label{defi: definition of Case phi k}
Let $k\in ([0,r']\cap\Z)$. We say that a pseudo-strongly divisible module $\fM'$ in $\mathbf{Case}~(k)$ satisfies \emph{the strong divisibility of type} $\mathbf{Case}_\phi(k)$ (resp. \emph{of type} $\mathbf{Case}_\phi(k+\tfrac{1}{2}))$
if it satisfies the strong divisibility \eqref{eq: strong divisibility, local} together with
\begin{equation*}
[\phi_r'(G'_2)]_{E'_1}\in(S_\cO')^\times\quad \big(\mbox{resp.}\,\,\, [\phi_r'(G'_1)]_{E'_1}\in(S_\cO')^\times\big).
\end{equation*}
If $k=\infty$, then we say that a pseudo-strongly divisible module $\fM'$ in $\mathbf{Case}~(\infty)$ satisfies \emph{the strong divisibility of type} $\mathbf{Case}_\phi(\infty)$ if it satisfies the strong divisibility \eqref{eq: strong divisibility, local}.
\end{defi}

Note that $\mathbf{Case}_\phi(0)$ is empty as $[\phi'(E'_2)]_{E'_1}=0$, and that if  a pseudo-strongly divisible module $\fM'$ satisfies the strong divisibility of type $\mathbf{Case}_\phi(\infty)$ then we have $[\phi_r'(G'_1)]_{E'_1}\in(S_\cO')^\times$, as $[\phi'(E'_2)]_{E'_1}=0$. We set $$\cK_{r'}:=\left\{k'\in\tfrac{1}{2}\Z\mid \tfrac{1}{2}\leq k'\leq r'+\tfrac{1}{2}\right\}\sqcup\{\infty\}.$$ By convention, we set $\lfloor \infty\rfloor:=\infty$, and if a pseudo-strongly divisible module $\fM'$ satisfies the strongly divisibility of type $\mathbf{Case}_\phi(k')$ for some $k'\in \cK_{r'}$ then the filtration of $\fM'$ is automatically determined by $k'$, i.e., it has the filtration $\Fil^{r;r'}_{k}\fM'$ where $k:=\lfloor k'\rfloor$. 

In Lemma~\ref{lemm: pseudo, condition for phi,N stable}, we have seen that $\tilde{\Delta}(T)$ is not relevant to $\phi'(\fM')\subseteq \fM'$, which implies that $\tilde{\Delta}(T)$ is not relevant to the strong divisibility in $\mathbf{Case}_\phi(r'+\tfrac{1}{2})$ and in $\mathbf{Case}_\phi(\infty)$. Moreover, it turns out that $\tilde{\Delta}(T)$ is not relevant to the strong divisibility in the other cases either.
\begin{lemm}\label{lemma formula for phi_r(F_kk) and phi_r(F_k0)}
Assume that $v_p(\Theta\Omega)\leq 0$, and further that if $p=3$ then $\tilde{\Delta}(T)$ is a constant in $E$. If $0\leq k\leq r'-1$ then we have
\begin{multline}\label{eq: formula for phi_r(F_k,k)}
\frac{1}{c^{r-r'}}\phi'_r(v^{r-r'}F'_{k,k})=\frac{\Lambda }{p^{r-k-1}}P_{k,k}(c,x)E'_1\\ 
+\frac{\Lambda}{p^{r-k+1}\Theta\Omega}P_{k,k}(c,x)\big(
    \tilde\Delta(\phi(\gamma)-1)-\tilde\Delta(-1)
    \big)E'_2 \\
    +\frac{\Lambda}{p^{r-k}\Omega}\big([P_{k,k}(c,x)cf(c)]_c^{(r-1)}-P_{k,k}(c,x)\Delta(\gamma-1)\big)E'_2,
\end{multline}
and if $1\leq k\leq r'$ then we have
\begin{multline}\label{eq: formula for phi_r(F_k,0)}
\frac{1}{c^{r-r'}}\phi'_r(v^{r-r'}F'_{k,0})=\frac{ \Lambda \Theta}{p^{k-1}}P_{k,0}(c,x)E'_1\\ 
+\frac{\Lambda}{p^{k+1}\Omega}P_{k,0}(c,x)\big(
    \tilde\Delta(\phi(\gamma)-1)-\tilde\Delta(-1) \big)E'_2 \\
    +\frac{\Lambda\Theta}{p^k\Omega} \big([P_{k,0}(c,x)cf(c)]_c^{(r-1)}-P_{k,0}(c,x)\Delta(\gamma-1)\big)E'_2.
\end{multline}
In particular, if $[\phi'_r(v^{r-r'}F'_{k,k})]_{E'_1}\in S_\cO'$ $($resp. if $[\phi'_r(v^{r-r'}F'_{k,0})]_{E'_1}\in S_\cO')$ then $\tilde{\Delta}(T)$ is not relevant to the strong divisibility.
\end{lemm}

Note that it is harmless to assume $v_p(\Omega)\leq 0$ and $v_p(\Theta)\leq0$, as we will see in \S\ref{subsec: set up}. We also note that the assumption on $\tilde{\Delta}$ is harmless either, as we will choose such $\tilde{\Delta}$ when $p=3$.

\begin{proof}
From the definition of $F'_{k,i}$ in \eqref{eq: definition of F'_k,i} together with \eqref{eq: definition of phi'} and \eqref{eq: definition of N'}, it is tedious to compute $\phi'_r(v^{r-r'}F'_{k,k})$ and $\phi'_r(v^{r-r'}F'_{k,0})$, and so we omit the detail.

For $0\leq k\leq r-1$, the coefficient of the second term of $\phi'_r(v^{r-r'}F'_{k,k})$ can be written as
\begin{equation}\label{eq: phi N stable equality}
    [\phi'_r(v^{r-r'}F'_{k,k})]_{E'_1}\cdot \frac{1}{p^2\Theta\Omega}\big(
    \tilde\Delta(\phi(\gamma)-1)-\tilde\Delta(-1)
    \big)
\end{equation}
and so it belongs to $p^{p-3} S_\cO'$ if $[\phi'_r(v^{r-r'}F'_{k,k})]_{E'_1}\in S_\cO'$, by the assumption \eqref{eq: condition for phi,N stable} together with the fact \eqref{eq: p^p-1 divides phi(gamma)}. Hence, $\tilde{\Delta}$ is not relevant to the strong divisibility if $p> 3$. If $p=3$ then the quantity \eqref{eq: phi N stable equality} vanishes by the hypothesis. Similarly, if $1\leq k\leq r$, then the coefficient of the second term of $\phi'_r(v^{r-r'}F'_{k,0})$ can be written as in \eqref{eq: phi N stable equality}, and so by the same argument we conclude that $\tilde{\Delta}$ is not relevant to the strong divisibility.
\end{proof}

Since $\tilde{\Delta}(T)$ is not relevant to the strong divisibility, the following definition makes sense.
\begin{defi}
Let $k'\in\cK_{r'}$ and $k:=\lfloor k'\rfloor$, and let $\fM'$ be a pseudo-strongly divisible module in $\mathbf{Case}~(k)$. For given quadruple $(\Lambda,\Theta,\Omega,x)\in (E^\times)^3\times E$ and $\Delta(T)\in E(T)_\Delta$ (possibly depending on $x$), we say that the quadruple $(\Lambda,\Theta,\Omega,x)$ is \emph{good for $\mathbf{Case}_\phi(k')$ with respect to $\Delta(T)$} if $\fM'$ satisfies the following:
\begin{itemize}[leftmargin=*]
\item $\phi'(\fM')\subseteq \fM'$ and $N'(\fM')\subseteq\fM'$, i.e., the inequalities in \eqref{eq: condition for phi,N stable};
\item $\Fil^{r;r'}_k\fM'\subseteq\fM'$, i.e., $F'_{k,i}\in \fM'$ for all $0\leq i\leq r'-1$;
\item the strong divisibility of type $\mathbf{Case}_\phi(k')$.
\end{itemize}
\end{defi}

Recall that $\delta_{l}(T,x)$ for $0\leq l\leq \lfloor\tfrac{r'}{2}\rfloor+1$ are defined in Definition~\ref{defi: definition of delta}. Note that $\Delta(T)$ of $\mathbf{Case}_\phi(k')$ will be chosen as follows: if we set
\begin{equation}\label{eq: definition of l(r;k')}
l(r';k'):=
\begin{cases}
\lceil k'\rceil & \mbox{if }\frac{1}{2}\leq k'\leq m+1;\\
r'+1-\lceil k'\rceil & \mbox{if }m+1< k'\leq r'+\frac{1}{2};\\
0 & \mbox{if }k'=\infty,
\end{cases}
\end{equation}
then we have 
\begin{equation}\label{eq: Choice of delta}
\Delta(T):=\delta_{l(r';k')}(T,x)
\end{equation} 
in each $\mathbf{Case}_\phi(k')$. We recall that $\delta_{l}(T,x)=\delta_l(T)$ unless either $r'=2m$ and $l=m+1$ or $l=0$, as we pointed out right after Definition~\ref{defi: definition of delta}.

\begin{rema}
The choice of $\Delta(T)$ in \eqref{eq: Choice of delta} is set to assure $[\phi_r(v^{r-r'}G'_i)]_{E'_2}\in S'_{\cO}$ for $i=1$ or $2$ depending on $\mathbf{Case}_\phi(k')$. More precisely, if $k'=k\in\Z$ then $\Delta(T)$ is chosen to have $$[P_{k,0}(c,x)cf(c)]_c^{(r-1)}-\Delta(\gamma-1)P_{k,0}(c,x)\in P_{k,0}(c,x)\fm S'_{\cO}$$ from \eqref{eq: formula for phi_r(F_k,0)}, while if $k'=k+\frac{1}{2}\in \frac{1}{2}+\Z$ then $\Delta(T)$ is chosen to have $$[P_{k,k}(c,x)cf(c)]_c^{(r-1)}-\Delta(\gamma-1)P_{k,k}(c,x)\in P_{k,k}(c,x)\fm S'_{\cO}$$ from \eqref{eq: formula for phi_r(F_k,k)}. We believe that the choice of $\Delta$ is essentially unique!
\end{rema}

The main results in this section is to find out the precise conditions on $(\Lambda,\Theta,\Omega,x)$ under which the quadruple is good for $\mathbf{Case}_\phi(k')$ with respect to $\Delta(T)$ determined in \eqref{eq: Choice of delta} for each $k'\in\cK_{r'}$. The cases $r'>1$ and $k'\neq\infty$ are treated in \S\ref{subsec: Strong divisibility} while the cases $r'=1$ or $k'=\infty$ are in \S\ref{subsec: exceptional cases}. Those conditions are summarized in following table, Table~\ref{tab:my_label}. Recall that we write $k$ for $\lfloor k'\rfloor$ for $k'\in\cK_{r'}$ with convention $\lfloor\infty\rfloor=\infty$.

\begin{defi}\label{defi: equation and inequalities, local}
By the \emph{equations and inequalities of $\mathbf{Case}_\phi(k')$} (resp. by \emph{$\Delta$ of $\mathbf{Case}_\phi(k')$}), we mean those on $(\Lambda,\Theta,\Omega,x)\in (E^\times)^3\times E$ in Table~\ref{tab:my_label} (resp. those defined in \eqref{eq: Choice of delta} or equivalently in Table~\ref{tab:my_label}).
\end{defi}
Note that we often write $\mathbf{Case}_\phi(r';k')$ for $\mathbf{Case}_\phi(k')$ to emphasize $r'$.

\begin{tiny}
\begin{table}[htbp]
    \centering
    \begin{tabular}{|c||c|c|c|c|}
      \hline
       $k'$
        & Eq. of $\mathbf{Case}_\phi(k')$ & Inequalities of $\mathbf{Case}_\phi(k')$ & $\Delta$ of $\mathbf{Case}_\phi(k')$ \\\hline\hline
       $k'=\frac{1}{2}$
        &
        \begin{tabular}{@{}c@{}}
         $v_p(\Lambda)=r'-1$, \\ $v_p(\Theta\Omega^{-1})=-r'+1$
        \end{tabular}
        &
        \begin{tabular}{@{}c@{}}
         $v_p(\Theta)\leq v_p(x)$, \\ $v_p(\Omega)\leq 0$
        \end{tabular}
        &
        \begin{tabular}{@{}c@{}}
         $\delta_1(T)$
        \end{tabular}
        \\\hline
        $\frac{1}{2}< k'=k+\frac{1}{2}\leq \frac{r'}{2}$
        &
        \begin{tabular}{@{}c@{}}
         $v_p(\Lambda)=r'-k-1$, \\ $v_p(\Theta\Omega^{-1})=-r'+2k'$
        \end{tabular}
        &
        \begin{tabular}{@{}c@{}}
         $v_p(\Theta)\leq v_p(x)$, \\ $-1\leq v_p(\Omega)\leq 0$
        \end{tabular}
        &
        \begin{tabular}{@{}c@{}}
         $\delta_{k+1}(T)$
        \end{tabular}
        \\\hline
        $\frac{r'+1}{2}\leq k'=k+\frac{1}{2}<r' $
        &
        \begin{tabular}{@{}c@{}}
         $v_p(\Lambda x)=r'-k-1$, \\ $v_p(\Theta\Omega^{-1} x^{-2})=-r'+2k'$
        \end{tabular}
        &
        \begin{tabular}{@{}c@{}}
         $0> v_p(x)\leq v_p(\Theta)\leq 0$, \\ $-1\leq v_p(\Omega)\leq 0$
        \end{tabular}
        &
        \begin{tabular}{@{}c@{}}
          $\delta_{m+1}(T,x)\mbox{ if }r'=2k$,\\ $\delta_{r-k}(T)\mbox{ otherwise}$
        \end{tabular}

        \\\hline
        $k'=r'+\frac{1}{2} $
        &
        \begin{tabular}{@{}c@{}}
         $v_p(\Lambda)=-1$, \\ $v_p(\Theta\Omega^{-1})=r'+1$
        \end{tabular}
        &
        \begin{tabular}{@{}c@{}}
         $v_p(x)+r'\leq v_p(\Theta)\leq 0$
        \end{tabular}
        &
        $\delta_0(T,x)$

        \\\hline
       $\frac{1}{2}< k'=k\leq \frac{r'}{2}$
        &
        \begin{tabular}{@{}c@{}}
         $v_p(\Lambda \Theta)=k-1$, \\ $v_p(\Theta\Omega)=-r'+2k'-1$
        \end{tabular}
        &
        \begin{tabular}{@{}c@{}}
         $v_p(\Theta)\leq v_p(x)$, \\ $-1\leq v_p(\Omega)\leq 0$
        \end{tabular}
        &
        \begin{tabular}{@{}c@{}}
         $\delta_k(T)$
        \end{tabular}
        \\\hline
       $\frac{r'+1}{2}\leq k'=k <r'$
        &
        \begin{tabular}{@{}c@{}}
         $v_p(\Lambda \Theta x^{-1})=k-1$, \\ $v_p(\Theta\Omega x^{-2})=-r'+2k'-1$
        \end{tabular}
        &
        \begin{tabular}{@{}c@{}}
         $0> v_p(x)\leq v_p(\Theta)\leq 0$, \\ $-1\leq v_p(\Omega)\leq 0$
        \end{tabular}
        &
        \begin{tabular}{@{}c@{}}
         $\delta_{k}(T,x)\mbox{ if }k=m+1$,\\ $\delta_{r+1-k}(T)\mbox{ otherwise}$
        \end{tabular}
        \\\hline
        $k'=r'$
        &
        \begin{tabular}{@{}c@{}}
         $v_p(\Lambda \Theta x^{-1})=r'-1$, \\ $v_p(\Theta\Omega x^{-2})=r'-1$
        \end{tabular}
        &
        \begin{tabular}{@{}c@{}}
         $v_p(x)\leq v_p(\Theta)\leq v_p(x)+r'$, \\ $v_p(\Omega)\leq 0$, $v_p(\Theta)\leq 0$
        \end{tabular}
        &
        \begin{tabular}{@{}c@{}}
         $\delta_2(T,x)\mbox{ if }r'=2$, \\$\delta_1(T)\mbox{ if }r'\neq2$
        \end{tabular}
        \\\hline
        $k'=\infty$
        &
        \begin{tabular}{@{}c@{}} $v_p(\Lambda)=-1$,\\ $v_p(\Theta\Omega^{-1})=r'+1$
        \end{tabular}
        &
        \begin{tabular}{@{}c@{}}
        $v_p(\Theta)\leq 0$
        \end{tabular}
        &
        \begin{tabular}{@{}c@{}}
         $\delta_{0}(T,x)$
        \end{tabular}
        \\ \hline
    \end{tabular}
    \caption{Equations, Inequalities, and $\Delta$ of $\mathbf{Case}_\phi(k')$}
    \label{tab:my_label}
\end{table}
\end{tiny}

We point out that for each $k'\in\cK_{r'}$ the equations and inequalities of $\mathbf{Case}_\phi(k')$ always imply that $v_p(\Theta)\leq 0$ and $v_p(\Omega)\leq 0$, and that if $k'=m+\frac{1}{2}$ and $r'=2m$ then the equations and inequalities in Proposition~\ref{prop-cM-Case-m} imply $-1\leq v_p(\Omega)$, which is added to Table~\ref{tab:my_label} for a uniform expression. From the second column of Table~\ref{tab:my_label} we immediately get
\begin{equation}\label{eq: Lambda Omega Theta = r'-1}
v_p(\Lambda^2\Theta\Omega^{-1})=r'-1.
\end{equation}
This phenomenon is novel and we didn't expect it. Moreover, we also point out that the equations and inequalities of $\mathbf{Case}_\phi(k')$ imply the inequalities in \eqref{eq: condition for phi,N stable} as well as the inequalities in Theorem~\ref{theo: filtration of pseudo}.

\begin{theo}\label{theo: strong divisibility of pseudo}
If the quadruple $(\Lambda,\Theta,\Omega,x)$ satisfies the equations and inequalities of $\mathbf{Case}_{\phi}(r';k')$ for some $k'\in\cK_{r'}$ then the quadruple is good for $\mathbf{Case}_\phi(k')$ with respect to $\Delta(T):=\delta_{l(r';k')}(T,x)$.
\end{theo}

We further compute the mod-$p$ reduction of $G'_1,G'_2$ as well as of $\phi'_r(G'_1),\phi'_r(G'_2)$ for each $k'\in\cK_{r'}$, which will be used to describe the corresponding Breuil modules, for the rest of this section. In fact, due to the simpler structures of the Breuil modules, it is more convenient to compute the mod-$p$ reduction of $F_1',F_2'$, where 
\begin{equation*}
    (F'_1,\, F'_2):=
    \begin{cases}
        \big(G_1'+\frac{x}{\Theta}G_2',\,G_2'  \big)  & \mbox{if }r'=2m\mbox{ and }k'=m; \\
        \big(G_1',\,\frac{\Theta}{x}G_1'+G_2'\big) & \mbox{if }r'=2m\mbox{ and }k'=m+\tfrac{1}{2}; \\
        \big(G_1',\,G_2'\big)  & \mbox{otherwise. }
    \end{cases}
\end{equation*}
This modification only arises in Proposition~\ref{prop-cM-Case-m}.

\subsection{Generators of the filtration}\label{subsec: Generators of the filtration}
In this subsection, we prove Theorem~\ref{theo: filtration of pseudo} when $r'\geq 2$ and $k'\neq\infty$. It is harmless to assume that $r=r'$, as we can recover the general case $r'\leq r$ by multiplying $v^{r-r'}$ to the generators. Assume that $r=r'\geq 2$, and set
    \begin{equation*}
        \Fil^r_{k}\fM':=\Fil^{r;r}_{k}\fM'
    \end{equation*}
for brevity.

We further define a $\barS'_\F$-module
    \begin{equation*}
        \cM':=\fM'/(\fm,\Fil^p S'_\cO)\fM',
    \end{equation*}
and write $\barE'_1$ and $\barE'_2$ for the image of $E'_1$ and $E'_2$, respectively, under the natural quotient map $\pi':\fM'\to\cM'$. We also set
$$\Fil^r_{k}\cM':=\pi'(\Fil^r_{k}\fM').$$

In this subsection, we shall investigate the conditions for $[F'_{k,i}]_{E'_j}$'s to be integral. Also, under such conditions, we shall compute their mod-$p$ reductions. Moreover, we will determine the $S'_\cO$-generators of $\Fil^r_{k}\fM'$ for each $\mathbf{Case}~(k)$.

\subsubsection{\textbf{In $\mathbf{Case}~(0)$}}
In this section of paragraph, we describe the $S'_\cO$-generators of $\Fil^r_{k}\fM'$ and the $\barS'_\F$-generators of $\Fil^r_{k}\cM'$ when $k=0$, i.e., in the case $\mathbf{Case}~(0)$.
\begin{prop}\label{prop: Fil-Case-0}
Assume $k=0$. Then $F'_{0,i}\in \fM'$ for all $0\leq i\leq r-1$ if and only if
\begin{equation}\label{eq: condition for k=0}
v_p(\Theta)\leq v_p(x)\quad\mbox{ and }\quad v_p(\Theta)\leq -r+1,
\end{equation}
in which case we have
    $$\Fil^r_{0}\fM'=S'_\cO (F'_{0,0},v^r E_2')+\Fil^p S'_\cO\fM'$$
    and
    $\Fil^r_{0}\cM'=\barS'_\F(\barF'_{0,0},u^r\barE'_2)$ with
    $$\barF'_{0,0}=\barE'_1 +\bigg(\frac{x}{\Theta}+\frac{1}{p^{r-1}\Theta}\frac{(-1)^{r-2}}{r-1}u^{r-1}\bigg)\barE'_2.$$
\end{prop}

\begin{proof}
If $k=0$, then by Lemma~\ref{lemma: k=0,k=r} we have
\begin{align*}
[F'_{0,i}]_{E'_1}=p^iP_{0,i}\Big(\frac{v}{p},x\Big)
    =v^i \,\,\mbox{ and }\,\,
[F'_{0,i}]_{E'_2}=\frac{p^i}{\Theta}Q_{0,i}\Big(\frac{v}{p},x\Big)
    =\frac{p^i}{\Theta}Q_{0,i}\Big(\dfrac{v}{p},0\Big)+\frac{x}{\Theta}v^i.
\end{align*}
Note that by Lemma~\ref{lemma: k=0,k=r} we have
\begin{itemize}[leftmargin=*]
  \item $Q_{0,r-1}(T,0)=0$;
  \item for $0\leq i\leq r-2$, $Q_{0,i}(T,0)=\sum_{n=0}^{r-2-i}\frac{(-1)^n}{n+1}T^{n+1+i}$.
\end{itemize}
Hence, we have
\begin{align*}
    \frac{p^i}{\Theta}Q_{0,i}\Big(\frac{v}{p},0\Big) =\frac{p^i}{\Theta}\sum_{n=0}^{r-2-i}\frac{(-1)^n}{n+1}\Big(\frac{v}{p}\Big)^{n+1+i}.
    \end{align*}
Thus, $[F'_{0,i}]_{E'_2}\in S'_\cO$ if and only if
\begin{align*}
    v_p(\Theta)\leq v_p(x)
    \qquad\mbox{and}\qquad
    v_p(\Theta)\leq -r+1,
\end{align*}
in which case, we have
\begin{align*}
F'_{0,0}
    &\equiv
    E'_1
    +\bigg(\frac{x}{\Theta}+\frac{1}{p^{r-1}\Theta}\frac{(-1)^{r-2}}{r-1}v^{r-1}\bigg)E'_2
    \pmod{\fm\fM'};
    \\
F'_{0,i}
    &\equiv
    v^iE'_1
    +\frac{x}{\Theta}v^iE'_2
    \pmod{\fm\fM'}\,\,\,\mbox{ if }0<i\leq r-1.
\end{align*}

We note that
\begin{align*}
F'_{0,i}
    &\equiv v^iF'_{0,0}-\frac{1}{p^{r-1}\Theta}\frac{(-1)^{r-2}}{r-1}v^{i+r-1}E'_2
    \pmod{\fm\fM'};\\
v^rE'_1
    &\equiv
     v^rF'_{0,0}-\bigg(
    \frac{x}{\Theta}+\frac{1}{p^{r-1}\Theta}\frac{(-1)^{r-2}}{r-1}v^{r-1}
    \bigg)v^{r}E'_2
    \pmod{\fm\fM'}.
\end{align*}
Hence, we deduce that $\Fil^r_{0}\cM'=S'_\F(\barF'_{0,0},u^r\barE'_2)$. Now, by Nakayama's Lemma,
\begin{align*}
    \Fil^r_{0}\fM'/\Fil^p S'_\cO\fM'  = S'_\cO/\Fil^p S'_\cO(F'_{0,0},v^rE'_2),
\end{align*}
and so we have
\begin{equation}\label{eq: Nakayama}
    \Fil^r_{0}\fM'
    =\SBr'(F'_{0,0},v^rE'_2)+\Fil^p\SBr'\fM',
\end{equation}
which completes the proof.
\end{proof}

\subsubsection{\textbf{In $\mathbf{Case}~(k)$ for $0< k<r/2$}}
In this section of paragraph, we describe the $S'_\cO$-generators of $\Fil^r_{k}\fM'$ and the $\barS'_\F$-generators of $\Fil^r_{k}\cM'$ in the case $\mathbf{Case}~(k)$ for $0< k<r/2$.

\begin{prop}\label{prop: Fil-Case-1-r/2}
Assume that $1\leq k<r/2$. Then $F'_{k,i}\in\fM'$ for all $0\leq i\leq r-1$ if and only if
\begin{equation}\label{eq: condition for 1 leq k <r/2}
-r+2k-1\leq v_p(\Theta)\leq -r+2k+1 \quad\mbox{ and }\quad v_p(\Theta)\leq v_p(x),
\end{equation}
in which case we have $$\Fil^r_{k}\fM'=S'_\cO(F'_{k,k},F'_{k,0})+\Fil^p S'_\cO\fM'$$ and $\Fil^r_{k}\cM'=\barS'_\F(\barF'_{k,k},\barF'_{k,0})$ with
    \begin{align*}
    \barF'_{k,k}&=u^k\barE'_1  +\bigg(\dfrac{1}{p^{r-2k-1}\Theta}\dfrac{1}{a_{k+1}}u^{r-k-1}+\dfrac{x}{\Theta}u^k\bigg)\barE'_2;    \\
    \barF'_{k,0}&=p^{r-2k+1}\Theta a_ku^{k-1}\barE'_1 +(u^{r-k}+p^{r-2k+1}xa_ku^{k-1})\barE'_2.
    \end{align*}
\end{prop}

We prove this proposition by a series of lemmas, dividing the values of $i$ into three cases: $0\leq i\leq k-1$, $k\leq i\leq r-k-1$, and $r-k\leq i\leq r-1$.
\begin{lemm}\label{lemma 0 < k <r/2. 0 leq i leq k-1}
Keep the assumption of Proposition~\ref{prop: Fil-Case-1-r/2}. Then $F'_{k,i}\in\fM'$ for $0\leq i\leq k-1$ if and only if
\begin{align*}
    v_p(\Theta)\geq -r+2k-1
    \quad\mbox{ and }\quad
    v_p(x)\geq -r+2k-1,
\end{align*}
in which case we have
\begin{align*}
F'_{k,0}
    &\equiv
    p^{r-2k+1}\Theta a_kv^{k-1}E'_1
    +(v^{r-k}+p^{r-2k+1}xa_kv^{k-1})E'_2
    \pmod{\fm\fM'};
    \\
F'_{k,i}
    &\equiv
    v^{i+r-k}E'_2
    \pmod{\fm\fM'}\,\,\mbox{ if }0<i\leq k-1.
\end{align*}
\end{lemm}

\begin{proof}
Assume that $0\leq i\leq k-1$. By Lemma~\ref{lemma-PQ-k leq m} we have
\begin{align*}
[F'_{k,i}]_{E'_1}&=p^{i+r-k}\Theta P_{k,i}\Big(\frac{v}{p},0\Big);
    \\
[F'_{k,i}]_{E'_2} &=p^{i+r-k}Q_{k,i}\Big(\frac{v}{p},0\Big)
    +p^{i+r-k}xP_{k,i}\Big(\frac{v}{p},0\Big).
\end{align*}

Note that $\deg_T P_{k,i}(T,0)\leq k-1$ by Lemma~\ref{lemma: PQdegree} and that $\deg_TP_{k,0}(T,0)=k-1$ with the leading coefficient $a_k\in\Z_{(p)}^{\times}$ by Lemma~\ref{lemma: B_k,0 k leq m}. Hence, we have
\begin{align*}
    p^{i+r-k}\Theta P_{k,i}\Big(\frac{v}{p},0\Big)
    \equiv
    \begin{cases}
        p^{r-2k+1}\Theta a_kv^{k-1}\pmod{p^{r-2k+1}\Theta\fm S'_\cO} & \mbox{if } i=0; \\
        0\pmod{p^{r-2k+1}\Theta\fm S'_\cO} & \mbox{otherwise},
    \end{cases}
\end{align*}
and
\begin{align*}
    p^{i+r-k}xP_{k,i}\Big(\frac{v}{p},0\Big)
    \equiv
    \begin{cases}
        p^{r-2k+1}xa_kv^{k-1}\pmod{p^{r-2k+1}x\fm S'_\cO} & \mbox{if } i=0; \\
        0 \pmod{p^{r-2k+1}x\fm S'_\cO} & \mbox{otherwise}.
    \end{cases}
\end{align*}
Moreover, as  $Q_{k,i}(T,0)$ is a monic polynomial of degree $i+r-k$ by Lemma~\ref{lemma: PQdegree}, we have
\begin{align*}
    p^{i+r-k}Q_{k,i}\Big(\frac{v}{p},0\Big)
    \equiv v^{i+r-k} \pmod{\fm S'_\cO}.
\end{align*}

Now, it is immediate that $F'_{k,i}\in\fM'$ for $0\leq i\leq k-1$ if and only if $v_p(\Theta)\geq -r+2k-1$ and $v_p(x)\geq -r+2k-1$. It is also immediate that if $v_p(\Theta)\geq -r+2k-1$ and $v_p(x)\geq -r+2k-1$ then $\barF'_{k,i}$ for $0\leq i\leq k-1$ is computed as in the statement of Lemma~\ref{lemma 0 < k <r/2. 0 leq i leq k-1}.
\end{proof}

\begin{lemm}\label{lemma 0 < k <r/2. k leq i leq r-k-1}
Keep the assumption of Proposition~\ref{prop: Fil-Case-1-r/2}. Then $F'_{k,i}\in\fM'$ for $k\leq i\leq r-k-1$ if and only if
\begin{align*}
    v_p(\Theta)\leq -r+2k+1
    \quad\mbox{ and }\quad
    v_p(\Theta)\leq v_p(x),
\end{align*}
in which case
\begin{align*}
F'_{k,k}
    &\equiv
    v^kE'_1+\bigg(
    \frac{1}{p^{r-2k-1}\Theta}\frac{1}{a_{k+1}}v^{r-k-1}+\frac{x}{\Theta}v^k
    \bigg)E'_2
    \pmod{\fm\fM'};
    \\
F'_{k,i}
    &\equiv
    v^iE'_1+\frac{x}{\Theta}v^iE_2 \pmod{\fm\fM'}\,\,\mbox{ if }k+1\leq i\leq r-k-1.
\end{align*}
\end{lemm}

\begin{proof}
Assume that $k\leq i\leq r-k-1$. By Lemma~\ref{lemma-PQ-k leq m} we have
\begin{align*}
[F'_{k,i}]_{E'_1}&=p^i P_{k,i}\Big(\frac{v}{p},0\Big);\\
[F'_{k,i}]_{E'_2}&=\frac{p^i}{\Theta}Q_{k,i}\Big(\frac{v}{p},0\Big)
    +\frac{p^ix}{\Theta}P_{k,i}\Big(\frac{v}{p},0\Big).
\end{align*}

As $P_{k,i}(T,0)$ is a monic polynomial of degree $i$ by Lemma~\ref{lemma: PQdegree}, we have
\begin{align*}
    p^i P_{k,i}\Big(\frac{v}{p},0\Big)\equiv v^i \pmod{\fm S'_\cO}
\quad\mbox{and}\quad
\frac{p^ix}{\Theta}P_{k,i}\Big(\frac{v}{p},0\Big)
    \equiv \frac{x}{\Theta}v^i \pmod{\frac{x}{\Theta}\fm S'_\cO}.
\end{align*}
Moreover, note that ${1}/{a_{k+1}}\in\Z_{(p)}$ by Lemma~\ref{lemma: B_k,0 k leq m}~(iii) and Lemma~\ref{lemma: a_k m+1 leq k}~(v), that $\deg_T Q_{k,i}(T,0)\leq r-k-1$, and that $\deg_T Q_{k,k}(T,0)=r-k-1$ with the leading coefficient $1/a_{k+1}\in\Z_{(p)}$ by Lemma~\ref{lemma: PQdegree} together with Lemma~\ref{lemma: B_k,0 k leq m}, and so we have
\begin{align*}
    \frac{p^i}{\Theta}Q_{k,i}\Big(\frac{v}{p},0\Big)
    \equiv
    \begin{cases}
        \frac{1}{p^{r-2k-1}\Theta}\frac{1}{a_{k+1}}v^{r-k-1}\pmod{\frac{1}{p^{r-2k-1}\Theta}\fm S'_\cO} & \mbox{if } i=k;\\
        0\pmod{\frac{1}{p^{r-2k-1}\Theta}\fm S'_\cO} & \mbox{otherwise}.
    \end{cases}
\end{align*}

Now, it is immediate that $F'_{k,i}\in\fM'$ for $k\leq i\leq r-k-1$ if and only if $v_p(\Theta)\leq -r+2k+1$ and $v_p(\Theta)\leq v_p(x)$. It is also immediate that if $v_p(\Theta)\leq -r+2k+1$ and $v_p(\Theta)\leq v_p(x)$ then $\barF'_{k,i}$ for $k\leq i\leq r-k-1$ is computed as in the statement of Lemma~\ref{lemma 0 < k <r/2. k leq i leq r-k-1}.
\end{proof}

From Lemma~\ref{lemma 0 < k <r/2. 0 leq i leq k-1} and Lemma~\ref{lemma 0 < k <r/2. k leq i leq r-k-1}, it is easy to see that
$F'_{k,i}\in\fM'$ for $0\leq i\leq r-k-1$ if and only if the inequalities in \eqref{eq: condition for 1 leq k <r/2} hold. In the next lemma, we will see that we don't get any extra inequalities from the case $r-k\leq i\leq r-1$.

\begin{lemm}\label{lemma 0 < k <r/2. r-k leq i leq r-1}
Keep the assumption of Proposition~\ref{prop: Fil-Case-1-r/2}, and assume further
\begin{equation*}
v_p(x)\geq -r+2k-1,\quad v_p(\Theta)-2v_p(x)\leq r-2k+1,\quad\mbox{and}\quad v_p(\Theta)\leq v_p(x).
\end{equation*}
Then we have
\begin{align*}
F'_{k,r-k}
&\equiv
    (v^{r-k}-p^{r-2k+1}xa_kv^{k-1})E'_1
    +\frac{p^{r-2k+1}x^2}{\Theta}v^{k-1}E'_2\pmod{\fm\fM'}; \\
F'_{k,i}&\equiv v^iE'_1\pmod{\fm\fM'}\,\,\mbox{ if }r-k<i\leq r-1.
\end{align*}
\end{lemm}

\begin{proof}
Assume that $r-k\leq i\leq r-1$. By Lemma~\ref{lemma-PQ-k leq m}, we have
\begin{align*}
[F'_{k,i}]_{E'_1}
    &=p^i P_{k,i}\Big(\frac{v}{p},0\Big)
    -p^ix P_{k,i-r+k}\Big(\frac{v}{p},0\Big);
    \\
[F'_{k,i}]_{E'_2}
    &=\frac{p^i}{\Theta}Q_{k,i}\Big(\frac{v}{p},0\Big)
    +\frac{p^ix}{\Theta}
    \Big(P_{k,i}\Big(\frac{v}{p},0\Big)
    -Q_{k,i-r+k}\Big(\frac{v}{p},0\Big)\Big)
    -\frac{p^ix^2}{\Theta}P_{k,i-r+k}\Big(\frac{v}{p},0\Big).
\end{align*}

Since $P_{k,i}(T,0)$ is a monic polynomial of degree $i$ by Lemma~\ref{lemma: PQdegree}, we have
\begin{align*}
    p^iP_{k,i}\Big(\frac{v}{p},0\Big)\equiv v^i \pmod{\fm S'_\cO}.
\end{align*}
We note that $\deg_T P_{k,i-r+k}(T,0)\leq k-1$ for all $r-k<i\leq r-1$ and that $\deg_TP_{k,0}(T,0)=k-1$ with the leading coefficient $a_k\in\Z_{(p)}^{\times}$, by Lemma~\ref{lemma: PQdegree} and Lemma~\ref{lemma: B_k,0 k leq m}, and so we have
\begin{align*}
    p^ix P_{k,i-r+k}\Big(\frac{v}{p},0\Big)\equiv
    \begin{cases}
        p^{r-2k+1}xa_kv^{k-1}\pmod{\fm S_\cO'} & \mbox{if } i=r-k; \\
        0 \pmod{\fm S_\cO'}& \mbox{otherwise}
    \end{cases}
\end{align*}
as $v_p(x)\geq -r+2k-1$, and
\begin{align*}
    \frac{p^ix^2}{\Theta}P_{k,i-r+k}\Big(\frac{v}{p},0\Big)\equiv
    \begin{cases}
        \frac{p^{r-2k+1}x^2}{\Theta}a_kv^{k-1}\pmod{\fm S_\cO'} & \mbox{if } i=r-k; \\
        0 \pmod{\fm S_\cO'}& \mbox{otherwise}
    \end{cases}
\end{align*}
as $v_p(\Theta)-2v_p(x)\leq r-2k+1$. Since $\deg_T Q_{k,i}(T,0)\leq r-k-1$ by Lemma~\ref{lemma: PQdegree} and $v_p(\Theta)\leq 0$, for all $r-k\leq i\leq r-1$ we have
\begin{align*}
    \frac{p^i}{\Theta}Q_{k,i}\Big(\frac{v}{p},0\Big) \equiv 0 \pmod{\fm\fM'}.
\end{align*}
As $\deg_T \big(P_{k,i}(T,0)-Q_{k,i-r+k}(T,0)\big)\leq r-k-1$ by Lemma~\ref{lemma: PQdegree}, for all $r-k\leq i\leq r-1$
we have
\begin{align*}
    \frac{p^ix}{\Theta}\Big(&
    P_{k,i}\Big(\frac{v}{p},0\Big)-Q_{k,i-r+k}\Big(\frac{v}{p},0\Big)
    \Big)
    \equiv
    0 \pmod{\fm S_\cO'}
\end{align*}
as $v_p(\Theta)\leq v_p(x)$.

Now, it is immediate that if the inequalities in the statement of Lemma~\ref{lemma 0 < k <r/2. r-k leq i leq r-1} hold then $\barF'_{k,i}$ for $r-k\leq i\leq r-1$ is computed as in the statement of Lemma~\ref{lemma 0 < k <r/2. r-k leq i leq r-1}.
\end{proof}

It is easy to see that the inequalities in \eqref{eq: condition for 1 leq k <r/2} imply the ones in Lemma~\ref{lemma 0 < k <r/2. r-k leq i leq r-1}. Hence, from the three lemmas above, we now conclude that $F'_{k,i}\in\fM'$ for all $0\leq i\leq r-1$ if and only if the inequalities in \eqref{eq: condition for 1 leq k <r/2} hold.

Finally, we need to check that $\Fil^r_{k}\cM'$ is generated by $\barF'_{k,k},\,\barF'_{k,0}$ over~ $\barS'_\F$.
\begin{proof}[Proof of Proposition~\ref{prop: Fil-Case-1-r/2}]
Assume that the inequalities in \eqref{eq: condition for 1 leq k <r/2} hold, and note that $\barF'_{k,i}$ for all $0\leq i\leq r-1$ are computed in the preceding lemmas.

We note from Lemma~\ref{lemma 0 < k <r/2. 0 leq i leq k-1} and Lemma~\ref{lemma 0 < k <r/2. k leq i leq r-k-1} that
\begin{align*}
    vF'_{k,0}-p^{r-2k+1}\Theta a_kF'_{k,k} \equiv   v^{r-k+1}E'_2  \pmod{\fm\fM'}.
\end{align*}
Hence, if $1\leq i\leq k-1$ then we have
\begin{align*}
    F'_{k,i}
    \equiv
    v^{i-1}(vF'_{k,0}-p^{r-2k+1}\Theta a_kF'_{k,k})
    \pmod{\fm\fM'},
\end{align*}
and if $k+1\leq i\leq r-k-1$ then we have
\begin{align*}
    F'_{k,i}
    &\equiv
    v^{i-k}F'_{k,k}
    -\frac{1}{p^{r-2k-1}\Theta a_{k+1}}v^{i-k-2}(vF'_{k,0}-p^{r-2k+1}\Theta a_kF'_{k,k})\\
    &\equiv
    v^{i-k}F'_{k,k}-\frac{1}{p^{r-2k-1}\Theta a_{k+1}}v^{i-k-1}F'_{k,0}
    \pmod{\fm\fM'}.
\end{align*}
Moreover, we have
\begin{align*}
    F'_{k,r-k}
    &\equiv
    v^{r-k}E'_1-\frac{x}{\Theta}\big(F'_{k,0}-v^{r-k}E'_2\big)\\
    &\equiv
    v^{r-2k}(v^kE'_1+\frac{x}{\Theta}v^kE'_2)
    -\frac{x}{\Theta}F'_{k,0}\\
    &\equiv
    v^{r-2k}F'_{k,k}-\frac{1}{p^{r-2k-1}\Theta a_{k+1}}v^{2r-3k-1}E'_2
    -\frac{x}{\Theta}F'_{k,0}\\
    &\equiv
     v^{r-2k}F'_{k,k}-\frac{1}{p^{r-2k-1}\Theta a_{k+1}}v^{r-2k}\left(vF'_{k,0}-p^{r-2k+1}\Theta a_kF'_{k,k}\right)
    -\frac{x}{\Theta}F'_{k,0}
\end{align*}
modulo $\fm\fM'$, and if $r-k+1\leq i\leq r-1$ then we also have
\begin{align*}
    F'_{k,i}
    &\equiv
    v^{i-k}F'_{k,k}
    -v^{i-k}\bigg(\frac{1}{p^{r-2k-1}\Theta a_{k+1}}v^{r-k-1}+\frac{x}{\Theta}v^k\bigg)E'_2\\
    &\equiv
    v^{i-k}F'_{k,k}
    -\bigg(\frac{1}{p^{r-2k-1}\Theta a_{k+1}}v^{i-k-2}+\frac{x}{\Theta}v^{i-r+k-1}\bigg)
    (vF'_{k,0}-p^{r-2k+1}\Theta a_kF'_{k,k})\\
    &\equiv
    (v^{i-k}-p^{r-2k+1}xa_kv^{i-r+k-1})F'_{k,k}
    -\bigg(\frac{1}{p^{r-2k-1}\Theta a_{k+1}}v^{i-k-1}+\frac{x}{\Theta}v^{i-r+k}\bigg)F'_{k,0}
\end{align*}
modulo $\fm\fM'$.

Finally, it is immediate that $$v^rE'_2
    \equiv
    v^{k}F'_{k,0}-p^{r-2k+1}\Theta a_kv^{k-1}F'_{k,k}\pmod{\fm\fM'},$$
and
\begin{align*}
v^rE'_1
    &\equiv v^{r-k}F'_{k,k}-\bigg(\frac{1}{p^{r-2k-1}\Theta a_{k+1}}v^{2r-2k-1}+\frac{x}{\Theta}v^r\bigg)E'_2 \\
    &\equiv v^{r-k}F'_{k,k}-\bigg(\frac{1}{p^{r-2k-1}\Theta a_{k+1}}v^{r-k-2}+\frac{x}{\Theta}v^{k-1}\bigg)\left(vF'_{k,0}-p^{r-2k+1}\Theta a_kF'_{k,k}\right)\\
    &\equiv
    (v^{r-k}+p^{r-2k+1}xa_kv^{k-1})F'_{k,k}
    -\bigg(\frac{1}{p^{r-2k-1}\Theta a_{k+1}}v^{r-k-1}+\frac{x}{\Theta}v^k\bigg)F'_{k,0}
\end{align*}
modulo $\fm\fM'$.

Hence, we deduce that $\Fil^r_{k}\cM'=\barS'_\F(\barF'_{k,k},\barF'_{k,0})$. Now, by the same argument as in \eqref{eq: Nakayama}, we have $\Fil^r_{k}\fM'=S'_\cO(F'_{k,k},F'_{k,0})+\Fil^pS'_\cO\fM'$, which completes the proof.
\end{proof}

\subsubsection{\textbf{In $\mathbf{Case}~(m)$ for $r=2m$}}
In this section of paragraph, we describe the $S'_\cO$-generators of $\Fil^r_{m}\fM'$ and the $\barS'_\F$-generators of $\Fil^r_{m}\cM'$ in the case $\mathbf{Case}~(m)$ when $r=2m$.

\begin{prop}\label{prop: Fil-Case-m}
Assume that $r=2m$ and $k=m$ and that $v_p(\Theta)\leq 0$. Then $F'_{m,i}\in\fM'$ for all $0\leq i\leq r-1$ if and only if
\begin{equation}\label{eq: condition for k=m r=2m}
-1\leq v_p(\Theta) \quad\mbox{ and }\quad v_p(\Theta)-2v_p(x)\leq 1,
\end{equation}
in which case we have
$$\Fil^r_{m}\fM'=S'_\cO(F'_{m,m},F'_{m,0})+\Fil^pS'_\cO\fM'$$ and $\Fil^r_{m}\cM'=\barS'_\F(\barF'_{m,m},\barF'_{m,0})$ with
    \begin{align*}
    \barF'_{m,m}&=(u^m-pxa_mu^{m-1})\barE'_1-\frac{px^2}{\Theta}a_mu^{m-1}\barE'_2; \\
    \barF'_{m,0}&=p\Theta a_mu^{m-1}E'_1+(u^m+pxa_mu^{m-1})\barE'_2.
    \end{align*}
\end{prop}

We prove this proposition by a series of lemmas, dividing the values of $i$ into two cases: $0\leq i\leq m-1$ and $m\leq i\leq r-1$.

\begin{lemm}\label{lemma k=m r=2m, 0 leq i leq m-1}
Keep the assumption of Proposition~\ref{prop: Fil-Case-m}. Then $F'_{m,i}\in\fM'$ for $0\leq i\leq m-1$ if and only if
\begin{align*}
    v_p(\Theta)\geq -1  \quad\mbox{ and }\quad v_p(x)\geq -1,
\end{align*}
in which case we have
\begin{align*}
F'_{m,0} &\equiv  p\Theta a_mv^{m-1}E'_1 +(v^m+pxa_mv^{m-1})E'_2  \pmod{\fm\fM'};
    \\
F'_{m,i} &\equiv  v^{i+m}E'_2 \pmod{\fm\fM'}\,\,\mbox{ if }1\leq i\leq m-1.
\end{align*}
\end{lemm}

\begin{proof}
Assume that $0\leq i\leq m-1$. By Lemma~\ref{lemma-PQ-k leq m} we have
\begin{align*}
[F'_{m,i}]_{E'_1} &=p^{i+m}\Theta P_{m,i}\Big(\frac{v}{p},0\Big);\\
[F'_{m,i}]_{E'_2} &=p^{i+m}Q_{m,i}\Big(\frac{v}{p},0\Big) +p^{i+m}xP_{m,i}\Big(\frac{v}{p},0\Big).
\end{align*}

Note that $\deg_T P_{m,i}(T,0)\leq m-1$ for all $0\leq i\leq m-1$ by Lemma~\ref{lemma: PQdegree}, and that $\deg_TP_{m,0}(T,0)=m-1$ with the leading coefficient $a_m\in\Z_{(p)}^\times$ by Lemma~\ref{lemma: B_k,0 k leq m}~(iii). Hence, we have
\begin{align*}
    p^{i+m}\Theta P_{m,i}\Big(\frac{v}{p},0\Big) &\equiv
    \begin{cases}
        p\Theta a_m v^{m-1}\pmod{p\Theta\fm S'_\cO} & \mbox{if } i=0; \\
        0\pmod{p\Theta\fm S'_\cO} & \mbox{otherwise}
    \end{cases}
\end{align*}
and
\begin{align*}
    p^{i+m}xP_{m,i}\Big(\frac{v}{p},0\Big) &\equiv
    \begin{cases}
        pxa_mv^{m-1}\pmod{px\fm S'_\cO} & \mbox{if } i=0; \\
        0 \pmod{px\fm S'_\cO} & \mbox{otherwise}.
    \end{cases}
\end{align*}
Moreover, since $Q_{k,i}(T,0)$ is a monic polynomial of degree $i+m$ by Lemma~\ref{lemma: PQdegree}, we have
\begin{align*}
    p^{i+m}Q_{m,i}\Big(\frac{v}{p},0\Big) \equiv v^{i+m} \pmod{\fm S_\cO'}.
\end{align*}

Now, it is immediate that $F'_{m,i}\in\fM'$ for $0\leq i\leq m-1$ if and only if $v_p(\Theta)\geq -1$ and $v_p(x)\geq -1$. It is also immediate that if $v_p(\Theta)\geq -1$ and $v_p(x)\geq -1$ then $\barF'_{m,i}$ for $0\leq i\leq m-1$ is computed as in the statement of Lemma~\ref{lemma k=m r=2m, 0 leq i leq m-1}.
\end{proof}

\begin{lemm}\label{lemma k=m r=2m, m leq i leq r-1}
Keep the assumption of Proposition~\ref{prop: Fil-Case-m}, and assume further that $v_p(\Theta)\geq -1$ and $v_p(x)\geq -1$. Then $F'_{m,i}\in\fM'$ for $m\leq i\leq r-1$ if and only if
\begin{align*}
   \quad v_p(\Theta)-2v_p(x)\leq 1,
\end{align*}
in which case we have
\begin{align*}
F'_{m,m} &\equiv (v^m-pxa_mv^{m-1})E'_1-\frac{px^2}{\Theta}a_mv^{m-1}E'_2 \pmod{\fm\fM'} ;\\
F'_{m,i} &\equiv v^iE'_1\pmod{\fm\fM'} \,\,\mbox{ if }m+1\leq i\leq r-1.
\end{align*}
\end{lemm}

\begin{proof}
Assume that $m\leq i\leq r-1$. By Lemma~\ref{lemma-PQ-k leq m} we have
\begin{align*}
[F'_{m,i}]_{E'_1} &=p^i P_{m,i}\Big(\frac{v}{p},0\Big)-p^ix P_{m,i-m}\Big(\frac{v}{p},0\Big); \\
[F'_{m,i}]_{E'_2} &=\frac{p^i}{\Theta}Q_{m,i}\Big(\frac{v}{p},0\Big) +\frac{p^ix}{\Theta}\Big(P_{m,i}\Big(\frac{v}{p},0\Big) -Q_{m,i-m}\Big(\frac{v}{p},0\Big)\Big) -\frac{p^ix^2}{\Theta}P_{m,i-m}\Big(\frac{v}{p},0\Big).
\end{align*}

Since $P_{m,i}(T,0)$ is a monic polynomial of degree $i$ by Lemma~\ref{lemma: PQdegree}, we have
\begin{align*}
    p^iP_{m,i}\Big(\frac{v}{p},0\Big)
    \equiv v^i \pmod{\fm S'_\cO}.
\end{align*}
Similarly, as $\deg_T Q_{m,i}(T,0)\leq m-1$ by Lemma~\ref{lemma: PQdegree} and $v_p(\Theta)\leq 0$, we have
\begin{align*}
    \frac{p^i}{\Theta}Q_{m,i}\Big(\frac{v}{p},0\Big) \equiv 0 \pmod{\fm S'_\cO}.
\end{align*}

Moreover, we note that by Lemma~\ref{lemma: PQdegree} $\deg_T P_{m,i-m}(T,0)\leq m-1$, and that by Lemma~\ref{lemma: B_k,0 k leq m} $\deg_T P_{m,0}(T,0)=m-1$ with the leading coefficient $a_m\in\Z_{(p)}^{\times}$. Hence, we have
\begin{align*}
    p^ix P_{m,i-m}\Big(\frac{v}{p},0\Big)
    &\equiv
    \begin{cases}
        pxa_mv^{m-1}\pmod{\fm S'_\cO} & \mbox{if } i=m; \\
        0\pmod{\fm S'_\cO} & \mbox{otherwise}
    \end{cases}
\end{align*}
as $v_p(x)\geq -1$, and
\begin{align*}
    \frac{p^ix^2}{\Theta}P_{m,i-m}\Big(\frac{v}{p},0\Big)
    &\equiv
    \begin{cases}
        \frac{px^2}{\Theta}a_mv^{m-1}\pmod{\frac{px^2}{\Theta}\fm S'_\cO} & \mbox{if } i=m; \\
        0\pmod{\frac{px^2}{\Theta}\fm S'_\cO} & \mbox{otherwise.}
    \end{cases}
\end{align*}

Finally, it remains to deal with the term $$\frac{p^ix}{\Theta}\Big(P_{m,i}\Big(\frac{v}{p},0\Big) -Q_{m,i-m}\Big(\frac{v}{p},0\Big)\Big).$$ By Lemma~\ref{lemma: PQdegree} we have $\deg_T(P_{m,i}(T,0)-T^i)\leq m-1$ and $\deg_T(Q_{m,i-m}(T,0)-T^i)\leq m-1$, and so $\deg_T\left(P_{m,i}(T,0)-Q_{m,i-m}(T,0)\right)\leq m-1$ for all $m\leq i \leq r-1$. If we write $b_m$ for the coefficient of $T^{m-1}$ in $P_{m,m}(T,0)-Q_{m,0}(T,0)$, then
\begin{align*}
    \frac{p^ix}{\Theta}
    \Big(P_{m,i}\Big(\frac{v}{p},0\Big)-Q_{m,i-m}\Big(\frac{v}{p},0\Big)\Big)
    \equiv
    \begin{cases}
        \frac{px}{\Theta}b_mv^{m-1}\pmod{\fm S'_\cO} & \mbox{if } i=m; \\
        0 \pmod{\fm S'_\cO} & \mbox{otherwise,}
    \end{cases}
\end{align*}
as the inequalities $v_p(x)\geq -1$ and $v_p(\Theta)\leq 0$ imply that $v_p(px/\Theta)\geq 0$.

In fact, it is easy to see that $v_p(\Theta)-2v_p(x)\leq 1$ implies that $v_p(px/\Theta)>0$. Indeed, if $v_p(x)>-1$ then it is obvious that $v_p(px/\Theta)>0$ as $v_p(\Theta)\leq 0$. If $v_p(x)=-1$, then we have $v_p(\Theta)\leq -1$ from $v_p(\Theta)-2v_p(x)\leq 1$, and so we have $v_p(\Theta)=-1$, as $v_p(\Theta)\geq -1$.
Hence, we conclude that if $v_p(x)=-1$ then $v_p(px/\Theta)=1$.

Now, it is immediate that $F'_{m,i}\in\fM'$ for $m\leq i\leq r-1$ if and only if $v_p(\Theta)-2v_p(x)\leq 1$. It is also immediate that if the inequalities in the statement of Lemma~\ref{lemma k=m r=2m, m leq i leq r-1} hold then $\barF'_{m,i}$ for $m\leq i\leq r-1$ is computed as in the statement of Lemma~\ref{lemma k=m r=2m, m leq i leq r-1}.
\end{proof}

From Lemma~\ref{lemma k=m r=2m, 0 leq i leq m-1} and Lemma~\ref{lemma k=m r=2m, m leq i leq r-1}, it is easy to see that
$F'_{k,i}\in\fM'$ for $0\leq i\leq r-1$ if and only if the inequalities in \eqref{eq: condition for k=m r=2m} hold. Note that $v_p(\Theta)-2v_p(x)\leq 1$ together with $v_p(\Theta)\geq -1$ implies that $v_p(x)\geq -1$.

Finally, we need to check that $\Fil^r_{m}\cM'$ is generated by $\barF'_{m,m},\,\barF'_{m,0}$ over~ $\barS'_\F$.

\begin{proof}[Proof of Proposition~\ref{prop: Fil-Case-m}]
Assume that the inequalities in \eqref{eq: condition for k=m r=2m} hold, and note that $\barF'_{m,i}$ for all $0\leq i\leq r-1$ are computed in the preceding lemmas.

We note from Lemma~\ref{lemma k=m r=2m, 0 leq i leq m-1} and Lemma~\ref{lemma k=m r=2m, m leq i leq r-1} that
\begin{align*}
F'_{m,0}
    &\equiv
    v^mE'_2+p\Theta a_mv^{m-1}\bigg(E'_1+\frac{x}{\Theta}E'_2\bigg)
    \pmod{\fm\fM'};\\
F'_{m,m}
    &\equiv
    v^mE'_1-pxa_mv^{m-1}\bigg(E'_1+\frac{x}{\Theta}E'_2\bigg)
    \pmod{\fm\fM'},
\end{align*}
and so we have
\begin{align*}
    p\Theta\bigg(F'_{m,m}+\frac{x}{\Theta}F'_{m,0}\bigg)
    \equiv
    p\Theta v^m\bigg(E'_1+\frac{x}{\Theta}E'_2\bigg)
    \pmod{\fm\fM'}.
\end{align*}
Hence, if $1\leq i\leq m-1$ then we have
\begin{align*}
    F'_{m,i}\equiv v^iF'_{m,0}-p\Theta a_mv^{i-1}\bigg(F'_{m,m}+\dfrac{x}{\Theta}F'_{m,0}\bigg)  \pmod{\fm\fM'},
\end{align*}
and if $m+1\leq i\leq r-1$ then we have
\begin{align*}
    F'_{m,i} \equiv v^{i-m}F'_{m,m}+pxa_mv^{i-m-1}\bigg(F'_{m,m}+\dfrac{x}{\Theta}F'_{m,0}\bigg)  \pmod{\fm\fM'}.
\end{align*}

Now, it is also immediate that $u^r\barE'_1,\,u^r\barE'_2\in \barS'_\F(\barF'_{m,m},\barF'_{m,0})$, as $$v^rE'_1\equiv v^{m-1}F'_{m,m+1}\pmod{\fm\fM'}\quad\mbox{ and }\quad v^rE'_2\equiv v^{m-1}F'_{m,1}\pmod{\fm\fM'}.$$ Hence, we deduce that $\Fil^r_{m}\cM'=\barS'_\F(\barF'_{m,m},\barF'_{m,0})$. Now, by the same argument as in \eqref{eq: Nakayama}, we have
$\Fil^r_{m}\fM' =S'_\cO(F'_{m,m},F'_{m,0})+\Fil^pS'_\cO\fM'$, which completes the proof.
\end{proof}

\subsubsection{\textbf{In $\mathbf{Case}~(k)$ for $m+1\leq k \leq r-1$}}
In this section of paragraph, we describe the $S'_\cO$-generators of $\Fil^r_{k}\fM'$ and the $\barS'_\F$-generators of $\Fil^r_{k}\cM'$ in the case $\mathbf{Case}~(k)$ for $m+1\leq k<r$. In this case, it is enough to consider $x\in E$ with $v_p(x)<0$.

\begin{prop}\label{prop: Fil-Case-m+1-r-1}
Assume that $m+1\leq k\leq r-1$ and that $v_p(\Theta)\leq 0$ and $v_p(x)<0$. Then $F'_{k,i}\in\fM'$ for all $0\leq i\leq r-1$ if and only if
\begin{equation}\label{eq: condition for m+1 leq k < r}
v_p(x)\leq v_p(\Theta)\quad\mbox{ and }\quad 2k-r-1\leq v_p(\Theta)-2v_p(x)\leq 2k-r+1,
\end{equation}
in which case
$$\Fil^r_{k}\fM'=S'_\cO(F'_{k,k},F'_{k,0})+\Fil^pS'_\cO\fM'$$ and $\Fil^r_{k}\cM'=\barS'_\F(\barF'_{k,k},\barF'_{k,0})$ with
    \begin{align*}
    \barF'_{k,k} &=(u^k-p^{2k-r+1}xa_{r-k}u^{r-k-1})\barE'_1-\frac{p^{2k-r+1}x^2}{\Theta}a_{r-k}u^{r-k-1}E'_2;\\
    \barF'_{k,0} &=\bigg(\frac{\Theta}{x}u^{r-k}-\frac{\Theta}{p^{2k-r-1}x^2}a'_{r,k}u^{k-1}\bigg)\barE'_1 +u^{r-k}\barE'_2,
    \end{align*}
    where $a'_{r,k}:=0$ if $(r,k)=(2m+1,m+1)$ and $a'_{r,k}:=1/a_{r-k+1}$ otherwise.
\end{prop}

We prove this proposition by a series of lemmas, dividing the values of $i$ into three cases: $0\leq i\leq 2k-r-1$, $2k-r\leq i\leq k-1$, and $k\leq i\leq r-1$. We will frequently use the following: if $v_p(x)<0$ then
\begin{equation}\label{eq: congruence of d_k(x)}
    \frac{x^s}{d_k(x)}\equiv \frac{1}{x^{2k-r-s}} \pmod{\frac{1}{x^{2k-r-s}}\fm}
\end{equation}
by Lemma~\ref{lemma-QAinv}~(iii).

\begin{lemm}\label{lemma m+1 leq k leq r-1, 0 leq i leq 2k-r-1}
Keep the assumption of Proposition~\ref{prop: Fil-Case-m+1-r-1}. Then $F'_{k,i}\in\fM'$ for $0\leq i\leq 2k-r-1$ if and only if
\begin{align*}
    v_p(x)\leq v_p(\Theta)\quad\mbox{ and }\quad v_p(\Theta)-2v_p(x)\geq 2k-r-1,
\end{align*}
in which case we have
\begin{align*}
    F'_{k,0}&\equiv  \bigg(\dfrac{\Theta}{x}v^{r-k}-\dfrac{\Theta}{p^{2k-r-1}x^2}a_{r,k}'v^{k-1}\bigg)E'_1+v^{r-k}E'_2  \pmod{\fm\fM'};\\
    F'_{k,i}&\equiv  \frac{\Theta}{x}v^{i+r-k}E'_1+v^{i+r-k}E'_2 \pmod{\fm\fM'}\,\,\mbox{ if }1\leq i \leq 2k-r-1,
\end{align*}
where $a'_{r,k}:=0$ if $(r,k)=(2m+1,m+1)$ and $a'_{r,k}:=1/a_{r-k+1}$ otherwise.
\end{lemm}

\begin{proof}
Assume that $0\leq i\leq 2k-r-1$. By Lemma~\ref{lemma: m+1 to r} we have
\begin{align*}
[F'_{k,i}]_{E'_1}&=
    p^{i+r-k}\Theta\frac{x^{2k-r-1}}{d_k(x)}P_{r-k,i+r-k}\Big(\frac{v}{p},0\Big)
    +p^{i+r-k}\Theta\sum_{s=0}^{2k-r-2}\frac{x^s}{d_k(x)} P_{k,i,s}\Big(\frac{v}{p}\Big);\\
[F'_{k,i}]_{E'_2}&=p^{i+r-k}\sum_{s=0}^{2k-r}
    \frac{x^s}{d_k(x)} Q_{k,i,s}\Big(\frac{v}{p}\Big)
\end{align*}

As $P_{r-k,i+r-k}(T,0)$ is a monic polynomial of degree $i+r-k$ and $\deg_xd_k(x)=2k-r$, we have
\begin{align*}
    p^{i+r-k}\Theta\frac{x^{2k-r-1}}{d_k(x)}P_{r-k,i+r-k}\Big(\frac{v}{p},0\Big)
    &\equiv
    \frac{\Theta}{x}p^{i+r-k}P_{r-k,i+r-k}\Big(\frac{v}{p},0\Big)\pmod{\frac{\Theta}{x}\fm S'_\cO}\\
    &\equiv \frac{\Theta}{x}v^{i+r-k} \pmod{\frac{\Theta}{x}\fm S'_\cO}
\end{align*}
where the first congruence is due to \eqref{eq: congruence of d_k(x)}. (Note that if $k=m+1$ and $r=2m+1$ then we have $2k-r-1=0$.)

We need to compute the remaining terms of $[F'_{k,i}]_{E'_1}$ if $(r,k)\neq(2m+1,m+1)$. We note that if $(r,k)\neq(2m+1,m+1)$ then
\begin{itemize}[leftmargin=*]
  \item $\deg P_{k,i,s}(T)\leq k-1$ for $s\leq 2k-r-2$ by Lemma~\ref{lemma: PQdegree};
  \item $\deg_TP_{k,0,2k-r-2}(T)=k-1$ with the leading coefficient $-1/a_{r-k+1}\in\Z_{(p)}^\times$ by (iii) and (iv) of  Lemma~\ref{lemma: a_k m+1 leq k}.
\end{itemize}
Hence, if $(r,k)\neq (2m+1,m+1)$ and $s\leq 2k-r-2$ then we have
\begin{align*}
    p^{i+r-k}\Theta\frac{x^s}{d_k(x)} P_{k,i,s}\Big(\frac{v}{p}\Big)
    &\equiv
    \frac{\Theta}{p^{2k-r-1-i}x^{2k-r-s}} p^{k-1} P_{k,i,s}\Big(\frac{v}{p}\Big)\\
    &\equiv
    \begin{cases}
        -\frac{\Theta}{p^{2k-r-1}x^2}\frac{1}{a_{r-k+1}}v^{k-1}
         & \mbox{if } (i,s)=(0,2k-r-2);\\
        0 & \mbox{otherwise}
    \end{cases}
\end{align*}
modulo $\frac{\Theta}{p^{2k-r-1}x^2}\fm S'_\cO$, where the first congruence is due to \eqref{eq: congruence of d_k(x)}.

As $v_p(x^s/d_k(x))\geq 0$ for $s\leq 2k-r$
\begin{align*}
Q_{k,i}(T,x)=\sum_{s=0}^{2k-r}\frac{x^s}{d_k(x)}Q_{k,i,s}(T)
\in\cO[T].
\end{align*}
Moreover, $Q_{k,i}(T,x)$ is a monic polynomial of degree $i+r-k$ in $T$ by Lemma~\ref{lemma: PQdegree}, and so we have
\begin{align*}
    [F'_{k,i}]_{E'_2}
    =p^{i+r-k}Q_{k,i}\Big(\frac{v}{p},x\Big)
    \equiv v^{i+r-k} \pmod{\fm S'_\cO}.
\end{align*}

Now, it is immediate that $F'_{k,i}\in\fM'$ for $0\leq i\leq 2k-r-1$ if and only if $v_p(x)\leq v_p(\Theta)$ and $v_p(\Theta)-2v_p(x)\geq 2k-r-1$. It is also immediate that if $v_p(x)\leq v_p(\Theta)$ and $v_p(\Theta)-2v_p(x)\geq 2k-r-1$ then $\barF'_{k,i}$ for $0\leq i\leq 2k-r-1$ is computed as in the statement of Lemma~\ref{lemma m+1 leq k leq r-1, 0 leq i leq 2k-r-1}.
\end{proof}

\begin{lemm}\label{lemma m+1 leq k leq r-1, 2k-r leq i leq k-1}
Keep the assumption of Proposition~\ref{prop: Fil-Case-m+1-r-1}, and assume further that $v_p(x)\leq v_p(\Theta)$. Then $F'_{k,i}\in\fM'$ for $2k-r\leq i\leq k-1$ if and only if
\begin{align*}
    v_p(x)\geq -2k+r-1,
\end{align*}
in which case we have
\begin{align*}
F'_{k,2k-r} &\equiv  p^{2k-r+1}\Theta a_{r-k}v^{r-k-1}E'_1 +\big(p^{2k-r+1}xa_{r-k}v^{r-k-1}+v^k\big)E'_2 \pmod{\fm\fM'};\\
F'_{k,i} &\equiv v^{i+r-k}E'_2 \pmod{\fm\fM'}\,\,\mbox{ if }2k-r+1\leq i\leq k-1.
\end{align*}
\end{lemm}

\begin{proof}
Assume that $2k-r\leq i\leq k-1$. By Lemma~\ref{lemma: m+1 to r} we have
\begin{align*}
[F'_{k,i}]_{E'_1} &=
    p^{i+r-k}\Theta\frac{x^{2k-r}}{d_k(x)}P_{r-k,i+r-2k}\Big(\frac{v}{p},0\Big)
    +p^{i+r-k}\Theta\sum_{s=0}^{2k-r-1}\frac{x^s}{d_k(x)}P_{k,i,s}\Big(\frac{v}{p}\Big);\\
[F'_{k,i}]_{E'_2} &=p^{i+r-k}
    \frac{x^{2k-r+1}}{d_k(x)}P_{r-k,i+r-2k}\Big(\frac{v}{p},0\Big)
    +p^{i+r-k}\sum_{s=0}^{2k-r}\frac{x^s}{d_k(x)}Q_{k,i,s}\Big(\frac{v}{p}\Big).
\end{align*}

As $\deg_T P_{r-k,i+r-2k}(T,0)\leq r-k-1$ by Lemma~\ref{lemma: PQdegree} and $\deg_TP_{r-k,0}(T,0)=r-k-1$ with the leading coefficient $a_{r-k}\in\Z_{(p)}^\times$ by Lemma~\ref{lemma: B_k,0 k leq m}, we have
\begin{align*}
    &p^{i+r-k}\Theta\frac{x^{2k-r}}{d_k(x)} P_{r-k,i+r-2k}\Big(\frac{v}{p},0\Big)\\
    &\equiv
    p^{i+r-k}\Theta P_{r-k,i+r-2k}\Big(\frac{v}{p},0\Big)\pmod{p^{2k-r+1}\Theta\fm S'_\cO}\\
    &\equiv
    \begin{cases}
        p^{2k-r+1}\Theta a_{r-k}v^{r-k-1} \pmod{p^{2k-r+1}\Theta\fm S'_\cO}& \mbox{if } i=2k-r; \\
        0 \pmod{p^{2k-r+1}\Theta\fm S'_\cO}& \mbox{otherwise}
    \end{cases}
\end{align*}
and
\begin{align*}
    &p^{i+r-k}\frac{x^{2k-r+1}}{d_k(x)} P_{r-k,i+r-2k}\Big(\frac{v}{p},0\Big)\\
    &\equiv
    p^{i+r-k}xP_{r-k,i+r-2k}\Big(\frac{v}{p},0\Big)\pmod{p^{2k-r+1}x\fm S'_\cO}\\
    &\equiv
    \begin{cases}
        p^{2k-r+1}xa_{r-k}v^{r-k-1}\pmod{p^{2k-r+1}x\fm S'_\cO} & \mbox{if } i=2k-r; \\
        0\pmod{p^{2k-r+1}x\fm S'_\cO} & \mbox{otherwise,}
    \end{cases}
\end{align*}
where the first congruence in each case is due to \eqref{eq: congruence of d_k(x)}. Moreover, as $\deg_T P_{k,i,s}(T)\leq k-1$ for $0\leq s\leq 2k-r-1$ by Lemma~\ref{lemma: PQdegree} we have
\begin{align*}
    p^{i+r-k}\Theta\frac{x^s}{d_k(x)}P_{k,i,s}\Big(\frac{v}{p}\Big)
    \equiv p^{i+r-k}\frac{\Theta}{x^{2k-r-s}}P_{k,i,s}\Big(\frac{v}{p}\Big)\equiv 0 \pmod{\frac{\Theta}{x}\fm S'_\cO}.
\end{align*}

Finally, by Lemma~\ref{lemma: m+1 to r} we have
\begin{align*}
    Q_{k,i}(T,x)-\frac{x^{2k-r+1}}{d_k(x)}P_{r-k,i+r-2k}(T,0)=
    \sum_{s=0}^{2k-r}\frac{x^s}{d_k(x)}Q_{k,i,s}(T)
    \in\cO[T],
\end{align*}
as $v_p(x^s/d_k(x))\geq 0$ for $s\leq 2k-r$, which is a monic polynomial of degree $i+r-k$ by Lemma~\ref{lemma: PQdegree}, and so
\begin{align*}
    p^{i+r-k}\sum_{s=0}^{2k-r}\frac{x^s}{d_k(x)}Q_{k,i,s}\Big(\frac{v}{p}\Big)
    \equiv v^{i+r-k} \pmod{\fm S'_\cO}.
\end{align*}

Hence, under the assumption $v_p(x)\leq v_p(\Theta)$, we have $F'_{k,i}\in\fM'$ for $2k-r\leq i\leq k-1$ if and only if $v_p(x)\geq -2k+r-1$. It is also immediate that if $v_p(x)\leq v_p(\Theta)$ and $v_p(x)\geq -2k+r-1$ then $\barF'_{k,i}$ for $2k-r\leq i\leq k-1$ is computed as in the statement of Lemma~\ref{lemma m+1 leq k leq r-1, 2k-r leq i leq k-1}.
\end{proof}

\begin{lemm}\label{lemma m+1 leq k leq r-1, k leq i leq r-1}
Keep the assumption of Proposition~\ref{prop: Fil-Case-m+1-r-1}. Then $F'_{k,i}\in\fM'$ for $k\leq i\leq r-1$ if and only if
\begin{align*}
    v_p(x)\geq -2k+r-1 \quad\mbox{ and }\quad v_p(\Theta)-2v_p(x)\leq 2k-r+1,
\end{align*}
in which case
\begin{align*}
F'_{k,k} &\equiv \big(v^k-p^{2k-r+1}xa_{r-k}v^{r-k-1}\big)E'_1 -\frac{p^{2k-r+1}x^2}{\Theta}a_{r-k}v^{r-k-1}E'_2 \pmod{\fm\fM'};\\
F'_{k,i} &\equiv v^iE'_1\pmod{\fm\fM'}\,\,\mbox{ if }k+1\leq i\leq r-1.
\end{align*}
\end{lemm}

\begin{proof}
Assume that $k\leq i\leq r-1$. By Lemma~\ref{lemma: m+1 to r} we have
\begin{align*}
[F'_{k,i}]_{E'_1}&=
    -p^i\frac{x^{2k-r+1}}{d_k(x)}P_{r-k,i-k}\Big(\frac{v}{p},0\Big)
    +p^i\sum_{s=0}^{2k-r}\frac{x^s}{d_k(x)}P_{k,i,s}\Big(\frac{v}{p}\Big);\\
[F'_{k,i}]_{E'_2}&=
    -\frac{p^i}{\Theta}\frac{x^{2k-r+2}}{d_k(x)}P_{r-k,i-k}\Big(\frac{v}{p},0\Big)
    +\frac{p^i}{\Theta}\sum_{s=0}^{2k-r+1}\frac{x^s}{d_k(x)}Q_{k,i,s}\Big(\frac{v}{p}\Big).
\end{align*}

By Lemma~\ref{lemma: PQdegree} $\deg_T P_{r-k,j-k}(T,0)\leq r-k-1$ for $k\leq j\leq r-1$ and by Lemma~\ref{lemma: B_k,0 k leq m} $\deg_TP_{r-k,0}(T,0)=r-k-1$ with the leading coefficient $a_{r-k}\in\Z_{(p)}^\times$, and so we have
\begin{align*}
    p^i \frac{x^{2k-r+1}}{d_k(x)} P_{r-k,i-k}\Big(\frac{v}{p},0\Big)
    &\equiv
    p^{i-r+k+1}x p^{r-k-1}P_{r-k,i-k}\Big(\frac{v}{p},0\Big)\pmod{p^{2k-r+1}x\fm S'_\cO}\\
    &\equiv
    \begin{cases}
        p^{2k-r+1}xa_{r-k}v^{r-k-1}\pmod{p^{2k-r+1}x\fm S'_\cO} & \mbox{if } i=k; \\
        0\pmod{p^{2k-r+1}x\fm S'_\cO} & \mbox{otherwise}
    \end{cases}
\end{align*}
and
\begin{align*}
    \frac{p^i}{\Theta}\frac{x^{2k-r+2}}{d_k(x)} P_{r-k,i-k}\Big(\frac{v}{p},0\Big)
    &\equiv
    \frac{p^{i-r+k+1}x^2}{\Theta}p^{r-k-1}P_{r-k,i-k}\Big(\frac{v}{p},0\Big)\pmod{\frac{p^{2k-r+1}x^2}{\Theta}\fm S'_\cO}\\
    &\equiv
    \begin{cases}
        \frac{p^{2k-r+1}x^2}{\Theta}a_{r-k}v^{r-k-1}\pmod{\frac{p^{2k-r+1}x^2}{\Theta}\fm S'_\cO} & \mbox{if } i=k; \\
        0\pmod{\frac{p^{2k-r+1}x^2}{\Theta}\fm S'_\cO} & \mbox{otherwise,}
    \end{cases}
\end{align*}
where the first congruence in each case is due to \eqref{eq: congruence of d_k(x)}. Moreover, by Lemma~\ref{lemma: m+1 to r} we have
\begin{align*}
    P_{k,i}(T,x)+\frac{x^{2k-r+1}}{d_k(x)}P_{r-k,i-k}(T,0)
    =\sum_{s=0}^{2k-r}\frac{x^s}{d_k(x)}P_{k,i,s}(T) \in\cO[T],
\end{align*}
as $v_p(x^s/d_k(x))\geq 0$ for $s\leq 2k-r$, which is a monic polynomial of degree $i$ by Lemma~\ref{lemma: PQdegree}, and so
\begin{align*}
    p^i\sum_{s=0}^{2k-r}\frac{x^s}{d_k(x)}P_{k,i,s}\Big(\frac{v}{p}\Big)
    \equiv v^i \pmod{\fm S'_\cO}.
\end{align*}

Finally, we have $\deg Q_{k,i,s}(T)\leq r-k-1$ by Lemma~\ref{lemma: PQdegree}, and so if $s\leq 2k-r+1$ then we have
\begin{align*}
    \frac{p^i}{\Theta}\frac{x^s}{d_k(x)}Q_{k,i,s}\Big(\frac{v}{p}\Big)\equiv \frac{p^{i}}{\Theta}\frac{1}{x^{2k-r-2}}Q_{k,i,s}\Big(\frac{v}{p}\Big) \equiv
    0 \pmod{\frac{p^{2k-r+1}x^2}{\Theta}\fm S'_\cO}.
\end{align*}

Hence, we have $F'_{k,i}\in\fM'$ for $k\leq i\leq r-1$ if and only if $v_p(x)\geq -2k+r-1$ and $v_p(X)-2v_p(x)\leq 2k-r+1$. It is also immediate that if $v_p(x)\geq -2k+r-1$ and $v_p(X)-2v_p(x)\leq 2k-r+1$ then $\barF'_{k,i}$ for $k\leq i\leq r-1$ is computed as in the statement of Lemma~\ref{lemma m+1 leq k leq r-1, k leq i leq r-1}.
\end{proof}

From Lemma~\ref{lemma m+1 leq k leq r-1, 0 leq i leq 2k-r-1}, Lemma~\ref{lemma m+1 leq k leq r-1, 2k-r leq i leq k-1}, and Lemma~\ref{lemma m+1 leq k leq r-1, k leq i leq r-1}, it is easy to see that $F'_{k,i}\in\fM'$ for $0\leq i\leq r-1$ if and only if the inequalities in \eqref{eq: condition for m+1 leq k < r} hold. Note that $v_p(\Theta)-2v_p(x)\leq 2k-r+1$ together with $v_p(x)\leq v_p(\Theta)$ implies $ v_p(x)\geq -2k+r-1$.

Finally, we need to check that $\Fil^r_{k}\cM'$ is generated by $\barF'_{k,k},\,\barF'_{k,0}$ over~ $\barS'_\F$.

\begin{proof}[Proof of Proposition~\ref{prop: Fil-Case-m+1-r-1}]
Assume that the inequalities in \eqref{eq: condition for m+1 leq k < r} hold, and note that $\barF'_{k,i}$ for all $0\leq i\leq r-1$ are computed in the preceding lemmas.

We first note that
\begin{equation}\label{eq: congruence for m+1 leq k < r, 1}
    \frac{\Theta}{p^{2k-r-1}x^2}F'_{k,k}
    \equiv
    \frac{\Theta}{p^{2k-r-1}x^2}v^kE'_1
    \pmod{\fm\fM'}
\end{equation}
and
\begin{equation}\label{eq: congruence for m+1 leq k < r, 2}
    vF'_{k,k}+\frac{p^{2k-r+1}x^2}{\Theta}a_{r-k}F'_{k,0}
    \equiv v^{k+1}E'_1  \pmod{\fm\fM'}.
\end{equation}

Hence, if $1\leq i\leq 2k-r-1$ then by \eqref{eq: congruence for m+1 leq k < r, 1} we have
\begin{align*}
F'_{k,i}
    &\equiv
    v^iF'_{k,0}+\frac{\Theta}{p^{2k-r-1}x^2}a'_{r,k}v^{i+k-1}E'_1\pmod{\fm\fM'}\\
    &\equiv
    v^iF'_{k,0}+\frac{\Theta}{p^{2k-r-1}x^2}a'_{r,k}v^{i-1}F'_{k,k}
    \pmod{\fm\fM'},
\end{align*}
and if $i=2k-r$ then by \eqref{eq: congruence for m+1 leq k < r, 1} we have
\begin{align*}
F'_{k,2k-r}
    &\equiv
    -\frac{\Theta}{x}F'_{k,k}
    +\frac{\Theta}{x}v^kE'_1+v^kE'_2\pmod{\fm\fM'}\\
    &\equiv
    -\frac{\Theta}{x}F'_{k,k}
    +v^{2k-r}\bigg(F'_{k,0}+\frac{\Theta}{p^{2k-r-1}x^2}a'_{r,k}v^{k-1}E'_1\bigg)\pmod{\fm\fM'}\\
    &\equiv
    -\bigg(\frac{\Theta}{x}-\frac{\Theta}{p^{2k-r-1}x^2}a'_{r,k}v^{2k-r-1}\bigg)F'_{k,k}
    +v^{2k-r}F'_{k,0}
    \pmod{\fm\fM'}.
\end{align*}
Moreover, if $2k-r+1\leq i\leq k-1$ then by \eqref{eq: congruence for m+1 leq k < r, 1} and \eqref{eq: congruence for m+1 leq k < r, 2}
\begin{align*}
F'_{k,i}&\equiv v^iF'_{k,0}-\frac{\Theta}{x}v^{i+r-k}E_1'+\frac{\Theta}{p^{2k-r-1}x^2}a'_{r,k}v^{i+k-1}E_1'\pmod{\fm\fM'}\\
&\equiv v^iF'_{k,0}-\frac{\Theta}{x}v^{i+r-2k}F'_{k,k} -p^{2k-r+1}xa_{r-k}v^{i+r-2k-1}F'_{k,0} \\
&\hspace{11em} +\frac{\Theta}{p^{2k-r-1}x^2}a'_{r,k}v^{i-1}F_{k,k}'\pmod{\fm\fM'},
\end{align*}
and if $k+1\leq i\leq r-1$ then by \eqref{eq: congruence for m+1 leq k < r, 2}
\begin{align*}
F'_{k,j}\equiv v^{j-k}F'_{k,k}+\frac{p^{2k-r+1}x^2}{\Theta}a_{r-k}v^{j-k-1}F'_{k,0}\pmod{\fm\fM'}.
\end{align*}

Now, it is also immediate that $u^r\barE'_1,\,u^r\barE'_2\in \barS'_\F(\barF'_{k,k},\barF'_{k,0})$, as $$v^rE'_1\equiv vF'_{k,r-1}\pmod{\fm\fM'}\quad\mbox{ and }\quad v^rE'_2\equiv vF'_{k,k-1}\pmod{\fm\fM'}.$$ Hence, we deduce that $\Fil^r_{k}\cM'=\barS'_\F(\barF'_{k,k},\barF'_{k,0})$. Now, by the same argument as in \eqref{eq: Nakayama}, we have
$\Fil^r_{k}\fM' =S'_\cO(F'_{k,k},F'_{k,0})+\Fil^pS'_\cO\fM'$, which completes the proof.
\end{proof}

\subsubsection{\textbf{In $\mathbf{Case}~(r)$}}
In this section of paragraph, we describe the $S'_\cO$-generators of $\Fil^r_{r}\fM'$ and the $\barS'_\F$-generators of $\Fil^r_{r}\cM'$ when $k=r$, i.e., in the case $\mathbf{Case}~(r)$.

\begin{prop}\label{prop: Fil-Case-r}
Assume that $k=r$ and that $v_p(\Theta)\leq 0$. Then $F'_{r,i}\in \fM'$ for all $0\leq i\leq r-1$ if and only if
\begin{equation}\label{eq: condition for k=r}
v_p(x)\leq v_p(\Theta) \quad\mbox{ and }\quad v_p(\Theta)-2v_p(x)\geq r-1,
\end{equation}
in which case
$$\Fil^r_{r}\fM'=S'_\cO (v^rE_1',F'_{r,0})+\Fil^pS'_\cO\fM'$$
and $\Fil^r_{r}\cM'=\barS'_\F(u^r\barE'_1,\barF'_{r,0})$ with
$$\barF'_{r,0}=\bigg(\frac{\Theta}{x}+\frac{\Theta}{p^{r-1}x^2}\frac{(-1)^{r-1}}{r-1}u^{r-1}\bigg)\barE'_1+\barE'_2.$$
\end{prop}

Note that the inequalities in \eqref{eq: condition for k=r} imply that $v_p(x)<0$.

\begin{proof}
If $k=r$, then by Lemma~\ref{lemma: k=0,k=r} we have
\begin{align*}
[F'_{r,i}]_{E'_1}
    =p^i\Theta P_{r,i}\Big(\frac{v}{p},x\Big)
    =\frac{\Theta}{x}v^i+\sum_{s=i}^{r-2}\frac{p^i\Theta}{x^{r-s}}P_{r,i,s}\Big(\frac{v}{p}\Big)\,\mbox{ and }\,
[F'_{r,i}]_{E'_2}
    =p^iQ_{r,i}\Big(\frac{v}{p}\Big)
    =v^i.
\end{align*}

We first claim that if $[F'_{r,i}]_{E'_1}\in\SBr'$ for all $0\leq i\leq r-1$ then $v_p(x)<0$. Indeed, as $P_{r,r-2,r-2}(T)=-T^{r-1}$ by Lemma~\ref{lemma: k=0,k=r}, we have
\begin{align*}
    [F'_{r,r-2}]_{E'_1}=\frac{\Theta}{x}v^{r-2}-\frac{\Theta}{px^2}v^{r-1}.
\end{align*}
Thus, if $[F'_{r,r-2}]_{E'_1}\in S'_\cO$, then $v_p(\Theta/px^2)\geq 0$, which gives
\begin{align*}
    2v_p(x)\leq v_p(\Theta)-1\leq -1<0
\end{align*}
and so $v_p(x)<0$.

Note that the following three are equivalent:
\begin{itemize}[leftmargin=*]
\item $\frac{p^i\Theta}{x^{r-s}}P_{r,i,s}\big(\frac{v}{p}\big)\in S'_\cO$ for all $0\leq i\leq s\leq r-2$;
\item $\frac{\Theta}{x^{r-s}}P_{r,0,s}\big(\frac{v}{p}\big)\in S'_\cO$ all $0\leq s\leq r-2$;
\item $\frac{\Theta}{x^{2}}P_{r,0,r-2}\big(\frac{v}{p}\big)\in S'_\cO$,
\end{itemize}
where the first two items are equivalent as $\deg_T P_{r,i,s}(T)\leq r-1$ for all $0\leq i\leq s\leq r-2$ by Lemma~\ref{lemma: k=0,k=r}, and the last two items are equivalent as $v_p(x)<0$.

Now, as $P_{r,0,r-2}(T)=-[Tf_r(T)]^{(r-1)}_T$ by Lemma~\ref{lemma: k=0,k=r}, we have
\begin{align*}
    \frac{\Theta}{x^{2}}P_{r,0,r-2}\Big(\frac{v}{p}\Big)
    \equiv
    \frac{\Theta}{p^{r-1}x^2}\frac{(-1)^{r-1}}{r-1}v^{r-1}
    \pmod{\frac{\Theta}{p^{r-1}x^{2}}\fm S'_\cO}.
\end{align*}
Hence, $F'_{r,i}\in\fM'$ for all $0\leq i\leq r-1$ if and only if
\begin{align*}
    v_p(x)\leq v_p(\Theta)
    \quad\mbox{ and }\quad
    v_p(\Theta)-2v_p(x)\geq r-1,
\end{align*}
in which case we have
\begin{align*}
F'_{r,0}
    &\equiv
    \bigg(\frac{\Theta}{x}+\frac{\Theta}{p^{r-1}x^2}\frac{(-1)^{r-1}}{r-1}v^{r-1}\bigg)E'_1 +E'_2 \pmod{\fm\fM'};
    \\
F'_{r,i}  &\equiv
    \frac{\Theta}{x}v^iE'_1
    +v^iE'_2 \pmod{\fm\fM'}\,\,\mbox{ if }0<i\leq r-1.
\end{align*}

We note that
\begin{align*}
F'_{r,i}
   & \equiv
    v^iF'_{r,0}-\frac{\Theta}{p^{r-1}x^2}\frac{(-1)^{r-1}}{r-1}v^{i-1}(v^rE'_1)
    \pmod{\fm\fM'}\,\,\mbox{ if }0<i\leq r-1;\\
v^rE'_2
    &\equiv
    v^rF'_{r,0}
    -\bigg(\frac{\Theta}{x}+\frac{\Theta}{p^{r-1}x^2}\frac{(-1)^{r-1}}{r-1}v^{r-1}\bigg)v^rE'_1
    \pmod{\fm\fM'}.
\end{align*}
Hence, we deduce that $\Fil^r_{r}\cM'=\barS'_\F(u^r\barE'_1,\barF'_{r,0})$. Now, by the same argument as in \eqref{eq: Nakayama}, we have $\Fil^r_{r}\fM'=S'_\cO(v^rE'_1,F'_{r,0})+\Fil^pS'_\cO\fM'$, which completes the proof.
\end{proof}

\subsection{Strong divisibility}\label{subsec: Strong divisibility}
In this subsection, we prove Theorem~\ref{theo: strong divisibility of pseudo} when $r'\geq 2$ and $k'\neq\infty$. We keep the notation of \S\ref{sec: Frame works}, \S\ref{subsec: pseudo sdm}, and \S\ref{subsec: Generators of the filtration}. In particular, we assume $r=r'\geq 2$. Note that it is harmless to assume $r=r'$, as one can recover the general case $r'\leq r$ by simply multiplying $c^{r-r'}$ to $\phi'_r(F'_1),\phi'_r(F'_2)$.

\subsubsection{\textbf{In $\mathbf{Case}~(0)$}}
In this section of paragraph, we describe the conditions for $(\Lambda,\Theta,\Omega,x)$ to be good for $\mathbf{Case}_\phi(k')$ when $k'=\frac{1}{2}$. We also determine the corresponding $\Delta$ and describe the image of $\phi'_r$ in $\cM'$ in this case.

\begin{prop}\label{prop-cM-Case-0}
Let $k'=\frac{1}{2}$. If the quadruple $(\Lambda,\Theta,\Omega,x)$ satisfies
\begin{enumerate}[leftmargin=*]
 \item $v_p(\Lambda)=r-1$ and $v_p(\Theta\Omega^{-1})=-r+1$;
 \item $v_p(\Theta)\leq v_p(x)$ and $v_p(\Omega)\leq 0$,
\end{enumerate}
then the quadruple is good for $\mathbf{Case}_\phi(\frac{1}{2})$ with respect to $\Delta(T)=\delta_1(T)$, in which case
$\Fil^r_{0}\cM'=\barS'_\F(\barF'_1,\barF'_2)$ where
$$\barF'_1:=\pi'(F'_{0,0})\quad\mbox{ and }\quad\barF'_2:=\pi'(v^rE'_2),$$
and $\phi'_r:\Fil^r_{0}\cM'\to \cM'$ is described as
    \begin{align*}
    \phi'_r(\barF'_1) &=\frac{\Lambda}{p^{r-1}} P_{0,0}(-1,0)\barE'_1 +\frac{\Lambda}{p^{r-1}\Omega}P_{0,0}(-1,0)\dot{\delta}_1(-1)u^{p-1}\barE'_2.;\\
    \phi'_r(\barF'_2) &=(-1)^r\frac{\Lambda\Theta}{\Omega}\barE'_2.
    \end{align*}
\end{prop}

\begin{proof}
We start the proof by noting that the inequalities in (ii) is equivalent with \eqref{eq: condition for k=0} under the equation (i) and that the equations in (i) and the inequalities in (ii) imply the inequalities in \eqref{eq: condition for phi,N stable}.

Let $k:=\lfloor k'\rfloor=0$ and $\Delta=\delta_1$. By Lemma~\ref{lemma formula for phi_r(F_kk) and phi_r(F_k0)}, it is immediate that if $[\phi'_r(F'_{0,0})]_{E'_1}\in S_{\cO}'$ then $\phi'_r(F'_{0,0})$ can be written as
\begin{equation*}
\phi'_r(F'_{0,0})\equiv\frac{\Lambda }{p^{r-1}}P_{0,0}(c,x)E'_1 +\frac{\Lambda}{p^{r}\Omega}\big([P_{0,0}(c,x)cf(c)]_c^{(r-1)}-P_{0,0}(c,x)\Delta(\gamma-1)\big)E'_2
\end{equation*}
modulo $\fm\fM'$. As $c-(\gamma-1)\equiv pu^{p-1}\pmod{p^2}$, by Lemma~\ref{lemma: k=0,k=r} one can readily see that $\phi'_r(\barF'_1)$ can be written as in the statement.

By Lemma~\ref{lemma: k=0,k=r} $P_{0,0}(c,x)=1$, and so if $v_p(\Lambda)=r-1$ then $[\phi'_r(\barF'_1)]_{\barE'_1}=\frac{\Lambda}{p^{r-1}}\in\F^\times$. It is easy to see that $v_p(\Omega)\leq 0$ and $v_p(\frac{\Lambda}{p^{r-1}})=0$ imply that $v_p(\frac{\Lambda}{p^{r-1}\Omega})\geq 0$, and so we further have $[\phi'_r(\barF'_1)]_{\barE'_2}\in \barS'_\F$ as $\dot{\delta}_1(-1)=r-1\in\F^{\times}$ by Lemma~\ref{lemma-delta-derivative}~(i).
Moreover, it is also immediate from \eqref{eq: definition of phi'} to have that $\phi_r'(v^rE'_2)=c^r\frac{\Lambda\Theta}{\Omega} E'_2$, and so we have $\phi'_r(\barF'_2)$ as in the statement and if $v_p(\Theta\Omega^{-1})=-r+1$ and $v_p(\Lambda)=r-1$ then $v_p(\frac{\Lambda\Theta}{\Omega})=0$ and so $[\phi'_r(\barF'_2)]_{\barE'_2}\in \F^\times$. Hence, we conclude that
the quadruple is good for $\mathbf{Case}_\phi(\frac{1}{2})$ with respect to $\Delta=\delta_1$.
\end{proof}

\subsubsection{\textbf{In $\mathbf{Case}~(k)$ with $1\leq k<r/2$}}
In this section of paragraph, we describe the conditions for $(\Lambda,\Theta,\Omega,x)$ to be good for $\mathbf{Case}_\phi(k')$ when $k'\in\{k,k+\frac{1}{2}\}$ and $k\in\Z$ with $1\leq k< r/2$. We also determine the corresponding $\Delta$ and describe the image of $\phi'_r$ in $\cM'$ in this case.

\begin{prop}\label{prop-cM-Case-1-r/2}
Let $k\in\Z$ with $1\leq k< r/2$, and let $k'\in\{k,k+\frac{1}{2}\}$. If the quadruple $(\Lambda,\Theta,\Omega,x)$ satisfies
    \begin{enumerate}[leftmargin=*]
      \item $\begin{cases}
          v_p(\Lambda \Theta)=k-1\mbox{ and } v_p(\Theta\Omega)=-r+2k-1 & \mbox{if } k'=k; \\
          v_p(\Lambda)=r-k-1\mbox{ and } v_p(\Theta\Omega^{-1})=-r+2k+1 & \mbox{if } k'=k+\frac{1}{2},
        \end{cases}$
      \item $v_p(\Theta)\leq v_p(x)$ and $-1\leq v_p(\Omega)\leq 0$,
    \end{enumerate}
    then the quadruple is good for $\mathbf{Case}_\phi(k')$ with respect to $\Delta(T)=\delta_{\lceil k' \rceil}(T)$, in which case $\Fil^r_{k}\cM'=\barS'_\F(\barF'_1,\barF'_2)$ where
$$\barF'_1:=\pi'(F'_{k,k})\quad\mbox{ and }\quad\barF'_2:=\pi'(F'_{k,0}),$$
and the map
$\phi'_r:\Fil^r_{k}\cM'\to \cM'$ is described as follows:
\begin{itemize}[leftmargin=*]
\item if $k'=k$ then
    \begin{align*}
    \phi'_r(\barF'_1)  &=\frac{\Lambda}{p^{r-k-1}} P_{k,k}(-1,0)\barE'_1 +\frac{\Lambda}{p^{r-k}\Omega} P_{k,k}(-1,0)\big(\delta_{k+1}(-1)-\delta_k(-1)\big)\barE'_2;\\
    \phi'_r(\barF'_2)  &=\frac{\Lambda\Theta}{p^{k-1}}  P_{k,0}(-1,0)\barE'_1   +\frac{\Lambda\Theta}{p^{k-1}\Omega}P_{k,0}(-1,0)\dot{\delta}_k(-1)u^{p-1}\barE'_2,
    \end{align*}
\item if $k'=k+\frac{1}{2}$ then
    \begin{align*}
    \phi'_r(\barF'_1)   &= \frac{\Lambda}{p^{r-k-1}} P_{k,k}(-1,0)\barE'_1 +\frac{\Lambda}{p^{r-k-1}\Omega} P_{k,k}(-1,0)\dot{\delta}_{k+1}(-1)u^{p-1}\barE'_2;\\
    \phi'_r(\barF'_2)   &= \frac{\Lambda \Theta}{p^{k-1}} P_{k,0}(-1,0)\barE'_1 +\frac{\Lambda\Theta}{p^k\Omega}P_{k,0}(-1,0)\big(
        \delta_k(-1)-\delta_{k+1}(-1)
        \big)\barE'_2.
    \end{align*}
\end{itemize}
\end{prop}

\begin{proof}
We start the proof by noting that the inequalities in (ii) imply the ones in \eqref{eq: condition for 1 leq k <r/2}, and that the equations in (i) and the inequalities in (ii) imply the inequalities in \eqref{eq: condition for phi,N stable}.

Assume first that $k'=k$, and let $\Delta=\delta_k$. Then we have $P_{k,0}(c,x)=P_{k,0}(c,0)$ by Lemma~\ref{lemma: B_k,0 k leq m}~(iii), and $P_{k,0}(-1,0)\in\Z_{(p)}^{\times}$ by Lemma~\ref{lemma-delta-denom-1}~(i), and so if $v_p(\Lambda\Theta)=k-1$ then $[\phi'_r(\barF'_2)]_{\barE_1'}=\frac{\Lambda\Theta}{p^{k-1}}P_{k,0}(-1,0)\in\F^{\times}$. As $c-(\gamma-1)\equiv pu^{p-1}\pmod{p^2}$, by Lemma~\ref{lemma: PQdegree}~(i) together with Lemma~\ref{lemma-PQ-k leq m}~(i) one can readily see that from \eqref{eq: formula for phi_r(F_k,0)}
\begin{align*}
[\phi'_r(\barF'_2)]_{\barE'_2}&= \frac{\Lambda\Theta}{p^k\Omega}\big(Q_{k,0}(c,0)-P_{k,0}(c,0)\delta_k(\gamma-1)\big)\\
&=\frac{\Lambda\Theta}{p^k\Omega}P_{k,0}(c,0)\big(\delta_k(c)-\delta_k(\gamma-1)\big)\\
&=\frac{\Lambda\Theta}{p^{k-1}\Omega}P_{k,0}(-1,0)\dot{\delta}_k(-1)u^{p-1}.
\end{align*}
Hence, $\phi'_r(\barF'_2)$ can be written as in the statement, and we have $[\phi'_r(\barF'_2)]_{\barE_2'}\in u^{p-1}S_\cO'$, by Lemma~\ref{lemma-delta-derivative} together with Lemma~\ref{lemma-delta-denom-1}, as $v_p(\Lambda\Theta\Omega^{-1})\geq k-1$.

To satisfy \eqref{eq: strong divisibility, local}, we should have $[\phi'_r(\barF'_1)]_{\barE_2'}\in (\barS_\F')^\times$ as $[\phi'_r(\barF'_2)]_{\barE_2'}\in u^{p-1}\barS_\F'$. However, by Lemma~\ref{lemma: PQdegree}~(i) together with Lemma~\ref{lemma-PQ-k leq m}~(i) one can readily see that from \eqref{eq: formula for phi_r(F_k,k)}
\begin{align*}
[\phi'_r(\barF'_1)]_{\barE'_2}&=\frac{\Lambda}{p^{r-k}\Omega}\left(Q_{k,k}(c,0) -P_{k,k}(c,0)\delta_{k}(\gamma-1)\right) \\ &=\frac{\Lambda}{p^{r-k}\Omega}P_{k,k}(c,0)\left(\delta_{k+1}(c)-\delta_k(\gamma-1)\right).
\end{align*}
Hence, $\phi'_r(\barF'_1)$ can be written as in the statement, and if $v_p(\Lambda\Theta)=k-1$ and $v_p(\Theta\Omega)=-r+2k-1$ then $v_p(\frac{\Lambda}{\Omega})=v_p(\Lambda\Theta)-v_p(\Theta\Omega)=r-k$ and so we have $[\phi'_r(\barF'_1)]_{\barE_2'}\in \F^\times$, by Lemma~\ref{lemma-delta-derivative}~(i) together with Lemma~\ref{lemma-delta-denom-1}~(ii). Moreover, we have $[\phi'_r(\barF'_1)]_{\barE'_1}\in \barS'_\F$ as $v_p(\frac{\Lambda}{p^{r-k-1}})=v_p(\frac{\Lambda}{p^{r-k}\Omega})+v_p(\Omega)+1\geq 0$.

Assume now that $k'=k+\frac{1}{2}$, and let $\Delta=\delta_{k+1}$. Then we have $P_{k,k}(c,x)=P_{k,k}(c,0)$ by Lemma~\ref{lemma-PQ-k leq m}~(i), and $P_{k,k}(-1,0)\in\Z_{(p)}^{\times}$ by Lemma~\ref{lemma-delta-denom-1}~(ii), and so if $v_p(\Lambda)=r-k-1$ then $[\phi'_r(\barF'_1)]_{\barE_1'}=\frac{\Lambda}{p^{r-k-1}}P_{k,k}(-1,0)\in\F^{\times}$. As $c-(\gamma-1)\equiv pu^{p-1}\pmod{p^2}$, by Lemma~\ref{lemma: PQdegree}~(i) together with Lemma~\ref{lemma-PQ-k leq m}~(i) one can readily see that from \eqref{eq: formula for phi_r(F_k,k)}
\begin{align*}
[\phi'_r(\barF'_1)]_{\barE'_2}&=\frac{\Lambda}{p^{r-k}\Omega}\left(Q_{k,k}(c,0) -P_{k,k}(c,0)\delta_{k+1}(\gamma-1)\right) \\ &=\frac{\Lambda}{p^{r-k}\Omega}P_{k,k}(c,0)\left(\delta_{k+1}(c)-\delta_{k+1}(\gamma-1)\right)\\
&=\frac{\Lambda}{p^{r-k-1}\Omega}P_{k,k}(c,0)\dot{\delta}_{k+1}(-1)u^{p-1}.
\end{align*}
Hence, $\phi'_r(\barF'_1)$ can be written as in the statement, and we have $[\phi'_r(\barF'_1)]_{\barE_2'}\in u^{p-1}S_\cO'$, by Lemma~\ref{lemma-delta-derivative}~(i) together with Lemma~\ref{lemma-delta-denom-1}~(ii), as $v_p(\frac{\Lambda}{p^{r-k-1}\Omega})=v_p(\frac{\Lambda}{p^{r-k-1}})-v_p(\Omega)\geq 0$.

To satisfy \eqref{eq: strong divisibility, local}, we should have $[\phi'_r(\barF'_2)]_{\barE_2'}\in (\barS_\F')^\times$ as $[\phi'_r(\barF'_1)]_{\barE_2'}\in u^{p-1}\barS_\F'$. However, by Lemma~\ref{lemma: PQdegree}~(i) together with Lemma~\ref{lemma-PQ-k leq m}~(i) one can readily see that from \eqref{eq: formula for phi_r(F_k,0)}
\begin{align*}
[\phi'_r(\barF'_2)]_{\barE'_2}&=\frac{\Lambda\Theta}{p^k\Omega} \big(Q_{k,0}(c,0)-P_{k,0}(c,0)\delta_{k+1}(\gamma-1)\big)\\
&=\frac{\Lambda\Theta}{p^k\Omega} P_{k,0}(c,0)\big(\delta_k(c)-\delta_{k+1}(\gamma-1)\big).
\end{align*}
Hence, $\phi'_r(\barF'_2)$ can be written as in the statement, and if $v_p(\Lambda)=r-k-1$ and $v_p(\Theta\Omega^{-1})=-r+2k+1$ then $v_p(\frac{\Lambda\Theta}{\Omega})=v_p(\Lambda)+v_p(\Theta\Omega^{-1})=k$ and so we have $[\phi'_r(\barF'_2)]_{\barE_2'}\in \F^\times$, by Lemma~\ref{lemma-delta-denom-1}~(i) together with Lemma~\ref{lemma-delta}~(ii). Moreover, we have $[\phi'_r(\barF'_2)]_{\barE'_1}\in \barS'_\F$ as $v_p(\frac{\Lambda\Theta}{p^{k-1}})=v_p(\frac{\Lambda\Theta}{p^k\Omega})+v_p(\Omega)+1\geq 0$. This completes the proof.
\end{proof}

\subsubsection{\textbf{In $\mathbf{Case}~(m)$ with $r=2m$}}
In this section of paragraph, we describe the conditions for $(\Lambda,\Omega,\Theta,x)$ to be good for $\mathbf{Case}_\phi(k')$ when $k'\in\{m,m+\frac{1}{2}\}$ and $r=2m$. We also determine the corresponding $\Delta$ and describe the image of $\phi'_r$ in $\cM'$ in this case.

\begin{prop}\label{prop-cM-Case-m}
Let $k=m$ with $r=2m$, let $k'\in\{m,m+\frac{1}{2}\}$, and assume that $v_p(\Theta)\leq 0$ for $k'=m+\frac{1}{2}$. If the quadruple $(\Lambda,\Theta,\Omega,x)$ satisfies
    \begin{enumerate}[leftmargin=*]
      \item $\begin{cases}
          v_p(\Lambda \Theta)=m-1\mbox{ and } v_p(\Theta\Omega)=-1 & \mbox{if } k'=m; \\
          v_p(\Lambda x)=m-1\mbox{ and } v_p(\Theta\Omega^{-1}x^{-2})=1 & \mbox{if }k'=m+\frac{1}{2},
        \end{cases}$
      \item $\begin{cases}
          v_p(\Theta)\leq v_p(x) \mbox{ and } -1\leq v_p(\Omega)\leq 0
           & \mbox{if } k'=m; \\
          v_p(\Theta)\geq v_p(x)<0 \mbox{ and } v_p(\Omega)\leq 0 & \mbox{if }k'=m+\frac{1}{2},
        \end{cases}$
    \end{enumerate}
    then the quadruple is good for $\mathbf{Case}_\phi(k')$ with respect to
    $$\Delta(T)=
        \begin{cases}
            \delta_{m}(T) & \mbox{if } k'=m; \\
            \delta_{m+1}(T,x) & \mbox{if } k'=m+\frac{1}{2},
        \end{cases}$$
in which case
$\Fil^r_{m}\cM'=\barS'_\F(\barF'_1,\barF'_2)$ where
$$
\begin{cases}
\barF'_1:=\pi'\big(F'_{m,m}+\frac{x}{\Theta}F'_{m,0}\big)\quad\mbox{ and }\quad\barF'_2:=\pi'(F'_{m,0}) &\mbox{if }k'=m;\\
\barF'_1:=\pi'(F'_{m,m})\quad\mbox{ and }\quad\barF'_2:=\pi'\big(\frac{\Theta}{x}F'_{m,m}+F'_{m,0}\big) &\mbox{if }k'=m+\frac{1}{2},
\end{cases}$$
and the map $\phi'_r:\Fil^r_{m}\cM'\to \cM'$ is described as follows:
\begin{itemize}[leftmargin=*]
\item if $k'=m$ then
    \begin{align*}
    \phi'_r(\barF'_1) &=\frac{\Lambda}{p^{m-1}} P_{m,m}(-1,0)\barE'_1 -\frac{\Lambda}{p^m\Omega} \frac{(-1)^r}{P_{m,0}(-1,0)}\barE'_2;  \\
    \phi'_r(\barF'_2)  &= \frac{\Lambda \Theta}{p^{m-1}} P_{m,0}(-1,0)\barE'_1 + \frac{\Lambda\Theta}{p^{m-1}\Omega} P_{m,0}(-1,0)\dot{\delta}_m(-1)u^{p-1}\barE'_2,
    \end{align*}
\item if $k'=m+\frac{1}{2}$ then
    \begin{align*}
    \phi'_r(\barF'_1) &= -\frac{\Lambda x}{p^{m-1}} P_{m,0}(-1,0)\barE'_1 -\frac{\Lambda x}{p^{m-1}\Omega} P_{m,0}(-1,0)\dot{\delta}_m(-1)u^{p-1}\barE'_2; \\
    \phi'_r(\barF'_2) &= -\frac{\Lambda\Theta x^{-1}}{p^m\Omega} \frac{(-1)^r}{P_{m,0}(-1,0)}\barE'_2.
    \end{align*}
\end{itemize}
\end{prop}

\begin{proof}
We start the proof by noting that the inequalities in (ii) imply the ones in \eqref{eq: condition for k=m r=2m} under the equations in (i), and that the equations in (i) and the inequalities in (ii) imply the inequalities in \eqref{eq: condition for phi,N stable}.

Assume first that $k'=m$, and let $\Delta=\delta_m$. Then we have $P_{m,0}(c,x)=P_{m,0}(c,0)$ by Lemma~\ref{lemma: B_k,0 k leq m}~(iii), and $P_{m,0}(-1,0)\in\Z_{(p)}^{\times}$ by Lemma~\ref{lemma-delta-denom-1}~(i), and so if $v_p(\Lambda\Theta)=m-1$ then $[\phi'_r(\barF'_2)]_{\barE_1'}=\frac{\Lambda\Theta}{p^{m-1}}P_{m,0}(-1,0)\in\F^{\times}$. As $c-(\gamma-1)\equiv pu^{p-1}\pmod{p^2}$, by Lemma~\ref{lemma: PQdegree}~(i) together with Lemma~\ref{lemma-PQ-k leq m}~(i) one can readily see that from \eqref{eq: formula for phi_r(F_k,0)}
\begin{align}\label{eq: case k=m r=2m strong div, 1}
\begin{split}
[\phi'_r(\barF'_2)]_{\barE'_2}&= \frac{\Lambda\Theta}{p^m\Omega} \big(Q_{m,0}(c,0)-P_{m,0}(c,0)\delta_m(\gamma-1)\big)\\
&=\frac{\Lambda\Theta}{p^m\Omega} P_{m,0}(c,0)\big(\delta_m(c)-\delta_m(\gamma-1)\big)\\
&=\frac{\Lambda\Theta}{p^{m-1}\Omega} P_{m,0}(-1,0)\dot{\delta}_m(-1)u^{p-1}.
\end{split}
\end{align}
Hence, $\phi'_r(\barF'_2)$ can be written as in the statement, and we have $[\phi'_r(\barF'_2)]_{\barE_2'}\in u^{p-1}\barS_\F'$, by Lemma~\ref{lemma-delta-derivative} together with Lemma~\ref{lemma-delta-denom-1}, as $v_p(\frac{\Lambda\Theta}{p^{m-1}\Omega})=v_p(\frac{\Lambda\Theta}{p^{m-1}})-v_p(\Omega)\geq0$.

To satisfy \eqref{eq: strong divisibility, local}, we should have $[\phi'_r(\barF'_1)]_{\barE_2'}\in (\barS_\F')^\times$ as $[\phi'_r(\barF'_2)]_{\barE_2'}\in u^{p-1}\barS_\F'$. However, one can readily see that from \eqref{eq: formula for phi_r(F_k,k)}
\begin{align*}
[\phi'_r(\barF'_{m,m})]_{\barE'_2}&=\frac{\Lambda}{p^m\Omega}  P_{m,m}(c,x)\big(\delta_{m+1}(c,x)-\delta_m(\gamma-1)\big)\\
            &=\frac{\Lambda}{p^m\Omega} P_{m,m}(c,x)\left[\big(\delta_{m+1}(c,x)-\delta_m(c)\big) +\big(\delta_m(c)-\delta_m(\gamma-1)\big)\right]\\
            &=\frac{\Lambda}{p^m\Omega} \frac{-c^r}{P_{m,0}(c,0)}
            +\frac{\Lambda}{p^m\Omega} \big(P_{m,m}(c,0)-xP_{m,0}(c,0)\big)
            \big(\delta_m(c)-\delta_m(\gamma-1)\big)\\
            &= -\frac{\Lambda}{p^m\Omega} \frac{c^r}{P_{m,0}(c,0)}
            +\frac{\Lambda}{p^{m-1}\Omega} \big(P_{m,m}(c,0)-xP_{m,0}(c,0)\big)
            \dot{\delta}_m(-1)u^{p-1}\\
            &= -\frac{\Lambda}{p^m\Omega} \frac{c^r}{P_{m,0}(c,0)} -\frac{\Lambda x}{p^{m-1}\Omega} P_{m,0}(c,0)\dot{\delta}_m(-1)u^{p-1},
\end{align*}
where the first equality is due to Lemma~\ref{lemma: PQdegree}~(i), Lemma~\ref{lemma: B_k,0 k leq m}, and Lemma~\ref{lemma-PQ-k leq m}~(ii), the third equality is due to Lemma~\ref{lemma-delta}~(iii) and Lemma~\ref{lemma-PQ-k leq m}~(ii), the forth equality is due to $c-(\gamma-1)\equiv pu^{p-1}\pmod{p^2}$, and the last equality is due to $v_p(\Lambda\Omega^{-1})=v_p(\Lambda\Theta)-v_p(\Theta\Omega)=m$. Hence, together with \eqref{eq: case k=m r=2m strong div, 1} it is immediate that $[\phi'_r(\barF'_1)]_{\barE'_2}$ can be written as in the statement, and if $v_p(\Lambda\Theta)=m-1$ and $v_p(\Theta\Omega)=-1$ then we have $[\phi'_r(\barF'_1)]_{\barE_2'}\in \F^\times$, by Lemma~\ref{lemma-delta-denom-1}~(i). Moreover, by Lemma~\ref{lemma-PQ-k leq m} and Lemma~\ref{lemma: B_k,0 k leq m}~(iii) we have
\begin{align*}
[\phi'_r(\barF'_1)]_{\barE'_1}=\frac{\Lambda }{p^{m-1}}P_{m,m}(c,x)+\frac{ \Lambda x}{p^{m-1}}P_{m,0}(c,x)=\frac{\Lambda }{p^{m-1}}P_{m,m}(c,0),
\end{align*}
and so $[\phi'_r(\barF'_1)]_{\barE'_1}=\frac{\Lambda }{p^{m-1}}P_{m,m}(-1,0)\in\F$ as $v_p(\frac{\Lambda}{p^{m-1}})=v_p(\frac{\Lambda\Theta}{p^{m-1}})-v_p(\Theta\Omega)+v_p(\Omega)\geq -(-1)+(-1)=0$.

Assume now that $k'=m+\frac{1}{2}$, and let $\Delta(T)=\delta_{m+1}(T,x)$. Then we have $P_{m,m}(c,x)=P_{m,m}(c,0)-xP_{m,0}(c,0)$ by Lemma~\ref{lemma-PQ-k leq m}~(ii), and $P_{m,0}(-1,0)\in\Z_{(p)}^{\times}$ by Lemma~\ref{lemma-delta-denom-1}~(i), and so if $v_p(\Lambda x)=m-1$ and $v_p(x)<0$ then $[\phi'_r(\barF'_1)]_{\barE_1'}=-\frac{\Lambda x}{p^{m-1}}P_{m,0}(-1,0)\in\F^{\times}$. Moreover, one can readily see that from \eqref{eq: formula for phi_r(F_k,k)}
\begin{align}\label{eq: case k=m r=2m strong div, 2}
\begin{split}
[\phi'_r(\barF'_1)]_{\barE'_2}&=\frac{\Lambda}{p^{m}\Omega} P_{m,m}(c,x)\left(\delta_{m+1}(c,x)-\delta_{m+1}(\gamma-1,x)\right)\\ &= \frac{\Lambda}{p^{m}\Omega} (P_{m,m}(c,0)-xP_{m,0}(c,0))\left(\delta_{m+1}(c,x)-\delta_{m+1}(\gamma-1,x)\right)\\
&=-\frac{\Lambda x}{p^{m-1}\Omega} P_{m,0}(c,0)\dot{\delta}_{m+1}(-1,x)u^{p-1}\\
&=-\frac{\Lambda x}{p^{m-1}\Omega} P_{m,0}(c,0)\dot{\delta}_{m}(-1)u^{p-1},
\end{split}
\end{align}
where the first equality is due to Lemma~\ref{lemma: PQdegree}~(i), Lemma~\ref{lemma: B_k,0 k leq m}, and Lemma~\ref{lemma-PQ-k leq m}~(ii), the second equality is due to Lemma~\ref{lemma-PQ-k leq m}~(ii), and the third equality is due to $c-(\gamma-1)\equiv pu^{p-1}\pmod{p^2}$ and $v_p(x)<0$, and the last equality is due to Lemma~\ref{lemma-delta}~(iii) and $v_p(x)<0$. Hence, $\phi'_r(\barF'_1)$ can be written as in the statement, and we have $\phi'_r(\barF'_1)_{\barE_2'}\in u^{p-1}\barS_\F'$, by Lemma~\ref{lemma-delta}~(ii) and Lemma~\ref{lemma-delta-denom-1}~(i), as $v_p(\frac{\Lambda x}{p^{m-1}\Omega})= v_p(\frac{\Lambda x}{p^{m-1}})-v_p(\Omega)\geq 0$.

To satisfy \eqref{eq: strong divisibility, local}, we should have $[\phi'_r(\barF'_2)]_{\barE_2'}\in (\barS_\F')^\times$ as $[\phi'_r(\barF'_1)]_{\barE_2'}\in u^{p-1}\barS_\F'$. However, one can readily see that from \eqref{eq: formula for phi_r(F_k,0)}
\begin{align*}
[\phi'_r(\barF'_{m,0})]_{\barE'_2}&=\frac{\Lambda\Theta}{p^m\Omega}
      P_{m,0}(c,0)\big(\delta_m(c)-\delta_{m+1}(\gamma-1,x)\big)\\
      &=\frac{\Lambda\Theta}{p^m\Omega}
      P_{m,0}(c,0)\big[\big(\delta_m(c)-\delta_{m+1}(c,x)\big) +\big(\delta_{m+1}(c,x)-\delta_{m+1}(\gamma-1,x)\big)\big]\\
      &=\frac{\Lambda\Theta}{p^m\Omega} \frac{c^r}{P_{m,m}(c,x)} + \frac{\Lambda\Theta}{p^{m-1}\Omega} P_{m,0}(c,0)\dot{\delta}_{m+1}(-1,x)u^{p-1}\\
      &=-\frac{\Lambda\Theta x^{-1}}{p^m\Omega} \frac{c^r}{P_{m,0}(c,0)} + \frac{\Lambda\Theta}{p^{m-1}\Omega} P_{m,0}(c,0)\dot{\delta}_m(-1)u^{p-1},
  \end{align*}
where the first equality is due to Lemma~\ref{lemma: PQdegree} and Lemma~\ref{lemma: B_k,0 k leq m}~(iii), and the third equality is due to Lemma~\ref{lemma-delta}~(iii) and $c-(\gamma-1)\equiv pu^{p-1}\pmod{p^2}$, and the last equality is due to Lemma~\ref{lemma-PQ-k leq m} and $v_p(x)<0$. Hence, together with \eqref{eq: case k=m r=2m strong div, 2} it is immediate that $[\phi'_r(\barF'_2)]_{\barE'_2}$ can be written as in the statement, and if 
$v_p(\Lambda\Theta\Omega^{-1}x^{-1})=v_p(\Lambda x)+v_p(\Theta\Omega^{-1}x^{-2})=m$ then we have $[\phi'_r(\barF'_2)]_{\barE_2'}\in \F^\times$, by Lemma~\ref{lemma-delta-denom-1}~(i). Moreover, by Lemma~\ref{lemma-PQ-k leq m} and Lemma~\ref{lemma: B_k,0 k leq m}~(iii) we have
\begin{align*}
[\phi'_r(\barF'_2)]_{\barE'_1}=\frac{\Lambda\Theta x^{-1} }{p^{m-1}}P_{m,m}(c,x)+\frac{\Lambda\Theta}{p^{m-1}}P_{m,0}(c,x)=\frac{\Lambda\Theta x^{-1}}{p^{m-1}}P_{m,m}(c,0),
\end{align*}
and so $[\phi'_r(\barF'_2)]_{\barE'_1}= \frac{\Lambda\Theta x^{-1}}{p^{m-1}}P_{m,m}(-1,0)=0$ as $v_p(\frac{\Lambda\Theta x^{-1}}{p^{m-1}})=v_p(\frac{\Lambda x}{p^{m-1}})+v_p(\Theta x^{-1})-v_p(x)>0$. This completes the proof.
\end{proof}

\subsubsection{\textbf{In $\mathbf{Case}~(m+1)$ with $m+1<r$}}
In this section of paragraph, we describe the conditions for $(\Lambda,\Theta,\Omega,x)$ to be good for $\mathbf{Case}_\phi(k')$ when $k'\in\{m+1,m+1+\frac{1}{2}\}$ and $r\geq 3$. We also determine the corresponding $\Delta$ and describe the image of $\phi'_r$ in $\cM'$ in this case.

For $0\leq k\leq r$, we define
\begin{equation}\label{eq: definition of b_r,k}
b_{r,k}(T):=-\frac{[P_{m+1,m+1,1}(T)Tf_r(T)]_{T}^{(r-1)} -P_{m+1,m+1,1}(T)\delta_m(T)}{P_{m,0}(T,0)}
\end{equation}
if $(r,k)=(2m+1,m+2)$, and $b_{r,k}(T):=0$ otherwise. Note that an identity for $b_{2m+1,m+2}$ has been illustrated in Lemma~\ref{lemma-b-2m+1-m+2}.

\begin{prop}\label{prop-cM-Case-m+1}
Let $k=m+1$ with $k<r$, let $k'\in\{k,k+\frac{1}{2}\}$, and assume that $v_p(\Theta)\leq0$ and $v_p(x)<0$.

If the quadruple $(\Lambda,\Theta,\Omega,x)$ satisfies
    \begin{enumerate}[leftmargin=*]
      \item $\begin{cases}
          v_p(\Lambda \Theta x^{-1})=m\mbox{ and } v_p(\Theta\Omega x^{-2})=-r+2m+1 & \mbox{if } k'=m+1; \\
          v_p(\Lambda x)=r-m-2\mbox{ and } v_p(\Theta\Omega^{-1} x^{-2})=-r+2m+3 & \mbox{if }k'=m+1+\frac{1}{2},
        \end{cases}$
      \item $v_p(\Theta)\geq v_p(x)$ and 
      $-1\leq v_p(\Omega)\leq 0$,
    \end{enumerate}
    then the quadruple is good for $\mathbf{Case}_\phi(k')$ with respect to
    $$\Delta(T)=
        \begin{cases}
            \delta_{m+1}(T,x) & \mbox{if } k'=m+1; \\
            \delta_{r-m-1}(T) & \mbox{if } k'=m+1+\frac{1}{2},
        \end{cases}$$
in which case $\Fil^r_{m+1}\cM'=\barS'_\F(\barF'_1,\barF'_2)$ where
$$\barF'_1:=\pi'(F'_{m+1,m+1})\quad\mbox{ and }\quad\barF'_2:=\pi'(F'_{m+1,0}),$$
and the map $\phi'_r:\Fil^r_{m+1}\cM'\to \cM'$ is described as follows:
\begin{itemize}[leftmargin=*]
\item if $k'=m+1$ then
    \begin{align*}
    \phi'_r(\barF'_1)&=
        -\frac{\Lambda x}{p^{r-m-2}}P_{r-m-1,0}(-1,0)\barE'_1\\
         &\hspace{2em}
         -\frac{\Lambda x}{p^{r-m-1}\Omega} P_{r-m-1,0}(-1,0)\big(\delta_{r-m-1}(-1)-\delta_{m+1}(-1,x) \big)\barE'_2;\\
    \phi'_r(\barF'_2) &=\frac{\Lambda \Theta x^{-1}}{p^m}P_{r-m-1,r-m-1}(-1,0)\barE'_1\\
     &\hspace{3em}
     +\frac{\Lambda\Theta x^{-1}}{p^m\Omega} P_{r-m-1,r-m-1}(-1,0)\dot{\delta}_{r-m}(-1)u^{p-1}\barE'_2,
    \end{align*}
\item if $k'=m+1+\frac{1}{2}$ then
    \begin{align*}
    \phi'_r(\barF'_1) &=-\frac{\Lambda x}{p^{r-m-2}}P_{r-m-1,0}(-1,0)\barE'_1\\
    &\hspace{2em}
    -\frac{\Lambda x}{p^{r-m-2}\Omega} P_{r-m-1,0}(-1,0)\Big(\dot{\delta}_{r-m-1}(-1)u^{p-1}+\frac{b_{r,m+2}(-1)}{px}\Big)\barE'_2;\\
    \phi'_r(\barF'_2) &=
        \frac{\Lambda \Theta x^{-1}}{p^m}P_{r-m-1,r-m-1}(-1,0)\barE'_1\\
        &\hspace{2em}
        +\frac{\Lambda\Theta x^{-1}}{p^{m+1}\Omega} P_{r-m-1,r-m-1}(-1,0)\big(
        \delta_{r-m}(-1)-\delta_{r-m-1}(-1)
        \big)\barE'_2.
    \end{align*}
\end{itemize}
\end{prop}

\begin{proof}
We start the proof by noting that the inequalities in (ii) imply the ones in \eqref{eq: condition for m+1 leq k < r} for $k=m+1$ under the equations in (i), and that the equations in (i) and the inequalities in (ii) imply the inequalities in \eqref{eq: condition for phi,N stable}.

We let $t=2(m+1)-r$ for brevity. Assume first that $k'=m+1$, and let $\Delta(T)=\delta_{m+1}(T,x)$. Then we have
\begin{equation}\label{eq: case k=m+1 strong div, 2}
P_{m+1,0}(T,x)=\frac{\tilde{P}_{m+1,0}(T,x)}{d_{m+1}(x)}\equiv x^{-1}P_{r-m-1,r-m-1}(T,0)\pmod{x^{-1}\fm}
\end{equation}
by Lemma~\ref{lemma: m+1 to r} and Lemma~\ref{lemma-delta-denom-1}~(ii), and so by Lemma~\ref{lemma-delta-denom-1} if $v_p(\Lambda\Theta x^{-1})=m$ and $v_p(x)<0$ then
\begin{align*}
[\phi'_r(\barF'_2)]_{\barE'_1}= \frac{ \Lambda \Theta}{p^{m}}\frac{\tilde{P}_{m+1,0}(c,x)}{d_{m+1}(x)} =\frac{ \Lambda \Theta x^{-1}}{p^{m}}P_{r-m-1,r-m-1}(-1,0)\in\F^{\times}.
\end{align*}
Moreover, one can readily see that from \eqref{eq: formula for phi_r(F_k,0)} together with \eqref{eq: case k=m+1 strong div, 2}
\begin{align*}
[\phi'_r(\barF'_2)]_{\barE'_2}&= \frac{\Lambda\Theta}{p^{m+1}\Omega} P_{m+1,0}(c,x)\big(\delta_{m+1}(c,x)-\delta_{m+1}(\gamma-1,x)\big)\\
&=\frac{\Lambda\Theta x^{-1}}{p^{m+1}\Omega} P_{r-m-1,r-m-1}(c,0)\big(\delta_{m+1}(c,x)-\delta_{m+1}(\gamma-1,x)\big)\\
&=\frac{\Lambda\Theta x^{-1}}{p^{m}\Omega} P_{r-m-1,r-m-1}(-1,0)\dot{\delta}_{m+1}(-1,x)u^{p-1}\\
&=\frac{\Lambda\Theta x^{-1}}{p^{m}\Omega} P_{r-m-1,r-m-1}(-1,0)\dot{\delta}_{r-m}(-1)u^{p-1},
\end{align*}
where the third equality is due to $c-(\gamma-1)\equiv pu^{p-1}\pmod{p^2}$ and the last equality is due to Lemma~\ref{lemma-delta-derivative}~(iv). Hence, $\phi'_r(\barF'_2)$ can be written as in the statement, and we have $[\phi'_r(\barF'_2)]_{\barE_2'}\in u^{p-1}\barS_\F'$, by Lemma~\ref{lemma-delta-denom-1} and Lemma~\ref{lemma-delta}, as $v_p(\frac{\Lambda\Theta x^{-1}}{p^m\Omega})=v_p(\frac{\Lambda\Theta x^{-1}}{p^m})-v_p(\Omega) \geq 0$.

To satisfy \eqref{eq: strong divisibility, local}, we should have $[\phi'_r(\barF'_1)]_{\barE_2'}\in (\barS_\F')^\times$ as $[\phi'_r(\barF'_2)]_{\barE_2'}\in u^{p-1}\barS_\F'$. However, we have
\begin{multline}\label{eq: case k=m+1 strong div, 1}
P_{m+1,m+1}(c,x)=\frac{\tilde{P}_{m+1,m+1}(c,x)}{d_{m+1}(x)}\equiv \\
-xP_{r-m-1,0}(c,0) +\frac{1}{x^t}\sum_{s=0}^{t}x^s P_{m+1,m+1,s}(c)\pmod{x\fm S'_\cO}
\end{multline}
by Lemma~\ref{lemma: m+1 to r} and Lemma~\ref{lemma-delta-denom-1}~(i), and so one can readily see that from \eqref{eq: formula for phi_r(F_k,k)}
\begin{align*}
[\phi'_r(\barF'_1)]_{\barE'_2}&=-\frac{\Lambda x}{p^{r-m-1}\Omega} \big([P_{r-m-1,0}(c,0)cf(c)]_c^{(r-1)}-P_{r-m-1,0}(c,0)\delta_{m+1}(\gamma-1,x)\big)\\
&=-\frac{\Lambda x}{p^{r-m-1}\Omega} P_{r-m-1,0}(-1,0)\big(\delta_{r-m-1}(-1)-\delta_{m+1}(-1,x)\big).
\end{align*}
As we have
\begin{align*}
&P_{r-m-1,0}(-1,0)\big(\delta_{r-m-1}(-1)-\delta_{m+1}(-1,x)\big)\\
&=\begin{cases}
   \frac{(-1)^r}{P_{m,m}(-1,0)}&\mbox{if }r=2m;\\
   P_{m-1,0}(-1,0)\big(\delta_{m-1}(-1)-\delta_{m}(-1)\big)&\mbox{if }r=2m+1
  \end{cases}
\end{align*}
by Lemma~\ref{lemma-delta}~(iii), if $v_p(\Lambda\Theta x^{-1})=m$ and $v_p(\Theta\Omega x^{-2})=-r+2m+1$ then $v_p(\frac{\Lambda x}{\Omega})=v_p(\Lambda\Theta x^{-1})-v_p(\Theta\Omega x^{-2})=r-m-1$ and so $[\phi'_r(\barF'_1)]_{\barE'_2}\in\F^\times$ by Lemma~\ref{lemma-delta-denom-1} and Lemma~\ref{lemma-delta}~(i). Moreover, we have
\begin{align*}
[\phi'_r(\barF'_1)]_{\barE'_1}=\frac{\Lambda }{p^{r-m-2}}P_{m+1,m+1}(c,x)=-\frac{\Lambda x}{p^{r-m-2}}P_{r-m-1,0}(-1,0)
\end{align*}
and so by Lemma~\ref{lemma-delta-denom-1} we have $[\phi'_r(\barF'_1)]_{\barE'_1}\in\F$ as $v_p(\Lambda x)=v_p(\frac{\Lambda x}{\Omega})+v_p(\Omega)\geq (r-m-1)+(-1)=r-m-2$.

Assume now that $k'=m+1+\frac{1}{2}$, and let $\Delta(T)=\delta_{r-m-1}(T)$. 
Then by \eqref{eq: case k=m+1 strong div, 1} and Lemma~\ref{lemma-delta-denom-1}, if $v_p(\Lambda x)=r-m-2$ and $v_p(x)<0$ then
\begin{align*}
[\phi'_r(\barF'_1)]_{\barE'_1}= \frac{\Lambda}{p^{r-m-2}}\frac{\tilde{P}_{m+1,m+1}(c,x)}{d_{m+1}(x)} = -\frac{\Lambda x}{p^{r-m-2}}P_{r-m-1,0}(-1,0)\in\F^{\times}
\end{align*}
Moreover, one can readily see that from \eqref{eq: formula for phi_r(F_k,k)}
\begin{align*}
[\phi'_r(\barF'_1)]_{\barE'_2}&=-\frac{\Lambda x}{p^{r-m-1}\Omega} \big([P_{r-m-1,r-m-1}(c,0)cf(c)]_{c}^{(r-1)}-P_{r-m-1,r-m-1}(c,0)\delta_{r-m-1}(\gamma-1)\big)\\
&\hspace{2em}+\frac{\Lambda}{p^{r-m-1}\Omega} \sum_{s=0}^{t}\frac{x^s}{x^{t}}\big([P_{m+1,m+1,s}(c)cf(c)]_{c}^{(r-1)}-P_{m+1,m+1,s}(c)\delta_{r-k+1}(\gamma-1)\big)\\
&=-\frac{\Lambda x}{p^{r-m-1}\Omega} P_{r-m-1,r-m-1}(c,0)\big(\delta_{r-m-1}(c)-\delta_{r-m-1}(\gamma-1)\big)\\
&\hspace{2em}+\frac{\Lambda}{p^{r-m-1}\Omega} \big([P_{m+1,m+1,t}(c)cf(c)]_{c}^{(r-1)}-P_{m+1,m+1,t}(c)\delta_{r-m-1}(\gamma-1)\big),
\end{align*}
where the first equality is due to \eqref{eq: case k=m+1 strong div, 1}, the second equality is due to Lemma~\ref{lemma: m+1 to r}, and the third equality is due to 
$v_p(\frac{\Lambda x^{-s}}{\Omega})=v_p(\Lambda x) +\frac{1}{2}v_p(\Theta\Omega^{-1}x^{-2}) -\frac{1}{2}\big(v_p(\Theta)+v_p(\Omega)\big) +v_p(x^{-s}) \geq (r-m-2)+\frac{1}{2}(t+3)+v_p(x^{-s}) =\frac{r-1}{2}+v_p(x^{-s}) >r-m-1$ if $s\geq 1$. Furthermore, we see that $v_p(\frac{\Lambda}{\Omega})> r-m-1$ unless $r=2m+1$. Hence, we have
\begin{align*}
[\phi'_r(\barF'_1)]_{\barE'_2}
&=-\frac{\Lambda x}{p^{r-m-1}\Omega} P_{r-m-1,0}(-1,0)\big(\delta_{r-m-1}(c)-\delta_{r-m-1}(\gamma-1)\big)\\
&\hspace{3em}
-\frac{\Lambda}{p^{r-m-1}\Omega}P_{r-m-1,0}(-1,0)b_{r,m+2}(-1)\\
&=-\frac{\Lambda x}{p^{r-m-2}\Omega} P_{r-m-1,0}(-1,0)\dot{\delta}_{r-m-1}(-1)u^{p-1}\\
&\hspace{3em}
-\frac{\Lambda}{p^{r-m-1}\Omega}P_{r-m-1,0}(-1,0)b_{r,m+2}(-1)
\end{align*}
where the last equality is due to $c-(\gamma-1)\equiv pu^{p-1}\pmod{p^2}$.
As we already know that $v_p(\frac{\Lambda}{\Omega})\geq r-m-1$ if $r=2m+1$ and as $v_p(\frac{\Lambda x}{p^{r-m-2}\Omega})=v_p(\frac{\Lambda x}{p^{r-m-2}})-v_p(\Omega)\geq 0$, we see that $[\phi'_r(\barF'_1)]_{\barE_2'}\in \barS_\F'$.

To satisfy \eqref{eq: strong divisibility, local}, we should have $[\phi'_r(\barF'_2)]_{\barE_2'}\in (\barS_\F')^\times$ as $[\phi'_r(\barF'_1)]_{\barE_2'}\in u^{p-1}\barS_\F'$. However, by \eqref{eq: case k=m+1 strong div, 2} one can readily see that from \eqref{eq: formula for phi_r(F_k,0)}
\begin{align*}
[\phi'_r(\barF'_2)]_{\barE'_2}& =\frac{\Lambda\Theta x^{-1}}{p^{m+1}\Omega} \big([P_{r-m-1,r-m-1}(c,0)cf(c)]_c^{(r-1)}-P_{r-m-1,r-m-1}(c,0)\delta_{r-m-1}(\gamma-1)\big)\\
&=\frac{\Lambda\Theta x^{-1}}{p^{m+1}\Omega} P_{r-m-1,r-m-1}(c,0)\big(\delta_{r-m}(c)-\delta_{r-m-1}(\gamma-1)\big)
\end{align*}
and so by Lemma~\ref{lemma-delta-denom-1} and Lemma~\ref{lemma-delta} if $v_p(\frac{\Lambda\Theta x^{-1}}{\Omega})=v_p(\Lambda x)+v_p(\Theta\Omega^{-1}x^{-2})=(r-m-2)+(-r+2m+3)=m+1$ then $[\phi'_r(\barF'_2)]_{\barE'_2}\in\F^{\times}$. Moreover, by \eqref{eq: case k=m+1 strong div, 2} we have $$[\phi'_r(\barF'_{2})]_{\barE'_1}=\frac{\Lambda \Theta x^{-1}}{p^{m}}P_{r-m-1,r-m-1}(c,0),$$ and so $[\phi'_r(\barF'_{2})]_{\barE'_1}\in\F$ by Lemma~\ref{lemma-delta-denom-1} as $v_p(\Lambda\Theta x^{-1})=v_p(\frac{\Lambda\Theta x^{-1}}{\Omega})+v_p(\Omega)\geq (m+1)+(-1)=m$. This completes the proof.
\end{proof}

\subsubsection{\textbf{In $\mathbf{Case}~(k)$ with $m+2\leq k<r$}}
In this section of paragraph, we describe the conditions for $(\Lambda,\Theta,\Omega,x)$ to be good for $\mathbf{Case}_\phi(k')$ when $k'\in\{k,k+\frac{1}{2}\}$ and $k\in\Z$ with $m+2\leq k< r$. We also determine the corresponding $\Delta$ and describe the image of $\phi'_r$ in $\cM'$ in this case.

Recall that for $0\leq k\leq r$ we have
\begin{equation*}
b_{r,k}(T):=\frac{[P_{m+2,0,1}(T)Tf_r(T)]_{T}^{(r-1)} -P_{m+2,0,1}(T)\delta_m(T)}{P_{m-1,m-1}(T,0)}
\end{equation*}
if $(r,k)=(2m+1,m+2)$, and $b_{r,k}(T)=0$ otherwise. Note that an identity for $b_{2m+1,m+2}$ has been illustrated in Lemma~\ref{lemma-b-2m+1-m+2}.

\begin{prop}\label{prop-cM-Case-m+2-r-1}
Let $k\in\Z$ with $m+2\leq k<r$, let $k'\in\{k,k+\frac{1}{2}\}$, and assume that $v_p(\Theta)\leq0$. If the quadruple $(\Lambda,\Theta,\Omega,x)$ satisfies
    \begin{enumerate}[leftmargin=*]
      \item $\begin{cases}
          v_p(\Lambda \Theta x^{-1})=k-1\mbox{ and } v_p(\Theta\Omega x^{-2})=-r+2k-1 & \mbox{if } k'=k; \\
          v_p(\Lambda x)=r-k-1\mbox{ and } v_p(\Theta\Omega^{-1} x^{-2})=-r+2k+1 & \mbox{if }k'=k+\frac{1}{2},
        \end{cases}$
      \item $v_p(\Theta)\geq v_p(x)$ and $-1\leq v_p(\Omega)\leq 0$,
      \end{enumerate}
    then the quadruple is good for $\mathbf{Case}_\phi(k')$ with respect to $\Delta(T)=\delta_{r+1-\lceil k' \rceil}(T)$, in which case
   $\Fil^r_{m+2}\cM'=\barS'_\F(\barF'_1,\barF'_2)$ where
    $$\barF'_1:=\pi'(F'_{k,k})\quad\mbox{ and }\quad \barF'_2=\pi'(F'_{k,0}),$$
and the map $\phi'_r:\Fil^r_{m+2}\cM'\to \cM'$ is described as follows:
\begin{itemize}[leftmargin=*]
\item if $k'=k$ then
    \begin{align*}
    \phi'_r(\barF'_1)&= -\frac{\Lambda x}{p^{r-k-1}}P_{r-k,0}(-1,0)\barE'_1 -\frac{\Lambda x}{p^{r-k}\Omega} P_{r-k,0}(-1,0) \big(\delta_{r-k}(-1)-\delta_{r-k+1}(-1)       \big)\barE'_2;\\
    \phi'_r(\barF'_2)
        &=\frac{\Lambda \Theta x^{-1}}{p^{k-1}}P_{r-k,r-k}(-1,0)\barE'_1
        \\&\hspace{3em}
        +\frac{\Lambda\Theta x^{-1}}{p^{k-1}\Omega} 
        P_{r-k,r-k}(-1,0)\Big(\dot{\delta}_{r-k+1}(-1)u^{p-1}
        +\frac{b_{r,k}(-1)}{px}
        \Big)\barE'_2,
    \end{align*}
\item if $k'=k+\frac{1}{2}$ then
    \begin{align*}
    \phi'_r(\barF'_1)&=-\frac{\Lambda x}{p^{r-k-1}}P_{r-k,0}(-1,0)\barE'_1-\frac{\Omega x}{p^{r-k-1}}P_{r-k,0}(-1,0)\dot{\delta}_{r-k}(-1)u^{p-1}\barE'_2;\\
    \phi'_r(\barF'_2)&=\frac{\Lambda \Theta x^{-1}}{p^{k-1}}P_{r-k,r-k}(-1,0)\barE'_1\\
     &\hspace{3em} +\frac{\Lambda\Theta x^{-1}}{p^k\Omega} P_{r-k,r-k}(-1,0)\big(\delta_{r-k+1}(-1)-\delta_{r-k}(-1)\big)\barE'_2.
    \end{align*}
\end{itemize}
\end{prop}

Note that from Proposition~\ref{prop-cM-Case-m+2-r-1} one can readily observe that
\begin{itemize}[leftmargin=*]
\item $[\phi'_r(\barF'_1)]_{\barE_1'}\cdot [\phi'_r(\barF'_2)]_{\barE_2'}=0$ if $k'=k$;
\item the second inequality in (ii) together with $v_p(\Theta)\leq 0$ implies $v_p(x)<0$.
\end{itemize}

\begin{proof}
We start the proof by noting that the inequalities in (ii) imply the ones in \eqref{eq: condition for m+1 leq k < r} for $m+2\leq k\leq r-1$ under the equations in (i), and that the equations in (i) and the inequalities in (ii) imply the inequalities in \eqref{eq: condition for phi,N stable}.

We let $t=2k-r$ for brevity. Assume first that $k'=k$, and let $\Delta(T)=\delta_{r+1-k}(T)$. Then we have
\begin{equation}\label{eq: case m+2 leq k <r strong div, 1}
P_{k,0}(T,x)=\frac{\tilde{P}_{k,0}(T,x)}{d_{k}(x)}\equiv\frac{1}{x^t}\sum_{s=0}^{t-1}x^s P_{k,0,s}(T)\pmod{\fm}
\end{equation}
by Lemma~\ref{lemma: m+1 to r} together with $v_p(x)<0$, and so by Lemma~\ref{lemma-delta-denom-1} if $v_p(\Lambda\Theta x^{-1})=k-1$ and $v_p(x)<0$ then
\begin{align*}
[\phi'_r(\barF'_2)]_{\barE'_1}= \frac{ \Lambda \Theta}{p^{k-1}}\frac{\tilde{P}_{k,0}(c,x)}{d_{k}(x)} =\frac{ \Lambda \Theta x^{-1}}{p^{k-1}}P_{r-k,r-k}(-1,0)\in\F^{\times}.
\end{align*}
Moreover, one can readily see that from \eqref{eq: formula for phi_r(F_k,0)}
\begin{align*}
[\phi'_r(\barF'_2)]_{\barE'_2}&=\frac{\Lambda\Theta}{p^k\Omega} \sum_{s=0}^{t-1}\frac{x^s}{x^t}\big([P_{k,0,s}(c)cf(c)]_{c}^{(r-1)}-P_{k,0,s}(c)\delta_{r-k+1}(\gamma-1)\big)\\
&=\frac{\Lambda\Theta x^{-1}}{p^k\Omega} \big([P_{r-k,r-k}(c,0)cf(c)]_{c}^{(r-1)}-P_{r-k,r-k}(c,0)\delta_{r-k+1}(\gamma-1)\big)\\
&\hspace{2em}+\frac{\Lambda\Theta}{p^k\Omega} \sum_{s=0}^{t-2}\frac{x^s}{x^{t}}\big([P_{k,0,s}(c)cf(c)]_{c}^{(r-1)}-P_{k,0,s}(c)\delta_{r-k+1}(\gamma-1)\big)\\
&=\frac{\Lambda\Theta x^{-1}}{p^k\Omega} P_{r-k,r-k}(c,0)\big(\delta_{r-k+1}(c)-\delta_{r-k+1}(\gamma-1)\big)\\
&\hspace{2em}+\frac{\Lambda\Theta x^{-2}}{p^k\Omega} \big([P_{k,0,t-2}(c)cf(c)]_{c}^{(r-1)}-P_{k,0,t-2}(c)\delta_{r-k+1}(\gamma-1)\big),
\end{align*}
where the first equality is due to \eqref{eq: case m+2 leq k <r strong div, 1}, the second equality is due to Lemma~\ref{lemma: m+1 to r}, and the third equality is due to $v_p(\frac{\Lambda\Theta x^{-s}}{\Omega})=v_p(\Lambda\Theta x^{-1})+\frac{1}{2}v_p(\Theta\Omega x^{-2})-\frac{1}{2}\big(v_p(\Theta)+3v_p(\Omega)\big)+v_p(x^{-(s-2)})\geq (k-1)-\frac{1}{2}(t-1)+v_p(x^{-(s-2)})=k+\frac{1}{2}(t-3)+v_p(x^{-(s-2)})>k$ if $s\geq 3$. Furthermore, we see that $v_p(\frac{\Lambda\Theta x^{-2}}{\Omega})>k$ unless $(r,k)=(2m+1,m+2)$. Hence, we have
\begin{align*}
[\phi'_r(\barF'_2)]_{\barE'_2}&=\frac{\Lambda\Theta x^{-1}}{p^k\Omega} P_{r-k,r-k}(c,0)\big(\delta_{r-k+1}(c)-\delta_{r-k+1}(\gamma-1)\big) + \frac{\Lambda\Theta x^{-2}}{p^k\Omega}P_{r-k,r-k}(c,0)b_{r,k}(-1)\\
&=\frac{\Lambda\Theta x^{-1}}{p^{k-1}\Omega} P_{r-k,r-k}(-1,0)\dot{\delta}_{r-k+1}(-1)u^{p-1} + \frac{\Lambda\Theta x^{-2}}{p^k\Omega} P_{r-k,r-k}(-1,0)b_{r,k}(-1)
\end{align*}
where the second equality is due to $c-(\gamma-1)\equiv pu^{p-1}\pmod{p^2}$. As we already know that $v_p(\frac{\Lambda\Theta x^{-2}}{\Omega})\geq k$ if $(r,k)=(2m+1,m+2)$ and as $v_p(\frac{\Lambda\Theta x^{-1}}{p^{k-1}\Omega})=v_p(\frac{\Lambda\Theta x^{-1}}{p^{k-1}})-v_p(\Omega)\geq 0$, we see that $[\phi'_r(\barF'_2)]_{\barE'_2}\in \barS'_\F$.

We now consider $\phi'_r(\barF'_1)$. By Lemma~\ref{lemma: m+1 to r} and Lemma~\ref{lemma-delta-denom-1}~(i) we have
\begin{equation}\label{eq: case m+2 leq k leq r-1 strong div, 2}
P_{k,k}(c,x)=\frac{\tilde{P}_{k,k}(c,x)}{d_{k}(x)}\equiv-xP_{r-k,0}(c,0)\pmod{x\fm S'_\cO}.
\end{equation}
Hence, from \eqref{eq: formula for phi_r(F_k,k)} we have
\begin{align*}
[\phi'_r(\barF'_1)]_{\barE'_1}=\frac{\Lambda }{p^{r-k-1}}P_{k,k}(c,x)=-\frac{\Lambda x}{p^{r-k-1}}P_{r-k,0}(-1,0)
\end{align*}
and so by Lemma~\ref{lemma-delta-denom-1} we have $[\phi'_r(\barF'_1)]_{\barE'_1}\in\F$ as $v_p(\Lambda x)=v_p(\Lambda\Theta x^{-1})-v_p(\Theta\Omega x^{-2})+v_p(\Omega)\geq (k-1)-(-r+2k-1)+(-1)=r-k-1$.

It is easy to see that $[\phi'_r(\barF'_1)]_{\barE'_1}\cdot[\phi'_r(\barF'_2)]_{\barE'_2}\in\F$ by observing the valuation of the coefficients, so that we should have $[\phi'_r(\barF'_1)]_{\barE_2'}\in (\barS_\F')^\times$ to satisfy \eqref{eq: strong divisibility, local}.
From \eqref{eq: formula for phi_r(F_k,k)} together with \eqref{eq: case m+2 leq k leq r-1 strong div, 2}, we have
\begin{align*}
[\phi'_r(\barF'_1)]_{\barE'_2}&=-\frac{\Lambda x}{p^{r-k}\Omega} \big([P_{r-k,0}(c,0)cf(c)]_c^{(r-1)}-P_{r-k,0}(c,0)\delta_{r+1-k}(\gamma-1)\big)\\
&=-\frac{\Lambda x}{p^{r-k}\Omega} P_{r-k,0}(-1,0)\big(\delta_{r-k}(-1)-\delta_{r+1-k}(-1)\big),
\end{align*}
and so if $v_p(\Lambda\Theta x^{-1})=k-1$ and $v_p(\Theta\Omega x^{-2})=-r+2k-1$ then $v_p(\frac{\Lambda x}{\Omega})=(k-1)-(-r+2k-1)=r-k$ and so $[\phi'_r(\barF'_1)]_{\barE'_2}\in\F^{\times}$, by Lemma~\ref{lemma-delta}~(i) and Lemma~\ref{lemma-delta-denom-1}~(i).

Assume now that $k'=k+\frac{1}{2}$, and let $\Delta(T)=\delta_{r-k}(T)$. Then by \eqref{eq: case m+2 leq k leq r-1 strong div, 2} together with Lemma~\ref{lemma-delta-denom-1} we have $$[\phi'_r(\barF'_1)]_{\barE'_1}=-\frac{\Lambda x}{p^{r-k-1}}P_{r-k,0}(-1,0)\in\F^\times$$ as $v_p(\Lambda x)=r-k-1$. Moreover, by \eqref{eq: case m+2 leq k leq r-1 strong div, 2} we have
\begin{align*}
[\phi'_r(\barF'_1)]_{\barE'_2}&=-\frac{\Lambda x}{p^{r-k}\Omega} \big([P_{r-k,0}(c,0)cf(c)]_c^{(r-1)}-P_{r-k,0}(c,0)\delta_{r-k}(\gamma-1)\big)\\
&=-\frac{\Lambda x}{p^{r-k}\Omega} P_{r-k,0}(-1,0)\big(\delta_{r-k}(c)-\delta_{r-k}(\gamma-1)\big)\\
&=-\frac{\Lambda x}{p^{r-k-1}\Omega} P_{r-k,0}(-1,0)\dot{\delta}_{r-k}(-1)u^{p-1}
\end{align*}
where the last equality is due to $c-(\gamma-1)\equiv pu^{p-1}\pmod{p^2}$. Hence, we have $[\phi'_r(\barF'_1)]_{\barE'_2}\in\barS'_\F$ as $v_p(\frac{\Lambda x}{\Omega})=v_p(\Lambda x)-v_p(\Omega)\geq r-k-1$.

To satisfy \eqref{eq: strong divisibility, local}, we should have $[\phi'_r(\barF'_2)]_{\barE_2'}\in (\barS_\F')^\times$ as $[\phi'_r(\barF'_1)]_{\barE_2'}\in u^{p-1}\barS_\F'$. However, one can readily see that from \eqref{eq: formula for phi_r(F_k,0)}
\begin{align*}
[\phi'_r(\barF'_2)]_{\barE'_2}& =\frac{\Lambda\Theta x^{-1}}{p^{k}\Omega} \big([P_{r-k,r-k}(c,0)cf(c)]_c^{(r-1)}-P_{r-k,r-k}(c,0)\delta_{r-k}(\gamma-1)\big)\\
&=\frac{\Lambda\Theta x^{-1}}{p^{k}\Omega} P_{r-k,r-k}(-1,0)\big(\delta_{r-k+1}(-1)-\delta_{r-k}(-1)\big)
\end{align*}
where the first equality is due to \eqref{eq: case m+2 leq k <r strong div, 1} and $v_p(x)<0$ together with Lemma~\ref{lemma: m+1 to r} and the second equality is due to Lemma~\ref{lemma: PQdegree}. Hence, if $v_p(\Lambda x)=r-k-1$ and $v_p(\Theta\Omega^{-1}x^{-2})=-r+2k+1$ then we have $v_p(\frac{\Lambda\Theta x^{-1}}{\Omega})=v_p(\Lambda x)+v_p(\Theta\Omega^{-1}x^{-2})=k$ and so $[\phi'_r(\barF'_2)]_{\barE'_2}\in\F^{\times}$. Moreover, by \eqref{eq: case m+2 leq k <r strong div, 1} we have
\begin{align*}
[\phi'_r(\barF'_2)]_{\barE'_1}= \frac{ \Lambda \Theta}{p^{k-1}}\frac{\tilde{P}_{k,0}(c,x)}{d_{k}(x)} =\frac{ \Lambda \Theta x^{-1}}{p^{k-1}}P_{r-k,r-k}(-1,0)\in\F
\end{align*}
as $v_p(\Lambda\Theta x^{-1})=v_p(\frac{\Lambda\Theta x^{-1}}{\Omega})+v_p(\Omega)\geq k-1$. This completes the proof.
\end{proof}

\subsubsection{\textbf{In $\mathbf{Case}~(r)$}}
In this section of paragraph, we describe the conditions for $(\Lambda,\Theta,\Omega,x)$ to be good for $\mathbf{Case}_\phi(k')$ when $k'=r$. We also determine the corresponding $\Delta$ and describe the image of $\phi'_r$ in $\cM'$ in this case.

Recall that $b_{r,k}$ has been defined in \eqref{eq: definition of b_r,k}.

\begin{prop}\label{prop-cM-Case-r}
Let $k=r\geq 2$, $k'\in\{r,r+\frac{1}{2}\}$, and assume that $v_p(\Theta)\leq 0$.

If the quadruple $(\Lambda,\Theta,\Omega,x)$ satisfies
\begin{enumerate}[leftmargin=*]
      \item $\begin{cases}
               v_p(\Lambda \Theta x^{-1})=r-1\mbox{ and } v_p(\Theta\Omega x^{-2})=r-1 & \mbox{if } k'=r;\\
               v_p(\Lambda)=-1\mbox{ and } v_p(\Theta\Omega^{-1})=r+1 & \mbox{if } k'=r+\frac{1}{2},
             \end{cases}$
      \item $\begin{cases}
               v_p(x)\leq v_p(\Theta)\leq v_p(x)+r\mbox{ and } v_p(\Omega)\leq 0 & \mbox{if } k'=r;\\
               v_p(\Theta)\geq v_p(x)+r & \mbox{if } k'=r+\frac{1}{2},
             \end{cases}$
\end{enumerate}
then the quadruple is good for $\mathbf{Case}_\phi(k')$ with respect to
    \begin{equation*}
        \Delta(T)=
        \begin{cases}
            \delta_1(T) & \mbox{if } k'=r\geq 3;\\
            \delta_2(T,x) & \mbox{if } k'=r=2;\\
            \delta_0(T,x) & \mbox{if } k'=r+\frac{1}{2},
        \end{cases}
    \end{equation*}
in which case $\Fil^r_{r}\cM'=\barS'_\F(\barF'_1,\barF'_2)$ where
    $$\barF'_1:=\pi'(v^rE'_1)\quad\mbox{ and }\quad\barF'_2:=\pi'(F'_{r,0}),$$
and the map $\phi'_r:\Fil^r_{r}\cM'\to \cM'$ is described as follows:
\begin{itemize}[leftmargin=*]
\item if $k'=r$ then
    \begin{align*}
    \phi'_r(\barF'_1)
        &=(-1)^r p\Lambda\barE'_1+(-1)^{r+1}\frac{\Lambda x}{\Omega}\barE'_2;\\
    \phi'_r(\barF'_2)
        &=\frac{\Lambda \Theta }{p^{r-1}x}\barE'_1+ \frac{\Lambda\Theta}{p^{r-1}\Omega x} \Big(\dot{\delta}(-1)u^{p-1}-\frac{b_{r,r}(-1)}{px}\Big)\barE'_2,
    \end{align*}
\item if $k'=r+\frac{1}{2}$ then
    \begin{align*}
    \phi'_r(\barF'_1)
        &=(-1)^rp\Lambda\barE'_1;\\
    \phi'_r(\barF'_2)
        &=\frac{\Lambda \Theta }{p^{r-1}x}\barE'_1+ \frac{\Lambda\Theta}{p^r\Omega} \barE'_2.
    \end{align*}
\end{itemize}
\end{prop}

Note that from Proposition~\ref{prop-cM-Case-r} one can readily observe the following:
\begin{itemize}[leftmargin=*]
\item $[\phi'_r(\barF'_1)]_{\barE_1'}\cdot [\phi'_r(\barF'_2)]_{\barE_2'}=0$ if $k'=r$;
\item the second inequality in (ii) together with $v_p(\Theta)\leq 0$ implies $v_p(x)<0$;
\item $b_{r,r}\neq 0$ if and only if $r=3$, by definition.
\end{itemize}

\begin{proof}
We start the proof by noting that the inequalities in (ii) implies the ones in \eqref{eq: condition for k=r} (including $v_p(x)<0$) under the equations in (i), and that the equations in (i) and the inequalities in (ii) imply the inequalities in \eqref{eq: condition for phi,N stable}, including $v_p(\Theta_0)\leq 0$ if $k'=r$.

Assume first that $k'=r>1$, and let $\Delta(T)=\delta_1(T)$ if $r\geq 3$ and $\Delta(T)=\delta_2(T,x)$ if $r=2$. By Lemma~\ref{lemma: k=0,k=r}, we have $P_{r,0}(c,x)\equiv x^{-1}\pmod{x^{-1}\fm S'_\cO}$ as $v_p(x)<0$, and so if $v_p(\Lambda\Theta x^{-1})=r-1$ then
$[\phi'_r(\barF'_{2})]_{\barE'_1}=\frac{\Lambda \Theta x^{-1}}{p^{r-1}}\in\F^\times$.

To compute $[\phi'_r(\barF'_{2})]_{\barE'_2}$ we separate it into two cases: $r=2$ and $r\geq 3$. If $r=2$ then from \eqref{eq: formula for phi_r(F_k,0)} we have
\begin{align*}
[\phi'_r(\barF'_{2})]_{\barE'_2}&= \frac{\Lambda\Theta}{p^r\Omega} P_{2,0}(c,x)\big(\delta_2(c,x)-\delta_2(\gamma-1,x)\big)\\
&= \frac{\Lambda\Theta x^{-1}}{p^r\Omega} \big(\delta_2(c,x)-\delta_2(\gamma-1,x)\big)\\
&= \frac{\Lambda\Theta x^{-1}}{p^{r-1}\Omega} \dot{\delta}_2(-1,x)u^{p-1}\\
&= \frac{\Lambda\Theta x^{-1}}{p^{r-1}\Omega} \dot{\delta}_1(-1)u^{p-1},
\end{align*}
where the second equality is due to Lemma~\ref{lemma: k=0,k=r} and $v_p(x)<0$, the third equality is due to $c-(\gamma-1)\equiv pu^{p-1}\pmod{p^2}$, and the last equality is due to Lemma~\ref{lemma-delta-derivative}~(iv).  Hence, $[\phi'_r(\barF'_{2})]_{\barE'_2}\in\barS'_\F$ as $v_p(\frac{\Lambda\Theta x^{-1}}{p^{r-1}\Omega})=v_p(\frac{\Lambda\Theta x^{-1}}{p^{r-1}})-v_p(\Omega)\geq 0$.
If $r\geq 3$ then from \eqref{eq: formula for phi_r(F_k,0)} together with $v_p(x)<0$ and $P_{r,0,s}(T)=[(-Tf_r(T))^{r-1-s}]_T^{(r-1)}$ we have
\begin{align*}
[\phi'_r(\barF'_{2})]_{\barE'_2}&=\frac{\Lambda\Theta}{p^r\Omega} x^{-r}\sum_{s=0}^{r-1}x^s\big([P_{r,0,s}(c)cf(c)]_c^{(r-1)}-P_{r,0,s}(c)\delta_1(\gamma-1)\big)\\
&=\frac{\Lambda\Theta}{p^r\Omega} x^{-r}\sum_{s=0}^{r-1}x^s\big(-P_{r,0,s-1}(c)-P_{r,0,s}(c)\delta_1(\gamma-1)\big),
\end{align*}
where the first equality is due to Lemma~\ref{lemma: k=0,k=r} and the second equality is due to the identity $[P_{r,0,s}(c)cf(c)]_{c}^{(r-1)}=-P_{r,0,s-1}(c)$. As $v_p(\frac{\Lambda\Theta x^{-s}}{\Omega})
=v_p(\Lambda\Theta x^{-1})+\frac{1}{2}v_p(\Theta\Omega x^{-2})-\frac{1}{2}\big(v_p(\Theta)+3v_p(\Omega)\big)+v_p(x^{-(s-2)})
\geq\frac{3}{2}(r-1)+v_p(x^{-(s-2)})=r+\frac{1}{2}(r-3)+v_p(x^{-(s-2)})>r$ if $s\geq 3$, we further have

\begin{align*}
[\phi'_r(\barF'_{2})]_{\barE'_2}&=\frac{\Lambda\Theta}{p^r\Omega} \left[x^{-1}\big(\delta_1(c)-\delta_1(\gamma-1)\big)+x^{-2}(-P_{r,0,r-3}(c)+\delta_1(c)\delta_1(\gamma-1))\right]\\
&=\frac{\Lambda\Theta}{p^r\Omega} \big[px^{-1}\dot{\delta}_1(-1)u^{p-1}-x^{-2}\big(P_{r,0,r-3}(-1)-\delta_1(-1)^2\big)\big],
\end{align*}
where the first equality is due to $P_{r,0,r-2}(c)=-\delta_1(c)$ and $P_{r,0,r-1}(c)=1$ and the second equality is due to $c-(\gamma-1)\equiv pu^{p-1}\pmod{p^2}$.
Hence, $[\phi'_r(\barF'_{2})]_{\barE'_2}\in\barS'_\F$ as 
$v_p(\frac{\Lambda\Theta x^{-1}}{p^{r-1}\Omega})=v_p(\frac{\Lambda\Theta x^{-1}}{p^{r-1}})-v_p(\Omega)\geq 0$ and 
$v_p(\frac{\Lambda\Theta x^{-2}}{\Omega})\geq r+\frac{1}{2}(r-3)\geq r$. Moreover,
$$\frac{\Lambda\Theta}{p^r\Omega}x^{-2}\big(P_{r,0,r-3}(-1)-\delta_1(-1)^2\big)=0$$
if $r>3$, as $v_p(\frac{\Lambda\Theta x^{-2}}{\Omega})>r$ if $r>3$, and it is easy to see that $b_{r,r}(-1)=P_{r,0,r-3}(-1)-\delta_1(-1)^2$ if $r=3$.

We now compute $\phi'_r(\barF'_1)$. From \eqref{eq: definition of phi'}, we have
\begin{align*}
\phi'_r(\barF'_1)=c^r\left(p\Lambda E'_1-\frac{\Lambda}{\Omega}\big(x+\delta(\gamma-1)\big)E'_2\right)=(-1)^rp\Lambda E'_1-(-1)^r\frac{\Lambda x}{\Omega}E'_2,
\end{align*}
as $v_p(\frac{\Lambda x}{\Omega})=v_p(\Lambda\Theta x^{-1})-v_p(\Omega\Theta x^{-2})=0$ and $v_p(\frac{\Lambda}{\Omega})=-v_p(x)>0$. It is easy to see that $[\phi'_r(\barF'_1)]_{\barE'_1}\cdot [\phi'_r(\barF'_2)]_{\barE'_2}=0$ by observing the valuation of the coefficient.
Hence, to satisfy \eqref{eq: strong divisibility, local}, we should have $[\phi'_r(\barF'_1)]_{\barE_2'}\in (\barS_\F')^\times$, and it holds as $v_p(\Lambda\Omega^{-1}x)=0$. Moreover, we have $[\phi'_r(\barF'_1)]_{\barE_1'}\in \F$ as $v_p(\Lambda)=v_p(\Lambda\Theta x^{-1})-v_p(\Theta x^{-1})\geq (r-1) -r=-1$.

Assume now that $k'=r+\frac{1}{2}$, and let $\Delta(T)=\delta_0(T,x)$. Then by \eqref{eq: definition of phi'} we have $$[\phi'_r(\barF'_1)]_{\barE'_1}=(-1)^rp\Lambda\in\F^\times$$ as $v_p(\Lambda)=-1$. Moreover, by \eqref{eq: definition of phi'} we have
\begin{align*}
[\phi'_r(\barF'_1)]_{\barE'_2}&=-(-1)^r\frac{\Lambda}{\Omega} \big(x+\delta_0(\gamma-1,x)\big)\\
&=-(-1)^r\frac{\Lambda}{\Omega}\big(x+(-x)\big)\\
&=0
\end{align*}
Hence, we have $[\phi'_r(\barF'_1)]_{\barE'_2}\in\barS'_\F$.

To satisfy \eqref{eq: strong divisibility, local}, we should have $[\phi'_r(\barF'_2)]_{\barE_2'}\in (\barS_\F')^\times$ as $[\phi'_r(\barF'_1)]_{\barE_2'}\in \barS_\F'$. However, one can readily see that from \eqref{eq: formula for phi_r(F_k,0)}
\begin{align*}
[\phi'_r(\barF'_2)]_{\barE'_2}& =\frac{\Lambda\Theta x^{-1}}{p^{r}\Omega}\big([P_{0,0}(c,0)cf(c)]_c^{(r-1)}-P_{0,0}(c,0)\delta_0(\gamma-1,x)\big)\\
&=\frac{\Lambda\Theta x^{-1}}{p^{r}\Omega}(-x)\\
&=-\frac{\Lambda\Theta}{p^r\Omega}
\end{align*}
where the first equality is due to Lemma~\ref{lemma: k=0,k=r} and $v_p(x)<0$, and the second equality is due to $P_{0,0}(c,0)=1$ and $v_p(x)<0$. Hence, if $v_p(\Lambda)=-1$ and $v_p(\Theta\Omega^{-1})=r+1$ then we have $v_p(\frac{\Lambda\Theta}{\Omega})=v_p(\Lambda)+v_p(\Theta\Omega^{-1})=r$ and so $[\phi'_r(\barF'_2)]_{\barE'_2}\in\F^{\times}$. Moreover, we have $$[\phi'_r(\barF'_2)]_{\barE'_1}= \frac{\Lambda \Theta x^{-1}}{p^{r-1}}\in\F$$ as $v_p(\Lambda\Theta x^{-1})=-1+v_p(\Theta x^{-1})\geq r-1$. This completes the proof.
\end{proof}

\subsection{The exceptional cases}\label{subsec: exceptional cases}
In this subsection, we prove Theorem~\ref{theo: filtration of pseudo} and Theorem~\ref{theo: strong divisibility of pseudo} for two exceptional cases: $r'=1$ and $\mathbf{Case}_\phi(\infty)$.

\subsubsection{\textbf{When $r'=1$}}
In this section of paragraph, we describe the filtration and the conditions for the strong divisibility when $r'=1$. Recall that we set $\delta_1(T,x)=0$ and $\delta_0(T,x)=-x$ if $r'=1$. Without loss of generality, we may assume $r=r'$, and we do so.

Set $\Fil^1_{k}\fM':=\Fil^{1;1}_{k}\fM'$ for brevity. It is clear from \eqref{eq: definition of FilfM'} that
\begin{equation}\label{eq: condition for filtration, r=1}
\begin{cases}
F'_{0,0}\in\fM'\,\,\mbox{ if and only if }\,\,v_p(\Theta)\leq v_p(x);\\
F'_{1,0}\in\fM'\,\,\mbox{ if and only if }\,\,v_p(\Theta)\geq v_p(x),
\end{cases}
\end{equation}
in which case we have
\begin{equation}\label{eq: generators of the filtration, r=1}
\Fil^{1}_{k}\fM'=
\begin{cases}
S'_\cO(F'_{0,0},vE'_2)+\Fil^pS'_\cO\cdot\fM'&\mbox{ if }k=0;\\
S'_\cO(vE'_1,F'_{1,0})+\Fil^pS'_\cO\cdot\fM'&\mbox{ if }k=1.
\end{cases}
\end{equation}

We first treat the case $\mathbf{Case}~(0)$.
\begin{prop}\label{prop: case phi 0.5, r=1}
Let $k'=\frac{1}{2}$. If the quadruple $(\Lambda,\Theta,\Omega,x)$ satisfies
    \begin{enumerate}[leftmargin=*]
    \item $v_p(\Lambda)=0$ and $v_p(\Theta\Omega^{-1})=0$;
    \item $v_p(\Theta)\leq v_p(x)$ and $v_p(\Omega)\leq 0$,
    \end{enumerate}
then the quadruple is good for $\mathbf{Case}_\phi(\frac{1}{2})$ with respect to $\Delta(T)=\delta_1(T)$, in which case $\Fil^1_{0}\cM'=\barS'_\F(\barF'_1,\barF'_2)$ with
    \begin{align*}
    \barF'_1:=\pi'(F'_{0,0})\quad\mbox{ and }\quad\barF'_2:=\pi'(vE'_2)
    \end{align*}
and the map $\phi'_1:\Fil^1_{0}\cM'\to\cM'$ is described as
    \begin{align*}
    \phi'_1(\barF'_1)=\Lambda\barE'_1\quad\mbox{ and }\quad \phi'_1(\barF'_2)= -\frac{\Lambda\Theta}{\Omega}\barE'_2.
    \end{align*}
\end{prop}

\begin{proof}
We start the proof by noting that the equations in (i) and the inequalities in (ii) imply the inequalities in \eqref{eq: condition for phi,N stable}, and that the first inequality in (ii) implies the one in \eqref{eq: condition for filtration, r=1}.

By \eqref{eq: definition of phi'}, it is tedious to see that
$$\phi'_1(F'_{0,0})\equiv\frac{1}{p}\bigg(\bigg(p\Lambda E'_1-\frac{\Lambda x}{\Omega} E'_2\bigg)+\frac{x}{\Theta}\bigg(\frac{\Lambda\Theta}{\Omega} E'_2\bigg)\bigg)= \Lambda E'_1 \pmod{\fm\fM'},
$$
as $\delta=0$ and \eqref{eq: p^p-1 divides phi(gamma)}. Hence, if $v_p(\Lambda)=0$ then $[\phi'_1(\barF'_1)]_{\barE'_1}\in\F^{\times}$.

As $[\phi'_1(\barF'_1)]_{\barE'_2}=0$, we should have $[\phi'_1(\barF'_2)]_{\barE'_2}\in(\barS'_\F)^\times$. But by \eqref{eq: definition of phi'} $\phi'_1(\barF'_2)=\frac{\Lambda\Theta}{\Omega}\barE'_2$ and so if $v_p(\frac{\Lambda\Theta}{\Omega})=v_p(\Lambda)+v_p(\Theta\Omega^{-1})=0$ then we have $[\phi'_1(\barF'_2)]_{\barE'_2}\in\F^{\times}$, which completes the proof.
\end{proof}

We now treat the case $\mathbf{Case}~(1)$.
\begin{prop}\label{prop: case phi 1, r=1}
Let $k'\in\{1,\frac{3}{2}\}$, and assume that $v_p(\Theta)\leq 0$. If the quadruple $(\Lambda,\Theta,\Omega,x)$ satisfies
\begin{enumerate}[leftmargin=*]
    \item $\begin{cases}
             v_p(\Lambda \Theta x^{-1})=0\mbox{ and } v_p(\Theta\Omega x^{-2})=0 & \mbox{if } k'=1; \\
             v_p(\Lambda)=-1\mbox{ and } v_p(\Theta\Omega^{-1})=2 & \mbox{if } k'=\frac{3}{2},
           \end{cases}$
    \item $\begin{cases}
             v_p(x)\leq v_p(\Theta)\leq v_p(x)+1 & \mbox{if } k'=1; \\
             v_p(\Theta)\geq v_p(x)+1 & \mbox{if } k'=\frac{3}{2},
           \end{cases}$
\end{enumerate}
then the quadruple is good for $\mathbf{Case}_\phi(k')$ with respect to
$$\Delta(T)=
        \begin{cases}
            \delta_1(T) & \mbox{if } k'=1; \\
            \delta_0(T,x) & \mbox{if } k'=\frac{3}{2},
        \end{cases}$$
in which case $\Fil^1_{1}\cM'=\barS'_\F(\barF'_1,\barF'_2)$ with
    \begin{align*}
    \barF'_1:=\pi'(vE'_1)\quad\mbox{ and }\quad\barF'_2:=\pi'(F'_{1,0})
    \end{align*}
    and the map $\phi'_1:\Fil^1_{1}\cM'\to\cM'$ is described as follows:
\begin{itemize}[leftmargin=*]
\item if $k'=1$ then
    \begin{align*}
    \phi'_1(\barF'_1)=-p\Lambda\barE'_1+\frac{\Lambda x}{\Omega}\barE'_2\quad\mbox{ and }\quad\phi'_1(\barF'_2)=\frac{\Lambda\Theta}{x}\barE'_1,
    \end{align*}
\item if $k'=\frac{3}{2}$ then
    \begin{align*}
    \phi'_1(\barF'_1)=-p\Lambda\barE'_1\quad\mbox{ and }\quad\phi'_1(\barF'_2)=\frac{\Lambda\Theta}{x}\barE'_1+\frac{\Lambda\Theta}{p\Omega}\barE'_2.
    \end{align*}
\end{itemize}
\end{prop}

\begin{proof}
Assume first that $k'=1$. We start the proof by noting that the equations in (i) and the inequalities in (ii) imply the inequalities in \eqref{eq: condition for phi,N stable}, and that the first inequality in (ii) implies the one in \eqref{eq: condition for filtration, r=1}.

By \eqref{eq: definition of phi'} it is tedious that
$$\phi'_1(F'_{1,0})\equiv \frac{1}{p}\bigg(\frac{\Theta}{x}\bigg(p\Lambda E'_1-\frac{\Lambda x}{\Omega}E'_2\bigg)-\bigg(\frac{\Lambda\Theta}{\Omega} E'_2\bigg)\bigg)\equiv \frac{\Lambda \Theta}{x} E'_1 \pmod{\fm\fM'},
$$
as $\Delta=0$ and \eqref{eq: p^p-1 divides phi(gamma)}. Hence, if $v_p(\Lambda\Theta x^{-1})=0$ then $[\phi'_1(\barF'_2)]_{\barE'_1}\in\F^{\times}$.

As $[\phi'_1(\barF'_2)]_{\barE'_2}=0$, we should have $[\phi'_1(\barF'_1)]_{\barE'_2}\in(\barS'_\F)^\times$. But by \eqref{eq: definition of phi'}
$$\phi'_1(vE'_1)\equiv c\bigg(p\Lambda E'_1-\frac{\Lambda x}{\Omega}E'_2\bigg)  \pmod{\fm\fM'},
$$
as $\Delta=0$ and \eqref{eq: p^p-1 divides phi(gamma)}, and so if $v_p(\Lambda\Theta x^{-1})=v_p(\Theta\Omega x^{-2})=0$ then we have 
$v_p(\frac{\Lambda x}{\Omega})=v_p(\Lambda\Theta x^{-1})-v_p(\Theta\Omega x^{-2})=0$ and so $[\phi'_1(\barF'_1)]_{\barE'_2}\in\F^{\times}$. Moreover, we also have $[\phi'_1(\barF'_1)]_{\barE'_1}\in\F$ as $v_p(\Lambda)\geq -1$.

Now assume $k'=\frac{3}{2}$. We note that the equations in (i), the inequalities in (ii) and the fact that $\Delta=\delta_0=-x$ imply the inequalities in \eqref{eq: condition for phi,N stable}, and that the inequality in (ii) implies the one in \eqref{eq: condition for filtration, r=1}.

By \eqref{eq: definition of phi'}
$$\phi'_1(vE'_1)\equiv cp\Lambda E'_1  \pmod{\fm\fM'},
$$
as $\Delta=-x$ and \eqref{eq: p^p-1 divides phi(gamma)}. Hence, if $v_p(\Lambda)=-1$ then $[\phi'_1(\barF'_2)]_{\barE'_1}\in\F^{\times}$.

As $[\phi'_1(\barF'_1)]_{\barE'_2}=0$, we should have $[\phi'_1(\barF'_2)]_{\barE'_2}\in(\barS'_\F)^\times$. But by \eqref{eq: definition of phi'} it is tedious that
$$\phi'_1(F'_{1,0})\equiv \frac{1}{p}\bigg(\frac{\Theta}{x}(p\Lambda E'_1)-\bigg(\frac{\Lambda\Theta}{\Omega} E'_2\bigg)\bigg)\equiv \frac{\Lambda \Theta}{x} E'_1-\frac{\Lambda\Theta}{p\Omega} E'_2 \pmod{\fm\fM'},
$$
as $\Delta=-x$ and \eqref{eq: p^p-1 divides phi(gamma)}, and so if $v_p(\Lambda)=-1$ and $v_p(\Theta\Omega^{-1})=2$ then we have $v_p(\frac{\Lambda\Theta}{\Omega})=v_p(\Lambda)+v_p(\Theta\Omega^{-1})=1$ and so $[\phi'_1(\barF'_2)]_{\barE'_2}\in\F^{\times}$. Moreover, we also have $[\phi'_1(\barF'_2)]_{\barE'_1}\in\F$ as $v_p(\Lambda \Theta x^{-1})=v_p(\Theta x^{-1})-1\geq 0$. This completes the proof.
\end{proof}

\subsubsection{\textbf{In $\mathbf{Case}_\phi(\infty)$}}
In this section of paragraph, we describe the filtration and the conditions for the strong divisibility when $k'=\infty$ and $1\leq r'\leq r <p-1$. Recall that we set $\delta_0(T,x)=-x$. Without loss of generality, we may assume $r=r$, and we do so.

Set $\Fil^r_\infty\fM':=\Fil^{r;r'}_\infty\fM'$ for brevity. It is clear from \eqref{eq: definition of FilfM'} that
\begin{align}\label{eq: generators of the filtration, infty}
\begin{split}
\Fil^r_\infty\fM'&=S'_\cO\big(E'_2\big)+\Fil^rS'_\cO\cdot\fM'\\
&=S'_\cO\big(v^rE'_1,E'_2\big)+\Fil^pS'_\cO\cdot\fM'.
\end{split}
\end{align}

\begin{prop}\label{prop: case infty}
Let $k'=\infty$. If the quadruple $(\Lambda,\Theta,\Omega,x)$ satisfies
    \begin{enumerate}[leftmargin=*]
    \item $v_p(\Lambda)=-1$ and $v_p(\Theta\Omega^{-1})=r+1$;
    \item $v_p(\Theta)\leq 0$,
    \end{enumerate}
then the quadruple is good for $\mathbf{Case}_\phi(\infty)$ with respect to $\Delta(T)=\delta_0(T,x)$, in which case
$\Fil^{r}_\infty\cM'=\barS'_\F(\barF'_1,\barF'_2)$ where
    \begin{align*}
    \barF'_1:=\pi'(v^rE'_1)\quad\mbox{ and }\quad \barF'_2:=\pi'(E'_2)
    \end{align*}
    and the map $\phi'_r:\Fil^{r}_\infty\cM'\to\cM'$ is described as
    \begin{align*}
        \phi'_r(\barF'_1)=(-1)^rp\Lambda\barE'_1 \quad\mbox{ and }\quad
        \phi'_r(\barF'_2)=\frac{\Lambda\Theta}{p^r\Omega}\barE'_2.
    \end{align*}
\end{prop}

\begin{proof}
We start the proof by noting that the equations in (i) and the inequalities in (ii) imply the inequalities in \eqref{eq: condition for phi,N stable}.

It is clear that $\phi'_r(\barF'_2)$ is computed as in the statement, by \eqref{eq: definition of phi'}. Hence, if $v_p(\Lambda)=-1$ and $v_p(\Theta\Omega^{-1})=r+1$ then we have $v_p(\frac{\Lambda\Theta}{\Omega})=v_p(\Lambda)+v_p(\Theta\Omega^{-1})=r$ and so $[\phi'_r(\barF'_2)]_{\barE_2'}\in\F^\times$. From \eqref{eq: definition of phi'}, it is also clear from \eqref{eq: p^p-1 divides phi(gamma)} that we have
$$\phi'(E'_1)\equiv p\Lambda E'_1\pmod{\fm\fM'}$$
as $v_p(\frac{\Lambda}{\Omega})\geq r-v_p(\Theta)\geq r$ and $\Delta=-x$.
Hence, if $v_p(\Lambda)=-1$ then $\phi'_r(\barF'_1)=(-1)^rp\Lambda\barE'_1\in\F^\times$. This completes the proof.
\end{proof}

\smallskip
\section{Main results}\label{sec: main results}
In this section, we state and prove the main results on Galois stable lattices. We also summarize the mod-$p$ reduction of the pseudo-strongly divisible modules.

\subsection{Typical examples of pseudo-strongly divisible modules}\label{subsec: set up}
In this subsection, we show that certain choice of $S_\cO$-generators in $\cD:=\cD(D)$, where $D$ is an admissible filtered $(\phi,N)$-module of rank $2$ in Example~\ref{exam: admissible filtered modules}, gives rise to pseudo-strongly divisible modules.

Throughout this subsection, we keep the notation of Example~\ref{exam: admissible filtered modules}, and, in particular, let $D=D(\lambda,\vL)$ where $\lambda\in\cO$ and $\vL:=(\fL_j)_{j\in\cJ}\in[\mathbb{P}^1(E)]^f$. We also set $\cD:=\cD(D)$. Assume that $0<r_j<p-1$ for all $j\in\cJ$, and fix $r$ with $p-1 > r\geq r_j$ for all $j\in\cJ$.

We let $E_1,E_2$ be the elements in $\cD=\oplus_{j\in\cJ}\cD^{(j)}$ of the following form
\begin{align}\label{eq: form of generators of sdm}
\begin{split}
    E_1&=\vth_1\bigg(\eta_1+\frac{\varphi(\vrho)+\vDel^\varphi(\gamma-1)}{p}\eta_2\bigg);\\
    E_2&=\vth_2\eta_2
\end{split}
\end{align}
for some $\vth_1,\vth_2,\vrho \in (E^\times)^f$, and $\vDel(T)\in (E(T)_\Delta)^f$. We write
$$\vrho=(\varrho_0,\cdots,\varrho_{f-1})\quad\mbox{ and }\quad\vDel(T)=(\Delta^{(0)}(T),\cdots,\Delta^{(f-1)}(T)).$$ 
Similarly, we write $\vth_1=(\theta_{1,0},\theta_{1,1},\cdots,\theta_{1,f-1})$ and $\vth_2=(\theta_{2,0},\theta_{2,1},\cdots,\theta_{2,f-1})$. 
We further require that $\vrho$ satisfies
\begin{equation}\label{eq: definition of varrho}
    \varrho_{j}=\fL_j \mbox{ for } j\not\in I_D.
\end{equation}

\begin{rema}\label{rema: why fL_j=varho_j}
If $j\in I_D$, i.e., $\fL_j=\infty$, then we no condition is imposed on $\varrho_j$, i.e., we can choose any $\varrho_j\in E$. By Lemma~\ref{lemm: computation of phi^diamond N^diamond} and Lemma~\ref{lemm: filtration of fM^diamond}, $\varrho_j$ only appears at $\phi^{\diamond}(E_1)e_{j+1}$ (see \eqref{eq: computation of phi}), if $j\in I_D$. By the identities \eqref{eq: connection between x and L} and \eqref{eq: definition of vx}, $\phi^{\diamond}(E_1)$ can be written as in \eqref{eq: phi(E1) and phi(E2)}. But remember that our choice of $\Delta^{(j)}(T)=-x_j$ if $j\in I_D$, and so the term $x_j+\Delta^{(j)}(\gamma-1)$ of $\phi^{\diamond}(E_1)e_{j+1}$ vanishes. Hence, it is free to choose any $\varrho_j$ in $E$.
\end{rema}

This choice of the two elements in $\cD$ induces the following additional structures:
\begin{itemize}[leftmargin=*]
\item $\fM^{\diamond}:=S_\cO(E_1,E_2)$ is the $S_\cO$-submodule of $\cD$ generated by $E_1,E_2$;
\item the Frobenius map $\phi^{\diamond}:\fM^{\diamond}\rightarrow \cD$ is the restriction of $\phi$ on $\cD$;
\item the monodromy operator $N^\diamond:\fM^{\diamond}\rightarrow \cD$ is the restriction of $N$ on $\cD$;
\item $\Fil^r\fM^\diamond:=\Fil^r\cD\cap \fM^\diamond$.
\end{itemize}
The quadruple 
\begin{equation}\label{eq: definition of fM_vth,vth,vDel}
(\fM^\diamond,\Fil^r\fM^\diamond,\phi^\diamond,N^\diamond)
\end{equation}
is denoted by $\fM_{(\vth_1,\vth_2,\vrho,\vDel)}$. 

\begin{rema}
The main reason we require that $\varrho_j$ satisfies \eqref{eq: definition of varrho} is that the quantity $$\tfrac{p\fL_j-\varrho_{j-1}-\Delta^{(j-1)}(-1)}{p} \quad (\mbox{resp}.\quad \tfrac{p\varrho_j-\varrho_{j-1}-\Delta^{(j-1)}(-1)}{p})$$ naturally appears in the filtration $\Fil^r\fM^\diamond$ (resp. in the formula $\phi^\diamond(E_1)$) as you will see in \eqref{eq: computation of filtration} (resp. in \eqref{eq: computation of phi}), and we wish that these two quantities are identical since the generic fiber of the strongly divisible modules is required to be isomorphic to $\cD(D)$. It is also reasonable to expect that $\Delta^{(j)}$ in \eqref{eq: form of generators of sdm} is a rational function in $\gamma-1$, as $c:=\frac{\phi(v)}{p}$ is congruent to $\gamma-1$ modulo $p$ and $\fM$ is generated by the image of $\phi_r$. If $j\in I_D$, whatever $\varrho_j$ we choose in $E$, it turns out that the resulting strongly divisible modules are homothetic.
\end{rema}

The Frobenius map $\phi^\diamond$ and the monodromy operator $N^\diamond$ can be described explicitly.
\begin{lemm}\label{lemm: computation of phi^diamond N^diamond}
\begin{enumerate}[leftmargin=*]
\item The Frobenius map $\phi^\diamond:\fM^\diamond\rightarrow\cD$ is induced by
\begin{align}\label{eq: computation of phi}
\begin{split}
\phi^\diamond(E_1)=&
    p\frac{\lambda\varphi(\vth_1)}{\vth_1}E_1
    -\frac{\lambda\varphi(\vth_1)}{p\vth_2} \big(p\varphi(\vrho)-\varphi^2(\vrho)-\vDel^{\varphi^2}(-1)\big)E_2\\
    &\qquad-\frac{\lambda\varphi(\vth_1)}{\vth_2}\vDel^\varphi(\gamma-1)E_2  +\frac{\lambda\varphi(\vth_1)}{p\vth_2}\big(
    \vDel^{\varphi^2}(\phi(\gamma)-1)-\vDel^{\varphi^2}(-1)\big)E_2;\\
\phi^\diamond(E_2)=&\frac{\lambda\varphi(\vth_2)}{\vth_2}E_2 =\frac{\lambda\varphi(\vth_1)}{\vth_1}\varphi\left(\frac{\vth_2}{\vth_1}\right)\frac{\vth_1}{\vth_2}E_2.
\end{split}
\end{align}
\item The monodromy operator $N^\diamond:\fM^\diamond\rightarrow\cD$ is induced by
\begin{align}\label{eq: computation of N}
\begin{split}
N^\diamond(E_1) &=\vth_1\bigg(\eta_2+\frac{N(\vDel^\varphi(\gamma-1))}{p}\eta_2\bigg);\\
N^\diamond(E_2) &=0.
\end{split}
\end{align}
\end{enumerate}
\end{lemm}

\begin{proof}
The proof easily follows from tedious computation, and so we omit it.
\end{proof}

Moreover, the filtration $\Fil^r\fM^\diamond$ can be described explicitly as well.
\begin{lemm}\label{lemm: filtration of fM^diamond}
The $S_\cO$-module $\Fil^r\fM^\diamond$ is described as 
\begin{equation}\label{eq: computation of filtration}
\Fil^r\fM^\diamond =\bigoplus_{j\in\cJ}[\Fil^r\fM^{\diamond}e_j]^{(r-1)}+\Fil^rS_\cO\cdot\fM^\diamond
\end{equation}
where for each $j\in\cJ$ 
\begin{enumerate}[leftmargin=*]
\item if $\fL_j=\infty$ then
\begin{equation*}
[\Fil^r\fM^{\diamond}e_j]^{(r-1)}= \big\{v^{r-r_j}C(v)E_2e_j \mid C(T)\in\cO[T]^{(r_j-1)}\big\}
\end{equation*} 
\item if $\fL_j\neq \infty$ then
$[\Fil^r\fM^{\diamond}e_j]^{(r-1)}$ consists of the elements of the form
$$v^{r-r_j}[C(v)\tfrac{\theta_{2,j}}{\theta_{1,j}} E_1+B(v)E_2]e_j$$
for the polynomials $C(T),B(T)\in\cO[T]^{(r_j-1)}$ satisfying $C(T)\tfrac{\theta_{2,j}}{\theta_{1,j}}\in\cO[T]$ and
\begin{equation}\label{eq: equation for C and B}
B(T)=\bigg[C(T)\bigg(x_j+\frac{T}{p}f_{r_j}\Big(\frac{T}{p}\Big) \bigg)\bigg]_{T}^{(r_j-1)}
\end{equation}
where
\begin{equation}\label{eq: definition of x}
x_j:=\frac{p\varrho_j-\varrho_{j-1}-\Delta^{(j-1)}(-1)}{p}.
\end{equation}
\end{enumerate}
\end{lemm}

\begin{proof}
From Example~\ref{exam: filtration of cD}, every element of $\left(\frac{\Fil^r\cD}{\Fil^rS_E\cD}\right)e_j$ can be written as
\begin{equation}\label{eq: elements in Filr cD}
v^{r-r_j}C(v)\theta_{2,j}\Big(
    \xi_j\eta_1+\nu_j\eta_2+\xi_j\frac{v}{p}f_{r_j}\Big(\frac{v}{p}\Big)\eta_2
    \Big)e_j
\end{equation}
for some polynomial $C(T)\in E(T)$ of degree $\leq r_j-1$, and this can be rewritten as
\begin{equation}\label{eq: filtration in general}
v^{r-r_j}C(v)\bigg(
    \xi_j\frac{\theta_{2,j}}{\theta_{1,j}}E_1-\xi_j\frac{\varrho_{j-1}+\Delta^{(j-1)}(-1)}{p}E_2+\nu_jE_2
    +\xi_j\frac{v}{p}f_{r_j}\Big(\frac{v}{p}\Big)E_2
    \bigg)e_j
\end{equation}
modulo $\Fil^rS_E\cD$ by using \eqref{eq: form of generators of sdm}. Recall that we write $\fL_j\in\mathbb{P}^1(E)$ for $\frac{\nu_j}{\xi_j}$, and note that it is harmless to multiply $\vth_2$ to the quantity in \eqref{eq: elements in Filr cD}. 

Now, if $j\in I_D$, i.e., $\fL_j=\infty$ then it is immediate from \eqref{eq: filtration in general} that every element of $\left(\frac{\Fil^r\cD}{\Fil^rS_E\cD}\right)e_j$ can be written as $v^{r-r_j}C(v)E_2e_j$. If $j\not\in I_D$, i.e., $\fL_j\neq\infty$, then due to \eqref{eq: definition of x} it is also easy to see from \eqref{eq: filtration in general} that every element of $\left(\frac{\Fil^r\cD}{\Fil^rS_E\cD}\right)e_j$ can be written as $v^{r-r_j}\left[C(v)\frac{\theta_{2,j}}{\theta_{1,j}} E_1+B(v)E_2\right]e_j$ where
$B(T)$ and $C(T)$ satisfy the equation \eqref{eq: equation for C and B}. This completes the proof.
\end{proof}

By identifying 
\begin{equation}\label{eq: connection between x and L}
\vLam=\frac{\lambda\vth_1}{\varphi^{-1}(\vth_1)}\quad\mbox{ and }\quad \vTh=\frac{\vth_2}{\vth_1}
\end{equation}
as well as
\begin{equation}\label{eq: definition of vx}
\vx=\frac{p\vrho-\varphi(\vrho)-\vDel^\varphi(-1)}{p}
\end{equation}
we may consider $\fM_{(\vth_1,\vth_2,\vrho,\vDel)}$ as parameterized by $(\vLam,\vTh,\vx)\in ((E^{\times})^f)^2\times E^f$ together with $\vDel(T)\in E(T)_\Delta$. (Note that the identity in \eqref{eq: definition of vx} is induced from \eqref{eq: definition of x}.) More precisely, by Lemma~\ref{lemm: filtration of fM^diamond} $\Fil^r\fM^\diamond$ is parameterized by $(\vTh,\vx)$, and by Lemma~\ref{lemm: computation of phi^diamond N^diamond} $\phi^\diamond$ and $N^\diamond$ can be written as
\begin{align}\label{eq: phi(E1) and phi(E2)}
\begin{split}
\phi^\diamond(E_1)&=p\varphi(\vLambda)E_1
-\frac{\varphi(\vLam)}{\vTh}\big(\varphi(\vx)+\vDel^\varphi(\gamma-1)\big)E_2\\
&\qquad\qquad\qquad\qquad+\frac{\varphi(\vLam)}{p\vTh}\big(\vDel^{\varphi^2}(\phi(\gamma)-1)-\vDel^{\varphi^2}(-1)\big)E_2;\\
\phi^\diamond(E_2)&=\frac{\varphi(\vLam\vTh)}{\vTh}E_2
\end{split}
\end{align}
and
\begin{align}\label{eq: N(E1) and N(E2)}
\begin{split}
N^\diamond(E_1) &=\frac{1}{\vTh}\bigg(1+\frac{N(\gamma)}{p}\dot{\vDel}^\varphi(\gamma-1)\bigg)E_2;\\
N^\diamond(E_2) &=0.
\end{split}
\end{align}

It is clear from the first identity in \eqref{eq: connection between x and L} that
\begin{equation}\label{eq: From the setting 2}
\prod_{j=0}^{f-1}\Lambda_j=\lambda^f.
\end{equation}
Moreover, the association $(\vth_1,\vth_2)\mapsto (\vLam,\vTh)$ defined by the identities in \eqref{eq: connection between x and L} gives rise to a bijection 
\begin{equation}\label{eq: Lam,Ome,The vs the1,the2}
\big((E^\times)^f\big)^{2}/E^{\times}\longrightarrow\bigg\{(\vLam,\vTh)\in \big((E^\times)^f\big)^{2}\mid \prod_{j=0}^{f-1}\Lambda_j=\lambda^f\bigg\}
\end{equation}
where the domain $\big((E^\times)^f\big)^{2}/E^{\times}$ is the quotient of $\big((E^\times)^f\big)^{2}$ by $E^\times$ via the diagonal action.

We now show that the isotypic components of $\fM_{(\vth_1,\vth_2,\vrho,\vDel)}$ are typical examples of pseudo-strongly divisible modules. Fix $j\in\cJ$, and consider the isotypic component of $\fM_{(\vth_1,\vth_2,\vrho,\vDel)}$ at $j$
\begin{equation}\label{eq: typical example of pseudo-SDM}
(\fM^\diamond e_j, \Fil^r \fM^\diamond e_j, \phi^{\diamond,(j)}: \fM^\diamond e_j \to \cD e_{j+1}, N^{\diamond,(j)}: \fM^\diamond e_j \to \cD e_j),
\end{equation}
that is parameterized by $(\Lambda_j,\Theta_j,\Theta_{j+1},x_j)$ together with $(\Delta^{(j)},\Delta^{(j-1)})$. 

\begin{defi}
If $\fL_j=\infty$ and so $\Fil^r \fM^\diamond e_j$ can be written as (i) of Lemma~\ref{lemm: filtration of fM^diamond}, then we say that \emph{$\Fil^r \fM^\diamond e_j$ is of $\mathbf{Case}~(\infty)$}. If $\fL_j\neq\infty$ and so $\Fil^r \fM^\diamond e_j$ can be written as (ii) of Lemma~\ref{lemm: filtration of fM^diamond}, then we consider $(r_j+1)$-choices of the generators for $\mathbb{M}$ as in \eqref{eq: definition of cases} (identifying the equation in \eqref{eq: equation for C and B} with the equation in \eqref{eq: B and C} as well as $(r_j,x_j,\Theta_j)=(r',x,\Theta)$), each of which we say that \emph{$\Fil^r \fM^\diamond e_j$ is of $\mathbf{Case}~(k)$}, where $k$ varies over the integers with $0\leq k\leq r_j$.
\end{defi}
Consider the natural isomorphisms of the $S_E'$-modules for all $j\in\cJ$
$$\psi^{(j)}:\cD'\rightarrow \cD e_j:\,E_1'\mapsto E_1 e_j,\,\,E'_2\mapsto E_2e_j.$$

\begin{prop}\label{prop: typical example of pseudo-SDM}
Fix $j\in\cJ$, and let $k\in([0,r_j]\cap\Z)\cup\{\infty\}$. Assume that $\Fil^r \fM^\diamond e_j$ is of $\mathbf{Case}~(k)$. By identifying $(r',\Lambda,\Theta,\Omega,x,\Delta,\tilde{\Delta})$ with $(r_j,\Lambda_j,\Theta_j,\Theta_{j+1},x_j,\Delta^{(j)},\Delta^{(j-1)})$,
we have the following commutative diagram
$$
\xymatrix{
\cD'\ar@{->}[d]_{\psi^{(j)}}^{\simeq}&&\fM'\ar@{->}[ll]_{N'}\ar@{->}[rr]^{\phi'}\ar@{->}[d]_{\psi^{(j)}}^{\simeq} &&\cD'\ar@{->}[d]_{\psi^{(j+1)}}^{\simeq} &&\Fil^{r;r'}_k\fM'\ar@{->}[ll]_{\phi_r'}\ar@{->}[d]_{\psi^{(j)}}^{\simeq}\\
\cD e_{j}&&\fM^{\diamond} e_j\ar@{->}[ll]^{N^{\diamond,(j)}}\ar@{->}[rr]_{\phi^{\diamond,(j)}}&&\cD e_{j+1}&&\Fil^r\fM^{\diamond}e_{j}\ar@{->}[ll]^{\phi_r^{\diamond,(j)}}
}
$$
where $\fM'$ is a pseudo-strongly divisible module in $\mathbf{Case}~(k)$ with the assumption that $\Fil^r\fM'$ is an $S'_\cO$-submodule of $\fM'$. 
\end{prop}

By Proposition~\ref{prop: typical example of pseudo-SDM}, the quadruple in \eqref{eq: typical example of pseudo-SDM} with $\Fil^r \fM^\diamond e_j$ of $\mathbf{Case}~(k)$ can be regarded as an example of a pseudo-strongly divisible module in $\mathbf{Case}~(k)$. We also note that the assumption $\Fil^{r;r'}_k\fM'\subseteq\fM'$ is illustrated in Theorem~\ref{theo: filtration of pseudo}.

\begin{proof}
The commutativity of the squares involving $\phi'$ and $N'$ are immediate from \eqref{eq: phi(E1) and phi(E2)} and $\eqref{eq: N(E1) and N(E2)}$, respectively. For the commutativity of the square involving $\phi'_r$, it is enough to check that the restriction $\phi^{(j)}:\Fil^{r;r'}_k\fM'\rightarrow \Fil^r\fM^{\diamond}e_{j}$ is a well-defined isomorphism, as $\phi'_r$ and $\phi_r^{\diamond,(j)}$ are induced from $\phi'$ and $\phi^{\diamond,j}$, respectively.

If $k=\infty$, then it is immediate by (i) of Lemma~\ref{lemm: filtration of fM^diamond}. Assume $k\neq \infty$. By Lemma~\ref{lemma: C and B in terms of P and Q}, it is easy to see that we have
$$\Theta C(v)E'_1+B(v)E'_2 =\sum_{i=0}^{r'-1} D_i F'_{k,i}.$$
Hence, by (ii) of Lemma~\ref{lemm: filtration of fM^diamond} we complete the proof.
\end{proof}

We will apply the properties of the first row of the diagram in Proposition~\ref{prop: typical example of pseudo-SDM} to the second row, by identifying $(r',\Lambda,\Theta,\Omega,x,\Delta,\tilde\Delta,k')$ with $(r_j,\Lambda_j,\Theta_j,\Theta_{j+1},x_j,\Delta^{(j)},\Delta^{(j-1)},k_j')$, respectively, for all $j\in\cJ$. In particular, one needs to determine when $(\Lambda_j,\Theta_j,\Theta_{j+1},x_j)$ satisfies the equations and inequalities of $\mathbf{Case}_\phi(k_j')$ for $k_j'\in\cK_{r_j}$ (see Definition~\ref{defi: equation and inequalities, local}).

\begin{defi}
Let $\vr:=(r_j)_{j\in\cJ}$ with $0<r_j<p-1$ for all $j\in\cJ$, and $\vk':=(k'_j)_{j\in\cJ}\in\prod_{j\in\cJ}\cK_{r_j}$.
By the \emph{equations and inequalities of $\mathbf{Case}_\phi(\vr;\vk')$}, we mean the equations and inequalities of $\mathbf{Case}_\phi(k_j')$ for all $j\in\cJ$. We say that $(\vLam,\vTh,\vx)\in ((E^\times)^f)^{2}\times E^f$ \emph{satisfies the equations and inequalities of $\mathbf{Case}_\phi(\vr;\vk')$} if $(\Lambda_j,\Theta_j,\Theta_{j+1},x_j)$ satisfies the equations and inequalities of $\mathbf{Case}_\phi(r_j;k_j')$ for all $j\in\cJ$.
\end{defi}

We now sharpen the equations and inequalities in Table~\ref{tab:my_label}. As each equation relates $v_p(\Theta_j)$ and $v_p(\Theta_{j+1})$, the equations and inequalities of $\mathbf{Case}_\phi(r_j;k'_j)$ are quite connected to those of $\mathbf{Case}_\phi(r_{j+1};k'_{j+1})$ for each $j\in\cJ$. We first define a sign function $s_j$ as follows:
\begin{equation*}
s_j:=
\begin{cases}
   (-1)^{2k'_j} & \mbox{if } k'_j\neq\infty; \\
   -1 & \mbox{if } k'_j=\infty.
\end{cases}
\end{equation*}

\begin{prop}\label{prop: system of equations and inequalities sharpened}
Let $\vr:=(r_j)_{j\in\cJ}$ be given with $0<r_j<p-1$ for all $j\in\cJ$ and set $\vk':=(k'_j)_{j\in\cJ}\in \prod_{j\in\cJ}\cK_{r_j}$. Then $(\vLam,\vTh,\vx)\in ((E^\times)^f)^{2}\times E^f$ satisfy the equations and inequalities of $\mathbf{Case}_\phi(\vr;\vk')$ if and only if the tuples satisfy the following: if we set $T_j=v_p(\Theta_j)$ and $t_j=v_p(x_j)$ for all $j\in\cJ$ then
\begin{enumerate}[leftmargin=*]
\item for all $j\in\cJ$ the equation
     \begin{equation}\label{eq: valuation of Lambda}
     v_p(\Lambda_j)=\frac{r_j-1-T_j+T_{j+1}}{2}
     \end{equation}
     hold;
\item for each $j\in\cJ$ the following linear equation
        \begin{equation}\label{eq: system of equations sharpened}
            T_j+s_jT_{j+1}
            =
            \begin{cases}
                2k'_j-r_j-\frac{1+s_j}{2} & \mbox{if } \frac{1}{2}\leq k'_j\leq \frac{r_j}{2}\,\,\mbox{or}\,\,k'_j=r_j+\frac{1}{2};\\
                2k'_j-r_j-\frac{1+s_j}{2}+2t_j  & \mbox{if } \frac{r_j+1}{2}\leq k'_j\leq r_j;\\
                r_j+1 & \mbox{if } k'_j=\infty
            \end{cases}
        \end{equation}
hold;
\item for each $j\in\cJ$ the following inequalities
        \begin{equation}\label{eq: system of inequailties sharpened}
            T_j\leq 0\qquad\mbox{and}\qquad
            \begin{cases}
                T_j\leq t_j  & \mbox{if }k'_j=\frac{1}{2};\\
                T_j\leq t_j,\, -1\leq T_{j+1} & \mbox{if } 1\leq k'_j\leq \frac{r_j}{2};\\
                T_j\geq t_j<0,\, -1\leq T_{j+1} &\mbox{if } \frac{r_j+1}{2}\leq k'_j\leq r_j-\frac{1}{2};\\
                t_j\leq T_j\leq t_j+r_j &\mbox{if }k'_j=r_j;\\
                T_j\geq t_j+r_j &\mbox{if }k'_j=r_j+\frac{1}{2};\\
                \mbox{none} &\mbox{if }k'_j=\infty
            \end{cases}
        \end{equation}
    hold.
\end{enumerate}
\end{prop}

Note that if $(\vTh,\vx)$ satisfies \eqref{eq: system of equations sharpened} and \eqref{eq: system of inequailties sharpened} then there always exist $\vLam$ satisfying the condition in \eqref{eq: valuation of Lambda}.
\begin{proof}
The equation \eqref{eq: valuation of Lambda} is immediate from the equation~\eqref{eq: Lambda Omega Theta = r'-1} by letting $\Omega=\Theta_{j+1}$. The equation \eqref{eq: system of equations sharpened} is nothing else than the second equation in Table~\ref{tab:my_label}. Similarly, one can immediately get the inequalities \eqref{eq: system of inequailties sharpened} from the inequalities in Table~\ref{tab:my_label}. The inverse direction is also straightforward by replacing $\Theta_{j+1}$ with $\Omega$.
\end{proof}

\subsection{Main results on Galois stable lattices}\label{subsec: main results}
In this subsection, we state and prove the main results on Galois stable lattices. We have two versions of the main results. The first one starts with a weakly admissible filtered $(\phi,N)$-module $D$ as in Example~\ref{exam: admissible filtered modules}, and then we give a candidate of strongly divisible modules and determine when this candidate becomes a strongly divisible module in $\cD(D)$ (see Theorem~\ref{theo: main}). The second one starts with the data $(\vLam,\vTh,\vx)$ satisfying all the necessary equations and inequalities for some $\vk'$, and then we show that these data together with $\vDel(T)$ (determined by $\vk'$; see \eqref{eq: Choice of delta}) determine a weakly admissible filtered $(\phi,N)$-module $D$ and a strongly divisible module in $\cD(D)$ (see Theorem~\ref{theo: main 2}).

Throughout this subsection, we keep the notation of Example~\ref{exam: admissible filtered modules}, and, in particular, let $D=D(\lambda,\vL)$ where $\vL:=(\fL_j)_{j\in\cJ}\in[\mathbb{P}^1(E)]^f$. Assume that $0<r_j<p-1$ for all $j\in\cJ$, and fix $r$ with $p-1 > r\geq r_j$ for all $j\in\cJ$. We also set $\cD:=\cD(D)$ and $m_j:=\lfloor\tfrac{r_j}{2}\rfloor$, and recall $\fM_{(\vth_1,\vth_2,\vrho,\vDel)}$ from \eqref{eq: definition of fM_vth,vth,vDel}.

\smallskip
We now prove the first version of our main results.
\begin{theo}\label{theo: main}
Assume that $0<r_j<p-1$ for all $j\in\cJ$, and fix an integer $r$ with $r_j\leq r<p-1$ for all $j\in\cJ$. Let $D=D(\lambda,\vL)$. If there exists $(k'_{j})_{\in\cJ}\in \prod_{j\in\cJ}\cK_{r_j}$ such that
    \begin{enumerate}[leftmargin=*]
    \item $k'_j=\infty$ if and only if $j\in I_D$;
    \item the triple $(\vLam,\vTh,\vx)$ induced by \eqref{eq: connection between x and L} and \eqref{eq: definition of vx} satisfies the equations and inequalities of $\mathbf{Case}_\phi(\vr;\vk')$;
    \item for each $j\in\cJ$ $\Delta^{(j)}$ is equal to $\Delta$ of $\mathbf{Case}_\phi(r_j;k'_j)$,
    \end{enumerate}
then the $S_\cO$-module $\fM_{(\vth_1,\vth_2,\vrho,\vDel)}$ is a strongly divisible module in $\cD$ of weight $r$.
\end{theo}

Note that the equations in \eqref{eq: From the setting 2} are automatically satisfied due to our definition of $(\vLam,\vTh)$ in \eqref{eq: connection between x and L}. We also note that the value of $\vDel^\varphi(-1)$ appearing in \eqref{eq: definition of vx} is explicitly computed in \S\ref{subsec: the elements delta}. More precisely, $\Delta^{(j)}(-1):=\delta_{l_j}(-1,x_j)$ where $l_j:=l(r_j;k'_j)$ is defined in \eqref{eq: definition of l(r;k')}, and $\delta_{l_j}(-1,x_j)$ is computed as follows:
\begin{equation}\label{eq: delta(-1)}
\delta_{l_j}(-1,x_j)=
\begin{cases}
   -x_j & \mbox{if } l_j=0; \\
   -H_{r_j-l_j}-H_{l_j-1} & \mbox{if } 1\leq l_j\leq \frac{r_j+1}{2}; \\
   -H_{m_j}-H_{m_j-1}+\frac{1}{m_j^2(x_j+H_{m_j}+H_{m_j-1})} & \mbox{if } l_j=m_j +1\mbox{ and }r_j=2m_j
\end{cases}
\end{equation}
where the first case is obvious by its definition, the second case is from Lemma~\ref{lemma-delta} (ii), and the last case is from Lemma~\ref{lemma-delta} (iii), Lemma~\ref{lemma-delta-denom-1} (i), and Lemma~\ref{lemma-P-2m-m}.

\begin{rema}
If $l_j=m_j+1$ and $r_j=2m_j$ then the identity \eqref{eq: definition of vx} together with \eqref{eq: delta(-1)} gives rise to an equation of degree $2$
$$px_{j+1}=p\fL_{j+1}-\fL_j+H_{m_j}+H_{m_j-1}-\frac{1}{m_j^2(x_j+H_{m_j}+H_{m_j-1})}.$$
This explains why the construction of some of the strongly divisible modules in \cite{GP} is more subtle than others. It also explain why a similar phenomenon didn't occur in \cite{BM}, as the Hodge--Tate weights treated in \cite{BM} are $(0,r)$ for odd positive integers $r$ with $r<p-1$.
\end{rema}

\begin{proof}[Proof of Theorem~\ref{theo: main}]
By the assumption (ii), we see that for each $j\in\cJ$ the quadruple $(\Lambda_j,\Theta_j,\Theta_{j+1},x_j)$ satisfies the equations and inequalities of $\mathbf{Case}_\phi(r_j;k'_{j})$, and so by Theorem~\ref{theo: strong divisibility of pseudo} the quadruple is good for $\mathbf{Case}_\phi(r_j;k'_{j})$ with respect to $\Delta^{(j)}$. In particular, it satisfies the inequalities in \eqref{eq: condition for phi,N stable} and so $\fM_{(\vth_1,\vth_2,\vrho,\vDel)}$ is $(\phi,N)$-stable by Lemma~\ref{lemm: pseudo, condition for phi,N stable}. We also note that it satisfies the inequalities in Theorem~\ref{theo: filtration of pseudo}. Hence, $\fM_{(\vth_1,\vth_2,\vrho,\vDel)}$ is a strongly divisible module in $\cD$ of weight $r$.
\end{proof}

For a given weakly admissible filtered $(\phi,N)$-module $D$, it is usually very difficult to construct a strongly divisible module in $\cD:=\cD(D)$. An advantage of Theorem~\ref{theo: main} is that one can reduce constructing a strongly divisible module in $\cD$ to finding a triple $(\vLam,\vTh,\vx)$ satisfying the necessary equations and inequalities from Theorem~\ref{theo: main} (ii) for some tuple $(k'_j)_{j\in\cJ}\in\prod_{j\in\cJ}\cK_{r_j}$. On the other hand, there is still an issue finding which tuple $(k'_j)_{j\in\cJ}\in\prod_{j\in\cJ}\cK_{r_j}$ we need to choose, and the second version resolves this issue by starting with a tuple $(k'_j)_{j\in\cJ}\in\prod_{j\in\cJ}\cK_{r_j}$.

\smallskip

For the second version, we start by noting that if $(\vLam,\vTh,\vx)$ satisfies the equations and inequalities of $\mathbf{Case}_\phi(\vr;\vk')$ then we have
$$\sum_{j\in\cJ} v_p(\Lambda_j)=\frac{1}{2}\sum_{j\in\cJ}(r_j-1)$$
due to \eqref{eq: valuation of Lambda}. This is a necessary condition for \eqref{eq: From the setting 2} and \eqref{eq: condition for lambda, weakly admissible}.

For a given $(k'_j)_{j\in\cJ}$, it is not difficult to compute the necessary equations and inequalities. More precisely, for a given subset $\cJ_0\subset\cJ$ with
\begin{equation}\label{eq: condition for infinity}
\sum_{j\in\cJ}(r_j-1)\geq 2\sum_{j\in\cJ_0}r_j
\end{equation}
(coming from \eqref{eq: condition for lambda, weakly admissible}) and for each tuple $\vk':=(k'_j)_{j\in\cJ}\in\prod_{j\in\cJ}\cK_{r_j}$ with  $\cJ_\infty(\vk')=\cJ_0$ where $$\cJ_\infty(\vk'):=\{j\in\cJ\mid k'_j=\infty\},$$ it is not difficult to find $(\vLam,\vTh,\vx)$ satisfying the equations and inequalities of $\mathbf{Case}_\phi(\vr;\vk')$. Furthermore, it is also not very difficult to describe $(\fL_j)_{j\in\cJ\setminus\cJ_\infty(\vk')}$ satisfying the equation \eqref{eq: definition of vx}. Before introducing the second version, we will further introduce the notation and the terminologies.

For a given tuple $\vr:=(r_j)_{j\in\cJ}\in\Z^f$ with $0<r_j<p-1$ for all $j\in\cJ$ we set
\begin{align*}
\cK(\vr):=
    \bigg\{
    \vk'=(k'_j)_{j\in\cJ}\in \prod_{j\in\cJ}\cK_{r_j}\mid
    \sum_{j\in\cJ}(r_j-1)\geq 2\sum_{j\in \cJ_\infty(\vk')}r_j
    \bigg\},
\end{align*}
and for an integer $r\in\Z$ with $0<r<p-1$ we also set
\begin{align*}
\cK(r):=\left\{(\vr;\vk')\mid 0<r_j\leq r\,\,\mbox{ and }\,\,  \vk'\in \cK(\vr)\right\}.
\end{align*}
For a given subset $\cJ_0\subset\cJ$ with \eqref{eq: condition for infinity}, we further define $$E_{\cJ_0}^f:=\{\vL\in [\mathbb{P}^1(E)]^f\mid \fL_j=\infty \mbox{ for } j\in \cJ_0\mbox{ and } \fL_j\neq\infty\mbox{ for }j\not\in\cJ_0\}$$ and
$$\left\{
  \begin{array}{ll}
    \cK(r,\cJ_0):=\{(\vr;\vk')\in \cK(r)\mid \cJ_\infty(\vk')=\cJ_0\};&\\
    \cK(\vr,\cJ_0):=\{\vk'\in \cK(\vr)\mid \cJ_\infty(\vk')=\cJ_0\}.&
  \end{array}
\right.$$
We will later use the map $\psi^{\cJ_0}:E^f\rightarrow E_{\cJ_0}^f$ sending $(y_j)_{j\in\cJ}$ to $(y'_j)_{j\in\cJ_0}$ where $y'_j:=y_j$ if $j\not\in\cJ_0$ and $y'_j:=\infty$ if $j\in\cJ_0$.

For $(\vr;\vk')\in\cK(r)$, we define a subset of $\hat\R^f(=\hat\R^\cJ)$, denoted by $R'(\vr;\vk')$, by the set of $(t_j)_{j\in \cJ}$ in $\hat\R^f$ such that $\big((T_j)_{j\in \cJ},(t_j)_{j\in\cJ}\big)$ satisfies
\eqref{eq: system of equations sharpened} and \eqref{eq: system of inequailties sharpened} for $\mathbf{Case}_\phi(\vr;\vk')$ for some $(T_j)_{j\in \cJ}\in \hat\R^f$. We expect that for each $\cJ_0\subset\cJ$ with \eqref{eq: condition for infinity} $$\bigcup_{\vk'\in \cK(\vr,\cJ_0)}R'(\vr;\vk')=\hat\R^f.$$
We further let $R'_{int}(\vr;\vk')$ be the topological interior of $R'(\vr;\vk')$ in $\hat\R^f$.
\begin{defi}
For a given $(\vr;\vk')\in\cK(r)$, we say that $\vk'$ is \emph{valid} if $R'(\vr;\vk')\neq\emptyset$. 
\end{defi}

We point out that the equation \eqref{eq: system of equations sharpened} and the inequality \eqref{eq: system of inequailties sharpened} do not impose any restriction on $t_j$ if $k'_j=\infty$. Hence, for any valid $\vk'$, there is a set $R(\vr;\vk')\subset \hat\R^{\cJ\setminus\cJ_\infty(\vk')}$ such that 
\begin{equation}\label{eq: R vs R'}
R'(\vr;\vk')=\hat\R^{\cJ_\infty(\vk')}\times R(\vr;\vk'),
\end{equation} 
where $\hat\R^{\cJ_\infty(\vk')}\times\hat\R^{\cJ\setminus\cJ_\infty(\vk')}$ is naturally identified with $\hat\R^{\cJ}$. We call this $R(\vr;\vk')\subset \hat\R^{\cJ\setminus\cJ_\infty(\vk')}$ \emph{the areal support of $\mathbf{Case}_\phi(\vr;\vk')$}, and we define $R_{int}(\vr;\vk')$ as the topological interior of $R(\vr;\vk')$ as well. It is obvious that $\vk'$ is valid if and only if $R(\vr;\vk')$ is nonempty. It is also immediate that $R(\vr;\vk)=R'(\vr;\vk')$ if $\cJ_\infty(\vk')=\emptyset$. We expect that for each $\cJ_0\subset\cJ$ with \eqref{eq: condition for infinity} $$\bigcup_{\vk'\in \cK(\vr;\cJ_0)} R(\vr;\vk')=\hat\R^{\cJ\setminus\cJ_0}.$$

We now attach $\vL\in[\mathbb{P}^1(E)]^f$ to $\vx$ via \eqref{eq: definition of vx} for each $\vk'$. Recall that we set $m_j:=\lfloor r_j/2\rfloor$. Since $\varphi^f=1$, we have the following identity $$(p-\varphi)^{-1}=\frac{1}{p^f-1}(p^{f-1}+p^{f-2}\varphi+\cdots+\varphi^{f-1}),$$ as a linear automorphism on $E^f$, so that the linear map $(p-\varphi)^{-1}$ is well-defined. For each tuple of integers $\vl=(l_j)_{j\in\cJ}$ with $l_j\in [0,m_j+1]$, the association $$\vx\mapsto p\vx+\vdel_{\vl}^{\varphi}(-1,\vx),$$ which is inspired by \eqref{eq: definition of vx}, also gives rise to a function from $E^f$ to itself, denoted by~$\psi_{\vl}$. (In fact, this map $\psi_{\vl}$ is ill-defined as its domain can not be the whole $E^f$ in general. But for our purpose for defining the support below it is good enough since the domain of $\psi_{\vl}$ always contains $v_p^{-1}(R'(\vr;\vk'))$.)
\begin{lemm}
If $(r_j;l_j)\not\in\{(r_j;0),(2m_j;m_{j}+1)\}$ for some $j\in\cJ$ then $\psi_{\vl}$ is injective.
\end{lemm}

Note that the condition for the lemma is required due to the second comments right after Definition~\ref{defi: definition of delta}.

\begin{proof}
Without loss of generality, we assume that the condition holds only for $j=f-1$. Then we have
$$\psi_{\vl}(\vx)=(px_0+\delta_{l_{f-1}}(-1),px_1+\delta_{l_0}(-1,x_0),\cdots,px_{f-1}+\delta_{l_{f-2}}(-1,x_{f-2})),$$
and so it is obvious that the injectivity of the first coordinate implies the others.
\end{proof}

For $\vl=(l_j)_{j\in\cJ}$ with $l_j\in [0,m_j+1]$ and for $\cJ_0\subset\cJ$ we have the following sequence of functions
\begin{equation*}
\xymatrix{
\hat\R^f && E^f\ar@{->}[ll]_{v_p} \ar@{->}[rr]^{\psi_{\vl}} && E^f\ar@{->}[rr]^{\,(p-\varphi)^{-1}} && E^f\ar@{->}[rr]^{\psi^{\cJ_0}} && E^f_{\cJ_0},
}
\end{equation*}
and for any valid $\vk'$ \emph{the support of $\mathbf{Case}_\phi(\vr;\vk')$} is defined as follows:
\begin{align}\label{eq: definition of support}
\Supp(\vr;\vk'):=\big(\psi^{\cJ_\infty(\vk')}\circ (p-\varphi)^{-1}\circ\psi_{\vl(\vr;\vk')}\big) \left(v_p^{-1}(R'(\vr;\vk'))\right)
\end{align}
where $\vl(\vr;\vk'):=(l(r_j;k_j'))_{j\in\cJ}$ and $l(r_j;k_j')$ is defined in \eqref{eq: definition of l(r;k')}. We emphasize that the composition of maps $\psi^{\cJ_\infty(\vk')}\circ (p-\varphi)^{-1}\circ\psi_{\vl(\vr;\vk')}$ depends on $\vk'$.

We are now ready to state and prove the second version. For a valid $\vk'\in\cK(\vr)$, if $(\vLam,\vTh,\vx)$ satisfies the equations and inequalities of $\mathbf{Case}_\phi(\vr;\vk')$ (see Proposition~\ref{prop: system of equations and inequalities sharpened}) then we have the following data:
\begin{enumerate}[leftmargin=*]
\item $\vL=\psi^{\cJ_\infty(\vk')}\circ (p-\varphi)^{-1}\circ\psi_{\vl(\vr;\vk')}(\vx)\in\Supp(\vr;\vk')$;
\item $\vrho=(p-\varphi)^{-1}\circ\psi_{\vl(\vr;\vk')}(\vx)$;
\item A choice of $\lambda\in\cO$ satisfying the equation \eqref{eq: From the setting 2};
\item $\Delta^{(j)}(T)=\delta_{l(r_j;k_j')}(T,x)$ for all $j\in\cJ$ (see \eqref{eq: Choice of delta});
\item A choice of $(\vth_1,\vth_2)$ via the bijection~\eqref{eq: Lam,Ome,The vs the1,the2}.
\end{enumerate}
Moreover, the data (i) and (iii) together with $\vr$ determine a weakly admissible filtered $(\phi,N)$-module $D=D(\lambda,\vL)$, and the data (ii), (iv), and (v) determine an $S_\cO$-submodule $\fM_{(\vth_1,\vth_2,\vrho,\vDel)}$ defined in \eqref{eq: definition of fM_vth,vth,vDel}. Note that the weakly admissible filtered $(\phi,N)$-module $D=D(\lambda,\vL)$ does not depend on the choice of $\lambda$, by Lemma~\ref{lemm: isomorphism class of filtered phi N modules}.

\begin{theo}\label{theo: main 2}
For a valid $\vk'\in\cK(\vr)$, if $(\vLam,\vTh,\vx)\in ((E^\times)^f)^2\times E^f$ satisfies the equations and inequalities of  $\mathbf{Case}_\phi(\vr;\vk')$ then the corresponding $\fM_{(\vth_1,\vth_2,\vrho,\vDel)}$ is a strongly divisible module in $\cD(D)$ of weight $r$.
\end{theo}

\begin{proof}
We prove this theorem by checking the conditions of Theorem~\ref{theo: main}. It is obvious that $\cJ_0:=\cJ_{\infty}(\vk')=I_D$, which implies the condition (i) of Theorem~\ref{theo: main}. The condition (iii) of Theorem~\ref{theo: main} is also immediate, as that is our choice of $\vDel$.
For the condition (ii) of Theorem~\ref{theo: main}, it is clear that $\psi^{\cJ_0}(\vrho)=\vL$, and it is also clear that $\vx$ and $\vrho$ satisfies the equation~\eqref{eq: definition of vx}. Moreover, $(\vth_1,\vth_2)$ determines $(\vLam,\vTh)$ via \eqref{eq: connection between x and L}, which completes the proof.
\end{proof}

An advantage of Theorem~\ref{theo: main 2} is that we do not have an issue to find an appropriate $\vk'$ for a given $D$. Moreover, for a given $\vk'$ it is easy to compute all the equations and inequalities of Proposition~\ref{prop: system of equations and inequalities sharpened}. However, one can still ask that in this way we can always find a strongly divisible module for each $D$. We expect yes! In other words, we expect that for each $\cJ_0\subset\cJ$ with \eqref{eq: condition for infinity}
\begin{equation}\label{eq: conjecture}
\bigcup_{\vk'\in \cK(\vr,\cJ_0)}\Supp(\vr;\vk')=E^f_{\cJ_0}.
\end{equation}
In fact, we expect more! If the mod-$p$ reduction is an extension of two distinct characters, then we expect that both non-homothetic lattices can be obtained in this way. The reason is that the only artificial assumption we made for the pseudo-strongly divisible modules is $v_p(x)\neq 0$ in some cases (see Table~\ref{tab:my_label} or Proposition~\ref{prop: system of equations and inequalities sharpened}), but we expect that such cases will be covered by other cases. We will illustrate that our expectation is correct for certain cases when $f=2$ in the next section, \S\ref{sec: examples} (see Remarks~\ref{rema: 2,2} and~\ref{rema: 1,5}). We will also show that we can recover the results of \cite{BM,GP,LP} of $f=1$ case.

Finally, we observe when $\fM_{(\vth_1,\vth_2,\vrho,\vDel)}$ are homothetic for different choice of the data.
\begin{coro}
Assume that $(\vLam,\vTh,\vx)$ and $(\vLam',\vTh',\vx')$ satisfy the equations and inequalities of $\mathbf{Case}_\phi(\vr;\vk')$, and choose $(\vth_1,\vth_2)$ and $(\vth_1',\vth_2')$, respectively, via the bijection~\eqref{eq: Lam,Ome,The vs the1,the2}. If
\begin{enumerate}
\item $v_p(\Lambda_j)=v_p(\Lambda'_j)$ for all $j\in\cJ$ and $\prod_{j\in\cJ}\Lambda_j=\prod_{j\in\cJ}\Lambda'_j$;
\item $v_p(\Theta_j)=v_p(\Theta'_j)$ for all $j\in\cJ$;
\item $\vx=\vx'$ and so $\vrho:=(p-\varphi)^{-1}\circ\psi_{\vl(\vr;\vk')}(\vx)=(p-\varphi)^{-1}\circ\psi_{\vl(\vr;\vk')}(\vx')$,
\end{enumerate}
then $\fM_{(\vth_1,\vth_2,\vrho,\vDel)}$ is homothetic to $\fM_{(\vth'_1,\vth'_2,\vrho,\vDel)}$.
\end{coro}
\begin{proof}
By (i) and (ii), there exist $\vec{a},\vec{b}\in(\cO^{\times})^f$ and $c\in E^\times$ such that $\vec{a}\cdot \vth_1=\vth_1'\cdot c$ and $\vec{b}\cdot \vth_2=\vth_2'\cdot c$. Since $\vec{a},\vec{b}\in(\cO^{\times})^f$, the multiplication by these two elements does not change $\fM_{(\vth_1,\vth_2,\vrho,\vDel)}$ and $\fM_{(\vth_1',\vth_2',\vrho',\vDel)}$ as $S_\cO$-modules. Moreover, the multiplication by $c$ is nothing else than homothety. Hence, the two modules are homothetic.
\end{proof}

\subsection{Isotypic components of Breuil modules}\label{subsec: mod p reduction}
In this subsection, we use the mod-$p$ reduction of the pseudo-strongly divisible modules $\fM'$ developed in \S\ref{sec: pseudo strongly divisible modules}
to compute the mod-$p$ reduction of the isotypic component of the $S_\cO$-module $\fM_{\vth_1,\vth_2,\vrho,\vDel}$ at $j\in\cJ$. Once $(\vLam,\vTh,\vx)\in ((E^\times)^f)^{2}\times E^f$ satisfies the equations and inequalities of $\mathbf{Case}_\phi(\vr;\vk')$ for some $\vk'\in\prod_{j\in\cJ}\cK_{r_j}$, one may directly give the structure of the mod-$p$ reduction $$\cM^\diamond:=\fM^\diamond/(\varpi,\Fil^pS_{\cO})\fM^\diamond=\barS_\F(\barE_1,\barE_2).$$

For each $j\in\cJ$, the isotypic component $\cM^\diamond e_j$ of $\cM^\diamond$ has the following additional structure:
$$\big(\Fil^r\cM^\diamond e_j,\,\,\phi_r^{\diamond,(j)}:\Fil^r\cM^\diamond e_j\rightarrow\cM^\diamond e_{j+1},\,\,N^{\diamond,(j)}:\cM^\diamond e_j\rightarrow \cM^\diamond e_j\big)$$ depending on $\mathbf{Case}_\phi(r_j;k'_j)$. We write $\baseE^{(j)}$ for the ordered basis $(\barE_1^{(j)}\ \barE_2^{(j)})$ of the isotypic component $\cM^\diamond e_j$ and we let $\barF^{(j)}=(\barF_1^{(j)}\ \barF_2^{(j)})$ be an ordered pair of generators for $\Fil^r\cM^\diamond e_j$ over $\barS_\F e_j(\cong\barS'_\F)$.

We set $k_j:=\lfloor k'_j \rfloor$, and to lighten the notation we will write
$$
\begin{cases}
M_{N,j}:=\Mat_{\baseE^{(j)},\baseE^{(j)}}(N^{\diamond,(j)});\\
M_{\Fil^r,j}:=\Mat_{\baseE^{(j)},\baseF^{(j)}}(\Fil^r\cM^\diamond e_j);\\
M_{\phi_r,j}:=\Mat_{\baseE^{(j+1)},\baseF^{(j)}}(\phi_r^{\diamond,(j)})
\end{cases}
$$
throughout this subsection. We further define the following pair of units $\alpha_j,\beta_j$ in $\F^\times$ for each $(r_j;k'_j)$ by
$$
(\alpha_j,\beta_j):=
\begin{cases}
    \big( \frac{\Lambda_j}{p^{r_j-k_j-1}},\ \frac{\Lambda_j\Theta_j}{p^{k_j}\Theta_{j+1}} \big)
     & \mbox{if } k'_j=k_j+\tfrac{1}{2}\in(\tfrac{1}{2}+\Z)\cap[\frac{1}{2},\tfrac{r_j}{2}]; \\
    \big( \frac{\Lambda_j\Theta_j}{p^{k_j-1}},\ \frac{\Lambda_j}{p^{r_j-k_j}\Theta_{j+1}} \big)
     & \mbox{if } k'_j=k_j\in\Z\cap[\frac{1}{2},\tfrac{r_j}{2}]; \\
    \big( \frac{\Lambda_jx_j}{p^{r_j-k_j-1}},\ \frac{\Lambda_j\Theta_j}{p^{k_j}x_j\Theta_{j+1}} \big)
     & \mbox{if } k'_j=k_j+\tfrac{1}{2}\in(\tfrac{1}{2}+\Z)\cap[\tfrac{r_j+1}{2},r_j]; \\
    \big( \frac{\Lambda_j\Theta_j}{p^{k_j-1}x_j},\ \frac{\Lambda_jx_j}{p^{r_j-k_j}\Theta_{j+1}} \big)
     & \mbox{if } k'_j=k_j\in\Z\cap[\tfrac{r_j+1}{2},r_j];\\
    \big( p\Lambda_j,\ \frac{\Lambda_j\Theta_j}{p^{r_j}\Theta_{j+1}} \big)
     & \mbox{if }k'_j=r_j+\frac{1}{2},\infty.
    \end{cases}
$$
It is easy to see from this definition that
\begin{equation}\label{eq: alpha_j/beta_j}
\left(\frac{\alpha_j}{\beta_j}\right)^{s_j}=
\begin{cases}
  \Theta_j\Theta_{j+1}^{s_j}p^{r_j-2k'_j+\frac{1+s_j}{2}} & \mbox{if } k'_j\in [\frac{1}{2},\frac{r_j}{2}]\cup\{r_j+\frac{1}{2}\};\\
  \Theta_j\Theta_{j+1}^{s_j}p^{r_j-2k'_j+\frac{1+s_j}{2}}x_j^{-2} & \mbox{if } k'_j\in [\frac{r_j+1}{2},r_j];\\
 \Theta_j\Theta_{j+1}^{s_j}p^{-r_j-1} & \mbox{if } k'_j=\infty
\end{cases}
\end{equation}
where $s_j$ is defined right before Proposition~\ref{prop: system of equations and inequalities sharpened}.

\subsubsection{\textbf{Description of $N^{\diamond,(j)}$}}
In this section of paragraph, we describe the matrix $M_{N,j}$.

As $N(\gamma)/p\equiv -u^{p-1} \pmod{(\varpi,\Fil^p S_{\cO})}$ one has
$$M_{N,j}=
\begin{bmatrix}
  0 & 0 \\
  \frac{1}{\Theta_j}\big(1-\dot\Delta^{(j-1)}(-1)u^{p-1}\big) & 0
\end{bmatrix},$$
coming from \eqref{eq: definition of N'}, where $\Delta^{(j)}(-1)=\delta_{l(r_j;k'_j)}(-1,x_j)$. Here, $l(r_j;k'_j)$ has been defined in \eqref{eq: Choice of delta}. One can find the value of $\dot\delta_{l}(-1)$ in Lemma~\ref{lemma-delta-derivative} for $\delta_l\neq\delta_0$ while one obviously has $\dot\delta_0(-1,x_j)=0$.

We introducing the following definition which lets us to describe $N^{(j)}$ shortly.
\begin{defi}\label{defi: monodromy type}
For a rank 2 Breuil module $\cM$, we define the \textit{monodromy type of $\cM$} by a tuple of integer $\mathrm{MT}=(\mathrm{MT}_0,\cdots,\mathrm{MT}_{f-1})$, where each $\mathrm{MT}_j$ is given by
    $$\mathrm{MT}_j=\operatorname{rank}_{\F}(N^{(j)}:\cM^{(j)}/u\cM^{(j)}\to\cM^{(j)}/u\cM^{(j)}).$$
In the same fashion, we define the monodromy type of a mod-$p$ reduction of pseudo-strongly divisible module.
\end{defi}

Recalling that $v_p(\Theta_j)\leq 0$ for all $j\in\cJ$, it is worth noting that $\mathrm{MT}_j=1$ if $v_p(\Theta_j)=0$, and $\mathrm{MT}_j=0$ if $v_p(\Theta_j)<0$.

\begin{rema}
By Theorem~4.3.2 of \cite{Caruso}, any semisimple rank 2 Breuil module has monodromy type $(0,\cdots,0)$. Hence, any Breuil module with non-trivial monodromy type is not semisimple. This fact will be frequently used to see non-splitness of rank 2 Breuil modules in our examples.
\end{rema}

\subsubsection{\textbf{Description of $\Fil^r\cM^\diamond e_j$ and $\phi_r^{\diamond,(j)}$ when $r_j\geq 2$}}
In this section of paragraph, we describe the matrices $M_{\Fil^r,j}$ and $M_{\phi_r,j}$ when $k'_j<\infty$ and $r_j\geq 2$.

\begin{prop}\label{prop: mod p reduction, r_j>=2, k'_j half integer}
Assume that $r_j\geq 2$. If $k'_j\in (\frac{1}{2}+\Z)$, then the matrices representing $\Fil^r\cM^\diamond e_j$ and $\phi_r^{\diamond,(j)}$ are given as follow:
\begin{enumerate}[leftmargin=*]
\item if $k'_j=\frac{1}{2}$ then
    \begin{align*}
    M_{\Fil^r,j}
        =
        \begin{bmatrix}
            1 & 0 \\
            \frac{x_j}{\Theta_j} & 1
        \end{bmatrix}
        \begin{bmatrix}
            u^{r-r_j} & 0 \\
            \frac{\alpha_j}{\beta_j\Theta_{j+1}}\frac{(-1)^{r_j}u^{r-1}}{r_j-1} & u^r
        \end{bmatrix}
    \end{align*}
    and
    \begin{align*}
    M_{\phi_r,j}
        =
        \begin{bmatrix}
            1 & 0 \\
            \frac{r_j-1}{\Theta_{j+1}}u^{p-1} & 1
        \end{bmatrix}
        \begin{bmatrix}
            \alpha_j (-1)^{r-r_j} & 0 \\
            0 & \beta_j (-1)^r
        \end{bmatrix};
    \end{align*}
\item if $k'_j=k_j+\frac{1}{2}\in[1,\frac{r_j-1}{2}]$ then
    \begin{align*}
    M_{\Fil^r,j}
        =
        \begin{bmatrix}
            1 & 0 \\
            \frac{x_j}{\Theta_j} & 1
        \end{bmatrix}
        \begin{bmatrix}
            u^{r-r_j+k_j} & 0 \\
            \frac{\alpha_j}{\beta_j\Theta_{j+1}}\frac{(-1)^{r_j}u^{r-k_j-1}}{(r_j-2k_j-1)\binom{r_j-k_j-1}{k_j}^2} & u^{r-k_j}
        \end{bmatrix}
    \end{align*}
    and
    \begin{align*}
    M_{\phi_r,j}
        =
        \begin{bmatrix}
            1 & p\Theta_{j+1}\frac{k_j(r_j-k_j)}{r_j-2k_j} \\
            \frac{2(r_j-k_j-1)k_j+r_j-1}{\Theta_{j+1}}u^{p-1} & 1
        \end{bmatrix}
        \begin{bmatrix}
            \alpha_j \frac{(-1)^{r-r_j}}{\binom{r_j-k_j-1}{k_j}} & 0 \\
            0 & \beta_j (-1)^r\binom{r_j-k_j-1}{k_j}
        \end{bmatrix};
    \end{align*}
\item if $k'_j=m_j+\frac{1}{2}$ with $r_j=2m_j$ then
    \begin{align*}
    M_{\Fil^r,j}
        =
        \begin{bmatrix}
            1 & \frac{\Theta_j}{x_j} \\
            0 & 1
        \end{bmatrix}
        \begin{bmatrix}
            u^{r-m_j} & 0 \\
            -\frac{\alpha_j}{\beta_j\Theta_{j+1}} m_j^2 u^{r-m_j-1}
             & u^{r-m_j}
        \end{bmatrix}
    \end{align*}
    and
    \begin{align*}
    M_{\phi_r,j}
        =
        \begin{bmatrix}
            1 & 0 \\
            \frac{2m_j^2-1}{\Theta_{j+1}}u^{p-1} & 1
        \end{bmatrix}
        \begin{bmatrix}
            \alpha_j (-1)^{r+1}m_j & 0 \\
            0 & \beta_j \frac{(-1)^{r+1}}{m_j}
        \end{bmatrix};
    \end{align*}
\item if $k'_j=m_j+\frac{1}{2}$ with $r_j=2m_j+1$ then
    \begin{align*}
    M_{\Fil^r,j}
        =
        \begin{bmatrix}
            1 & 0 \\
            \frac{x_j}{\Theta_j}+\frac{\alpha_j}{\beta_j\Theta_{j+1}}2H_{m_j} & 1
        \end{bmatrix}
        \begin{bmatrix}
            u^{r-m_j-1} & 0 \\
            0 & u^{r-m_j}
        \end{bmatrix}
    \end{align*}
    and
    \begin{align*}
    M_{\phi_r,j}
        =
        \begin{bmatrix}
            1 & p\Theta_{j+1}m_j(m_j+1) \\
            \frac{2m_j(m_j+1)}{\Theta_{j+1}}u^{p-1} & 1
        \end{bmatrix}
        \begin{bmatrix}
            \alpha_j (-1)^{r+1} & 0 \\
            0 & \beta_j (-1)^r
        \end{bmatrix};
    \end{align*}
\item if $k'_j=k_j+\frac{1}{2}\in [m_j+1,r_j]$ then
    \begin{align*}
    M_{\Fil^r,j}
        =
        \begin{bmatrix}
            1 & \frac{\Theta_j}{x_j} \\
            0 & 1
        \end{bmatrix}
        \begin{bmatrix}
            u^{r-r_j+k_j} & 0 \\
            \frac{\alpha_j}{\beta_j\Theta_{j+1}}(-1)^{r_j-1}(2k_j-r_j+1)\binom{k_j}{r_j-k_j-1}^2 u^{r-k_j-1} & u^{r-k_j}
        \end{bmatrix}
    \end{align*}
    and
    \begin{multline*}
    M_{\phi_r,j}
        =
        \begin{bmatrix}
            1 & p\Theta_{j+1}\frac{k_j(r_j-k_j)}{r_j-2k_j} \\
            \frac{2(r_j-k_j-1)k_j+r_j-1}{\Theta_{j+1}}u^{p-1} +\frac{b_{r_j,k_j+1}(-1)}{p\Theta_{j+1}x_j} & 1
        \end{bmatrix}\\
        \times
        \begin{bmatrix}
            \alpha_j (-1)^{r+1}(r_j-k_j)\binom{k_j}{r_j-k_j} & 0 \\
            0 & \beta_j \frac{(-1)^{r-r_j+1}}{(r_j-k_j)\binom{k_j}{r_j-k_j}}
        \end{bmatrix}
    \end{multline*}
    where $b_{r_j,k_j+1}(-1)=-\frac{2m_j^2+2m_j+1}{m_j^2(m_j+1)^2}$ if $(r_j,k_j)=(2m_j+1,m_j+1)$ and $b_{r_j,k_j+1}(-1)=0$ otherwise;
 \item if $k'_j=r_j+\frac{1}{2}$ then
    \begin{align*}
    M_{\Fil^r,j}
        =
        \begin{bmatrix}
            u^r & 0 \\
            0 & u^{r-r_j}
        \end{bmatrix}
    \end{align*}
    and
    \begin{align*}
    M_{\phi_r,j}
        =
        \begin{bmatrix}
          1 & \frac{\alpha_j}{\beta_j}\frac{\Theta_j}{p^{r_j}x_j} \\
          0 & 1
        \end{bmatrix}
        \begin{bmatrix}
            \alpha_j (-1)^r & 0 \\
            0 & \beta_j (-1)^{r-r_j}
        \end{bmatrix}.
    \end{align*}
\end{enumerate}
\end{prop}

\begin{proof}
This proposition follows immediately from the results in \S\ref{sec: pseudo strongly divisible modules}. Note that the $r=r_j$ was assumed for the results in \S\ref{sec: pseudo strongly divisible modules}. But by simply multiplying $v^{r-r_j}$ to $G_1',G_2'$ as in \eqref{eq: filtration when r' leq r} and $(-1)^{r-r_j}$ to the image of $\phi_r(\bar{F}_1),\phi(\bar{F}_2)$, we could have the description as in the statement. More precise reference details are provided in the following.
\begin{enumerate}[leftmargin=*]
\item Let $k'_j=\frac{1}{2}$. Then $M_{\Fil^r,j}$ is computed in Proposition~\ref{prop: Fil-Case-0}, and $M_{\phi_r,j}$ is given by Proposition~\ref{prop-cM-Case-0}, where $P_{0,0}(-1,0)=1$ by Lemma~\ref{lemma-delta-denom-1} (i) and $\dot\delta_1(-1)=r_j-1$ by Lemma~\ref{lemma-delta-derivative} (iii).
\item Let $k'_j=k_j+\frac{1}{2}\in[1,\frac{r_j-1}{2}]$. Then $M_{\Fil^r,j}$ is computed in Proposition~\ref{prop: Fil-Case-1-r/2}, where $\frac{1}{a_{k_j+1}}=\frac{(-1)^{r_j}}{(r_j-2k_j-1)\binom{r_j-k_j-1}{k_j}^2}$ by Lemma~\ref{lemma-delta-denom-1} (i) and (ii) and $p^{r_j-2k_j+1}\Theta_j=0$ in $\F$. Also, $M_{\phi_r,j}$ is given by Proposition~\ref{prop-cM-Case-1-r/2}, where $P_{k_j,k_j}(-1,0)=\binom{r_j-k_j-1}{k_j}^{-1}$ and $P_{k_j,0}(-1,0)=(-1)^{r_j}k_j\binom{r_j-k_j}{k_j}$ by Lemma~\ref{lemma-delta-denom-1} (i) and (ii), $\delta_{k_j}(-1)-\delta_{k_j+1}(-1)=\frac{r_j-2k_j}{k_j(r_j-k_j)}$ by Lemma~\ref{lemma-delta} (i), and $\dot\delta_{k_j+1}(-1)=2(r_j-k_j-1)k_j+r_j-1$ by Lemma~\ref{lemma-delta-derivative} (iii).
\item Let $k'_j=m_j+\frac{1}{2}$ with $r_j=2m_j$. Then $M_{\Fil^r,j}$ is computed in Proposition~\ref{prop: Fil-Case-m} and Proposition~\ref{prop-cM-Case-m}, where $a_{m_j}=m_j^2$ by Lemma~\ref{lemma-delta-denom-1} (i) and (ii). Also, $M_{\phi_r,j}$ is given by Proposition~\ref{prop-cM-Case-m}, where $P_{m_j,0}(-1,0)=m_j$ by Lemma~\ref{lemma-delta-denom-1} (i) and $\dot\delta_{m_j}(-1)=2m_j^2-1$ by Lemma~\ref{lemma-delta-derivative} (iii).
\item Let $k'_j=m_j+\frac{1}{2}$ with $r_j=2m_j+1$. Then $M_{\Fil^r,j}$ is computed in Proposition~\ref{prop: Fil-Case-1-r/2}, where $\frac{1}{a_{m_j+1}}=2H_{m_j}$ by Lemma~\ref{lemma: a_k m+1 leq k} (v) and $p^{r_j-2m_j+1}\Theta_j=0$ in $\F$. Also, $M_{\phi_r,j}$ is given by Proposition~\ref{prop-cM-Case-m}, where $P_{m_j,m_j}(-1,0)=1$ and $P_{m_j,0}(-1,0)=-m_j(m_j+1)$ by Lemma~\ref{lemma-delta-denom-1} (i) and (ii), $\delta_{m_j}(-1)-\delta_{m_j+1}(-1)=\frac{1}{m_j(m_j+1)}$ by Lemma~\ref{lemma-delta} (i), and $\dot\delta_{m_j+1}(-1)=2m_j(m_j+1)$ by Lemma~\ref{lemma-delta-derivative} (iii).
\item Let $k'_j=k_j+\frac{1}{2}\in[m_j+1,r_j]$. Then $M_{\Fil^r,j}$ is computed from in Proposition~\ref{prop: Fil-Case-m+1-r-1}, where $a_{r_j-k_j}=(-1)^{r_j}(2k_j-r_j+1)\binom{k_j}{r_j-k_j-1}^2$ by Lemma~\ref{lemma-delta-denom-1} (i) and (ii) and $\frac{\Theta_j}{p^{2k_j-r_j-1}x_j^2}=0$ in $\F$. Also, $M_{\phi_r,j}$ is given by Proposition~\ref{prop-cM-Case-m+1} and Proposition~\ref{prop-cM-Case-m+2-r-1}, where $P_{r_j-k_j,0}(-1,0)=(-1)^{r_j}(r_j-k_j)\binom{k_j}{r_j-k_j}$ and $P_{r_j-k_j,r_j-k_j}(-1,0)=\binom{k_j-1}{r_j-k_j}^{-1}$ by Lemma~\ref{lemma-delta-denom-1} (i) and (ii), $\delta_{r_j-k_j+1}(-1)-\delta_{r_j-k_j}(-1)=\frac{r_j-2k_j}{k_j(r_j-k_j)}$ by Lemma~\ref{lemma-delta} (i), $\dot\delta_{r_j-k_j}(-1)=2(r_j-k_j+1)k_j+r_j-1$ by Lemma~\ref{lemma-delta-derivative} (iii), and $b_{2m_j+1,m_j+2}(-1)=-\frac{2m_j^2+2m_j+1}{m_j^2(m_j+1)^2}$ by Lemma~\ref{lemma-b-2m+1-m+2}.
\item Let $k'_j=r_j+\frac{1}{2}$. $M_{\Fil^r,j}$ is computed in Proposition~\ref{prop: Fil-Case-r} where $\frac{\Theta_j}{x_j}=0$ in $\F$, and the form of $M_{\phi_r,j}$ is given by Proposition~\ref{prop-cM-Case-r}. 
\end{enumerate}
This completes the proof.
\end{proof}

\begin{prop}\label{prop: mod p reduction, r_j>=2, k'_j integer}
Assume that $r_j\geq 2$. If $k'_j\in \Z$, then the matrices representing $\Fil^r\cM^\diamond e_j$ and $\phi_r^{\diamond,(j)}$ are given as follow:
\begin{enumerate}[leftmargin=*]
\item if $k'_j=k_j\in [\frac{1}{2},\frac{r_j-1}{2}]$ then
    \begin{align*}
    M_{\Fil^r,j}
        =
        \begin{bmatrix}
            1 & 0 \\
            \frac{x_j}{\Theta_j} & 1
        \end{bmatrix}
        \begin{bmatrix}
            u^{r-r_j+k_j} & \frac{\alpha_j}{\beta_j\Theta_{j+1}}(-1)^{r_j}(r_j-2k_j+1)\binom{r_j-k_j}{k_j-1}^2 u^{r-r_j+k_j-1} \\
            0 & u^{r-k_j}
        \end{bmatrix}
    \end{align*}
    and
    \begin{multline*}
    M_{\phi_r,j}
        =
        \begin{bmatrix}
            1 & p\Theta_{j+1}\frac{-k_j(r_j-k_j)}{r_j-2k_j} \\
            \frac{2(r_j-k_j)(k_j-1)+r_j-1}{\Theta_{j+1}}u^{p-1} & 1
        \end{bmatrix}
        \begin{bmatrix}
            0 & \alpha_j (-1)^rk_j\binom{r_j-k_j}{k_j} \\
            \beta_j\frac{(-1)^{r-r_j+1}}{k_j\binom{r_j-k_j}{k_j}} & 0
        \end{bmatrix};
    \end{multline*}
\item if $k'_j=m_j$ with $r_j=2m_j$ then
    \begin{align*}
    M_{\Fil^r,j}
        =
        \begin{bmatrix}
            1 & 0 \\
            \frac{x_j}{\Theta_j} & 1
        \end{bmatrix}
        \begin{bmatrix}
            u^{r-m_j} & \frac{\alpha_j}{\beta_j\Theta_{j+1}} m_j^2 u^{r-m_j-1} \\
            0 & u^{r-m_j}
        \end{bmatrix}
    \end{align*}
    and
    \begin{align*}
    M_{\phi_r,j}=
        \begin{bmatrix}
            1 & -p\Theta_{j+1}m_j^2(H_{m_j}+H_{m_j-1}) \\
            \frac{2m_j^2-1}{\Theta_{j+1}}u^{p-1} & 1
        \end{bmatrix}
        \begin{bmatrix}
            0 & \alpha_j (-1)^r m_j \\
            \beta_j \frac{(-1)^{r+1}}{m_j} & 0
        \end{bmatrix};
    \end{align*}
\item if $k'_j=m_j+1$ with $r_j=2m_j+1$ then
    \begin{align*}
    M_{\Fil^r,j}
        =
        \begin{bmatrix}
            1 & \frac{\Theta_j}{x_j} \\
            0 & 1
        \end{bmatrix}
        \begin{bmatrix}
            u^{r-m_j} & 0 \\
            0 & u^{r-m_j-1}
        \end{bmatrix}
    \end{align*}
    and
    \begin{align*}
    M_{\phi_r,j}
        =
        \begin{bmatrix}
            1 & p\Theta_{j+1}m_j(m_j+1) \\
            \frac{2m_j(m_j+1)}{\Theta_{j+1}}u^{p-1} & 1
        \end{bmatrix}
        \begin{bmatrix}
            0 & \alpha_j (-1)^{r+1} \\
            \beta_j (-1)^{r+1} & 0
        \end{bmatrix};
    \end{align*}
\item if $k'_j=k_j\in [\frac{r_j+2}{2}, r_j-1]$ then
    \begin{align*}
    M_{\Fil^r,j}
        =
        \begin{bmatrix}
            1 & \frac{\Theta_j}{x_j} \\
            0 & 1
        \end{bmatrix}
        \begin{bmatrix}
            u^{r-r_j+k_j}
             & \frac{\alpha_j}{\beta_j\Theta_{j+1}}\frac{(-1)^{r_j-1}u^{r-r_j+k_j-1}}{(2k_j-r_j-1)\binom{k_j-1}{r_j-k_j}^2} \\
            0 & u^{r-k_j}
        \end{bmatrix}
    \end{align*}
    and
    \begin{multline*}
    M_{\phi_r,j}
        =
        \begin{bmatrix}
            1 & p\Theta_{j+1}\frac{k_j(r_j-k_j)}{2k_j-r_j} \\
            \frac{2(r_j-k_j)(k_j-1)+r_j-1}{\Theta_{j+1}}u^{p-1}+\frac{b_{r_j,k_j}(-1)}{p\Theta_{j+1}x_j} & 1
        \end{bmatrix}\\
        \times
        \begin{bmatrix}
            0 & \alpha_j \frac{(-1)^{r-r_j}}{\binom{k_j-1}{r_j-k_j}}\\
            \beta_j (-1)^{r+1}\binom{k_j-1}{r_j-k_j} & 0
        \end{bmatrix}
    \end{multline*}
    where $b_{r_j,k_j}(-1)=-\frac{2m_j^2+2m_j+1}{m_j^2(m_j+1)^2}$ if $(r_j,k_j)=(2m_j+1,m_j+2)$ and $b_{r_j,k_j}(-1)=0$ otherwise;
\item if $k'_j=r_j$ then
    \begin{align*}
    M_{\Fil^r,j}
        =
        \begin{bmatrix}
            1 & \frac{\Theta_j}{x_j} \\
            0 & 1
        \end{bmatrix}
        \begin{bmatrix}
            u^r & \frac{\alpha_j}{\beta_j\Theta_{j+1}}\frac{(-1)^{r_j-1}u^{r-1}}{r_j-1} \\
            0 & u^{r-r_j}
        \end{bmatrix}
    \end{align*}
    and
    \begin{align*}
    M_{\phi_r,j}
        =
        \begin{bmatrix}
            1 & -\frac{\alpha_j}{\beta_j}\frac{p^{r_j}x_j}{\Theta_j} \\
            \frac{r_j-1}{\Theta_{j+1}}u^{p-1}+\frac{b_{r_j,r_j}(-1)}{p\Theta_{j+1}x_j} & 1
        \end{bmatrix}
        \begin{bmatrix}
            0 & \alpha_j (-1)^{r-r_j} \\
            \beta_j (-1)^{r+1} & 0
        \end{bmatrix}
    \end{align*}
    where $b_{r_j,r_j}(-1)=-\frac{5}{4}$ if $r_j=3$ and $b_{r_j,r_j}(-1)=0$ otherwise.
\end{enumerate}
\end{prop}

\begin{proof}
By the same reason as Proposition~\ref{prop: mod p reduction, r_j>=2, k'_j half integer}, we get the description as in the statement. More precise reference details are provided in the following.
\begin{enumerate}[leftmargin=*]
\item Let $k'_j=k_j\in[\frac{1}{2},\frac{r_j-1}{2}]$. Then $M_{\Fil^r,j}$ is computed in Proposition~\ref{prop: Fil-Case-1-r/2}, where $a_{k_j}=(-1)^{r_j}(r_j-2k_j+1)\binom{r_j-k_j}{k_j-1}^2$ by Lemma~\ref{lemma-delta-denom-1} (i) and (ii) and $\frac{1}{p^{r_j-2k_j-1}\Theta_j}=0$ in $\F$. Also, $M_{\phi_r,j}$ is given by Proposition~\ref{prop-cM-Case-1-r/2}, where $P_{k_j,0}(-1,0)=(-1)^{r_j}k_j\binom{r_j-k_j}{k_j}$ and $P_{k_j,k_j}(-1,0)=\binom{r_j-k_j-1}{k_j}^{-1}$ by Lemma~\ref{lemma-delta-denom-1} (i) and (ii), $\delta_{k_j+1}(-1)-\delta_{k_j}(-1)=\frac{2k_j-r_j}{k_j(r_j-k_j)}$ by Lemma~\ref{lemma-delta} (i), and $\dot\delta_{k_j}(-1)=2(r_j-k_j)(k_j-1)+r_j-1$ by Lemma~\ref{lemma-delta-derivative} (iii).
\item Let $k'_j=m_j$ with $r_j=2m_j$. Then $M_{\Fil^r,j}$ is computed in Proposition~\ref{prop: Fil-Case-m}, where $a_{m_j}=m_j^2$ by Lemma~\ref{lemma-delta-denom-1} (i) and (ii). Also, $M_{\phi_r,j}$ is given by Proposition~\ref{prop-cM-Case-m}, where $P_{m_j,0}(-1,0)=m_j$ by Lemma~\ref{lemma-delta-denom-1} (i), $P_{m_j,m_j}(-1,0)=-m_j(H_{m_j}+H_{m_j-1})$ by Lemma~\ref{lemma-P-2m-m}, and $\dot\delta_{m_j}(-1)=2m_j^2-1$ by Lemma~\ref{lemma-delta-derivative} (iii).
\item Let $k'_j=m_j+1$ with $r_j=2m_j+1$. Then $M_{\Fil^r,j}$ is computed in Proposition~\ref{prop: Fil-Case-m+1-r-1}, where $\frac{p^{2(m_j+1)-r_j+1}x_j^2}{\Theta_j}=0$ in $\F$. Also, $M_{\phi_r,j}$ is given by Proposition~\ref{prop-cM-Case-m+1}, where $P_{r_j-m_j-1,r_j-m_j-1}(-1,0)=1$ and $P_{r_j-m_j-1,0}(-1,0)=-m_j(m_j+1)$ by Lemma~\ref{lemma-delta-denom-1} (i) and (ii), $\delta_{r_j-m_j-1}(-1)-\delta_{m_j+1}(-1,x_j)=\delta_{m_j}(-1)-\delta_{m_j+1}(-1)=\frac{1}{m_j(m_j+1)}$ by Lemma~\ref{lemma-delta} (i), and $\dot\delta_{r_j-m_j}(-1)=\dot\delta_{m_j+1}(-1)=2m_j(m_j+1)$ by Lemma~\ref{lemma-delta-derivative} (iii).
\item Let $k'_j=k_j\in[\frac{r_j+2}{2},r_j-1]$. Then $M_{\Fil^r,j}$ is computed in Proposition~\ref{prop: Fil-Case-m+1-r-1}, where $\frac{1}{a_{r_j-k_j+1}}=\frac{(-1)^{r_j}}{(2k_j-r_j-1)\binom{k_j-1}{r_j-k_j}^2}$ by Lemma~\ref{lemma-delta-denom-1} (i) and (ii) and $\frac{p^{2k_j-r_j+1}x_j^2}{\Theta_j}=0$ in $\F$. If $(r_j;k_j)=(2m_j;m_j+1)$, then the form of $M_{\phi_r,j}$ is given by Proposition~\ref{prop-cM-Case-m+1} where $\delta_{r_j-m_j-1}(-1)-\delta_{m_j+1}(-1,x_j)=\delta_{m_j-1}(-1)-\delta_{m_j}(-1)$ in $\F$ by Lemma~\ref{lemma-delta} (iii). Otherwise, the form of $M_{\phi_r,j}$ is given by Proposition~\ref{prop-cM-Case-m+2-r-1}. Here, we have $P_{r_j-k_j,r_j-k_j}(-1,0)=\binom{k_j-1}{r_j-k_j}^{-1}$ and $P_{r_j-k_j,0}(-1,0)=(-1)^{r_j}(r_j-k_j)\binom{k_j}{r_j-k_j}$ by Lemma~\ref{lemma-delta-denom-1} (i) and (ii), $\delta_{r_j-k_j}(-1)-\delta_{r_j-k_j+1}(-1)=\frac{2k_j-r_j}{k_j(r_j-k_j)}$ by Lemma~\ref{lemma-delta} (i), $\dot\delta_{r_j-k_j}(-1)=2(r_j-k_j+1)k_j+r_j-1$ by Lemma~\ref{lemma-delta-derivative} (iii), and $b_{2m_j+1,m_j+2}(-1)=-\frac{2m_j^2+2m_j+1}{m_j^2(m_j+1)^2}$ by Lemma~\ref{lemma-b-2m+1-m+2}.
\item Let $k'_j=r_j$. Then $M_{\Fil^r,j}$ is computed in Proposition~\ref{prop: Fil-Case-r}. Also, $M_{\phi_r,j}$ is given by Proposition~\ref{prop-cM-Case-r}, where $b_{3,3}(-1)=-\frac{5}{4}$ by Lemma~\ref{lemma-b-2m+1-m+2} and $\dot\delta_1(-1)=r_j-1$ for $r_j\geq 3$ and $\dot\delta_2(-1,x)=\dot\delta_1(-1)=1$ for $r_j=2$ by Lemma~\ref{lemma-delta-derivative} (iii) and (iv).
\end{enumerate}
This completes the proof.
\end{proof}

\subsubsection{\textbf{Description of $\Fil^r\cM^\diamond e_j$ and $\phi_r^{\diamond,(j)}$ for exceptional cases}}
In this section of paragraph, we describe the matrices $M_{\Fil^r,j}$ and $M_{\phi_r,j}$ when either $r_j=1$ or $k'_j=\infty$.

\begin{prop}\label{prop: mod p reduction, r_j=1}
Assume that $r_j=1$. The matrices representing $\Fil^r\cM^\diamond e_j$ and $\phi_r^{\diamond,(j)}$ are given as follow:
\begin{enumerate}[leftmargin=*]
\item if $k'_j=\frac{1}{2}$ then
    \begin{align*}
    M_{\Fil^r,j}
        =
        \begin{bmatrix}
          1 & 0 \\
          \frac{x_j}{\Theta_j} & 1
        \end{bmatrix}
        \begin{bmatrix}
            u^{r-1} & 0 \\
            0 & u^r
        \end{bmatrix}
    \end{align*}
    and
    \begin{align*}
    M_{\phi_r,j}
        =
        \begin{bmatrix}
            \alpha_j (-1)^{r-1} & 0 \\
            0 & \beta_j (-1)^r
        \end{bmatrix};
    \end{align*}
\item if $k'_j=1$ then
    \begin{align*}
    M_{\Fil^r,j}
        =
        \begin{bmatrix}
          1 & \frac{\Theta_j}{x_j} \\
          0 & 1
        \end{bmatrix}
        \begin{bmatrix}
            u^r & 0 \\
            0 & u^{r-1}
        \end{bmatrix}
    \end{align*}
    and
    \begin{align*}
    M_{\phi_r,j}
        =
        \begin{bmatrix}
          1 & -\frac{\alpha_j}{\beta_j}\frac{px_j}{\Theta_j} \\
          0 & 1
        \end{bmatrix}
        \begin{bmatrix}
            0 & \alpha_j (-1)^{r-1} \\
            \beta_j (-1)^{r-1} & 0
        \end{bmatrix};
    \end{align*}
\item if $k'_j=\frac{3}{2}$ then
    \begin{align*}
    M_{\Fil^r,j}
        =
        \begin{bmatrix}
            u^r & 0 \\
            0 & u^{r-1}
        \end{bmatrix}
    \end{align*}
    and
    \begin{align*}
    M_{\phi_r,j}
        =
        \begin{bmatrix}
          1 & \frac{\alpha_j}{\beta_j}\frac{\Theta_j}{px_j} \\
          0 & 1
        \end{bmatrix}
        \begin{bmatrix}
            \alpha_j (-1)^r & 0 \\
            0 & \beta_j (-1)^{r-1}
        \end{bmatrix}.
    \end{align*}
\end{enumerate}
\end{prop}

\begin{proof}
By the same reason as Proposition~\ref{prop: mod p reduction, r_j>=2, k'_j half integer}, we get the description as in the statement. More precisely,
\begin{enumerate}[leftmargin=*]
  \item if $k'_j=\frac{1}{2}$, then $M_{\phi_r,j}$ is given by Proposition~\ref{prop: case phi 0.5, r=1};
  \item if $k'_j=1$ or $\frac{3}{2}$, then $M_{\phi_r,j}$ is given by Proposition~\ref{prop: case phi 1, r=1}, and for $M_{\Fil^r,j}$ with $k'_j=\frac{3}{2}$, $\frac{\Theta_j}{x_j}=0$ in $\F$.
\end{enumerate}
This completes the proof.
\end{proof}

\begin{prop}\label{prop: mod p reduction, k'_j=infty}
Assume that $r_j\geq 1$ and $k'_j=\infty$. The matrices representing $\Fil^r\cM^\diamond e_j$ and $\phi_r^{\diamond,(j)}$ are given by
\begin{align*}
M_{\Fil^r,j}
    =
    \begin{bmatrix}
        u^r & 0 \\
        0 & u^{r-r_j}
    \end{bmatrix}
\end{align*}
and
\begin{align*}
M_{\phi_r,j}
    =
    \begin{bmatrix}
        \alpha_j (-1)^r & 0 \\
        0 & \beta_j (-1)^{r-r_j}
    \end{bmatrix}.
\end{align*}
\end{prop}

\begin{proof}
By the same reason as Proposition~\ref{prop: mod p reduction, r_j>=2, k'_j half integer}, we get the description as in the statement. More precisely, $M_{\Fil^r,j}$ and $M_{\phi_r,j}$ is given by Proposition~\ref{prop: case infty}.
\end{proof}

\smallskip

\section{Applications}\label{sec: examples}
In this section, we apply our main results to give mod-$p$ reduction of some $2$-dimensional semi-stable representations. We first show that we can recover the known results of the mod-$p$ reduction in \cite{BM,GP,LP}. In fact, our method provides more strongly divisible modules when $r$ is even, which are not discovered before.  We also use the main results to compute the mod-$p$ reduction of the semi-stable non-crystalline representations when $\vr=(2,2)$ and $\vr=(1,5)$. We note that we do not try to compute the unramified characters due to the size of the paper.

Recall that if $v_p(\vx)\in R'(\vr;\vk')$ then the association $\vx\mapsto \vL$ depends on $\vk'$ (see \eqref{eq: definition of support}). In other words, $\vL$ is determined by the equation \eqref{eq: definition of vx}, and $\vDel(-1)$ depends on $l(\vr;\vk')$ (and so on $\vk'$) as explicitly described in \eqref{eq: delta(-1)}. Hence, we will often write $\vx^{l(\vr;\vk')}$ for $\vx$ in $\mathbf{Case}_\phi(\vr;\vk')$ to distinguish $\vx$ in other cases, and its $j$-th coordinate will be written as $x^{l(\vr;\vk')}_j$. For instance, if $f=2$, $\vr=(2,2)$, and $\vk'=(\frac{1}{2},1)$ then we have $l(\vr;\vk')=(1,1)$ and so we write $\vx^{(1,1)}=(x^{(1,1)}_0,x^{(1,1)}_1)$ for $\vx=(x_0,x_1)$ in this case. We also recall that we have defined the monodromy type of the Breuil modules in Definition~\ref{defi: monodromy type}.

For each $\mathbf{Case}_\phi(\vr;\vk')$ in the following examples, we denote $\cM$ be the Breuil module obtained from $\fM=\fM_{(\vth_1,\vth_2,\vrho,\vDel)}$, which is given by Theorem~\ref{theo: main 2}. Also, we denote $\rhobar=\Tst^\ast(\cM)$ for each example.

\subsection{Recovering the case $f=1$}\label{subsec: f=1}
For the case $f=1$, it is not difficult to see that Theorem~\ref{theo: main} recovers the results on strongly divisible modules in \cite{BM,GP,LP}, which is the goal of this subsection. In fact, we even found strongly divisible modules that were not discovered in \cite{GP}. As $f=1$, we write $(\Lambda,\Theta,x)$ for $(\Lambda_0,\Theta_0,x_0)$ and $(r,k',l)$ for $(r_0,k'_0,l(r_0;k'_0))$. We also write $\Delta$ for $\Delta^{(0)}$ and $\mathrm{MT}$ for $\mathrm{MT}_0$. 

By \eqref{eq: From the setting 2}, we have $\Lambda=\lambda$ with $v_p(\lambda)=\frac{r-1}{2}$. Moreover, if $(r,l)\neq(2m,m+1)$ we have $x=\frac{1}{p}(p\fL-\fL-\delta_l(-1))$, and if $(r,l)=(2m,m+1)$ then $x$ is a solution of the equation $x=\frac{1}{p}(p\fL-\fL-\delta_{m+1}(-1,x))$, where $\delta_{l}(-1,x)$ is explicitly written in \eqref{eq: delta(-1)}.

\subsubsection{\textbf{The case $r=1$}}
If $r=1$ then it is easy to see from Proposition~\ref{prop: system of equations and inequalities sharpened} that only the cases $k'=\frac{1}{2}$ and $k'=1$ survive as in Table~\ref{tab: r=1}.

\begin{small}
\begin{table}[htbp]
  \centering
  \begin{tabular}{|c||c|c|c|}
  \hline
    $k'$ & $v_p(\Theta)$ & $v_p(x)$ & $\Delta(-1)$ \\ \hline\hline
    $k'=\frac{1}{2}$ & $(-\infty,\min\{0,v_p(x)\}]$ & $(-\infty,\infty]$ & $\delta_1(-1)=0$ \\\hline
    $k'=1$ & $v_p(x)$ & $(-\infty,0]$ & $\delta_1(-1)=0$ \\\hline
  \end{tabular}
  \caption{For $r=1$}    \label{tab: r=1}
\end{table}
\end{small}
We let $\xi:=\frac{1}{p}(p\fL-\fL)$, which is the same as $x^{(1)}$. By looking at the results in Proposition~\ref{prop: case phi 0.5, r=1} and Proposition~\ref{prop: case phi 1, r=1}, the case $k'=1$ is contained in the case $k'=\frac{1}{2}$. Also, by looking at the Breuil modules in Proposition~\ref{prop: mod p reduction, r_j=1}, we have
$\rhobar|_{I_{\Qp}}\cong
\begin{tiny}
\begin{bmatrix}
    \omega & * \\ 0 & 1
\end{bmatrix}
\end{tiny}$ 
and the extension classes are given by the following:
\begin{itemize}[leftmargin=*]
\item Assume $v_p(\xi)\geq 0$ in $\mbox{Case}_\phi(1;\frac{1}{2})$, so that $v_p(\Theta)\leq 0$.
\begin{itemize}[leftmargin=*]
    \item if we choose $v_p(\Theta)<0$, then $\rhobar$ is split.
    \item if we choose $v_p(\Theta)=0$, then $\rhobar$ is non-split with $\mathrm{MT}=1$. Note that if $v_p(\Theta)=0=v_p(\xi)$, then this case coincides with $\mbox{Case}_\phi(1;1)$.
\end{itemize}
\item Assume $v_p(\xi)<0$ in $\mbox{Case}_\phi(1;\frac{1}{2})$, so that $v_p(\Theta)\leq v_p(\xi)$.
\begin{itemize}[leftmargin=*]
    \item if we choose $v_p(\Theta)<v_p(\xi)$, then $\rhobar$ is split.
    \item if we choose $v_p(\Theta)=v_p(\xi)$, then $\rhobar$ is non-split with the monodromy type is $\mathrm{MT}=0$. Note that this case coincides with $\mbox{Case}_\phi(1;1)$.
\end{itemize}
\end{itemize}
Finally, we point out that $v_p(\xi)\geq 0$ if and only if $v_p(\fL)\geq 1$, and we summarize the non-split mod-$p$ reduction when $r=1$ in the table, Table~\ref{tab: r=1, reduction}.
\begin{small}
\begin{table}
  \begin{tabular}{|c||c|c|}
  \hline
      $\xi$ & $\rhobar|_{I_{\Qp}}$ & $\mathrm{MT}$ \\\hline\hline
      $v_p(\xi)\geq 0$
      & $
      \begin{bmatrix}
          \omega & * \\ 0 & 1
      \end{bmatrix}$
      & $1$ \\\hline
      $v_p(\xi)<0$
      & $
      \begin{bmatrix}
          \omega & * \\ 0 & 1
      \end{bmatrix}$
      & $0$ \\\hline
  \end{tabular}
  \caption{Non-split mod-$p$ reduction when $r=1$.}    \label{tab: r=1, reduction}
\end{table}
\end{small}

\subsubsection{\textbf{The case $r=2$}}
If $r=2$ then it is also easy to see from Proposition~\ref{prop: system of equations and inequalities sharpened} that only the cases in the table, Table~\ref{tab: r=2}, survives.
\begin{small}
\begin{table}[htbp]
  \begin{tabular}{|c||c|c|c|}
  \hline
    $k'$ & $v_p(\Theta)$ & $v_p(x)$ & $\Delta(-1)$ \\ \hline\hline
    $k'=1$ & $-\frac{1}{2}$ & $[-\frac{1}{2},\infty]$ & $\delta_1(-1)=-1$ \\\hline
    $k'=\frac{3}{2}$ & $[-\frac{1}{2},0]$ & $-\frac{1}{2}$ & $\delta_2(-1,x)=-1+\frac{1}{x+1}$ \\\hline
    $k'=2$ & $\frac{1}{2}+v_p(x)$ & $(-\infty,-\frac{1}{2}]$ & $\delta_2(-1,x)=-1+\frac{1}{x+1}$ \\ \hline
  \end{tabular}
  \caption{For $r=2$}    \label{tab: r=2}
\end{table}
\end{small}

We set $\xi:=\frac{1}{p}(p\fL-\fL+1)$, which is the same as $x^{(1)}$. We also set $\xi':=x^{(2)}$, which satisfies the equation $p\xi'=p\fL-\fL+1-\frac{1}{\xi'+1}=p\xi-\frac{1}{\xi'+1}$. Note that the existence of $\xi'$ depends on $\fL$ while $\fL$ is uniquely determined by $\xi'$.
By looking at the Breuil modules in Proposition~\ref{prop: mod p reduction, r_j>=2, k'_j half integer} and Proposition~\ref{prop: mod p reduction, r_j>=2, k'_j integer}, we compute $\rhobar|_{I_{\Qp}}$ as follow:
\begin{itemize}[leftmargin=*]
\item Assume $v_p(\xi)\geq -\frac{1}{2}$.
\begin{itemize}[leftmargin=*]
    \item We first consider $\mathbf{Case}_\phi(2;1)$. In this case it is easy to see that $\rhobar|_{I_{\Qp}}^{ss}\cong\omega\oplus\omega$ with $\mathrm{MT}=0$, by looking at the Breuil modules. By tedious computation on $\Mat(\phi_2)$, one can see that $\rhobar$ is non-split if and only if $v_p(p\xi^2-4)>0$.
    \item We now consider $\mathbf{Case}_\phi(2;\frac{3}{2})$. \emph{Note that the strongly divisible modules in this case was not discovered before.} Assuming $E$ is large enough, one can see that there exists $\xi'$ with $v_p(\xi')=-\frac{1}{2}$ such that $p\xi'=p\xi-\frac{1}{\xi'+1}$. It is also easy to see that $\rhobar|_{I_{\Qp}}^{ss}\cong \omega\oplus\omega$, by looking at the Breuil modules, and we can further sharpen the extension class as follow:
    \begin{itemize}[leftmargin=*]
        \item if $v_p(\Theta)=-\frac{1}{2}$, then we have $\mathrm{MT}=0$, and $\rhobar$ is non-split if and only if $v_p(p\xi^2-4)>0$. Note that this case coincides with $\mathbf{Case}_\phi(2;1)$.
        \item if $v_p(\Theta)\in (-\frac{1}{2},0)$, we have $\mathrm{MT}=0$ and one can see that $\rhobar$ is split.
        \item if $v_p(\Theta)=0$, we have $\mathrm{MT}=1$ and so $\rhobar$ is non-split. Note that this case coincides with $\mathbf{Case}_\phi(2;2)$ when $v_p(\xi')=-\frac{1}{2}$.
    \end{itemize}
\end{itemize}
\item Assume $v_p(\xi)<-\frac{1}{2}$. In this case we consider $\mathbf{Case}_\phi(2;2)$, and one can prove that there exists $\xi'\in E$ with $v_p(\xi)=v_p(\xi')$ satisfying $p\xi'=p\xi-\frac{1}{\xi'+1}$. It is easy to see that $\rhobar$ is irreducible satisfying $\rhobar|_{I_{\Qp}}\cong\omega_2^2\oplus\omega_2^{2p}$.
\end{itemize}
Finally, we point out that $v_p(\xi)\geq -\frac{1}{2}$ if and only if $v_p(\fL-1)\geq \frac{1}{2}$, and we summarize the non-split mod-$p$ reduction when $r=2$ in the table, Table~\ref{tab: r=2, reduction}.

\begin{small}
\begin{table}[htbp]
  \begin{tabular}{|c||c|c|}
  \hline
    $\xi$ & $\rhobar|_{I_{\Qp}}$ & $\mathrm{MT}$ \\\hline\hline
    $v_p(\xi)\geq -\frac{1}{2}$
    & $
    \begin{bmatrix}
        \omega & * \\ 0 & \omega
    \end{bmatrix}$
    & $1$ \\\hline
    \begin{tabular}{c}
        $v_p(\xi)\geq -\frac{1}{2}$, $v_p(p\xi^2-4)>0$
    \end{tabular}
    & $
    \begin{bmatrix}
        \omega & * \\ 0 & \omega
    \end{bmatrix}$
    & $0$ \\\hline
    $v_p(\xi)<-\frac{1}{2}$
    & $\omega_2^2\oplus\omega_2^{2p}$
    & $0$ \\\hline
  \end{tabular}
  \caption{Non-split mod-$p$ reduction when $r=2$.}    \label{tab: r=2, reduction}
\end{table}
\end{small}

\subsubsection{\textbf{The case $r=2m+1\geq 3$}}
We now consider the case $r=2m+1\geq 3$. It is also easy to see from Proposition~\ref{prop: system of equations and inequalities sharpened} that only the cases in the following table, Table~\ref{tab: r=2m+1}, survive.
\begin{small}
\begin{table}[htbp]
\centering
\begin{tabular}{|c||c|c|c|}
\hline
  $k'$
   & $v_p(\Theta)$ & $v_p(x)$ & $\Delta(-1)$ \\ \hline\hline
  $k'=m$
   & $-1$ & $[-1,\infty]$ & $\delta_m(-1)=-H_{m-1}-H_{m+1}$ \\\hline
  $k'=m+\frac{1}{2}$
   & $[-1,\min\{0,v_p(x)\}]$ & $[-1,\infty]$ & $\delta_{m+1}(-1)=-2H_m$ \\\hline
  $m+1\leq k'=k\leq r-1$
   & $k-m-1+v_p(x)$ & $[m-k,m-k+1]$ & $\delta_{r-k+1}(-1)=-H_{k-1}-H_{r-k}$\\\hline
  $k'=r$
   & $m+v_p(x)$ & $(-\infty,-m]$ & $\delta_1=-H_{r-1}$\\ \hline
\end{tabular}
\caption{For $r=2m+1\geq 3$}     \label{tab: r=2m+1}
\end{table}
\end{small}

We set $\xi:=\frac{1}{p}(p\fL-\fL+H_{m-1}+H_{m+1})$ and $\xi':=\frac{1}{p}(p\fL-\fL+2H_m)$. It is immediate that we have $\xi=x^{(m)}$ and $\xi'=x^{(m+1)}$. By looking at the Breuil modules in the corresponding propositions in Proposition~\ref{prop: mod p reduction, r_j>=2, k'_j half integer} and Proposition~\ref{prop: mod p reduction, r_j>=2, k'_j integer}, we compute $\rhobar|_{I_{\Qp}}$ as follow:
\begin{itemize}[leftmargin=*]
\item Assume $v_p(\xi)>-1$. In this case we consider $\mathbf{Case}_\phi(r;m+\frac{1}{2})$. Note that we have $v_p(\xi')=-1$ as $v_p(\xi-\xi')=-1$. It is easy to see that $\rhobar|_{I_{\Qp}}\cong\omega_2^{m+1+mp}\oplus\omega_2^{m+(m+1)p}$, by looking at the Breuil modules. Note that this case coincides with $\mathbf{Case}_\phi(r;m)$ when $v_p(x^{(m)})>-1$.
\item Assume $v_p(\xi)=-1$. In this case there are two non-homothetic lattices whose mod-$p$ reductions are non-split and reducible.
\begin{itemize}[leftmargin=*]
    \item We first consider $\mathbf{Case}_\phi(r;m+2)$. In this case we have $\xi=x^{(m)}$, and it is easy to see that $\rhobar|_{I_{\Qp}}\cong
\begin{tiny}\begin{bmatrix}
    \omega^m & * \\ 0 & \omega^{m+1}
\end{bmatrix}\end{tiny}$ with $\mathrm{MT}=1$.
    \item We also consider $\mathbf{Case}_\phi(r;m+\frac{1}{2})$. Note that $\xi'=x^{(m+1)}$ and $v_p(\xi')\geq -1$ as $v_p(\xi-\xi')=-1$. In this case it is easy to see that $\rhobar|_{I_{\Qp}}\cong
        \begin{tiny}\begin{bmatrix}
         \omega^{m+1} & * \\ 0 & \omega^{m}
        \end{bmatrix}\end{tiny}$, and we can further sharpen the extension classes as follow:
       \begin{itemize}[leftmargin=*]
       \item if $-1\leq v_p(\Theta)= v_p(\xi')<0$ then $\rhobar$ is non-split with $\mathrm{MT}=0$. Note that this case coincides with $\mathbf{Case}_\phi(r;m+1)$, and if $v_p(\xi')=-1$ this case also coincides with $\mathbf{Case}_\phi(r;m)$.
       \item if $v_p(\Theta)=0\leq v_p(\xi')$ then we have $\mathrm{MT}=1$.
       \item otherwise, $\rhobar|_{I_{\Q_p}}\cong \omega^{m+1}\oplus\omega^m$. In particular, it is split.
       \end{itemize}
\end{itemize}
\item Assume $v_p(\xi)\in (m-k,m-k+1)$ for $m+2\leq k\leq r-1$. Equivalently, we have $v_p(x^{(r-k+1)})\in (m-k,m-k+1)$. In this case, we consider $\mathbf{Case}_\phi(r;k)$, and it is easy to see that $\rhobar|_{I_{\Q_p}}\cong \omega_2^{k+(r-k)p}\oplus\omega_2^{(r-k)+kp}$, by looking at the Breuil modules.
\item Assume $v_p(\xi)=m-k$ for $m+2\leq k\leq r-1$. Equivalently, we have $v_p(x^{(r-k+1)})=v_p(x^{(r-k)})=m-k$. In this case there are two non-homothetic lattices whose mod-$p$ reductions are non-split and reducible.
\begin{itemize}[leftmargin=*]
    \item We first consider $\mathbf{Case}_\phi(r;k)$. In this case it is easy to see that $\rhobar$ is non-split and $\rhobar|_{I_{\Qp}}\cong
    \begin{tiny}\begin{bmatrix}
    \omega^k & * \\ 0 & \omega^{r-k}
    \end{bmatrix}\end{tiny}$ with $\mathrm{MT}=0$.
    \item We also consider $\mathbf{Case}_\phi(r;k+1)$. In this case we have $\rhobar|_{I_{\Qp}}\cong
    \begin{tiny}\begin{bmatrix}
    \omega^{r-k} & * \\ 0 & \omega^{k}
    \end{bmatrix}\end{tiny}$ with $\mathrm{MT}=1$.
\end{itemize}
\item Assume $v_p(\xi)<-m$. Equivalently, we have $v_p(x^{(1)})<-m$. From $\mathbf{Case}_\phi(r;r)$, it is easy to see that $\rhobar|_{I_{\Q_p}}\cong \omega_2^r\oplus\omega_2^{rp}$, by looking at the Breuil modules.
\end{itemize}
Finally, we summarize the non-split mod-$p$ reduction when $r=2m+1\geq 3$ in the table, Table~\ref{tab: r=2m+1, reduction}.
\begin{small}
\begin{table}[htbp]
\centering
\begin{tabular}{|c||c|c|}
\hline
    $\xi$ & $\rhobar|_{I_{\Qp}}$ & $\mathrm{MT}$ \\\hline\hline
    $v_p(\xi)>-1$
    & $\omega_2^{m+1+mp}\oplus\omega_2^{m+(m+1)p}$
    & $0$ \\\hline
    \multirow{2}{*}[-1em]{$v_p(\xi)=-1$}
    & $
    \begin{bmatrix}
        \omega^m & * \\ 0 & \omega^{m+1}
    \end{bmatrix}$
    & $1$ \\\cline{2-3}
    & $
    \begin{bmatrix}
        \omega^{m+1} & * \\ 0 & \omega^m
    \end{bmatrix}$
    & $
    \begin{cases}
        0 &\mbox{if } -1\leq v_p(\xi')<0;\\
        1 &\mbox{if } v_p(\xi')\geq 0
    \end{cases}$ \\\hline
    \begin{tabular}{c}
        $m-k<v_p(\xi)<m-k+1$\\ $(m+2\leq k\leq r-1)$
    \end{tabular}
    & $\omega_2^{k+(r-k)p}\oplus\omega_2^{r-k+kp}$
    & $0$ \\\hline
    \multirow{2}{*}[-0.5em]{
    \begin{tabular}{c}
        $v_p(\xi)=m-k$ \\ $(m+2\leq k\leq r-1)$
    \end{tabular}}
    & $
    \begin{bmatrix}
        \omega^k & * \\ 0 & \omega^{r-k}
    \end{bmatrix}$
    & $0$ \\\cline{2-3}
    & $
    \begin{bmatrix}
        \omega^{r-k} & * \\ 0 & \omega^k
    \end{bmatrix}$
    & $1$ \\\hline
    $v_p(\xi)<-m$
    & $\omega_2^{r}\oplus\omega_2^{rp}$
    & $0$ \\\hline
\end{tabular}
\caption{Non-split mod-$p$ reduction when $r=2m+1\geq 3$}     \label{tab: r=2m+1, reduction}
\end{table}
\end{small}

\subsubsection{\textbf{The case $r=2m> 2$}}
Finally, we treat the case $r=2m\geq 4$. It is again easy to see from Proposition~\ref{prop: system of equations and inequalities sharpened} that only the cases in the table, Table~\ref{tab: r=2m}, survive.

\begin{small}
\begin{table}[htbp]
\centering
\begin{tabular}{|c||c|c|c|}
\hline
  $k'$
    & $v_p(\Theta)$ & $v_p(x)$ & $\Delta(-1)$ \\\hline\hline
  $k'=m$
    & $-\frac{1}{2}$ & $[-\frac{1}{2},\infty]$
    & $\delta_m(-1)=-H_{m-1}-H_m$ \\\hline
  $k'=m+\frac{1}{2}$
    & $[-\frac{1}{2},0]$ & $-\frac{1}{2}$
    & \multirow{2}{*}{
    \begin{tabular}{r}
        $\delta_{m+1}(-1)=-H_{m-1}-H_m$ \\ $+\frac{1}{m^2(x+H_m+H_{m-1})}$
    \end{tabular}} \\[0.7ex]\cline{1-3}
  $k'=m+1$
    & $\frac{1}{2}+v_p(x)$ & $[-\frac{3}{2},-\frac{1}{2}]$ &  \\[0.7ex]\hline
  \begin{tabular}{c}
    $k'=k+\frac{1}{2}$ \\ $(m+1\leq k\leq r-1)$
  \end{tabular}
    & $[-1,0]$ & $m-k-\frac{1}{2}$
    & $\delta_{r-k}(-1)=-H_k-H_{r-k-1}$ \\\hline
  \begin{tabular}{c}
    $k'=k$\\$(m+2\leq k\leq r-1)$
  \end{tabular}
    & \begin{tabular}{c} $k-m-\frac{1}{2}$\\ $\,\,\,+v_p(x)$\end{tabular} & $[m-k-\frac{1}{2},m-k+\frac{1}{2}]$
    & $\delta_{r-k+1}(-1)=-H_{k-1}-H_{r-k}$ \\\hline
  $k'=r$
    & $m-\frac{1}{2}+v_p(x)$ & $(-\infty,-m+\frac{1}{2}]$
    & $\delta_1(-1)=-H_{r-1}$
   \\ \hline
\end{tabular}
\caption{For $r=2m\geq 4$}     \label{tab: r=2m}
\end{table}
\end{small}

We set $\xi:=\frac{1}{p}(p\fL-\fL+H_{m-1}+H_m)$, which is the same as $x^{(m)}$. We also set $\xi':=x^{(m+1)}$, which satisfies the equation $p\xi'=p\xi-\frac{1}{m^2(\xi'+H_m+H_{m-1})}$. Note that the existence of $\xi'$ depends on $\fL$, while $\fL$ is uniquely determined by $\xi'$. By looking at the Breuil modules in the corresponding propositions in Proposition~\ref{prop: mod p reduction, r_j>=2, k'_j half integer} and Proposition~\ref{prop: mod p reduction, r_j>=2, k'_j integer}, we compute $\rhobar|_{I_{\Qp}}$ as follow:
\begin{itemize}[leftmargin=*]
\item Assume $v_p(\xi)\geq -\frac{1}{2}$.
\begin{itemize}[leftmargin=*]
    \item We first consider $\mathbf{Case}_\phi(r;m)$. In this case it is easy to see that $\rhobar|_{I_{\Qp}}^{ss}\cong \omega^m\oplus\omega^m$ with $\mathrm{MT}=0$, by looking at the Breuil modules. By tedious computation on $\Mat(\phi_r)$, one can see that this $\rhobar$ is non-split if and only if $v_p(p\xi^2-\frac{4}{m^2})>0$.
    \item We now consider $\mathbf{Case}_\phi(r;m+\frac{1}{2})$. \emph{Note that the strongly divisible modules in this case was not discovered before.} Assuming $E$ is large enough, one can see that there exists $\xi'$ with $v_p(\xi')=-
    \frac{1}{2}$ such that $p\xi'=p\xi-\frac{1}{m^2(\xi'+H_m+H_{m-1})}$. It is easy to see that $\rhobar|_{I_{\Qp}}^{ss}\cong\omega^m\oplus\omega^m$, and we further sharpen the extension classes as follow:
    \begin{itemize}[leftmargin=*]
        \item if $v_p(\Theta)=-\frac{1}{2}$, then we have $\mathrm{MT}=0$, and $\rhobar$ is non-split if and only if $v_p(p\xi^2-\frac{4}{m^2})>0$. Note that this case coincides with $\mathbf{Case}_\phi(r;m)$.
        \item if $v_p(\Theta)=0$, then we have $\mathrm{MT}=1$. Note that this case coincides with $\mathbf{Case}_\phi(r;m+1)$ when $v_p(x^{(m+1)})=-\frac{1}{2}$.
        \item if $v_p(\Theta)\in(-\frac{1}{2},0)$, then $\rhobar|_{I_{\Qp}}\cong\omega^m\oplus\omega^m$. In particular, it is split.
    \end{itemize}
\end{itemize}
\item Assume $v_p(\xi)\in(-\frac{3}{2},-\frac{1}{2})$. We consider $\mathbf{Case}_\phi(r;m+1)$, and one can check that there exists $\xi'$ with $v_p(\xi)=v_p(\xi')$ satisfying $p\xi'=p\xi-\frac{1}{m^2(\xi'+H_m+H_{m-1})}$. In this case it is easy to see that $\rhobar|_{I_{\Qp}}\cong\omega_2^{m+1+(m-1)p}\oplus\omega_2^{m-1+(m+1)p}$, by looking at the Breuil modules.
\item Assume $v_p(\xi)=-\frac{3}{2}$. Equivalently, we have $v_p(x^{(m-1)})=-\frac{3}{2}$. In this case, we consider $\mathbf{Case}_\phi(r;m+\frac{3}{2})$, and there are two non-homothetic lattices whose mod-$p$ reductions are non-split and reducible.
\begin{itemize}[leftmargin=*]
    \item If $v_p(\Theta)=-1$, then it is easy to see that $\rhobar$ is non-split and $\rhobar|_{I_{\Qp}}\cong
    \begin{tiny}\begin{bmatrix}
    \omega^{m+1} & * \\ 0 & \omega^{m-1}
    \end{bmatrix}\end{tiny}$ with $\mathrm{MT}=0$. Note that this case coincides with $\mathbf{Case}_\phi(r;m+1)$ when $v_p(x^{(m+1)})=-\frac{3}{2}$.
    \item If $v_p(\Theta)=0$, then it is easy to see that $\rhobar|_{I_{\Qp}}\cong
    \begin{tiny}\begin{bmatrix}
    \omega^{m-1} & * \\ 0 & \omega^{m+1}
    \end{bmatrix}\end{tiny}$ with $\mathrm{MT}=1$. Note that this case coincides with $\mathbf{Case}_\phi(r;m+2)$ when $v_p(x^{(m-1)})=-\frac{3}{2}$.
    \item If $v_p(\Theta)\in(-1,0)$, then $\rhobar|_{I_{\Q_p}}\cong \omega^{m+1}\oplus\omega^{m-1}$. In particular, it is split.
\end{itemize}
\item Assume $v_p(\xi)\in (m-k-\frac{1}{2},m-k+\frac{1}{2})$ for $m+2\leq k\leq r-1$. Equivalently, we have $v_p(x^{(r-k+1)})\in (m-k-\frac{1}{2},m-k+\frac{1}{2})$. In this case we consider $\mathbf{Case}_\phi(r;k)$, and it is easy to see that $\rhobar|_{I_{\Qp}}\cong\omega_2^{k+(r-k)p}\oplus\omega_2^{r-k+kp}$, by looking at the Breuil modules.
\item Assume $v_p(\xi)=m-k-\frac{1}{2}$ for $m+2\leq k\leq r-1$. Equivalently, we have $v_p(x^{(r-k+1)})=v_p(x^{(r-k)})=m-k-\frac{1}{2}$. In this case we consider $\mathbf{Case}_\phi(r;k+\frac{1}{2})$, and there are two non-homothetic lattices whose mod-$p$ reductions are non-split and reducible.
\begin{itemize}[leftmargin=*]
    \item If $v_p(\Theta)=-1$, then it is easy to see that $\rhobar$ is non-split and $\rhobar|_{I_{\Qp}}\cong
    \begin{tiny}\begin{bmatrix}
    \omega^{k} & * \\ 0 & \omega^{r-k}
    \end{bmatrix}\end{tiny}$ with $\mathrm{MT}=0$. Note that this case coincides with $\mathbf{Case}_\phi(r;k)$ when $v_p(x^{(r-k+1)})=m-k-\frac{1}{2}$.

    \item If $v_p(\Theta)=0$, then it is easy to see that $\rhobar$ is non-split with $\rhobar|_{I_{\Qp}}\cong
    \begin{tiny}\begin{bmatrix}
    \omega^{r-k} & * \\ 0 & \omega^{k}
    \end{bmatrix}\end{tiny}$ and $\mathrm{MT}=1$. Note that this case coincides with $\mathbf{Case}_\phi(r;k+1)$ when $v_p(x^{(r-k)})=m-k-\frac{1}{2}$.
    \item If $v_p(\Theta)\in(-1,0)$, then $\rhobar|_{I_{\Q_p}}\cong\omega^{k}\oplus\omega^{r-k}$. In particular, it is split.
\end{itemize}
\item Assume $v_p(\xi)<-m+\frac{1}{2}$. Equivalently, we have $v_p(x^{(1)})<-m+\frac{1}{2}$. In this case it is easy to see that  $\rhobar|_{I_{\Qp}}\cong\omega_2^r\oplus\omega_2^{rp}$, by looking at the Breuil modules.
\end{itemize}
Finally, we summarize the non-split mod-$p$ reduction when $r=2m\geq 4$ in the table, Table~\ref{tab: r=2m, reduction}.

\begin{small}
\begin{table}[htbp]
\centering
\begin{tabular}{|c||c|c|}
\hline
    $\xi$ & $\rhobar|_{I_{\Qp}}$ & $\mathrm{MT}$ \\\hline\hline
    $v_p(\xi)\geq -\frac{1}{2}$
    & $
    \begin{bmatrix}
        \omega^{m+1} & * \\ 0 & \omega^{m+1}
    \end{bmatrix}$
    & $1$ \\\hline
    \begin{tabular}{c}
        $v_p(\xi)\geq -\frac{1}{2}$, $v_p(p\xi^2-\frac{4}{m^2})>0$
    \end{tabular}
    & $
    \begin{bmatrix}
        \omega^{m+1} & * \\ 0 & \omega^{m+1}
    \end{bmatrix}$
    & $0$ \\\hline
    \begin{tabular}{c}
        $m-k-\frac{1}{2}<v_p(\xi)<m-k+\frac{1}{2}$\\ $(m+1\leq k\leq r-1)$
    \end{tabular}
    & $\omega_2^{k+(r-k)p}\oplus\omega_2^{r-k+kp}$
    & $0$ \\\hline
    \multirow{2}{*}[-0.5em]{
    \begin{tabular}{c}
        $v_p(\xi)=m-k-\frac{1}{2}$ \\ $(m+1\leq k\leq r-1)$
    \end{tabular}}
    & $
    \begin{bmatrix}
        \omega^k & * \\ 0 & \omega^{r-k}
    \end{bmatrix}$
    & $0$ \\\cline{2-3}
    & $
    \begin{bmatrix}
        \omega^{r-k} & * \\ 0 & \omega^k
    \end{bmatrix}$
    & $1$ \\\hline
    $v_p(\xi)<-m+\frac{1}{2}$
    & $\omega_2^{r}\oplus\omega_2^{rp}$
    & $0$ \\\hline
\end{tabular}
\caption{Non-split mod-$p$ reduction when $r=2m\geq 4$}     \label{tab: r=2m, reduction}
\end{table}
\end{small}

\subsection{An example of the case $f=2$: $\vr=(2,2)$}
In this subsection, we apply our main result, Theorem~\ref{theo: main 2}, to compute the mod-$p$ reduction of $2$-dimensional semi-stable representations of $G_{\Q_{p^2}}$ with parallel Hodge--Tate weights $(0,2)$. We set $\vr=(2,2)$ and $r=2$. We assume that $2<p-1$.

We first need to compute the valid $\vk'\in\cK(\vr)$, and it is immediate that $\cJ_\infty(\vk')=\emptyset$ for all $\vk'\in\cK(\vr)$ as $\vr$ is parallel. Recall that we say that $\vk'$ is valid if $R(\vr;\vk')\neq\emptyset$, i.e., the equations and inequalities have common solutions. It turns out that it is enough to consider those $\vk'\in\cK(\vr)$ such that $R_{int}(\vr;\vk')\neq\emptyset$, and such $\vk'$ is listed as follow:
$$(\tfrac{1}{2},1),(\tfrac{1}{2},2),(1,\tfrac{1}{2}),(1,1),(1,\tfrac{3}{2}), (\tfrac{3}{2},1),(\tfrac{3}{2},2),(2,\tfrac{1}{2}),(2,\tfrac{3}{2}).$$
Moreover, we describe the areal support $R(\vr;\vk')$ for each such $\vk'$ in the table, Table~\ref{tab: J((2,2))}.
\begin{small}
\begin{table}[htbp]
  \centering
  \begin{tabular}{|c||c|c|}
    \hline
    $\vk'$ & $R(\vr;\vk')$ & $\vx$ \\\hline\hline
    $(\tfrac{1}{2},1)$ & $t_0\geq -1$, $t_1\geq 0$ & {$\vx^{(1,1)}$} \\\hline
    $(\tfrac{1}{2},2)$ & $t_1\leq -1$, $t_0-t_1\geq 0$ & {$\vx^{(1,1)}$} \\\hline
    $(1,\tfrac{1}{2})$ & $t_0\geq 0$, $t_1\geq -1$ & {$\vx^{(1,1)}$} \\\hline
    $(1,1)$ & $t_0\geq -1$, $t_1\geq -1$, $t_0+t_1\geq -1$ & {$\vx^{(1,1)}$} \\\hline
    $(1,\tfrac{3}{2})$ & $-1\leq t_1<0$, $t_0+t_1\geq -1$ & {$\vx^{(1,2)}$} \\\hline
    $(\tfrac{3}{2},1)$ & $-1\leq t_0<0$, $t_0+t_1\geq -1$ & {$\vx^{(2,1)}$} \\\hline
    $(\tfrac{3}{2},2)$ &
        \begin{tabular}{c}
          $t_0+t_1\leq -1$, $t_0-t_1\geq 0$,\\ $t_1\geq 0$, $t_0<0$
        \end{tabular} & {$\vx^{(2,2)}$} \\\hline
    $(2,\tfrac{1}{2})$ & $t_0\leq -1$, $t_0-t_1\leq 0$ & {$\vx^{(2,1)}$} \\\hline
    $(2,\tfrac{3}{2})$ &
        \begin{tabular}{c}
          $t_0+t_1\leq -1$, $t_0-t_1\leq 0$, \\ $t_0\geq 0$, $t_1<0$
        \end{tabular} & {$\vx^{(2,2)}$} \\
    \hline
  \end{tabular}
  \caption{Areal support for $\vr=(2,2)$}\label{tab: J((2,2))}
\end{table}
\end{small}

We further note that $\Supp(\vr;\vk')\subset\Supp(\vr;(1,1))$ for $\vk'\in\{(\tfrac{1}{2},1),(1,\tfrac{1}{2}),(1,\tfrac{3}{2}),(\tfrac{3}{2},1)\}$. Hence, we exclude these redundancies and only consider
\begin{equation}\label{eq: list of strongly valid for 2,2}
(\tfrac{1}{2},2),(1,1),(\tfrac{3}{2},2),(2,\tfrac{1}{2}),(2,\tfrac{3}{2}),
\end{equation}
and we write $J(\vr)$ for the set of $\vk'$ listed in \eqref{eq: list of strongly valid for 2,2}.

Each areal support $R(\vr;\vk')$ for $\vk'\in J(\vr)$ is depicted in Figure~\ref{fig: vr=(2,2)}. We point out that the pairs of (half) integers in the figure indicate $\vk'\in J(\vr)$, and that each polygon indicate the areal support $R(\vr;\vk')$. Note that the figure does not carefully describe the boundary of each $R(\vr;\vk')$. More precisely, the point $(-1,0)$ (resp. $(0,-1)$) is not included in $R(\vr;(2,\frac{3}{2}))$ (resp. $R(\vr;(\frac{3}{2},2))$). 
\begin{figure}[htbp]
  \centering
  \includegraphics[scale=0.5]{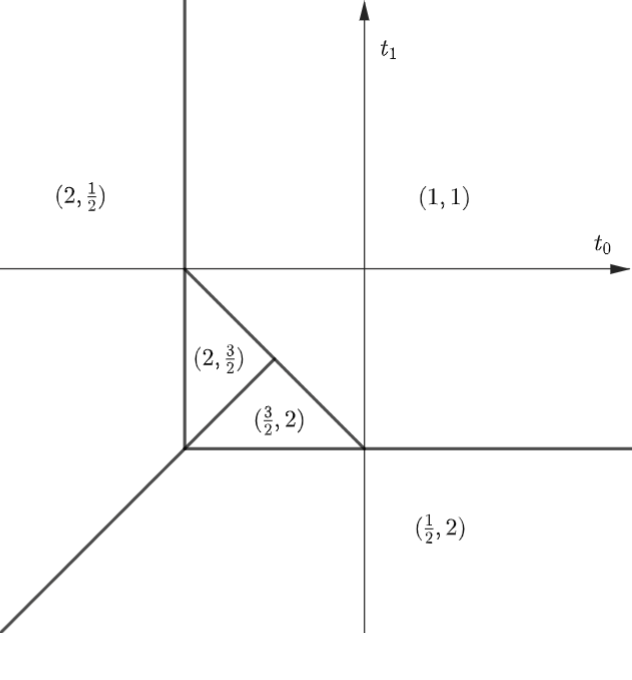}
  \caption{$R(\vr;\vk')$ for $\vr=(2,2)$ and $\vk'\in J(\vr)$}\label{fig: vr=(2,2)}
\end{figure}

For the rest of this subsection, we compute the mod-$p$ reduction of semi-stable representations in each $\mathbf{Case}_\phi(\vr;\vk')$ for $\vk'\in J(\vr)$. Then we summarize the mod-$p$ reduction when $\vr=(2,2)$ at the end of this subsection, \S\ref{subsubsec: summary for 2,2}.

\subsubsection{\textbf{For $\vk'=(\frac{1}{2},2)$ or $\vk'=(2,\frac{1}{2})$}}\label{subsubsec: vr=(2,2), vk=(1/2,2),(2,1/2)}
We only consider the case $\vk'=(\tfrac{1}{2},2)$, as the case $\vk'=(2,\frac{1}{2})$ is symmetric. In this section of paragraph, by $\vk'$ we always mean $(\frac{1}{2},2)$.

In this case, we have $\vx=\vx^{(1,2)}$, and we will write $\cM$ for the Breuil modules. From the results in \S\ref{subsec: mod p reduction}, the Breuil module $\cM$ is determined by the following data:
\begin{itemize}[leftmargin=*]
  \item $M_{N,0}=
  \begin{bmatrix}
    0 & 0 \\
    \frac{1-u^{p-1}}{\Theta_0} & 0
  \end{bmatrix}\quad\&\quad  M_{N,1}=
  \begin{bmatrix}
    0 & 0 \\
    \frac{1-u^{p-1}}{\Theta_1} & 0
  \end{bmatrix}$;
  \item $M_{\Fil^2,0}=
  \begin{bmatrix}
    1 & 0 \\
    \frac{x_0}{\Theta_0} & 1
  \end{bmatrix}
  \begin{bmatrix}
    1 & 0 \\
    \frac{\alpha_0 u}{\beta_0\Theta_1} & u^2
  \end{bmatrix}\quad\&\quad M_{\Fil^2,1}=
  \begin{bmatrix}
    1 & \frac{\Theta_1}{x_1} \\
    0 & 1
  \end{bmatrix}
  \begin{bmatrix}
    u^2 & -\frac{\alpha_1 u}{\beta_1\Theta_0} \\
    0 & 1
  \end{bmatrix}$;
  \item $M_{\phi_2,0}=
  \begin{bmatrix}
    1 & 0 \\
    \frac{u^{p-1}}{\Theta_1} & 1
  \end{bmatrix}
  \begin{bmatrix}
    \alpha_0 & 0 \\
    0 & \beta_0
  \end{bmatrix}\quad\&\quad
  M_{\phi_2,1}=
  \begin{bmatrix}
    1 & -\frac{\alpha_1}{\beta_1}\frac{p^2x_1}{\Theta_1} \\
    \frac{u^{p-1}}{\Theta_0} & 1
  \end{bmatrix}
  \begin{bmatrix}
    0 & \alpha_1 \\
    -\beta_1 & 0
  \end{bmatrix}.$
\end{itemize}
We further have the identities $v_p(\Theta_0)=v_p(x_1)$ and $v_p(\Theta_1)=1+v_p(x_1)$, and so we conclude that $\frac{1}{\Theta_0}=\frac{\Theta_1}{x_1}=\frac{p^2x_1}{\Theta_1}=0$ in $\F$. We will divide $R(\vr;\vk')$ into the following four areas, and in each area we will determine the mod-$p$ reduction:
\begin{enumerate}[leftmargin=*]
  \item $\vec{t}\in R_{int}(\vr;\vk')$: It is easy to see that $\cM\cong\cM_1((0,2),(2,0);(\alpha_0,\alpha_1),(\beta_0,-\beta_1))$, and so $\rhobar|_\IQptwo\cong\omega_4^{2+2p}\oplus\omega_4^{2p^2+2p^3}$ by Lemma~\ref{lemm: Tst^r for irred BrMod}~(ii).
  \item $t_1=-1$, $t_0>-1$: We have $\frac{1}{\Theta_1}\in\F^\times$ and $\frac{x_0}{\Theta_0}=0$ in $\F$. In particular, the monodromy type is $(0,1)$. Put $\cM'=\barS_{\F}(\barE_1^{(0)},\barE_2^{(1)})\subset\cM$. It is easy to see that $\cM'\cong\cM(1,0;-\frac{\alpha_0}{\Theta_1},\alpha_1)$ and $\cM/\cM'\cong\cM(1,2;\beta_0\Theta_1,-\beta_1)$, and so we conclude that $\rhobar^{ss}|_\IQptwo\cong \omega_2^{1+2p}\oplus\omega_2$, by Lemma~\ref{lemm: Tst^r for irred BrMod}~(i), and $\rhobar$ is non-split as the monodromy type is nonzero.
  \item $t_0=t_1=-1$: We have $\frac{1}{\Theta_1},\frac{x_0}{\Theta_0}\in\F^\times$, and so in particular its monodromy type is $(0,1)$. Put $\cM'=\barS_{\F}(\barE_1^{(0)},\barE_2^{(1)})\subseteq\cM$. It is easy to see that $\cM'\cong\cM(2,0;-\frac{\beta_0x_0}{\Theta_0},\alpha_1)$ and $\cM/\cM'\cong\cM(0,2;\frac{\alpha_0\Theta_0}{x_0},-\beta_1)$, and so we conclude that $\rhobar^{ss}|_\IQptwo \cong\omega_2^{2p}\oplus\omega_2^2$, by Lemma~\ref{lemm: Tst^r for irred BrMod}~(i), and $\rhobar$ is non-split as the monodromy type is nonzero.
  \item $t_0-t_1=0$, $t_0<-1$: We have $\frac{x_0}{\Theta_0}\in\F^\times$ and $\frac{1}{\Theta_1}=0$ in $\F$, and so in particular its monodromy type is $(0,0)$. Put $\cM'\subseteq \cM$ as in (iii). Using the same isomorphisms as in (iii) we have $\rhobar^{ss}|_\IQptwo \cong\omega_2^{2p}\oplus\omega_2^2$, by Lemma~\ref{lemm: Tst^r for irred BrMod}~(i). It is not difficult to show that the short exact sequence determined by $\cM'\hookrightarrow \cM$ is non-split.
\end{enumerate}

We summarize the mod-$p$ reduction for $\vk'=(\frac{1}{2},2)$ in the first table of Table~\ref{tab:   (1/2,2),(2,1/2)}. Moreover, the case $\vk'=(2,\frac{1}{2})$ is summarized in the second table. Note that by $\ast$ in the following tables we mean a non-split extension.
\begin{small}
\begin{table}[htbp]
  \centering
  \begin{tabular}{|c||c|c|}
    \hline
    $(t_0,t_1)$ & $\rhobar|_{I_{\Q_{p^2}}}$ & $\vec{\mathrm{MT}}$ \\\hline\hline
    $R_{int}(\vr;(\frac{1}{2},2))$
      & $\omega_4^{2+2p}\oplus\omega_4^{2p^2+2p^3}$ & $(0,0)$ \\\hline
    \begin{tabular}{c}
      $t_1=-1$, \\ $t_0>-1$
    \end{tabular}
      & $
      \begin{bmatrix}
        \omega_2 & * \\
        0 & \omega_2^{1+2p}
      \end{bmatrix}$ & $(0,1)$ \\\hline
    $t_0=t_1=-1$
      & $
      \begin{bmatrix}
        \omega_2^2 & * \\
        0 & \omega_2^{2p}
      \end{bmatrix}$ & $(0,1)$ \\\hline
    \begin{tabular}{c}
      $t_0-t_1=0$, \\ $t_0<-1$
    \end{tabular}
      & $
      \begin{bmatrix}
        \omega_2^2 & * \\
        0 & \omega_2^{2p}
      \end{bmatrix}$ & $(0,0)$ \\\hline
  \end{tabular}
  \qquad
  \begin{tabular}{|c||c|c|}
    \hline
    $(t_0,t_1)$ & $\rhobar|_{I_{\Q_{p^2}}}$ & $\vec{\mathrm{MT}}$ \\\hline\hline
    $R_{int}(\vr;(2,\frac{1}{2}))$
      & $\omega_4^{2p+2p^2}\oplus\omega_4^{2+2p^3}$ & $(0,0)$ \\\hline
    \begin{tabular}{c}
      $t_0=-1$, \\ $t_1>-1$
    \end{tabular}
      & $
      \begin{bmatrix}
        \omega_2^p & * \\
        0 & \omega_2^{2+p}
      \end{bmatrix}$ & $(1,0)$ \\\hline
    $t_0=t_1=-1$
      & $
      \begin{bmatrix}
        \omega_2^{2p} & * \\
        0 & \omega_2^2
      \end{bmatrix}$ & $(1,0)$ \\\hline
    \begin{tabular}{c}
      $t_0-t_1=0$, \\ $t_1<-1$
    \end{tabular}
      & $
      \begin{bmatrix}
        \omega_2^{2p} & * \\
        0 & \omega_2^2
      \end{bmatrix}$ & $(0,0)$ \\\hline
  \end{tabular}
  \caption{$\rhobar$ for $\vk'\in\{(\frac{1}{2},2)$, $(2,\frac{1}{2})\}$.}\label{tab: (1/2,2),(2,1/2)}
\end{table}
\end{small}

\subsubsection{\textbf{For $\vk'=(1,1)$}}\label{subsubsec: vr=(2,2), vk=(1,1)}
In this section of paragraph, by $\vk'$ we always mean $(1,1)$. In this case, we have $\vx=\vx^{(1,1)}$, and we will write $\cM$ for the Breuil modules. From the results in \S\ref{subsec: mod p reduction}, the Breuil modules $\cM$ is determined by the following data:
\begin{itemize}[leftmargin=*]
  \item $M_{N,0}=
  \begin{bmatrix}
    0 & 0 \\
    \frac{1-u^{p-1}}{\Theta_0} & 0
  \end{bmatrix}\quad\&\quad
M_{N,1}=
  \begin{bmatrix}
    0 & 0 \\
    \frac{1-u^{p-1}}{\Theta_1} & 0
  \end{bmatrix}$;
  \item $M_{\Fil^2,0}=
  \begin{bmatrix}
    1 & 0 \\
    \frac{x_0}{\Theta_0} & 1
  \end{bmatrix}
  \begin{bmatrix}
    u & \frac{\alpha_0}{\beta_0\Theta_1} \\
    0 & u
  \end{bmatrix}\quad\&\quad M_{\Fil^2,1}=
  \begin{bmatrix}
    1 & 0 \\
    \frac{x_1}{\Theta_1} & 1
  \end{bmatrix}
  \begin{bmatrix}
    u & \frac{\alpha_1}{\beta_1\Theta_0} \\
    0 & u
  \end{bmatrix}$;
  \item $M_{\phi_2,0}=
  \begin{bmatrix}
    1 & -p\Theta_1 \\
    \frac{u^{p-1}}{\Theta_1} & 1
  \end{bmatrix}
  \begin{bmatrix}
    0 & \alpha_0\\
    -\beta_0 & 0
  \end{bmatrix}\quad\&\quad M_{\phi_2,1}=
  \begin{bmatrix}
    1 & -p\Theta_0 \\
    \frac{u^{p-1}}{\Theta_0} & 1
  \end{bmatrix}
  \begin{bmatrix}
    0 & \alpha_1\\
    -\beta_1 & 0
  \end{bmatrix}.$
\end{itemize}
Note that one can choose any $\vTh$ satisfying
$$\max\{-1,-1-v_p(x_1)\}\leq v_p(\Theta_0)=1-v_p(\Theta_1)\leq \min\{0,v_p(x_0)\}.$$
We will divide $R(\vr;\vk')$ into following four areas, and in each area we will determine the mod-$p$ reduction:
\begin{enumerate}[leftmargin=*]
  \item $\vec{t}\in R_{int}(\vr;\vk')$: Note that we have either $v_p(\Theta_0)\not\in\{0,v_p(x_0)\}$ or $v_p(\Theta_0)\not\in\{-1,-1-v_p(x_1)\}$, in this case. Consider two $\barS_{\F}$-submodules    $\cM'=\barS_{\F}(\barE_1^{(0)},\barE_2^{(1)})$ and $\cM''=\barS_{\F}(\barE_2^{(0)},\barE_1^{(1)})$ of $\cM$. Then one can observe that
    \begin{itemize}[leftmargin=*]
      \item if $v_p(\Theta_0)\not\in\{0,v_p(x_0)\}$, then $\cM'$ is a Breuil submodule of $\cM$ such that $\cM'\cong\cM(1,1;-\beta_0,\alpha_1)$ and $\cM/\cM'\cong\cM(1,1;\alpha_0,-\beta_1)$;
      \item if $v_p(\Theta_0)\not\in\{-1,-1-v_p(x_1)\}$, then $\cM''$ is a Breuil submodule of $\cM$ such that $\cM''\cong\cM(1,1;\alpha_0,-\beta_1)$ and $\cM/\cM''\cong\cM(1,1;-\beta_0,\alpha_1)$.
    \end{itemize}
    In particular, if $v_p(\Theta_0)\not\in\{-1,0,v_p(x_0),-1-v_p(x_1)\}$, then $\cM=\cM'\oplus\cM''$ as Breuil modules. Otherwise, It is not difficult to show that the short exact sequence determined by $\cM'\hookrightarrow \cM$ or $\cM''\hookrightarrow\cM$ is non-split. So, $\rhobar^{ss}|_\IQptwo\cong\omega_2^{1+p}\oplus\omega_2^{1+p}$, by Lemma~\ref{lemm: Tst^r for irred BrMod}~(i). Furthermore, the monodromy type is given as follow.
    \begin{enumerate}[leftmargin=*]
      \item If $v_p(\Theta_0)=0$, then $t_0\geq 0$, and the monodromy type is $(1,0)$.
      \item If $v_p(\Theta_0)=-1$, then $t_1\geq 0$, and the monodromy type is $(0,1)$.
      \item Otherwise, the monodromy type is $(0,0)$.
    \end{enumerate}
  \item $t_1=-1$, $t_0\geq 0$: We have $v_p(\Theta_0)=0=-1-v_p(x_1)$, and so $\frac{1}{\Theta_0},\frac{x_1}{\Theta_1}\in\F^\times$ and $\frac{1}{\Theta_1}=0$ in $\F$. In particular, the monodromy type is $(1,0)$. Put $\cM'=\barS_{\F}(\barE_2^{(0)},\barE_1^{(1)})$. It is easy to see that $\cM'\cong\cM(1,2;\alpha_0,-px_1\beta_1)$ and $\cM/\cM'\cong\cM(1,0;-\beta_0,\frac{\alpha_1}{px_1})$, and so we conclude that $\rhobar^{ss}|_\IQptwo\cong\omega_2\oplus\omega_2^{1+2p}$, by Lemma~\ref{lemm: Tst^r for irred BrMod}~(i), and $\rhobar$ is non-split as the monodromy type is nonzero.
  \item $t_0+t_1=-1$, $-1<t_0<0$: We have $v_p(\Theta_0)=v_p(x_0)=-1-v_p(x_1)\not\in\{-1,0\}$, and so $\frac{x_0}{\Theta_0},\frac{x_1}{\Theta_1}\in\F^\times$ and $\frac{1}{\Theta_0}=\frac{1}{\Theta_1}=0$ in $\F$. In particular, the monodromy type is $(0,0)$. Put
     \begin{align*}
      A:=
      \begin{bmatrix}
        -\frac{x_1}{\Theta_1}\alpha_1 & \alpha_1 \\
        -\beta_1 & 0
      \end{bmatrix}
      \begin{bmatrix}
        -\frac{x_0}{\Theta_0}\alpha_0 & \alpha_0 \\
        -\beta_0 & 0
      \end{bmatrix}
      =
      \begin{bmatrix}
        \frac{x_0x_1}{\Theta_0\Theta_1}\alpha_0\alpha_1-\beta_0\alpha_1 & -\frac{x_1}{\Theta_1}\alpha_0\alpha_1 \\
        \frac{x_0}{\Theta_0}\alpha_0\beta_1 & -\alpha_0\beta_1
      \end{bmatrix}.
    \end{align*}
    Under the base change
    $$\baseE'^{(1)}=\baseE^{(1)}
    \begin{bmatrix}
      0 & -\frac{1}{\beta_1} \\
      \frac{1}{\alpha_1} & -\frac{x_1}{\beta_1\Theta_1}
    \end{bmatrix},\,
    \baseF'^{(0)}=\baseF^{(0)}
    \begin{bmatrix}
      1 & 0 \\
      -\frac{x_0}{\Theta_0} & 1
    \end{bmatrix},\,
    \baseF'^{(1)}=\baseF^{(1)}
    \begin{bmatrix}
      0 & -\frac{1}{\beta_1} \\
      \frac{1}{\alpha_1} & 0
    \end{bmatrix},$$
    we have
    \begin{itemize}[leftmargin=*]
      \item $\Mat_{\baseE^{(0)},\baseF'^{(0)}}(\Fil^r\cM^{(0)})=\Mat_{\baseE'^{(1)},\baseF'^{(1)}}(\Fil^r\cM^{(1)})=uI_2$;
      \item $\Mat_{\baseE'^{(1)},\baseF'^{(0)}}(\phi_2^{(0)})=A\quad\&\quad \Mat_{\baseE^{(0)},\baseF'^{(1)}}(\phi_2^{(1)})=I_2$.
    \end{itemize}
    Note that we have $\frac{\alpha_0}{\beta_0}=\frac{\alpha_1}{\beta_1}=p\Theta_0\Theta_1$ and $\alpha_0\alpha_1=\lambda^2\Theta_0\Theta_1$, and so one can describe the mod-$p$ reduction as follow:
    \begin{itemize}[leftmargin=*]
    \item if $px_0x_1\neq 4$ in $\F$, then $A$ is diagonalizable. So, extending $\F$ large enough, we have $\rhobar|_\IQptwo\cong\omega_2^{1+p}\oplus\omega_2^{1+p}$ by Lemma~\ref{lemm: Tst^r for irred BrMod}~(i).
    \item if $px_0x_1=4$ in $\F$, then $A$ is similar to $
      \begin{tiny}\begin{bmatrix}
        \frac{\lambda^2}{p} & 1\\
        0 & \frac{\lambda^2}{p}
      \end{bmatrix}\end{tiny}$
    and so $\rhobar$ is non-split with $\rhobar|_\IQptwo^{ss}\cong\omega_2^{1+p}\oplus\omega_2^{1+p}$ by Lemma~\ref{lemm: Tst^r for irred BrMod}~(i).
    \end{itemize}
  \item $t_0=-1$, $t_1\geq 0$: We have $v_p(\Theta_0)=-1=v_p(x_0)$, and so $\frac{1}{\Theta_1},\frac{x_0}{\Theta_0}\in\F^\times$ and $\frac{1}{\Theta_0}=0$ in $\F$. In particular, the monodromy type is $(0,1)$. Put $\cM'=\barS_{\F}(\barE_1^{(0)},\barE_2^{(1)})$. It is easy to see that $\cM'\cong\cM(2,1;-px_0\beta_0,\alpha_1)$ and $\cM/\cM'\cong\cM(0,1;\frac{\alpha_0}{px_0},-\beta_1)$, and so we conclude that $\rhobar^{ss}|_\IQptwo\cong\omega_2^p\oplus\omega_2^{2+p}$, by Lemma~\ref{lemm: Tst^r for irred BrMod}~(i), and $\rhobar$ is non-split as the monodromy type is nonzero.
\end{enumerate}

We summarize the mod-$p$ reduction for $\vk'=(1,1)$ in the following table, Table~\ref{tab: (1,1)}. Note that by $\ast$ in the following table we mean a non-split extension.
\begin{small}
\begin{table}[htbp]
  \centering
  \begin{tabular}{|c|c||c|c|}
    \hline
    \multicolumn{2}{|c||}{$(t_0,t_1)$} & $\rhobar|_{I_{\Q_{p^2}}}$ & $\vec{\mathrm{MT}}$ \\\hline\hline
    \multirow{2}{*}[-3em]{$R_{int}(\vr;(1,1))$}
        & 
        \begin{tabular}{c}
            $v_p(\Theta_0)\not\in\{-1,0,$\\$v_p(x_0),-1-v_p(x_1)\}$ 
        \end{tabular}
        & $
        \begin{bmatrix}
            \omega_2^{1+p} & 0 \\ 0 & \omega_2^{1+p}
        \end{bmatrix}$
        & $(0,0)$ \\\cline{2-4}
        & otherwise
        & $
        \begin{bmatrix}
            \omega_2^{1+p} & * \\ 0 & \omega_2^{1+p}
        \end{bmatrix}$
        & $
        \begin{cases}
            (1,0) & \mbox{if }v_p(\Theta_0)=0\\& \quad\mbox{ and } t_0\geq 0;\\
            (0,1) & \mbox{if }v_p(\Theta_0)=-1\\& \quad\mbox{ and } t_1\geq 0;\\
            (0,0) & \mbox{otherwise}
        \end{cases}$ \\\hline
    \multicolumn{2}{|c||}{$t_1=-1$, $t_0\geq 0$}
        & $
        \begin{bmatrix}
          \omega_2^{1+2p} & * \\
          0 & \omega_2
        \end{bmatrix}$ & $(1,0)$ \\\hline
    \multirow{3}{*}{
    \begin{tabular}{c}
         $t_0+t_1=-1$, \\ $-1<t_0<0$ 
    \end{tabular}
    }
        & $v_p(px_0x_1-4)=0$
        & $
        \begin{bmatrix}
            \omega_2^{1+p} & 0 \\ 0 & \omega_2^{1+p}
        \end{bmatrix}$
        & $(0,0)$ \\\cline{2-4}
        & $v_p(px_0x_1-4)>0$
        & $
        \begin{bmatrix}
            \omega_2^{1+p} & * \\ 0 & \omega_2^{1+p}
        \end{bmatrix}$
        & $(0,0)$ \\\hline
    \multicolumn{2}{|c||}{$t_0=-1$, $t_1\geq 0$}
        & $
        \begin{bmatrix}
          \omega_2^{2+p} & * \\
          0 & \omega_2^p
        \end{bmatrix}$ & $(0,1)$ \\\hline
  \end{tabular}
  \caption{$\rhobar$ for $\vk'=(1,1)$}\label{tab: (1,1)}
\end{table}
\end{small}

\subsubsection{\textbf{For $\vk'=(\frac{3}{2},2)$ or $\vk'=(2,\frac{3}{2})$}}\label{subsubsec: vr=(2,2), vk=(3/2,2),(2,3/2)}
We only consider the case $\vk'=(\tfrac{3}{2},2)$, as the case $\vk'=(2,\frac{3}{2})$ is symmetric. In this section of paragraph, by $\vk'$ we always mean $(\frac{3}{2},2)$.

In this case, we have $\vx=\vx^{(2,2)}$, and we will write $\cM$ for the Breuil modules. From the results in \S\ref{subsec: mod p reduction}, the Breuil module $\cM$ is determined by the following data:
\begin{itemize}[leftmargin=*]
  \item $M_{N,0}=
  \begin{bmatrix}
    0 & 0 \\
    \frac{1-u^{p-1}}{\Theta_0} & 0
  \end{bmatrix}\quad\&\quad
  M_{N,1}=
  \begin{bmatrix}
    0 & 0 \\
    \frac{1-u^{p-1}}{\Theta_1} & 0
  \end{bmatrix}$;
  \item $M_{\Fil^2,0}=
  \begin{bmatrix}
    1 & \frac{\Theta_0}{x_0} \\
    0 & 1
  \end{bmatrix}
  \begin{bmatrix}
    u & 0 \\
    -\frac{\alpha_0}{\beta_0\Theta_1} & u
  \end{bmatrix}\quad\&\quad
  M_{\Fil^2,1}=
  \begin{bmatrix}
    1 & \frac{\Theta_1}{x_1} \\
    0 & 1
  \end{bmatrix}
  \begin{bmatrix}
    u^2 & -\frac{\alpha_1 u}{\beta_1\Theta_0} \\
    0 & 1
  \end{bmatrix}$;
  \item $M_{\phi_2,0}=
  \begin{bmatrix}
    1 & 0 \\
    \frac{u^{p-1}}{\Theta_1} & 1
  \end{bmatrix}
  \begin{bmatrix}
    -\alpha_0 & 0 \\
    0 & -\beta_0
  \end{bmatrix}\quad\&\quad
  M_{\phi_2,1}=
  \begin{bmatrix}
    1 & -\frac{\alpha_1}{\beta_1}\frac{p^2x_1}{\Theta_1} \\
    \frac{u^{p-1}}{\Theta_0} & 1
  \end{bmatrix}
  \begin{bmatrix}
    0 & \alpha_1\\
    -\beta_1 & 0
  \end{bmatrix}.$
\end{itemize}

We further have the identities $v_p(\Theta_0)=1+v_p(x_0)+v_p(x_1)$ and $v_p(\Theta_1)=-v_p(x_0)+v_p(x_1)$, and so we conclude that $\frac{\Theta_1}{x_1}=\frac{p^2x_1}{\Theta_1}=0$ in $\F$. We will divide $R(\vr;\vk')$ into the following five areas, and in each area we will determine the mod-$p$ reduction:
\begin{enumerate}[leftmargin=*]
  \item $\vec{t}\in R_{int}(\vr;\vk')$: It is easy to see that $\cM\cong\cM_1((1,2),(1,0);(-\alpha_0,-\beta_0),(\alpha_1,\beta_1))$, and so we conclude that $\rhobar|_\IQptwo\cong \omega_4^{1+2p+p^2}\oplus\omega_4^{1+p^2+2p^3}$ by Lemma~\ref{lemm: Tst^r for irred BrMod}~(ii).
  \item $t_0+t_1=-1$, $-\frac{1}{2}<t_0<0$: We have $\frac{1}{\Theta_0}\in\F^\times$ and $\frac{1}{\Theta_1}=\frac{\Theta_0}{x_0}=0$ in $\F$. In particular, the monodromy type is $(1,0)$. Put $\cM'=\barS_{\F}(\barE_2)\subset\cM$. It is easy to see that $\cM'\cong\cM(1,1;-\beta_0,-\frac{\alpha_1}{\Theta_0})$ and $\cM/\cM'\cong\cM(1,1;-\alpha_0,-\beta_1\Theta_0)$, and so we conclude that $\rhobar^{ss}|_\IQptwo \cong \omega_2^{1+p}\oplus\omega_2^{1+p}$, by Lemma~\ref{lemm: Tst^r for irred BrMod}~(i), and $\rhobar$ is non-split as the monodromy type is nonzero.
  \item $t_0=t_1=-\frac{1}{2}$: We have $\frac{1}{\Theta_0},\frac{1}{\Theta_1}\in\F^\times$ and $\frac{\Theta_0}{x_0}=0$ in $\F$. In particular, the monodromy type is $(1,1)$. Put $\cM'$ as in (ii). Using the same isomorphisms as in (ii), we conclude that $\rhobar^{ss}|_\IQptwo\cong \omega_2^{1+p}\oplus\omega_2^{1+p}$, by Lemma~\ref{lemm: Tst^r for irred BrMod}~(i), and $\rhobar$ is non-split as the monodromy type is nonzero.
  \item $t_0-t_1=0$, $-1\leq t_0<-\frac{1}{2}$: We have $\frac{1}{\Theta_1}\in\F^\times$ and $\frac{1}{\Theta_0}=0$ in $\F$. In particular, the monodromy type is $(0,1)$. Put $\cM'=\barS_{\F}(\barE_1^{(0)},\barE_2^{(1)})\subset\cM$. It is easy to see that $\cM'\cong \cM(2,0;-\frac{\alpha_0}{\Theta_1},\alpha_1)$ and $\cM/\cM'\cong\cM(0,2;\beta_0\Theta_1,-\beta_1)$, and so we conclude that $\rhobar^{ss}|_\IQptwo \cong\omega_2^{2p}\oplus\omega_2^2$, by Lemma~\ref{lemm: Tst^r for irred BrMod}~(i), and $\rhobar$ is non-split as the monodromy type is nonzero.
  \item $t_1=-1$, $-1<t_0<0$: We have $\frac{\Theta_0}{x_0}\in\F^\times$ and $\frac{1}{\Theta_0}=\frac{1}{\Theta_1}=0$ in $\F$. In particular, the monodromy type is $(0,0)$. Put $\cM'=\barS_{\F}(\barE_2^{(0)},\barE_1^{(1)}-\frac{\beta_0x_0}{\alpha_0\Theta_0}\barE_2^{(1)})$. It is easy to see that $\cM'\cong\cM(1,2;\frac{\alpha_0x_0}{\Theta_0},-\beta_1)$ and $\cM/\cM'\cong\cM(1,0;-\frac{\beta_0x_0}{\Theta_0},\alpha_1)$, and so we conclude that $\rhobar^{ss}|_\IQptwo \cong\omega_2\oplus\omega_2^{1+2p}$, by Lemma~\ref{lemm: Tst^r for irred BrMod}~(i). It is not difficult to show that the short exact sequence determined by $\cM'\hookrightarrow \cM$ is non-split.
\end{enumerate}
We summarize the mod-$p$ reduction for $\vk'=(\frac{3}{2},2)$ in the first table of Table~\ref{tab:   (3/2,2),(2,3/2)}. Moreover, the case $\vk'=(2,\frac{3}{2})$ is summarized in the second table. Note that by $\ast$ in the following tables we mean a non-split extension.
\begin{small}
\begin{table}[htbp]
  \centering
  \begin{tabular}{|c||c|c|}
    \hline
    $(t_0,t_1)$ & $\rhobar|_{I_{\Q_{p^2}}}$ & $\vec{\mathrm{MT}}$ \\\hline\hline
    $R_{int}(\vr;(\frac{3}{2},2))$
      & \begin{tabular}{c}
         $\omega_4^{1+2p+p^2}$\\ $\oplus$\\ $\omega_4^{1+p^2+2p^3}$
         \end{tabular}
      & $(0,0)$ \\\hline
    \begin{tabular}{c}
      $t_0+t_1=-1$, \\ $-\frac{1}{2}<t_0<0$
    \end{tabular}
      & $
      \begin{bmatrix}
        \omega_2^{1+p} & * \\
        0 & \omega_2^{1+p}
      \end{bmatrix}$ & $(1,0)$ \\\hline
    $t_0=t_1=-\frac{1}{2}$
      & $
      \begin{bmatrix}
        \omega_2^{1+p} & * \\
        0 & \omega_2^{1+p}
      \end{bmatrix}$ & $(1,1)$ \\\hline
    \begin{tabular}{c}
      $t_0-t_1=0$, \\ $-1\leq t_0<-\frac{1}{2}$
    \end{tabular}
      & $
      \begin{bmatrix}
        \omega_2^2 & * \\
        0 & \omega_2^{2p}
      \end{bmatrix}$ & $(0,1)$ \\\hline
    \begin{tabular}{c}
      $t_1=-1$, \\ $-1<t_0<0$
    \end{tabular}
      & $
      \begin{bmatrix}
        \omega_2^{1+2p} & * \\
        0 & \omega_2
      \end{bmatrix}$ & $(0,0)$ \\\hline
  \end{tabular}\quad
  \begin{tabular}{|c||c|c|}
    \hline
    $(t_0,t_1)$ & $\rhobar|_{I_{\Q_{p^2}}}$ & $\vec{\mathrm{MT}}$ \\\hline\hline
    $R_{int}(\vr;(2,\frac{3}{2}))$
      & \begin{tabular}{c}
        $\omega_4^{p+2p^2+p^3}$ \\$\oplus$\\$\omega_4^{2+p+p^3}$
        \end{tabular}
      & $(0,0)$ \\\hline
    \begin{tabular}{c}
      $t_0+t_1=-1$, \\ $-\frac{1}{2}<t_1<0$
    \end{tabular}
      & $
      \begin{bmatrix}
        \omega_2^{1+p} & * \\
        0 & \omega_2^{1+p}
      \end{bmatrix}$ & $(0,1)$ \\\hline
    $t_0=t_1=-\frac{1}{2}$
      & $
      \begin{bmatrix}
        \omega_2^{1+p} & * \\
        0 & \omega_2^{1+p}
      \end{bmatrix}$ & $(1,1)$ \\\hline
    \begin{tabular}{c}
      $t_0-t_1=0$, \\ $-1\leq t_1<-\frac{1}{2}$
    \end{tabular}
      & $
      \begin{bmatrix}
        \omega_2^{2p} & * \\
        0 & \omega_2^2
      \end{bmatrix}$ & $(1,0)$ \\\hline
    \begin{tabular}{c}
      $t_0=-1$, \\ $-1<t_1<0$
    \end{tabular}
      & $
      \begin{bmatrix}
        \omega_2^{2+p} & * \\
        0 & \omega_2^p
      \end{bmatrix}$ & $(0,0)$ \\\hline
  \end{tabular}
  \caption{$\rhobar$ for $\vk'=(\frac{3}{2},2)$, $(2,\frac{3}{2})$.}\label{tab: (3/2,2),(2,3/2)}
\end{table}
\end{small}

\subsubsection{\textbf{Summary on mod-$p$ reduction when $\vr=(2,2)$}}\label{subsubsec: summary for 2,2}
In this section of paragraph, we summarize the results on mod-$p$ reduction when $\vr=(2,2)$. We set $\vec{\xi}:=(\xi_0,\xi_1)\in E^2$ where
$$\xi_0=\frac{1}{p}\left(p\fL_0-\fL_1+1\right) \qquad\mbox{and}\qquad \xi_1=\frac{1}{p}(p\fL_1-\fL_0+1).$$

We first partition $E^2(=E^2_\emptyset)$ in terms of $\vec{\xi}$, such that within each subset the semi-simplification of $\rhobar|_{I_{\Q_{p^2}}}$ remains constant, as follow:
\begin{itemize}[leftmargin=*]
\item $S_1:=\{\vL\in E^2\mid v_p(\xi_1)<-1,\ v_p(\xi_0)>v_p(\xi_1)\}$;
\item $S_2:=\{\vL\in E^2\mid v_p(\xi_0)<-1,\ v_p(\xi_0)<v_p(\xi_1)\}$;
\item $S_3:=\{\vL\in E^2\mid v_p(\xi_0)>-1,\ v_p(\xi_1)>-1, v_p(\xi_0)+v_p(\xi_1)>-1\}$;
\item $S_4:=\{\vL\in E^2\mid v_p(\xi_0)+v_p(\xi_1)<-1,\ v_p(\xi_0)>v_p(\xi_1)>-1\}$;
\item $S_5:=\{\vL\in E^2\mid v_p(\xi_0)+v_p(\xi_1)<-1,\ v_p(\xi_1)>v_p(\xi_0)>-1\}$;
\item $S_6:=\{\vL\in E^2\mid v_p(\xi_1)=-1,\ v_p(\xi_0)>-1\}$;
\item $S_7:=\{\vL\in E^2\mid v_p(\xi_0)=-1,\ v_p(\xi_1)>-1\}$;
\item $S_8:=\{\vL\in E^2\mid v_p(\xi_0)+v_p(\xi_1)=-1,\ -1<v_p(\xi_0)<0\}$;
\item $S_9:=\{\vL\in E^2\mid v_p(\xi_0)=v_p(\xi_1)<-\frac{1}{2}\}$.
\end{itemize}
It is not difficult to show that $S_i\cap S_j=\emptyset$ if $i\neq j$ and that $\cup_{i=1}^9S_i=E^2$. The following picture, Figure~\ref{fig: partition of E^2 for vr=(2,2)}, describes those sets $S_1,\cdots,S_9$ on $(v_p(\xi_0),v_p(\xi_1))$-plane, emphasizing the disjointness. These sets are carefully defined to describe the mod-$p$ reduction on each set.
\begin{figure}[htbp]
  \centering
  \includegraphics[scale=0.5]{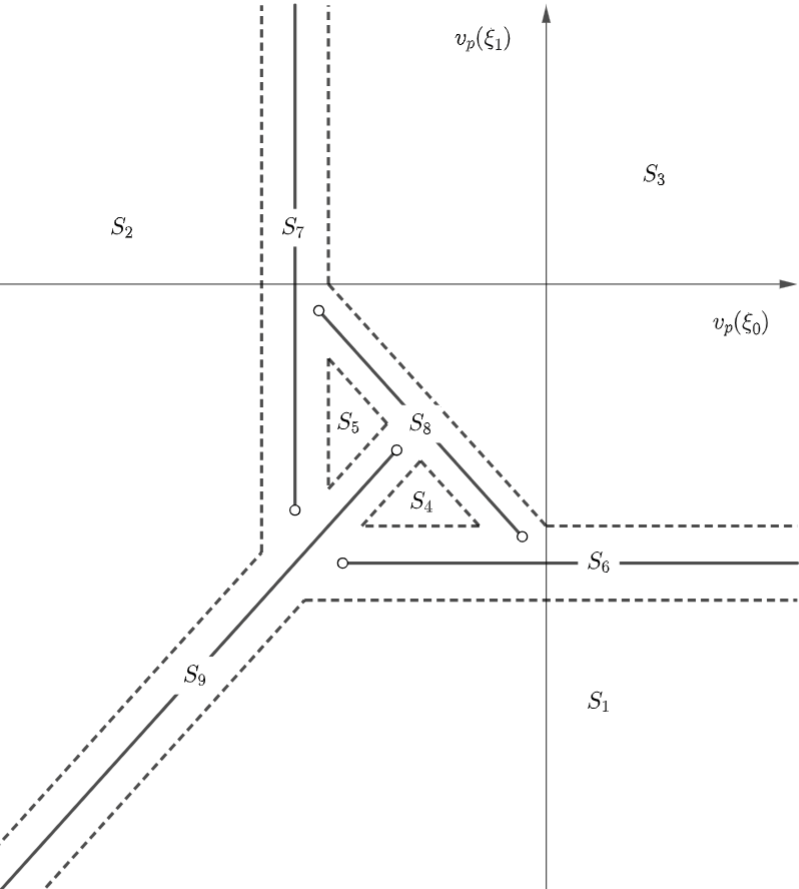}
  \caption{Partition by $S_1,\cdots,S_9$, displayed on $v_p(\vec{\xi})$-plane.}\label{fig: partition of E^2 for vr=(2,2)}
\end{figure}
We further define subsets of $S_3$ and $S_8$ as follow:
\begin{itemize}[leftmargin=*]
\item $S'_3=\{\vL\in S_3\mid -1<v_p(\xi_0+1)-v_p(\xi_1+1)<1\}\subset S_3$;
\item $S'_8=\{\vL\in S_8\mid v_p(p\xi_0\xi_1-4)>0\}\subset S_8$.
\end{itemize}
We summarize the non-split mod-$p$ reduction for all sets $S_1,\cdots,S_9,S_3',S_8'$ in Table~\ref{tab:summary for vr=(2,2)}.

\begin{table}[htbp]
    \centering
    \begin{tabular}{|c||c|c|}
    \hline
        $\vL$ & $\rhobar|_\IQptwo$ & $\vec{\mathrm{MT}}$ \\\hline\hline
    $S_1$
            & $\omega_4^{2+2p}\oplus\omega_4^{2p^2+2p^3}$
            & $(0,0)$ \\\hline
    $S_2$
            & $\omega_4^{2p+2p^2}\oplus\omega_4^{2+2p^3}$
            & $(0,0)$ \\\hline
    \multirow{2}{*}[-0.5em]{$S_3$}
        & $
        \begin{bmatrix}
            \omega & * \\ 0 & \omega
        \end{bmatrix}
        $
        & $
        \begin{cases}
            (1,0) & \mbox{if } v_p(\xi_0)\geq 0;\\
            (0,0) & \mbox{if } -1<v_p(\xi_0)<0
        \end{cases}
        $ \\\cline{2-3}
        & $
        \begin{bmatrix}
            \omega & * \\ 0 & \omega
        \end{bmatrix}
        $
        & $
        \begin{cases}
            (0,1) & \mbox{if } v_p(\xi_1)\geq 0;\\
            (0,0) & \mbox{if } -1<v_p(\xi_1)<0
        \end{cases}$ \\\hline
    $S'_3$
        &
        $
        \begin{bmatrix}
            \omega & * \\ 0 & \omega
        \end{bmatrix}
        $
        & 
        $
        \begin{cases}
            (0,1) &\mbox{if } -1<v_p(\xi_0+1)-v_p(\xi_1+1)<0; \\
            (1,1) &\mbox{if } v_p(\xi_0+1)=v_p(\xi_1+1);\\
            (1,0) &\mbox{if }0<v_p(\xi_0+1)-v_p(\xi_1+1)<1 
        \end{cases}
        $
        \\\hline
    $S_4$
            & $\omega_4^{1+2p+p^2}\oplus\omega_4^{1+p^2+2p^3}$
            & $(0,0)$ \\\hline
    $S_5$
            & $\omega_4^{p+2p^2+p^3}\oplus\omega_4^{2+p+p^3}$
            & $(0,0)$ \\\hline
    \multirow{2}{*}[-0.5em]{$S_6$}
            & $
            \begin{bmatrix}
                \omega_2 & * \\ 0 & \omega_2^{1+2p}
            \end{bmatrix}$
            & $(0,1)$\\\cline{2-3}
            & $
            \begin{bmatrix}
                \omega_2^{1+2p} & * \\ 0 & \omega_2
            \end{bmatrix}$
            & $(0,0)$ \\\hline
    \multirow{2}{*}[-0.5em]{$S_7$}
            & $
            \begin{bmatrix}
                \omega_2^p & * \\ 0 & \omega_2^{2+p}
            \end{bmatrix}$
            & $(1,0)$\\\cline{2-3}
            & $
            \begin{bmatrix}
                \omega_2^{2+p} & * \\ 0 & \omega_2^p
            \end{bmatrix}$
            & $(0,0)$ \\\hline
    $S_8$
        & $
        \begin{bmatrix}
            \omega & * \\ 0 & \omega
        \end{bmatrix}
        $
        & $
        \begin{cases}
            (0,1) &\mbox{if }-1<v_p(\xi_0)<-\frac{1}{2};\\
            (1,1) &\mbox{if } v_p(\xi_0)=-\frac{1}{2};\\
            (1,0) &\mbox{if } -\frac{1}{2}<v_p(\xi_0)<0
        \end{cases}$
        \\\hline
    $S'_8$
        & $
        \begin{bmatrix}
            \omega & * \\ 0 & \omega
        \end{bmatrix}$
        & $(0,0)$ \\\hline
    \multirow{2}{*}[-0.5em]{$S_9$}
            & $
            \begin{bmatrix}
                \omega_2^2 & * \\ 0 & \omega_2^{2p}
            \end{bmatrix}$
            & $(0,1)$\\\cline{2-3}
            & $
            \begin{bmatrix}
                \omega_2^{2p} & * \\ 0 & \omega_2^2
            \end{bmatrix}$
            & $(1,0)$ \\
        \hline
    \end{tabular}
    \caption{Non-split mod-$p$ reduction for $\vr=(2,2)$}
    \label{tab:summary for vr=(2,2)}
\end{table}

\begin{rema}\label{rema: 2,2}
One can extract the following results from Table~\ref{tab:summary for vr=(2,2)}:
\begin{itemize}[leftmargin=*]
\item our method provides with at least one Galois stable lattice in each $2$-dimensional semi-stable non-crystalline representation of $G_{\Q_{p^2}}$ with parallel Hodge--Tate weights $(0,2)$, as $\cup_{i=1}^9S_i=E^2$;
\item if the mod-$p$ reduction is an extension of two distinct characters then our method provides the two non-homothetic lattices (For instance, if $\vL\in S_8\setminus S_8'$ then there is one more lattice whose mod-$p$ reduction is split);
\item if $\vL\in S_3'$ then three non-homothetic lattices are provided, as $S_3'\subset S_3$. In particular, if $\vL\in S_3'$ then $\rhobar$ is an extension of the same characters.
\end{itemize}
\end{rema}

For the rest of this subsection, we explain how one can induce Table~\ref{tab:summary for vr=(2,2)} from the results in the previous sections of paragraph.
We start by noting the following identities $\vx^{(1,1)}=(\xi_0,\xi_1)$, $\vx^{(1,2)}=(\xi_0,\xi_1-\frac{1}{p(\xi_0+1)})$, and $\vx^{(2,1)}=(\xi_0-\frac{1}{p(\xi_1+1)},\xi_1)$. Moreover, if we let $\vec{\xi}'=(\xi'_0,\xi'_1)$ for the solutions of the equation
\begin{equation}\label{eq: equations for xi' when 2,2}
\xi'_0=\xi_0-\frac{1}{p(\xi'_1+1)}
\quad\mbox{and}\quad
\xi'_1=\xi_1-\frac{1}{p(\xi'_0+1)},
\end{equation}
then we have $\vec{\xi}'=\vx^{(2,2)}$, and we describe its necessary properties in the following lemma.
\begin{lemm}\label{lemma: implicit vx^(2,2)}
Let $\vec{\xi}\in E^2$ and let $\vec{\xi}'$ be a solution of the equations in \eqref{eq: equations for xi' when 2,2}.
\begin{enumerate}[leftmargin=*]
\item If $v_p(\xi_0)>-1$, $v_p(\xi_1)>-1$ and $v_p(\xi_0)+v_p(\xi_1)>-1$, then there exists a solution $\vec{\xi}'\in E^2$ whose valuation satisfies
$$v_p(\xi'_0)=\frac{v_p(\xi_0+1)-v_p(\xi_1+1)-1}{2}
\quad\mbox{and}\quad
v_p(\xi'_1)=\frac{v_p(\xi_1+1)-v_p(\xi_0+1)-1}{2}.$$
\item If $-1<v_p(\xi_0)<0$ and $v_p(\xi_0)+v_p(\xi_1)=-1$, then there exits a solution $\vec{\xi}'\in E^2$ whose valuation satisfies $v_p(\vec{\xi}')=v_p(\vec{\xi})$, assuming $E$ is large enough.
\item If $v_p(\xi_0)\geq -1$, $v_p(\xi_1)\geq -1$ and $v_p(\xi_0)+v_p(\xi_1)<-1$, then there exists a unique solution $\vec{\xi}'\in E^2$ whose valuation satisfies $v_p(\vec{\xi}')=v_p(\vec{\xi})$.
\end{enumerate}
\end{lemm}

\begin{proof}
The equation of $\vec{\xi}'$ can be translated to
$$p(\xi'_0+1)(\xi'_1+1)+1=p(\xi_0+1)(\xi'_1+1)=p(\xi'_0+1)(\xi_1+1).$$
Put $X=p(\xi'_0+1)(\xi'_1+1)$ and $A=p(\xi_0+1)(\xi_1+1)$, then one have $(X+1)^2=AX$ and $\xi'_j+1=\frac{X+1}{p(\xi_{j+1}+1)}$ for $j=0,1$.
\begin{enumerate}[leftmargin=*]
\item If $v_p(\xi_0)>-1$, $v_p(\xi_1)>-1$ and $v_p(\xi_0)+v_p(\xi_1)>-1$, then $v_p(A)>0$. By Hensel's Lemma, two solutions for $X$ can be found in $-1+\fm$. For this $X$, $v_p(X+1)=\frac{1}{2}v_p(AX)=\frac{v_p(\xi_0+1)+v_p(\xi_1+1)+1}{2}$. This with $\xi'_j+1=\frac{X+1}{p(\xi_{j+1}+1)}$ gives the result.
\item If $-1<v_p(\xi_0)<0$ and $v_p(\xi_0)+v_p(\xi_1)=-1$, then $v_p(A)=0$. Extending $E$ large enough, two solutions for $X$ can be found in $\cO^\times$. For this $X$, $v_p(X+1)=\frac{1}{2}v_p(AX)=-\frac{1}{2}$, giving the desired result.
\item If $v_p(\xi_0)\geq -1$, $v_p(\xi_1)\geq -1$ and $v_p(\xi_0)+v_p(\xi_1)<-1$, then $v_p(A)>0$. By Hensel's Lemma, two solutions $X=X_1$ and $X=X_2$ are found in $E$, where $v_p(X_1)=v_p(A^{-1})$ and $v_p(X_2)=v_p(A)$. This gives $v_p(X_1+1)=0$ and $v_p(X_2+1)=v_p(A)$. Now it is easy to see that only $\vec{\xi}'$ from $X_2$ can satisfy the condition $v_p(\xi_0)\geq -1$, $v_p(\xi_1)\geq -1$ and $v_p(\xi_0)+v_p(\xi_1)<-1$, and moreover $v_p(\vec{\xi}')=v_p(\vec{\xi})$.
\end{enumerate}
This completes the proof.
\end{proof}

We are now ready to verify how one can induce Table~\ref{tab:summary for vr=(2,2)}.
\begin{itemize}[leftmargin=*]
\item Assume $\vL\in S_1$. Equivalently, we have $v_p(x_0^{(1,2)})<-1$ and $v_p(x_0^{(1,2)})>v_p(x_1^{(1,2)})$. Then we have $\rhobar|_\IQptwo\cong\omega_4^{2+2p}\oplus\omega_4^{2p^2+2p^3}$ from $\mathbf{Case}_\phi(\vr;(\frac{1}{2},2))$ (see \S\ref{subsubsec: vr=(2,2), vk=(1/2,2),(2,1/2)}~(i)).

\item Assume $\vL\in S_2$. Equivalently, we have $v_p(x_1^{(2,1)})<-1$ and $v_p(x_0^{(2,1)})<v_p(x_1^{(2,1)})$. Then we have $\rhobar|_\IQptwo\cong\omega_4^{2p+2p^2}\oplus\omega_4^{2+2p^3}$ from $\mathbf{Case}_\phi(\vr;(\frac{1}{2},2))$ (see \S\ref{subsubsec: vr=(2,2), vk=(1/2,2),(2,1/2)}~(i), as $S_2=\varphi(S_1)$).

\item Assume $\vL\in S_3$, i.e, $v_p(\xi_0)>-1$, $v_p(\xi_1)>-1$, and $v_p(\xi_0)+v_p(\xi_1)>-1$. In this case there are two non-homothetic lattices whose mod-$p$ reductions are non-split and reducible.
    \begin{itemize}[leftmargin=*]
        \item We first consider the lattice from $v_p(\Theta_0)\in\{0,v_p(\xi_0)\}$ of $\mathbf{Case}_\phi(\vr;(1,1))$ (see \S\ref{subsubsec: vr=(2,2), vk=(1,1)}~(i)). This case gives $\rhobar|_\IQptwo\cong
        \begin{bmatrix}
            \omega & * \\ 0 & \omega
        \end{bmatrix}$ with $\vec{\mathrm{MT}}=(1,0)$ if $v_p(\xi_0)\geq 0$ and $\vec{\mathrm{MT}}=(0,0)$ otherwise.
        \item We also consider the lattices from $v_p(\Theta_0)\in\{-1,-1-v_p(\xi_1)\}$ of $\mathbf{Case}_\phi(\vr;(1,1))$ (see \S\ref{subsubsec: vr=(2,2), vk=(1,1)}~(i)). This case gives $\rhobar|_\IQptwo\cong
        \begin{bmatrix}
            \omega & * \\ 0 & \omega
        \end{bmatrix}$ with $\vec{\mathrm{MT}}=(0,1)$ if $v_p(\xi_1)\geq 0$ and $\vec{\mathrm{MT}}=(0,0)$ otherwise.
    \end{itemize}

Now we assume $\vL\in S'_3\subset S_3$. By Lemma~\ref{lemma: implicit vx^(2,2)}~(i), we have $v_p(\xi'_0)+v_p(\xi'_1)=-1$ and $-1<v_p(\xi'_0)<0$. In this case, there are another lattice giving non-split mod-$p$ reduction, which is different from the above two. For this lattice, we have $\rhobar|_\IQptwo\cong
\begin{bmatrix}
    \omega & * \\ 0 & \omega
\end{bmatrix}$, with $\vec{\mathrm{MT}}=(0,1)$ if $-1<v_p(\xi_0+1)-v_p(\xi_1+1)<0$, $\vec{\mathrm{MT}}=(1,1)$ if $v_p(\xi_0+1)=v_p(\xi_1+1)$, and $\vec{\mathrm{MT}}=(1,0)$ if $0<v_p(\xi_0+1)-v_p(\xi_1+1)<1$. This is from \S\ref{subsubsec: vr=(2,2), vk=(3/2,2),(2,3/2)}~(ii) if $v_p(\xi'_0)\neq -\frac{1}{2}$ (i.e., $v_p(\xi_0+1)\neq v_p(\xi_1+1)$ in $S'_3$), and from \S\ref{subsubsec: vr=(2,2), vk=(3/2,2),(2,3/2)}~(iii) if $v_p(\xi'_0)=-\frac{1}{2}$ (i.e., $v_p(\xi_0+1)=v_p(\xi_1+1)$ in $S'_3$).

\item Assume $\vL\in S_4$. Equivalently, we have $v_p(\xi'_0)+v_p(\xi'_1)<-1$, $v_p(\xi'_0)>v_p(\xi'_1)$ and $v_p(\xi'_1)>0$ by Lemma~\ref{lemma: implicit vx^(2,2)} (iii). Then $\rhobar|_\IQptwo\cong\omega_4^{1+2p+p^2}\oplus\omega_4^{1+p^2+2p^3}$ from $\mathbf{Case}_\phi(\vr;(\frac{3}{2},2))$ (see \S\ref{subsubsec: vr=(2,2), vk=(3/2,2),(2,3/2)}~(i)).

\item Assume $\vL\in S_5$. Equivalently, we have $v_p(\xi_0)+v_p(\xi_1)<-1$, $v_p(\xi_0)<v_p(\xi_1)$, and $v_p(\xi_0)>0$ by Lemma~\ref{lemma: implicit vx^(2,2)}. Then, $\rhobar|_\IQptwo\cong\omega_4^{p+2p^2+p^3}\oplus\omega_4^{2+p+p^3}$ from $\mathbf{Case}_\phi(\vr;(2,\frac{3}{2}))$ (see \S\ref{subsubsec: vr=(2,2), vk=(3/2,2),(2,3/2)} (i), as $S_5=\varphi(S_4)$).

\item Assume $\vL\in S_6$. Equivalently, we have $v_p(x_1^{(1,2)})=-1$ and $v_p(x_0^{(1,2)})>-1$. Also, if $-1<v_p(\xi_0)<0$, then $-1<v_p(\xi'_0)<0$ and $v_p(\xi'_1)=-1$, by Lemma~\ref{lemma: implicit vx^(2,2)} (iii). In this case there are two non-homothetic lattices whose mod-$p$ reductions are non-split and reducible.
    \begin{itemize}[leftmargin=*]
    \item We first consider $\rhobar$ from $\mathbf{Case}_\phi(\vr;(\frac{1}{2},2))$ (see \S\ref{subsubsec: vr=(2,2), vk=(1/2,2),(2,1/2)}~(ii)). In this case we have $\rhobar|_\IQptwo\cong
    \begin{bmatrix}
        \omega_2 & * \\ 0 & \omega_2^{1+2p}
    \end{bmatrix}$ with $\vec{\mathrm{MT}}=(0,1)$.
    \item We also consider $\rhobar$ from  $\mathbf{Case}_\phi(\vr;(\frac{3}{2},2))$ or $\mathbf{Case}_\phi(\vr;(1,1))$ (see \S\ref{subsubsec: vr=(2,2), vk=(3/2,2),(2,3/2)}~(v) if $-1<v_p(\xi_0)<0$ and from \S\ref{subsubsec: vr=(2,2), vk=(1,1)}~(ii) if $v_p(\xi_0)\geq 0$). In this case we have $\rhobar_2|_\IQptwo\cong
    \begin{bmatrix}
        \omega_2^{1+2p} & * \\ 0 & \omega_2
    \end{bmatrix}$ with $\vec{\mathrm{MT}}=(0,0)$.
    \end{itemize}

\item Assume $\vL\in S_7$. Equivalently, we have $v_p(x_0^{(2,1)})=-1$ and $v_p(x_1^{(2,1)})>-1$. Also, if $-1<v_p(\xi_1)<0$, then $-1<v_p(\xi'_1)<0$ and $v_p(\xi'_0)=-1$, by Lemma~\ref{lemma: implicit vx^(2,2)} (iii). In this case there are two non-homothetic lattices whose mod-$p$ reductions are non-split and reducible.
    \begin{itemize}[leftmargin=*]
    \item We first consider $\rhobar$ from $\mathbf{Case}_\phi(\vr;(2,\frac{1}{2}))$ (see \S\ref{subsubsec: vr=(2,2), vk=(1/2,2),(2,1/2)}~(ii), as $S_7=\varphi(S_6)$). In this case we have $\rhobar|_\IQptwo\cong
    \begin{bmatrix}
        \omega_2^p & * \\ 0 & \omega_2^{2+p}
    \end{bmatrix}$ with $\vec{\mathrm{MT}}=(1,0)$.
    \item We also consider $\rhobar$ from $\mathbf{Case}_\phi(\vr;(2,\frac{3}{2}))$ or $\mathbf{Case}_\phi(\vr;(1,1))$ (see \S\ref{subsubsec: vr=(2,2), vk=(3/2,2),(2,3/2)}~(v) if $-1<v_p(\xi_1)<0$ and from \S\ref{subsubsec: vr=(2,2), vk=(1,1)}~(ii) if $v_p(\xi_1)\geq 0$, as $S_7=\varphi(S_6)$). In this case we have $\rhobar_2|_\IQptwo\cong
    \begin{bmatrix}
        \omega_2^{2+p} & * \\ 0 & \omega_2^p
    \end{bmatrix}$ with $\vec{\mathrm{MT}}=(0,0)$.
    \end{itemize}

\item Assume $\vL\in S_8$. Extending $E$ large enough, we have $-1<v_p(\xi'_0)<0$ and $v_p(\xi'_0)+v_p(\xi'_1)=-1$, by Lemma~\ref{lemma: implicit vx^(2,2)} (ii). In this case, we consider either $\mathbf{Case}_\phi(\vr;(\frac{3}{2},2))$ or $\mathbf{Case}_\phi(\vr;(2,\frac{3}{2}))$, which coincides if $v_p(\xi_0)=-\frac{1}{2}$ (see \S\ref{subsubsec: vr=(2,2), vk=(3/2,2),(2,3/2)}~(ii)). It is shown that $\rhobar|_\IQptwo\cong
\begin{bmatrix}
    \omega & * \\ 0 & \omega
\end{bmatrix}$ with $\vec{\mathrm{MT}}=(0,1)$ if $-1<v_p(\xi_0)<-\frac{1}{2}$, $\vec{\mathrm{MT}}=(1,1)$ if $v_p(\xi_0)=-\frac{1}{2}$, and $\vec{\mathrm{MT}}=(1,0)$ if $-\frac{1}{2}<v_p(\xi_0)<0$. 

Now we consider $\vL\in S'_8\subset S_8$, consisting $\vL\in S_8$ with $v_p(p\xi_0\xi_1-4)>0$. In this case, we also consider $\mathbf{Case}_\phi(\vr;(1,1))$ to get another lattice giving non-split reducible mod-$p$ reduction (see \S\ref{subsubsec: vr=(2,2), vk=(1,1)}~(iii)). It is shown that $\rhobar|_\IQptwo\cong
\begin{bmatrix}
    \omega & * \\ 0 & \omega
\end{bmatrix}$ with $\vec{\mathrm{MT}}=(0,0)$.

\item Assume $\vL\in S_9$. Equivalently, we have
\begin{itemize}[leftmargin=*]
    \item $v_p(x_0^{\vl})=v_p(x_1^{\vl})\leq 0$ for $\vl=(1,2)$ and $(2,1)$, if $v_p(\xi_0)\leq -1$;
    \item $-1\leq v_p(\xi'_0)=v_p(\xi'_1)<-\frac{1}{2}$, if $-1\leq v_p(\xi_0)<-\frac{1}{2}$, by Lemma~\ref{lemma: implicit vx^(2,2)}.
\end{itemize}
In this case there are two non-homothetic lattices whose mod-$p$ reductions are non-split and reducible.
    \begin{itemize}[leftmargin=*]
    \item We first consider $\mathbf{Case}_\phi(\vr;(\frac{1}{2},2))$ or $\mathbf{Case}_\phi(\vr;(\frac{3}{2},2))$, which coincide if $v_p(\xi_0)=-1$ (see \S\ref{subsubsec: vr=(2,2), vk=(1/2,2),(2,1/2)}~(iii)(iv) if $v_p(\xi_0)\leq -1$ and \S\ref{subsubsec: vr=(2,2), vk=(3/2,2),(2,3/2)}~(iii),(iv) if $-1\leq v_p(\xi_0)<-\frac{1}{2}$). It is shown that $\rhobar|_\IQptwo\cong
    \begin{bmatrix}
        \omega_2^2 & * \\ 0 & \omega_2^{2p}
    \end{bmatrix}$ with $\vec{\mathrm{MT}}=(0,0)$ if $v_p(\xi_0)<-1$ and $\vec{\mathrm{MT}}=(0,1)$ if $-1\leq v_p(\xi_0)<-\frac{1}{2}$.
    \item We also consider $\mathbf{Case}_\phi(\vr;(2,\frac{1}{2}))$ or $\mathbf{Case}_\phi(\vr;(2,\frac{3}{2}))$, which coincide if $v_p(\xi_1)=-1$ (see \S\ref{subsubsec: vr=(2,2), vk=(1/2,2),(2,1/2)}~(iii)(iv) if $v_p(\xi_1)\leq -1$ and \S\ref{subsubsec: vr=(2,2), vk=(3/2,2),(2,3/2)}~(iii),(iv) if $-1\leq v_p(\xi_1)<-\frac{1}{2}$). It is shown that $\rhobar|_\IQptwo\cong
    \begin{bmatrix}
        \omega_2^{2p} & * \\ 0 & \omega_2^2
    \end{bmatrix}$ with $\vec{\mathrm{MT}}=(0,0)$ if $v_p(\xi_1)<-1$ and $\vec{\mathrm{MT}}=(1,0)$ if $-1\leq v_p(\xi_1)<-\frac{1}{2}$.
    \end{itemize}
\end{itemize}

\subsection{An example of the case $f=2$: $\vr=(1,5)$}
In this subsection, we apply our main result, Theorem~\ref{theo: main 2}, to compute the mod-$p$ reduction of $2$-dimensional semi-stable representations of $G_{\Q_{p^2}}$ with $j$-labeled Hodge--Tate weights $(0,r_j)$ where $\vr=(1,5)$. We set $r=5$ and assume $5<p-1$.

We first need to compute the valid $\vk'\in\cK(\vr)$, and it is immediate from \eqref{eq: condition for infinity} that $\cJ_\infty(\vk')\in\{\emptyset,\{0\}\}$ for all $\vk'\in\cK(\vr)$. Recall that we say that $\vk'$ is valid if $R(\vr;\vk')\neq\emptyset$, i.e., the equations and inequalities have common solutions. Following the notation in \eqref{eq: condition for infinity}, by $\cJ_0=\emptyset$ (resp. by $\cJ_0=\{0\}$) we mean we will consider $\vk'$ with $\cJ_\infty(\vk')=\emptyset$ (resp. with $\cJ_\infty(\vk')=\{0\}$). It turns out that it is enough to consider those $\vk'\in\cK(\vr)$ such that $R_{int}(\vr;\vk')\neq\emptyset$, and such $\vk'$ is listed as follow:
\begin{itemize}[leftmargin=*]
\item if $\cJ_\infty(\vk')=\emptyset$ then $\vk'$ is one of the following
$$(\tfrac{1}{2},2),(\tfrac{1}{2},\tfrac{5}{2}),(\tfrac{1}{2},3),(\tfrac{1}{2},4),(\tfrac{1}{2},5),(1,\tfrac{3}{2}),(1,\tfrac{5}{2}), (1,\tfrac{7}{2}),(1,\tfrac{9}{2}),(\tfrac{3}{2},1),(\tfrac{3}{2},\tfrac{3}{2}),(\tfrac{3}{2},2),(\tfrac{3}{2},4),(\tfrac{3}{2},5);$$
\item if $\cJ_\infty(\vk')=\{0\}$ then $\vk'$ is one of the following
$$(\infty,1),(\infty,\tfrac{3}{2}),(\infty,2),(\infty,4),(\infty,5).$$
\end{itemize}
Moreover, we describe the areal support $R(\vr;\vk')$ for each such $\vk'$ in the table, Table~\ref{tab: J((1,5))}.
\begin{table}[htbp]
  \centering
  \begin{tabular}{|c||c|c|}
    \hline
    $\vk'$ & $R(\vr;\vk')$ & $\vx$ \\\hline\hline
    $(\frac{1}{2},2)$ & $t_0\geq -1$, $t_1\geq -1$ & {$\vx^{(1,2)}$} \\\hline
    $(\frac{1}{2},\frac{5}{2})$ & $t_0\geq -1$, $t_1\geq -1$ & {$\vx^{(1,3)}$} \\\hline
    $(\frac{1}{2},3)$ & $-1\leq t_1<0$, $t_0-t_1\geq 0$ & {$\vx^{(1,3)}$} \\\hline
    $(\frac{1}{2},4)$ & $-2\leq t_1\leq -1$, $t_0-t_1\geq 1$ & {$\vx^{(1,2)}$} \\\hline
    $(\frac{1}{2},5)$ & $t_1\leq -2$, $t_0-t_1\geq 2$ & {$\vx^{(1,1)}$} \\\hline
    $(1,\frac{3}{2})$ & $-2\leq t_0\leq -1$, $t_0-t_1\leq 1$ & {$\vx^{(1,2)}$} \\\hline
    $(1,\frac{5}{2})$ & $-1\leq t_0\leq 0$, $t_0-t_1\leq 0$ & {$\vx^{(1,3)}$} \\\hline
    $(1,\frac{7}{2})$ & $t_0\leq -1$, $t_1\geq -1$, $t_0-t_1\leq 1$ & {$\vx^{(1,2)}$} \\\hline
    $(1,\frac{9}{2})$ & $-3\leq t_1\leq -2$, $1\leq t_0-t_1\leq 2$ & {$\vx^{(1,1)}$} \\\hline
    $(\frac{3}{2},1)$ & $t_0\leq -2$, $t_1\geq -3$ & {$\vx^{(0,1)}$} \\\hline
    $(\frac{3}{2},\frac{3}{2})$ & $t_0\leq -1$, $t_1\geq -3$, $t_0-t_1\leq 1$ & {$\vx^{(0,2)}$} \\\hline
    $(\frac{3}{2},2)$ & $t_0\leq -1$, $t_1\geq -2$ & {$\vx^{(0,2)}$} \\\hline
    $(\frac{3}{2},4)$ & $t_0-t_1\leq 1$, $-3\leq t_1\leq -2$ & {$\vx^{(0,2)}$} \\\hline
    $(\frac{3}{2},5)$ & $t_0-t_1\leq 2$, $t_1\leq -3$ & {$\vx^{(0,1)}$} \\\hline
    $(\infty,1)$ & $t_1\geq -3$ & {$\vx^{(0,1)}$} \\\hline
    $(\infty,\frac{3}{2})$ & $t_1\geq -3$ & {$\vx^{(0,2)}$} \\\hline
    $(\infty,2)$ & $t_1\geq -2$ & {$\vx^{(0,2)}$} \\\hline
    $(\infty,4)$ & $-3\leq t_1\leq 2$ &  {$\vx^{(0,2)}$}\\\hline
    $(\infty,5)$ & $t_1\leq -3$ & {$\vx^{(0,1)}$} \\
    \hline
  \end{tabular}
  \caption{Areal support for $\vr=(1,5)$}\label{tab: J((1,5))}
\end{table}

We further note that
\begin{itemize}[leftmargin=*]
\item $\Supp(\vr;\vk')\subset \Supp(\vr;(\frac{1}{2},\frac{5}{2}))$ for $\vk'=(\frac{1}{2},2),(\frac{1}{2},3),(1,\frac{5}{2})$;
\item $\Supp(\vr;\vk')\subset \Supp(\vr;(\frac{3}{2},\frac{3}{2}))$ for $\vk'=(1,\frac{3}{2})$, $(\frac{3}{2},1)$, $(\frac{3}{2},2)$, and $(\frac{3}{2},4)$;
\item $\Supp(\vr;\vk')\subset \Supp(\vr;(\infty,\frac{3}{2}))$ for $\vk'=(\infty,1)$, $(\infty,2)$, and $(\infty,4)$.
\end{itemize}
Hence, we exclude these redundancies and consider only
\begin{equation}\label{eq: list of strongly valid for 1,5}
    \{(\tfrac{1}{2},\tfrac{5}{2}),(\tfrac{1}{2},4),(\tfrac{1}{2},5),(1,\tfrac{7}{2}),(1,\tfrac{9}{2}),(\tfrac{3}{2},\tfrac{3}{2}),(\tfrac{3}{2},5)\} \quad\mbox{and}\quad
    \{(\infty,\tfrac{3}{2}),(\infty,5)\}.
\end{equation}
We will write $J(\vr;\emptyset)$ (resp. $J(\vr;\{0\})$) for the set of $\vk'$ with $\cJ_\infty(\vk')=\emptyset$ (resp. with $\cJ_\infty(\vk')=\{0\}$) listed in \eqref{eq: list of strongly valid for 1,5}. Moreover, we let $J(\vr):=J(\vr;\emptyset)\cup J(\vr;\{0\})$.

Each areal support $R(\vr;\vk')$ for $\vk'\in J(\vr)$ is depicted in Figure~\ref{fig: vr=(1,5)}. We point out that the pairs of (half) integers in the figure indicate $\vk'\in J(\vr)$, and that each polygon indicate the areal support $R(\vr;\vk')$.
\begin{figure}[htbp]
  \centering
  \includegraphics[scale=0.5]{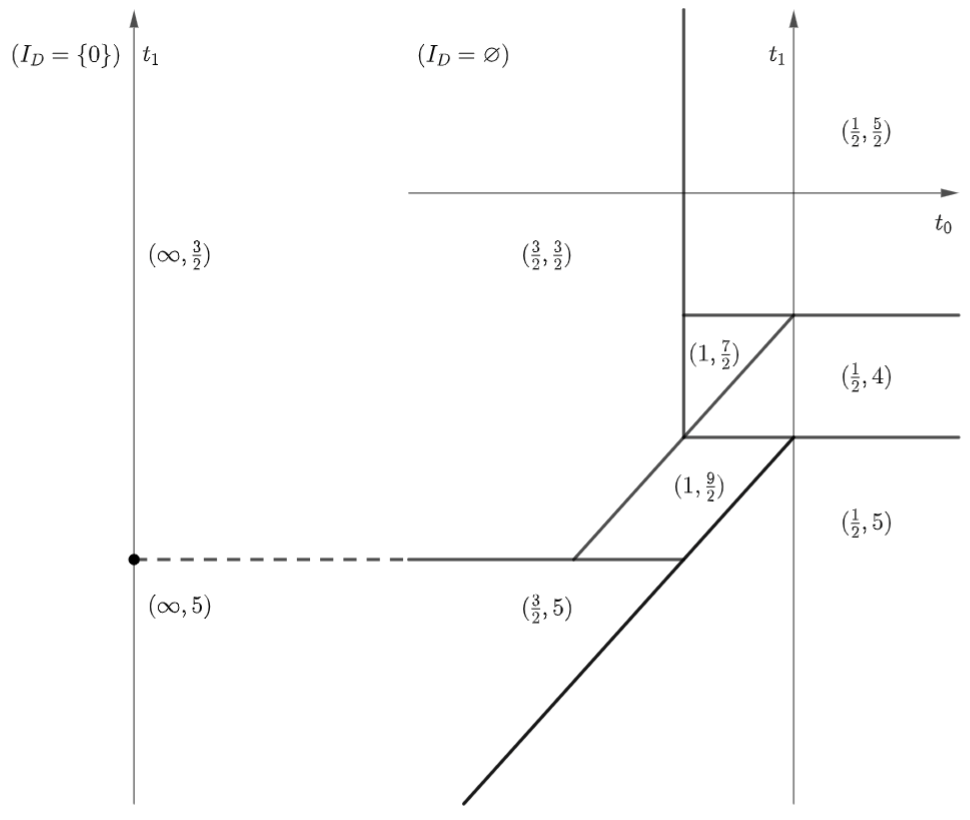}
  \caption{$R(\vr;\vk')$ for $\vr=(1,5)$ and $\vk'\in J(\vr)$}\label{fig: vr=(1,5)}
\end{figure}

For the rest of this subsection, we compute the mod-$p$ reduction of semi-stable representations in each $\mathbf{Case}_\phi(\vr;\vk')$ for $\vk'\in J(\vr)$. Then we summarize the mod-$p$ reduction when $\vr=(1,5)$ at the end of this subsection, \S\ref{subsubsec: summary for 1,5}.

\subsubsection{\textbf{For $\vk'=(\frac{1}{2},\frac{5}{2})$}}\label{subsubsec: vr=(1,5), vk=(1/2,5/2)}
In this section of paragraph, by $\vk'$ we always mean $(\frac{1}{2},\frac{5}{2})$. In this case, we have $\vx=\vx^{(1,3)}$, and we will write $\cM$ for the Breuil modules. From the results in \S\ref{subsec: mod p reduction}, the Breuil modules $\cM$ is determined by the following data:
\begin{itemize}[leftmargin=*]
  \item $M_{N,0}=
  \begin{bmatrix}
    0 & 0 \\
    \frac{1-12u^{p-1}}{\Theta_0} & 0
  \end{bmatrix}\quad\&\quad
  M_{N,1}=
  \begin{bmatrix}
    0 & 0 \\
    \frac{1}{\Theta_1} & 0
  \end{bmatrix}$;
  \item $M_{\Fil^5,0}=
  \begin{bmatrix}
    1 & 0 \\
    \frac{x_0}{\Theta_0} & 1
  \end{bmatrix}
  \begin{bmatrix}
    u^4 & 0 \\
    0 & u^5
  \end{bmatrix}\quad\&\quad
  M_{\Fil^5,1}=
  \begin{bmatrix}
    1 & 0 \\
    \frac{x_1}{\Theta_1}+\frac{3\alpha_1}{\beta_1\Theta_0} & 1
  \end{bmatrix}
  \begin{bmatrix}
    u^2 & 0 \\
    0 & u^3
  \end{bmatrix}$;
  \item $M_{\phi_5,0}=
  \begin{bmatrix}
    \alpha_0 & 0 \\
    0 & -\beta_0
  \end{bmatrix}\quad\&\quad
  M_{\phi_5,1}=
  \begin{bmatrix}
    1 & 6p\Theta_0 \\
    \frac{12u^{p-1}}{\Theta_0} & 1
  \end{bmatrix}
  \begin{bmatrix}
    \alpha_1 & 0 \\
    0 & -\beta_1
  \end{bmatrix}.$
\end{itemize}
Note that one can choose any $\vTh$ satisfying
$$-1\leq v_p(\Theta_0)=v_p(\Theta_1)\leq \min\{0,v_p(x_0),v_p(x_1)\}.$$
We will divide $R(\vr;\vk')$ into following four areas, and in each area we will determine the mod-$p$ reduction:
\begin{enumerate}[leftmargin=*]
  \item $\vec{t}\in R_{int}(\vr;\vk')$: Note that we have either $v_p(\Theta_0)\not\in\{0,v_p(x_0),v_p(x_1)\}$ or $v_p(\Theta_0)\neq -1$, in this case. Consider two $\barS_{\F}$-submodules $\cM'=\barS_{\F}(\barE_1)$ and $\cM''=\barS_{\F}(\barE_2)$ of $\cM$. Then one can observe that
    \begin{itemize}[leftmargin=*]
      \item if $v_p(\Theta_0)\not\in\{0,v_p(x_0),v_p(x_1)\}$, then $\cM'$ is a Breuil submodule of $\cM$ such that $\cM'\cong\cM(4,2;\alpha_0,\alpha_1)$ and the submodule $\cM'''=\barS_\F(\barE_2^{(0)},6p\Theta_0\barE_1^{(1)}+\barE_2^{(1)})$ gives $\cM'''\cong\cM(5,3;-\beta_0,-\beta_1)$ and $\cM=\cM'\oplus\cM'''$;
      \item if $v_p(\Theta_0)\neq -1$, then $\cM''$ is a Breuil submodule of $\cM$ and it satisfies $\cM''\cong\cM(5,3;-\beta_0,-\beta_1)$ and $\cM/\cM''\cong\cM(4,2;\alpha_0,\alpha_1)$. It is not difficult to show that the short exact sequence determined by $\cM''\hookrightarrow\cM$ is non-split.
    \end{itemize}
    So $\rhobar^{ss}|_\IQptwo\cong\omega_2^{1+3p}\oplus\omega_2^{2p}$, by Lemma~\ref{lemm: Tst^r for irred BrMod}~(i). In addition, the monodromy type is given as follow.
    \begin{enumerate}[leftmargin=*]
      \item If $v_p(\Theta_0)=0$, then $t_0\geq 0$ and $t_1\geq 0$ and the monodromy type is $(1,1)$.
      \item If $v_p(\Theta_0)\neq 0$, then the monodromy type is $(0,0)$. Note that, in particular, if $\cM''$ is a Breuil submodule of $\cM$ but $\cM'$ is not, then either $v_p(\Theta_0)=v_p(x_0)$ or $v_p(x_1)$, and so we have either $t_0<0$ or $t_1<0$.
    \end{enumerate}
  \item $t_1=-1$, $t_0>-1$: We have $v_p(\Theta_0)=-1=v_p(x_1)\neq v_p(x_0)$, and so $\frac{x_1}{\Theta_1},p\Theta_0,p\Theta_1\in\F^\times$ and $\frac{1}{\Theta_0}=\frac{1}{\Theta_1}=\frac{x_0}{\Theta_0}=0$ in $\F$. In particular, the monodromy type is $(0,0)$. Put $\cM'=\barS_{\F}(6p\Theta_0\barE_1^{(0)}+\barE_2^{(0)},\barE_2^{(1)})$. It is easy to see that $\cM'\cong\cM(5,3;-\beta_0,-\beta_1)$ and $\cM/\cM'=\barS_{\F}(\barE_1+\cM')\cong\cM(4,2;\alpha_0,\alpha_1)$, and so we conclude that $\rhobar^{ss}|_\IQptwo\cong\omega_2^{2p}\oplus\omega_2^{1+3p}$, by Lemma~\ref{lemm: Tst^r for irred BrMod}~(i). It is not difficult to show that the short exact sequence determined by $\cM'\hookrightarrow \cM$ is non-split.
  \item $t_0=t_1=-1$: We have $v_p(\Theta_0)=-1=v_p(x_0)=v_p(x_1)$, and so $\frac{x_0}{\Theta_0},\frac{x_1}{\Theta_1},p\Theta_0\in\F^\times$ and $\frac{1}{\Theta_0}=\frac{1}{\Theta_1}=0$ in $\F$. In particular, the monodromy type is $(0,0)$.
      \begin{enumerate}[leftmargin=*]
        \item Assume $px_0=\frac{1}{6}$ in $\F$, and put $\cM'=\barS_{\F}(\barE_1^{(0)}+\frac{x_0}{\Theta_0}\barE_2^{(0)},\barE_1^{(1)})$. It is easy to see that $\cM'$ satisfies $\cM'\cong\cM(4,3;\alpha_0,\frac{6p\beta_1\Theta_0x_1}{\Theta_1})$ and $\cM/\cM'=\barS_{\F}(\barE_2+\cM') \cong\cM(5,2;-\beta_0,-\frac{\alpha_1\Theta_1x_0}{\Theta_0x_1})$, and so $\rhobar^{ss}|_\IQptwo\cong\omega_2^{1+2p}\oplus\omega_2^{3p}$, by Lemma~\ref{lemm: Tst^r for irred BrMod}~(i). It is not difficult to show that the short exact sequence determined by $\cM'\hookrightarrow \cM$ is non-split.
        \item Assume $px_0\neq\frac{1}{6}$ in $\F$, put $\cM'=\barS_{\F}(6p\Theta_0\barE_1^{(0)}+\barE_2^{(0)},\barE_2^{(1)})$. It is easy to see that $\cM'\cong\cM(5,3;-\beta_0(1-6px_0),-\beta_1)$ and $\cM/\cM'=\barS_{\F}(\barE_1+\cM')\cong\cM(4,2;\frac{\alpha_0}{1-6px_0},\alpha_1)$, and so $\rhobar^{ss}|_\IQptwo\cong\omega_2^{2p}\oplus\omega_2^{1+3p}$, by Lemma~\ref{lemm: Tst^r for irred BrMod}~(i).It is not difficult to show that the short exact sequence determined by $\cM'\hookrightarrow \cM$ is non-split.
      \end{enumerate}
  \item $t_0=-1$, $t_1>-1$: We have $v_p(\Theta_0)=-1=v_p(x_0)\neq v_p(x_1)$, and so $\frac{x_0}{\Theta_0},p\Theta_0\in\F^\times$ and $\frac{1}{\Theta_0}=\frac{1}{\Theta_1}=\frac{x_1}{\Theta_1}=0$ in $\F$. In particular, the monodromy type is $(0,0)$.
      \begin{enumerate}[leftmargin=*]
        \item Assume $px_0=\frac{1}{6}$ in $\F$. It is easy to see that changing a basis to $\baseE'^{(0)}=(\barE_1^{(0)}, 6p\Theta_0\barE_1^{(0)}+\barE_2^{(0)})$ and $\baseF'^{(0)}=(u\barF_1^{(0)}-\frac{x_0}{\Theta_0}\barF_2^{(0)},6p\Theta_0\barF_1^{(1)})$ gives $\cM\cong\cM_0((5,2),(4,3); (6p\Theta_0\alpha_0,\alpha_1),(\frac{\beta_0x_0}{\Theta_0},-\beta_1))$, and so we conclude that $\rhobar|_\IQptwo \cong\omega_4^{3p+p^2+2p^3}\oplus\omega_4^{1+2p+3p^3}$, by Lemma~\ref{lemm: Tst^r for irred BrMod}~(ii).
        \item Assume $px_0\neq\frac{1}{6}$ in $\F$, and put $\cM'$ as in (iii)(b). By the same isomorphisms for $\cM'$ and $\cM/\cM'$, we conclude that $\rhobar^{ss}|_\IQptwo\cong\omega_2^{2p}\oplus\omega_2^{1+3p}$, by Lemma~\ref{lemm: Tst^r for irred BrMod}~(i).It is not difficult to show that the short exact sequence determined by $\cM'\hookrightarrow \cM$ is non-split.
      \end{enumerate}
\end{enumerate}
We summarize the mod-$p$ reduction for $\vk'=(\frac{1}{2},\frac{5}{2})$ in the following table, Table~\ref{tab: (1/2,5/2)}. Note that by $\ast$ in the following table we mean a non-split extension.
\begin{small}
\begin{table}[htbp]
  \centering
  \begin{tabular}{|c|c||c|c|}
    \hline
    \multicolumn{2}{|c||}{$(t_0,t_1)$} & $\rhobar|_{I_{\Q_{p^2}}}$ & $\vec{\mathrm{MT}}$ \\\hline\hline
    \multirow{2}{*}[-1.5em]{$R_{int}(\vr;(\frac{1}{2},\frac{5}{2}))$}
        & $v_p(\Theta_0)\not\in\{0,v_p(x_0),v_p(x_1)\}$
        & $
        \begin{bmatrix}
            \omega_2^{1+3p} & 0 \\ 0 & \omega_2^{2p}
        \end{bmatrix}$
        & $(0,0)$ \\\cline{2-4}
        & otherwise
        & $
        \begin{bmatrix}
            \omega_2^{1+3p} & * \\ 0 & \omega_2^{2p}
        \end{bmatrix}$
        & $
        \begin{cases}
            (1,1) &\mbox{if } t_0\geq 0 \\&\mbox{ and }t_1\geq 0;\\
            (0,0) &\mbox{otherwise}
        \end{cases}$ \\\hline
    \multicolumn{2}{|c||}{$t_1=-1$, $t_0>-1$}
        & $
        \begin{bmatrix}
          \omega_2^{1+3p} & * \\
          0 & \omega_2^{2p}
        \end{bmatrix}$ & $(0,0)$ \\\hline
    \multirow{2}{*}[-0.5em]{$t_0=t_1=-1$}
      & $v_p(px_0-\frac{1}{6})>0$
        & $
        \begin{bmatrix}
          \omega_2^{3p} & * \\
          0 & \omega_2^{1+2p}
        \end{bmatrix}$ & $(0,0)$ \\\cline{2-4}
      & $v_p(px_0-\frac{1}{6})=0$
        & $
        \begin{bmatrix}
          \omega_2^{1+3p} & * \\
          0 & \omega_2^{2p}
        \end{bmatrix}$ & $(0,0)$ \\\hline
    \multirow{2}{*}[-0.5em]{
    \begin{tabular}{c}
      $t_0=-1$, \\ $t_1>-1$
    \end{tabular}}
      & $v_p(px_0-\frac{1}{6})>0$
        & $\omega_4^{3p+p^2+2p^3}\oplus\omega_4^{1+2p+3p^3}$ & $(0,0)$ \\\cline{2-4}
      & $v_p(px_0-\frac{1}{6})=0$
        & $
        \begin{bmatrix}
          \omega_2^{1+3p} & * \\
          0 & \omega_2^{2p}
        \end{bmatrix}$ & $(0,0)$ \\\hline
  \end{tabular}
  \caption{$\rhobar$ for $\vk'=(\frac{1}{2},\frac{5}{2})$}\label{tab: (1/2,5/2)}
\end{table}
\end{small}

\subsubsection{\textbf{For $\vk'=(\frac{1}{2},4)$}}\label{subsubsec: vr=(1,5), vk=(1/2,4)}
In this section of paragraph, by $\vk'$ we always mean $(\frac{1}{2},4)$. In this case, we have $\vx=\vx^{(1,2)}$, and we will write $\cM$ for the Breuil modules. From the results in \S\ref{subsec: mod p reduction}, the Breuil modules $\cM$ is determined by the following data:
\begin{itemize}[leftmargin=*]
  \item $M_{N,0}=
  \begin{bmatrix}
    0 & 0 \\
    \frac{1-10u^{p-1}}{\Theta_0} & 0
  \end{bmatrix}\quad\&\quad
  M_{N,1}=
  \begin{bmatrix}
    0 & 0 \\
    \frac{1}{\Theta_1} & 0
  \end{bmatrix}$;
  \item $M_{\Fil^5,0}=
  \begin{bmatrix}
    1 & 0 \\
    \frac{x_0}{\Theta_0} & 1
  \end{bmatrix}
  \begin{bmatrix}
    u^4 & 0 \\
    0 & u^5
  \end{bmatrix}\quad\&\quad
  M_{\Fil^5,1}=
  \begin{bmatrix}
    1 & \frac{\Theta_1}{x_1} \\
    0 & 1
  \end{bmatrix}
  \begin{bmatrix}
    u^4 & \frac{\beta_1u^3}{18\alpha_1\Theta_0} \\
    0 & u
  \end{bmatrix}$;
  \item $M_{\phi_5,0}=
  \begin{bmatrix}
    \alpha_0 & 0 \\
    0 & -\beta_0
  \end{bmatrix}\quad\&\quad
  M_{\phi_5,1}=
  \begin{bmatrix}
    1 & \frac{4p\Theta_0}{3} \\
    \frac{10u^{p-1}}{\Theta_0}-\frac{13}{36p\Theta_0x_1} & 1
  \end{bmatrix}
  \begin{bmatrix}
    0 & \frac{\alpha_1}{3} \\
    3\beta_1 & 0
  \end{bmatrix}.$
\end{itemize}
Note that we have $v_p(\Theta_0)=v_p(\Theta_1)=1+v_p(x_1)$, and so $\frac{\Theta_1}{x_1}=0$ in $\F$. We will divide $R(\vr;\vk')$ into following five areas, and in each area we will determine the mod-$p$ reduction:
\begin{enumerate}[leftmargin=*]
  \item $\vec{t}\in R_{int}(\vr;\vk')$: It is easy to see that $\cM\cong\cM_1((4,4),(5,1);(\alpha_0,\frac{\alpha_1}{3}),(-\beta_0,3\beta_1))$, and so we conclude that $\rhobar|_\IQptwo \cong\omega_4^{1+4p+p^3}\oplus\omega_4^{p+p^2+4p^3}$ by Lemma~\ref{lemm: Tst^r for irred BrMod}~(ii).
  \item $t_1=-1$, $t_0\geq 0$: We have $\frac{1}{\Theta_0},\frac{1}{\Theta_1},\frac{1}{px_1}\in\F^\times$ and $p\Theta_0=0$ in $\F$. In particular, the monodromy type is $(1,1)$. Put $\cM'=\barS_{\F}(\barE_2)\subset\cM$. It is easy to see that $\cM'\cong\cM(5,2;-\beta_0,-\frac{\alpha_1}{18\Theta_0})$ and $\cM/\cM'\cong\cM(4,3;\alpha_0,18\beta_1\Theta_0)$, and so we conclude that $\rhobar^{ss}|_\IQptwo \cong\omega_2^{3p}\oplus\omega_2^{1+2p}$, by Lemma~\ref{lemm: Tst^r for irred BrMod}~(i), and $\rhobar$ is non-split as the monodromy type is nonzero.
  \item $t_0-t_1=1$, $-1<t_0<0$: We have $\frac{x_0}{\Theta_0}\in\F^\times$ and $\frac{1}{\Theta_0}=\frac{1}{\Theta_0}=p\Theta_0=\frac{1}{px_1}=0$ in $\F$. In particular, the monodromy type is $(0,0)$. Put $\cM'=\barS_{\F}(\barE_1^{(0)},\barE_2^{(1)})\subseteq\cM$. It is easy to see that $\cM'\cong\cM(5,1;\frac{\beta_0x_0}{\Theta_0},\alpha_1)$ and $\cM/\cM'\cong\cM(4,4;\frac{\alpha_0\Theta_0}{x_0},\beta_1)$, and so we conclude that $\rhobar^{ss}|_\IQptwo\cong \omega_2^{4p}\oplus\omega_2^{1+p}$, by Lemma~\ref{lemm: Tst^r for irred BrMod}~(i).It is not difficult to show that the short exact sequence determined by $\cM'\hookrightarrow \cM$ is non-split.
  \item $t_0=-1$, $t_1=-2$: We have $\frac{x_0}{\Theta_0},p\Theta_0\in\F^\times$ and $\frac{1}{\Theta_0}=\frac{1}{\Theta_1}=\frac{1}{px_1}=0$ in $\F$. In particular, the monodromy type is $(0,0)$.
     \begin{enumerate}[leftmargin=*]
      \item Assume $px_0=\frac{3}{4}$ in $\F$. Put $\cM':=\barS_{\F}\big(\barE_1^{(0)}+\frac{x_0}{\Theta_0}\barE_2^{(0)},\barE_1^{(1)}\big)$ and $\cM'':=\barS_{\F}(\barE_1^{(0)},\barE_2^{(1)})$. Then one can readily see that
          $\cM'\cong\cM(4,4;\alpha_0,4p\beta_1\Theta_0)$ and $\cM''\cong\cM\big(5,1;\frac{\beta_0x_0}{\Theta_0}, \frac{\alpha_1}{3}\big)$, and so $\cM=\cM'\oplus\cM''$. Hence, we conclude that $\rhobar|_\IQptwo\cong\omega_2^{1+p}\oplus\omega_2^{4p}$, by Lemma~\ref{lemm: Tst^r for irred BrMod}~(i).
      \item Assume $px_0\neq \frac{3}{4}$ in $\F$. Using $\cM''$ in (a), it is easy to see that $\cM''\cong\cM(5,1;\frac{\beta_0x_0}{\Theta_0},\frac{\alpha_1}{3})$ and $\cM/\cM''\cong\cM(4,4;\alpha_0,4p\beta_1\Theta_0)$, and so we conclude that $\rhobar^{ss}|_\IQptwo\cong\omega_2^{4p}\oplus\omega_2^{1+p}$, by Lemma~\ref{lemm: Tst^r for irred BrMod}~(i). It is not difficult to show that the short exact sequence determined by $\cM'\hookrightarrow \cM$ is non-split.
      \end{enumerate}
  \item $t_1=-2$, $t_0>-1$: We have $p\Theta_0\in\F^\times$ and $\frac{1}{\Theta_0}=\frac{1}{\Theta_1}=\frac{x_0}{\Theta_0}=\frac{1}{px_1}=0$ in $\F$. In particular, the monodromy type is $(0,0)$. Changing a basis to $\baseE'^{(0)}=(\barE_1^{(0)}, \frac{4p\Theta_0}{3}\barE_1^{(0)}+\barE_2^{(0)})$ and $\baseF'^{(0)}=(\barF_1^{(0)},\frac{4p\Theta_0u}{3}\barF_1^{(0)}+\barF_2^{(0)})$
      gives rise to $\cM\cong\cM_1((4,4),(5,1);(\alpha_0,\frac{\alpha_1}{3}),(-\beta_0,3\beta_1))$, and so we conclude that $\rhobar|_\IQptwo \cong\omega_4^{1+4p+p^3}\oplus\omega_4^{p+p^2+4p^3}$ by Lemma~\ref{lemm: Tst^r for irred BrMod}~(ii).
\end{enumerate}
We summarize the mod-$p$ reduction for $\vk'=(\frac{1}{2},4)$ in the following table, Table~\ref{tab: (1/2,4)}. Note that by $\ast$ in the following table we mean a non-split extension.
\begin{table}[htbp]
  \centering
  \begin{tabular}{|c|c||c|c|}
    \hline
    \multicolumn{2}{|c||}{$(t_0,t_1)$} & $\rhobar|_{I_{\Q_{p^2}}}$ & $\vec{\mathrm{MT}}$ \\\hline\hline
    \multicolumn{2}{|c||}{$R_{int}(\vr;(\frac{1}{2},4))$}
      & $\omega_4^{1+4p+p^3}\oplus\omega_4^{p+p^2+4p^3}$ & $(0,0)$ \\\hline
    \multicolumn{2}{|c||}{$t_1=-1$, $t_0\geq 0$}
      & $
      \begin{bmatrix}
        \omega_2^{1+2p} & * \\
        0 & \omega_2^{3p}
      \end{bmatrix}$ & $(1,1)$ \\\hline
    \multicolumn{2}{|c||}{$t_0-t_1=1$, $-1<t_0<0$}
      & $
      \begin{bmatrix}
        \omega_2^{1+p} & * \\
        0 & \omega_2^{4p}
      \end{bmatrix}$ & $(0,0)$ \\\hline
    \multirow{2}{*}[-0.5em]{
    \begin{tabular}{c}
      $t_0=-1$, \\ $t_1=-2$
    \end{tabular}}
      & $v_p(px_0-\frac{3}{4})>0$
        & $
        \begin{bmatrix}
          \omega_2^{4p} & 0 \\
          0 & \omega_2^{1+p}
        \end{bmatrix}$ & $(0,0)$ \\\cline{2-4}
      & $v_p(px_0-\frac{3}{4})=0$
        & $
        \begin{bmatrix}
          \omega_2^{1+p} & * \\
          0 & \omega_2^{4p}
        \end{bmatrix}$ & $(0,0)$ \\\hline
    \multicolumn{2}{|c||}{$t_1=-2$, $t_0>-1$}
        & $\omega_4^{1+4p+p^3}\oplus\omega_4^{p+p^2+4p^3}$ & $(0,0)$ \\\hline
  \end{tabular}
  \caption{$\rhobar$ for $\vk'=(\frac{1}{2},4)$}\label{tab: (1/2,4)}
\end{table}

\subsubsection{\textbf{For $\vk'=(\frac{1}{2},5)$}}\label{subsubsec: vr=(1,5), vk=(1/2,5)}
In this section of paragraph, by $\vk'$ we always mean $(\frac{1}{2},5)$. In this case, we have $\vx=\vx^{(1,1)}$, and we will write $\cM$ for the Breuil modules. From the results in \S\ref{subsec: mod p reduction}, the Breuil modules $\cM$ is determined by the following data:
\begin{itemize}[leftmargin=*]
  \item $M_{N,0}=
  \begin{bmatrix}
    0 & 0 \\
    \frac{1-4u^{p-1}}{\Theta_0} & 0
  \end{bmatrix}\quad\&\quad
  M_{N,1}=
  \begin{bmatrix}
    0 & 0 \\
    \frac{1}{\Theta_1} & 0
  \end{bmatrix}$;
  \item $M_{\Fil^5,0}=
  \begin{bmatrix}
    1 & 0 \\
    \frac{x_0}{\Theta_0} & 1
  \end{bmatrix}
  \begin{bmatrix}
    u^4 & 0 \\
    0 & u^5
  \end{bmatrix}\quad\&\quad
   M_{\Fil^5,1}=
  \begin{bmatrix}
    1 & \frac{\Theta_1}{x_1} \\
    0 & 1
  \end{bmatrix}
  \begin{bmatrix}
    u^5 & \frac{\alpha_1u^4}{4\beta_1\Theta_0} \\
    0 & 1
  \end{bmatrix}$;
  \item $M_{\phi_5,0}=
  \begin{bmatrix}
    \alpha_0 & 0 \\
    0 & -\beta_0
  \end{bmatrix}\quad\&\quad
  M_{\phi_5,1}=
  \begin{bmatrix}
    1 & -\frac{\alpha_1}{\beta_1}\frac{p^5x_1}{\Theta_1} \\
    \frac{4u^{p-1}}{\Theta_0} & 1
  \end{bmatrix}
  \begin{bmatrix}
    0 & \alpha_1 \\
    \beta_1 & 0
  \end{bmatrix}.$
\end{itemize}
Note that we have $v_p(\Theta_0)=v_p(\Theta_1)=2+v_p(x_1)$, and so $\frac{\Theta_1}{x_1}=\frac{p^5x_1}{\Theta_1}=0$ in $\F$. We will divide $R(\vr;\vk')$ into following three areas, and in each area we will determine the mod-$p$ reduction:
\begin{enumerate}[leftmargin=*]
  \item $\vec{t}\in R_{int}(\vr;\vk')$: It is easy to see that $\cM\cong\cM_1((4,5),(5,0);(\alpha_0,\alpha_1),(-\beta_0,\beta_1))$, and so we conclude that $\rhobar^{ss}\cong\omega_4^{1+5p}\oplus\omega_4^{p^2+5p^3}$ by Lemma~\ref{lemm: Tst^r for irred BrMod}~(ii).
  \item $t_1=-2$, $t_0\geq 0$: We have $\frac{1}{\Theta_0},\frac{1}{\Theta_1}\in\F^\times$. In particular, the monodromy type is $(1,1)$. Put $\cM'=\barS_{\F}(\barE_2)\subset\cM$. It is easy that $\cM'\cong\cM(5,1;-\beta_0,-\frac{\alpha_1}{4\Theta_0})$ and $\cM/\cM'\cong\cM(4,4;\alpha_0,4\beta_1\Theta_0)$, and so we conclude that $\rhobar^{ss}|_\IQptwo \cong\omega_2^{4p}\oplus\omega_2^{1+p}$, by Lemma~\ref{lemm: Tst^r for irred BrMod}~(i), and $\rhobar$ is non-split as the monodromy type is nonzero.
  \item $t_0-t_1=2$, $t_0<0$: We have $\frac{x_0}{\Theta_0}\in\F^\times$ and $\frac{1}{\Theta_0}=\frac{1}{\Theta_1}=0$ in $\F$. In particular, the monodromy type is $(0,0)$. Put $\cM'=\barS_{\F}(\barE_1^{(0)},\barE_2^{(1)})\subset\cM$. It is easy to see that $\cM'\cong\cM(5,0;\frac{\beta_0x_0}{\Theta_0},\alpha_1)$ and $\cM/\cM'\cong\cM(4,5;\frac{\alpha_0\Theta_0}{x_0},\beta_1)$, and so we conclude that $\rhobar^{ss}|_\IQptwo \cong \omega_2^{5p}\oplus\omega_2$, by Lemma~\ref{lemm: Tst^r for irred BrMod}~(i). It is not difficult to show that the short exact sequence determined by $\cM'\hookrightarrow \cM$ is non-split.
\end{enumerate}
We summarize the mod-$p$ reduction for $\vk'=(\frac{1}{2},5)$ in the following table, Table~\ref{tab: (1/2,5)}. Note that by $\ast$ in the following table we mean a non-split extension.
\begin{table}[htbp]
  \centering
  \begin{tabular}{|c||c|c|}
    \hline
    $(t_0,t_1)$ & $\rhobar|_{I_{\Q_{p^2}}}$ & $\vec{\mathrm{MT}}$ \\\hline\hline
    $R_{int}(\vr;(\frac{1}{2},5))$
      & $\omega_4^{1+5p}\oplus\omega_4^{p^2+5p^3}$ & $(0,0)$ \\\hline
    $t_1=-2$, $t_0\geq 0$
      & $
      \begin{bmatrix}
        \omega_2^{1+p} & * \\
        0 & \omega_2^{4p}
      \end{bmatrix}$ & $(1,1)$ \\\hline
    $t_0-t_1=2$, $t_0<0$
      & $
      \begin{bmatrix}
        \omega_2 & * \\
        0 & \omega_2^{5p}
      \end{bmatrix}$ & $(0,0)$ \\\hline
  \end{tabular}
  \caption{$\rhobar$ for $\vk'=(\frac{1}{2},5)$}\label{tab: (1/2,5)}
\end{table}

\subsubsection{\textbf{For $\vk'=(1,\frac{7}{2})$}}\label{subsubsec: vr=(1,5), vk=(1,7/2)}
In this section of paragraph, by $\vk'$ we always mean $(1,\frac{7}{2})$. In this case, we have $\vx=\vx^{(1,2)}$, and we will write $\cM$ for the Breuil modules. From the results in \S\ref{subsec: mod p reduction}, the Breuil modules $\cM$ is determined by the following data:
\begin{itemize}[leftmargin=*]
  \item $M_{N,0}=
  \begin{bmatrix}
    0 & 0 \\
    \frac{1-10u^{p-1}}{\Theta_0} & 0
  \end{bmatrix}\quad\&\quad
  M_{N,1}=
  \begin{bmatrix}
    0 & 0 \\
    \frac{1}{\Theta_1} & 0
  \end{bmatrix}$;
  \item $M_{\Fil^5,0}=
  \begin{bmatrix}
    1 & \frac{\Theta_0}{x_0} \\
    0 & 1
  \end{bmatrix}
  \begin{bmatrix}
    u^5 & 0 \\
    0 & u^4
  \end{bmatrix}\quad\&\quad
  M_{\Fil^5,1}=
  \begin{bmatrix}
    1 & \frac{\Theta_1}{x_1} \\
    0 & 1
  \end{bmatrix}
  \begin{bmatrix}
    u^3 & 0 \\
    \frac{18\alpha_1u}{\beta_1\Theta_0} & u^2
  \end{bmatrix}$;
  \item $M_{\phi_5,0}=
  \begin{bmatrix}
    1 & -\frac{\alpha_0}{\beta_0}\frac{px_0}{\Theta_0} \\
    0 & 1
  \end{bmatrix}
  \begin{bmatrix}
    0 & \alpha_0 \\
    \beta_0 & 0
  \end{bmatrix}\quad\&\quad
  M_{\phi_5,1}=
  \begin{bmatrix}
    1 & -6p\Theta_0 \\
    \frac{10u^{p-1}}{\Theta_0} -\frac{13}{36p\Theta_0x_1} & 1
  \end{bmatrix}
  \begin{bmatrix}
    6\alpha_1 & 0 \\
    0 & -\frac{\beta_1}{6}
  \end{bmatrix}.$
\end{itemize}
Note that we have $v_p(\Theta_0)=-1+v_p(x_0)-v_p(x_1)$ and $v_p(\Theta_1)=1+v_p(x_0)+v_p(x_1)$. We will divide $R(\vr;\vk')$ into following five areas, and in each area we will determine the mod-$p$ reduction:
\begin{enumerate}[leftmargin=*]
  \item $\vec{t}\in R_{int}(\vr;\vk')$: It is easy to see that $\cM\cong\cM_0((5,3),(4,2);(\alpha_0,6\alpha_1),(\beta_0,-\frac{\beta_1}{6}))$, and so we conclude that $\rhobar|_\IQptwo\cong\omega_4^{2p+p^2+3p^3}\oplus\omega_4^{1+3p+2p^3}$ by Lemma~\ref{lemm: Tst^r for irred BrMod}~(ii).
  \item $t_1=-1$, $-1<t_0<0$: We have $\frac{\Theta_0}{x_0},\frac{1}{px_1}\in\F^\times$ and $\frac{1}{\Theta_0}=\frac{1}{\Theta_1}=\frac{\Theta_1}{x_1}=\frac{px_0}{\Theta_0}=p\Theta_0=0$ in $\F$. In particular, the monodromy type is $(0,0)$. Put $\cM'=\barS_{\F}(\barE_2)\subset\cM$. It is easy to see that $\cM'\cong\cM(5,2;-\frac{\beta_0\Theta_0}{x_0},-\frac{\beta_1}{6})$ and $\cM/\cM'\cong\cM(4,3;\frac{\alpha_0x_0}{\Theta_0},6\alpha_1)$, and so we conclude that $\rhobar^{ss}|_\IQptwo\cong \omega_2^{3p}\oplus\omega_2^{1+2p}$, by Lemma~\ref{lemm: Tst^r for irred BrMod}~(i). It is not difficult to show that the short exact sequence determined by $\cM'\hookrightarrow \cM$ is non-split.
  \item $t_0=0$, $t_1=-1$: We have $\frac{1}{\Theta_0},\frac{1}{\Theta_1},\frac{\Theta_0}{x_0},\frac{1}{px_1}\in\F^\times$ and $\frac{\Theta_1}{x_1}=\frac{px_0}{\Theta_0}=p\Theta_0=0$ in $\F$. In particular, the monodromy type is $(1,1)$.  Put $\cM'$ as in (ii). By the same isomorphisms for $\cM'$ and $\cM/\cM'$, we conclude that $\rhobar^{ss}|_\IQptwo\cong \omega_2^{3p}\oplus\omega_2^{1+2p}$, by Lemma~\ref{lemm: Tst^r for irred BrMod}~(i), and $\rhobar$ is non-split as the monodromy type is nonzero.
  \item $t_0-t_1=1$, $-1\leq t_0<0$: We have $\frac{1}{\Theta_0}\in\F^\times$ and $\frac{1}{\Theta_1}=\frac{\Theta_0}{x_0}=p\Theta_0=\frac{1}{px_1}=0$ in $\F$. In particular, the monodromy type is $(1,0)$.
      Put $\cM'=\barS_{\F}(\barE_2^{(0)},\barE_1^{(1)})\subset\cM$. It is easy to see that $\cM'\cong\cM(4,4;\alpha_0,\frac{3\alpha_1}{\Theta_0})$ and $\cM/\cM'\cong\cM(5,1;\beta_0,\frac{\beta_1\Theta_0}{3})$, so we conclude that $\rhobar^{ss}|_\IQptwo\cong\omega_2^{1+p}\oplus\omega_2^{4p}$, by Lemma~\ref{lemm: Tst^r for irred BrMod}~(i), and $\rhobar$ is non-split as the monodromy type is nonzero.
  \item $t_0=-1$, $-2<t_1\leq -1$: We have $\frac{\Theta_1}{x_1}\in\F^\times$ and $\frac{1}{\Theta_0}=\frac{1}{\Theta_1}=\frac{px_0}{\Theta_0}=0$ in $\F$. In particular, the monodromy type is $(0,0)$. Put $\cM'=\barS_{\F}(\barE_1^{(0)},\barE_2^{(1)})\subset\cM$. It is easy to see that $\cM'\cong\cM(5,3;\beta_0,\frac{6\alpha_1\Theta_1}{x_1})$ and $\cM/\cM'\cong\cM(4,2;\alpha_0,-\frac{\beta_1x_1}{6\Theta_1})$, and so we conclude that $\rhobar^{ss}|_\IQptwo\cong\omega_2^{2p}\oplus\omega_2^{1+3p}$, by Lemma~\ref{lemm: Tst^r for irred BrMod}~(i). It is not difficult to show that the short exact sequence determined by $\cM'\hookrightarrow \cM$ is non-split.
\end{enumerate}
We summarize the mod-$p$ reduction for $\vk'=(1,\frac{7}{2})$ in the following table, Table~\ref{tab: (1,7/2)}. Note that by $\ast$ in the following table we mean a non-split extension.
\begin{table}[htbp]
  \centering
  \begin{tabular}{|c||c|c|}
    \hline
    $(t_0,t_1)$ & $\rhobar|_{I_{\Q_{p^2}}}$ & $\vec{\mathrm{MT}}$ \\\hline\hline
    $R_{int}(\vr;(1,\frac{7}{2}))$
      & $\omega_4^{2p+p^2+3p^3}\oplus\omega_4^{1+3p+2p^3}$ & $(0,0)$ \\\hline
    $t_1=-1$, $-1<t_0<0$
      & $
      \begin{bmatrix}
        \omega_2^{1+2p} & * \\
        0 & \omega_2^{3p}
      \end{bmatrix}$ & $(0,0)$ \\\hline
    $t_0=0$, $t_1=-1$
      & $
      \begin{bmatrix}
        \omega_2^{1+2p} & * \\
        0 & \omega_2^{3p}
      \end{bmatrix}$ & $(1,1)$ \\\hline
    $t_0-t_1=1$, $-1\leq t_0<0$
      & $
      \begin{bmatrix}
        \omega_2^{4p} & * \\
        0 & \omega_2^{1+p}
      \end{bmatrix}$ & $(1,0)$ \\\hline
    $t_0=-1$, $-2<t_1\leq -1$
      & $
      \begin{bmatrix}
        \omega_2^{1+3p} & * \\
        0 & \omega_2^{2p}
      \end{bmatrix}$ & $(0,0)$ \\\hline
  \end{tabular}
  \caption{$\rhobar$ for $\vk'=(1,\frac{7}{2})$}\label{tab: (1,7/2)}
\end{table}

\subsubsection{\textbf{For $\vk'=(1,\frac{9}{2})$}}\label{subsubsec: vr=(1,5), vk=(1,9/2)}
In this section of paragraph, by $\vk'$ we always mean $(1,\frac{9}{2})$. In this case, we have $\vx=\vx^{(1,1)}$, and we will write $\cM$ for the Breuil modules. From the results in \S\ref{subsec: mod p reduction}, the Breuil modules $\cM$ is determined by the following data:
\begin{itemize}[leftmargin=*]
  \item $M_{N,0}=
  \begin{bmatrix}
    0 & 0 \\
    \frac{1-4u^{p-1}}{\Theta_0} & 0
  \end{bmatrix}\quad\&\quad
  M_{N,1}=
  \begin{bmatrix}
    0 & 0 \\
    \frac{1}{\Theta_1} & 0
  \end{bmatrix}$;
  \item $M_{\Fil^5,0}=
  \begin{bmatrix}
    1 & \frac{\Theta_0}{x_0} \\
    0 & 1
  \end{bmatrix}
  \begin{bmatrix}
    u^5 & 0 \\
    0 & u^4
  \end{bmatrix}\quad\&\quad
  M_{\Fil^5,1}=
  \begin{bmatrix}
    1 & \frac{\Theta_1}{x_1} \\
    0 & 1
  \end{bmatrix}
  \begin{bmatrix}
    u^4 & 0 \\
    \frac{4\alpha_1}{\beta_1\Theta_0} & u
  \end{bmatrix}$;
  \item $M_{\phi_5,0}=
  \begin{bmatrix}
    1 & -\frac{\alpha_0}{\beta_0}\frac{px_0}{\Theta_0} \\
    0 & 1
  \end{bmatrix}
  \begin{bmatrix}
    0 & \alpha_0 \\
    \beta_0 & 0
  \end{bmatrix}\quad\&\quad
  M_{\phi_5,1}=
  \begin{bmatrix}
    1 & -\frac{4p\Theta_0}{3} \\
    \frac{4u^{p-1}}{\Theta_0} & 1
  \end{bmatrix}
  \begin{bmatrix}
    4\alpha_1 & 0 \\
    0 & -\frac{\beta_1}{4}
  \end{bmatrix}.$
\end{itemize}
Note that we have $v_p(\Theta_0)=-2+v_p(x_0)-v_p(x_1)$ and $v_p(\Theta_1)=2+v_p(x_0)+v_p(x_1)$. We will divide $R(\vr;\vk')$ into following eight areas, and in each area we will determine the mod-$p$ reduction:
\begin{enumerate}[leftmargin=*]
  \item $(t_0,t_1)\in R_{int}(\vr;\vk')$: It is easy to see that $\cM\cong\cM_0((5,4),(4,1);(\alpha_0,4\alpha_1),(\beta_0,-\frac{\beta_1}{4}))$, and so we conclude that $\rhobar^{ss}|_\IQptwo\cong\omega_4^{p+p^2+4p^3}\oplus\omega_4^{1+4p+p^3}$ by Lemma~\ref{lemm: Tst^r for irred BrMod}~(ii).
  \item $t_1=-2$, $-1<t_0<0$: We have $\frac{\Theta_0}{x_0}\in\F^\times$ and $\frac{1}{\Theta_0}=\frac{1}{\Theta_1}=\frac{\Theta_1}{x_1}=\frac{px_0}{\Theta_0}=p\Theta_0=0$ in $\F$. In particular, the monodromy type is $(0,0)$. Put $\cM'=\barS_{\F}(\barE_2)\subset\cM$. It is easy to see that $\cM'\cong\cM(5,1;-\frac{\beta_0\Theta_0}{x_0},-\frac{\beta_1}{4})$ and $\cM/\cM'\cong\cM(4,4;\frac{\alpha_0x_0}{\Theta_0},4\alpha_1)$, and so we conclude that $\rhobar^{ss}|_\IQptwo\cong \omega_2^{4p}\oplus\omega_2^{1+p}$, by Lemma~\ref{lemm: Tst^r for irred BrMod}~(i). It is not difficult to show that the short exact sequence determined by $\cM'\hookrightarrow \cM$ is non-split.
  \item $t_0=0$, $t_1=-2$: We have $\frac{1}{\Theta_0},\frac{1}{\Theta_1},\frac{\Theta_0}{x_0}\in\F^\times$ and $\frac{\Theta_1}{x_1}=\frac{px_0}{\Theta_0}=p\Theta_0=0$ in $\F$. In particular, the monodromy type is $(1,1)$. Put $\cM'$ as in (ii). By the same isomorphisms for $\cM'$ and $\cM/\cM'$ we have $\rhobar^{ss}|_\IQptwo\cong \omega_2^{4p}\oplus\omega_2^{1+p}$, by Lemma~\ref{lemm: Tst^r for irred BrMod}~(i), and $\rhobar$ is non-split as the monodromy type is nonzero.
  \item $t_0-t_1=2$, $-1\leq t_0<0$: We have $\frac{1}{\Theta_0}\in\F^\times$ and $\frac{1}{\Theta_1}=\frac{\Theta_0}{x_0}=\frac{\Theta_1}{x_1}=p\Theta_0=0$ in $\F$. In particular, the monodromy type is $(1,0)$. Put $\cM'=\barS_{\F}(\barE_2^{(0)},\barE_1^{(1)})\subset\cM$. It is easy to see that $\cM'\cong\cM(4,5;\alpha_0,\frac{\alpha_1}{\Theta_0})$ and $\cM/\cM'\cong\cM(5,0;\beta_0,\beta_1\Theta_0)$, and so we conclude that $\rhobar^{ss}|_\IQptwo\cong\omega_2\oplus\omega_2^{5p}$, by Lemma~\ref{lemm: Tst^r for irred BrMod}~(i), and $\rhobar$ is non-split as the monodromy type is nonzero.
  \item $t_1=-3$, $-2<t_0<-1$: We have $\frac{px_0}{\Theta_0}\in\F^\times$ and $\frac{1}{\Theta_0}=\frac{1}{\Theta_1}=\frac{\Theta_0}{x_0}=\frac{\Theta_1}{x_1}=p\Theta_0=0$ in $\F$. In particular, the monodromy type is $(0,0)$. Put $\cM'=\barS_{\F}(\barE_1^{(0)},\barE_1^{(1)}-\frac{\beta_0}{\alpha_0}\frac{\Theta_0}{px_0}\barE_2^{(1)})$. It is easy to see that $\cM'\cong\cM(5,4;-\frac{p\alpha_0x_0}{\Theta_0},4\alpha_1)$ and $\cM/\cM'=\barS_{\F}(\barE_2+\cM')\cong\cM(4,1;\frac{\beta_0\Theta_0}{px_0},-\frac{\beta_1}{4})$, and so we conclude that $\rhobar^{ss}|_\IQptwo\cong\omega_2^p\oplus\omega_2^{1+4p}$, by Lemma~\ref{lemm: Tst^r for irred BrMod}~(i). It is not difficult to show that the short exact sequence determined by $\cM'\hookrightarrow \cM$ is non-split.
  \item $t_0=-2$, $t_1=-3$: We have $\frac{\Theta_1}{x_1},\frac{px_0}{\Theta_0},p\Theta_0\in\F^\times$ and $\frac{1}{\Theta_0}=\frac{1}{\Theta_1}=\frac{\Theta_0}{x_0}=0$ in $\F$. In particular, the monodromy type is $(0,0)$. We note that $\frac{\alpha_0}{\beta_0}=\frac{\Theta_0\Theta_1}{x_0^2}$ from \eqref{eq: alpha_j/beta_j}.
      \begin{enumerate}[leftmargin=*]
        \item Assume $p^2(px_1+x_0)=0$ in $\F$. Put $\cM'=\barS_{\F}\big(-\frac{4p\Theta_0}{3}\barE_1^{(0)}+\barE_2^{(0)}, \frac{\Theta_1}{x_1}\barE_1^{(1)}+\barE_2^{(1)}\big)$. It is easy to see that  $\cM'\cong\cM(5,1;-\frac{4p\beta_0\Theta_0}{3},-\frac{\beta_1}{4})$ and $\cM/\cM'=\barS_{\F}(\barE_1+\cM')\cong \cM(4,4;\frac{3\alpha_0}{4p\Theta_0},4\alpha_1)$, and so we conclude that $\rhobar^{ss}|_\IQptwo\cong\omega_2^{4p}\oplus\omega_2^{1+p}$, by Lemma~\ref{lemm: Tst^r for irred BrMod}~(i). It is not difficult to show that the short exact sequence determined by $\cM'\hookrightarrow \cM$ is non-split.
        \item Assume $p^2(px_1+x_0)\neq0$ in $\F$. Put
            $\cM'=\barS_{\F}(\barE_1^{(0)}, -\frac{p\Theta_1}{x_0}\barE_1^{(1)}+\barE_2^{(1)})$.
            It is easy to see that $\cM'\cong\cM(5,4;\beta_0,-\frac{4\alpha_1\Theta_1(px_1+x_0)}{x_0x_1})$ and $\cM/\cM'=\barS_{\F}(\barE_2^{(0)}+\cM', \frac{\Theta_1}{x_1}\barE_1^{(1)}+\barE_2^{(1)}+\cM') \cong\cM(4,1;\frac{\alpha_0x_0x_1}{\Theta_1(px_1+x_0)},-\frac{\beta_1}{4})$, and so we conclude that $\rhobar^{ss}|_\IQptwo\cong\omega_2^p\oplus\omega_2^{1+4p}$, by Lemma~\ref{lemm: Tst^r for irred BrMod}~(i). It is not difficult to show that the short exact sequence determined by $\cM'\hookrightarrow \cM$ is non-split.
      \end{enumerate}
  \item $t_0-t_1=1$, $-2<t_0<-1$: We have $p\Theta_0\in\F^\times$ and $\frac{1}{\Theta_0}=\frac{1}{\Theta_1}=\frac{\Theta_0}{x_0}=\frac{\Theta_1}{x_1}=\frac{px_0}{\Theta_0}=0$ in $\F$. In particular, the monodromy type is $(0,0)$. Put $\cM'=\barS_{\F}(-\frac{4p\Theta_0}{3}\barE_1^{(0)}+\barE_2^{(0)},\barE_2^{(1)})$. It is easy to see that $\cM'\cong\cM(5,1;-\frac{4p\beta_0\Theta_0}{3},-\frac{\beta_1}{4})$ and $\cM/\cM'=\barS_{\F}(\barE_1+\cM'\cong\cM(4,4;\frac{3\alpha_0}{4p\Theta_0},4\alpha_1)$, and so we conclude that $\rhobar^{ss}|_\IQptwo\cong\omega_2^{4p}\oplus\omega_2^{1+p}$, by Lemma~\ref{lemm: Tst^r for irred BrMod}~(i). It is not difficult to show that the short exact sequence determined by $\cM'\hookrightarrow \cM$ is non-split.
  \item $t_0=-1$, $t_1=-2$: We have $\frac{\Theta_0}{x_0},p\Theta_0\in\F^\times$ and $\frac{1}{\Theta_0}=\frac{1}{\Theta_1}=\frac{\Theta_1}{x_1}=\frac{px_0}{\Theta_0}=0$ in $\F$. In particular, the monodromy type is $(0,0)$.

      \begin{enumerate}[leftmargin=*]
        \item Assume $px_0=-\frac{3}{4}$ in $\F$. Choosing basis $\baseE'^{(0)}=(\barE_1^{(0)},\frac{\Theta_0}{x_0}\barE_1^{(0)}+\barE_2^{(0)})$, we see that $\cM\cong\cM_0((5,4),(4,1);(\alpha_0,4\alpha_1),(\beta_0,\frac{\beta_1}{4}))$, and so we conclude that $\rhobar|_\IQptwo\cong\omega_4^{p+p^2+4p^3}\oplus\omega_4^{1+4p+p^3}$, by Lemma~\ref{lemm: Tst^r for irred BrMod}~(ii).
        \item Assume $px_0\neq-\frac{3}{4}$ in $\F$. Put $\cM'=\barS_{\F}(-\frac{4p\Theta_0}{3}\barE_1^{(0)}+\barE_2^{(0)},\barE_2^{(1)})$. It is easy to see that $\cM'\cong\cM(5,1;-\frac{\beta_0\Theta_0(3+4x_0)}{3x_0},\frac{\beta_1}{4})$ and $\cM/\cM'=\barS_{\F}(\frac{\Theta_0}{x_0}\barE_1^{(0)}+\barE_2^{(0)}+\cM',\barE_1^{(1)}+\cM') \cong\cM(4,4;\frac{3\alpha_0x_0}{\Theta_0(3+4x_0)},4\alpha_1)$, and so we conclude that $\rhobar^{ss}|_\IQptwo\cong\omega_2^{4p}\oplus\omega_2^{1+p}$, by Lemma~\ref{lemm: Tst^r for irred BrMod}~(i). It is not difficult to show that the short exact sequence determined by $\cM'\hookrightarrow \cM$ is non-split.
      \end{enumerate}
\end{enumerate}
We summarize the mod-$p$ reduction for $\vk'=(1,\frac{9}{2})$ in the following table, Table~\ref{tab: (1,9/2)}. Note that by $\ast$ in the following table we mean a non-split extension.
\begin{table}[htbp]
  \centering
  \begin{tabular}{|c|c||c|c|}
    \hline
    \multicolumn{2}{|c||}{$(t_0,t_1)$} & $\rhobar|_{I_{\Q_{p^2}}}$ & $\vec{\mathrm{MT}}$ \\\hline\hline
    \multicolumn{2}{|c||}{$R_{int}(\vr;(1,\frac{9}{2}))$}
      & $\omega_4^{p+p^2+4p^3}\oplus\omega_4^{1+4p+p^3}$ & $(0,0)$ \\\hline
    \multicolumn{2}{|c||}{$t_1=-2$, $-1<t_0<0$}
      & $
      \begin{bmatrix}
        \omega_2^{1+p} & * \\
        0 & \omega_2^{4p}
      \end{bmatrix}$ & $(0,0)$ \\\hline
    \multicolumn{2}{|c||}{$t_0=0$, $t_1=-2$}
      & $
      \begin{bmatrix}
        \omega_2^{1+p} & * \\
        0 & \omega_2^{4p}
      \end{bmatrix}$ & $(1,1)$ \\\hline
    \multicolumn{2}{|c||}{$t_0-t_1=2$, $-1\leq t_0<0$}
      & $
      \begin{bmatrix}
        \omega_2^{5p} & * \\
        0 & \omega_2
      \end{bmatrix}$ & $(1,0)$ \\\hline
    \multicolumn{2}{|c||}{$t_1=-3$, $-2<t_0<-1$}
        & $
        \begin{bmatrix}
          \omega_2^{1+4p} & * \\
          0 & \omega_2^p
        \end{bmatrix}$ & $(0,0)$ \\\hline
    \multirow{2}{*}[-0.5em]{
    \begin{tabular}{c}
      $t_0=-2$, \\ $t_1=-3$
    \end{tabular}}
      & $v_p(px_1+x_0)>-2$
        & $
        \begin{bmatrix}
          \omega_2^{1+p} & * \\
          0 & \omega_2^{4p}
        \end{bmatrix}$ & $(0,0)$ \\\cline{2-4}
      & $v_p(px_1+x_0)=-2$
        & $
        \begin{bmatrix}
          \omega_2^{1+4p} & * \\
          0 & \omega_2^p
        \end{bmatrix}$ & $(0,0)$ \\\hline
    \multicolumn{2}{|c||}{$t_0-t_1=1$, $-2<t_0<-1$}
        & $
        \begin{bmatrix}
          \omega_2^{1+p} & * \\
          0 & \omega_2^{4p}
        \end{bmatrix}$ & $(0,0)$ \\\hline
    \multirow{2}{*}[-0.5em]{
    \begin{tabular}{c}
      $t_0=-1$, \\ $t_1=-2$
    \end{tabular}}
      & $v_p(px_0+\frac{3}{4})>0$
        & $\omega_4^{p+p^2+4p^3}\oplus\omega_4^{1+4p+p^3}$ & $(0,0)$ \\\cline{2-4}
      & $v_p(px_0+\frac{3}{4})=0$
        & $
        \begin{bmatrix}
          \omega_2^{1+p} & * \\
          0 & \omega_2^{4p}
        \end{bmatrix}$ & $(0,0)$ \\\hline
  \end{tabular}
  \caption{$\rhobar$ for $\vk'=(1,\frac{9}{2})$}\label{tab: (1,9/2)}
\end{table}

\subsubsection{\textbf{For $\vk'=(\frac{3}{2},\frac{3}{2})$}}\label{subsubsec: vr=(1,5), vk=(3/2,3/2)}
In this section of paragraph, by $\vk'$ we always mean $(\frac{3}{2},\frac{3}{2})$. In this case, we have $\vx=\vx^{(0,2)}$, and we will write $\cM$ for the Breuil modules. From the results in \S\ref{subsec: mod p reduction}, the Breuil modules $\cM$ is determined by the following data:
\begin{itemize}[leftmargin=*]
  \item $M_{N,0}=
  \begin{bmatrix}
    0 & 0 \\
    \frac{1-10u^{p-1}}{\Theta_0} & 0
  \end{bmatrix}\quad\&\quad
  M_{N,1}=
  \begin{bmatrix}
    0 & 0 \\
    \frac{1}{\Theta_1} & 0
  \end{bmatrix}$;
  \item $M_{\Fil^5,0}=
  \begin{bmatrix}
    u^5 & 0 \\
    0 & u^4
  \end{bmatrix}\quad\&\quad
  M_{\Fil^5,1}=
  \begin{bmatrix}
    1 & 0 \\
    \frac{x_1}{\Theta_1} & 1
  \end{bmatrix}
  \begin{bmatrix}
    u & 0 \\
    -\frac{\alpha_1u^3}{18\beta_1\Theta_0} & u^4
  \end{bmatrix}$;
  \item $M_{\phi_5,0}=
  \begin{bmatrix}
    1 & \frac{\alpha_0}{\beta_0}\frac{\Theta_0}{px_0}\\
    0 & 1
  \end{bmatrix}
  \begin{bmatrix}
    -\alpha_0 & 0\\
    0 & \beta_0
  \end{bmatrix}\quad\&\quad
  M_{\phi_5,1}=
  \begin{bmatrix}
    1 & \frac{4p\Theta_0}{3} \\
    \frac{10u^{p-1}}{\Theta_0} & 1
  \end{bmatrix}
  \begin{bmatrix}
    \frac{\alpha_1}{3} & 0\\
    0 & -3\beta_1
  \end{bmatrix}$.
\end{itemize}
Note that one can choose any $\vTh$ satisfying
$$\max\{-1, v_p(x_0)+1\}\leq v_p(\Theta_0)=v_p(\Theta_1)+2\leq \min\{0,v_p(x_1)+2\},$$
and in particular $\frac{1}{\Theta_1}=0$ in $\F$. We will divide $R(\vr;\vk')$ into following five areas, and in each area we will determine the mod-$p$ reduction:
\begin{enumerate}[leftmargin=*]
  \item $\vec{t}\in R_{int}(\vr;\vk')$: Note that we have either $v_p(\Theta_0)\not\in\{0,v_p(x_1)+2\}$ or $v_p(\Theta_0)\not\in\{-1,v_p(x_0)+1\}$. Consider two $\barS_{\F}$-submodules $\cM'=\barS_{\F}(\barE_1)$ and $\cM''=\barS_{\F}(\barE_2)$ of $\cM$. Then one can observe that
    \begin{itemize}[leftmargin=*]
      \item if $v_p(\Theta_0)\not\in\{ 0,v_p(x_1)+2\}$, then $\cM'$ is a Breuil submodule of $\cM$ satisfying $\cM'\cong\cM(5,1;-\alpha_0,\frac{\alpha_1}{3})$ and $\cM/\cM'\cong\cM(4,4;\beta_0,-3\beta_1)$;
      \item if $v_p(\Theta_0)\not\in\{ -1,v_p(x_0)+1\}$, then $\cM''$ is a Breuil submodule of $\cM$ satisfying $\cM''\cong\cM(4,4;\beta_0,-3\beta_1)$ and $\cM/\cM''\cong\cM(5,1;-\alpha_0,\frac{\alpha_1}{3})$.
    \end{itemize}
    In particular, if $v_p(\Theta_0)\not\in\{-1,0,v_p(x_0)+1,v_p(x_1)+2\}$, then $\cM=\cM'\oplus\cM''$ as Breuil modules, and if $v_p(\Theta)=v_p(x_0)+1\not\in\{-1,0,v_p(x_1)+2\}$, then the short exact sequence determined by $\cM'\hookrightarrow\cM$ is split. Otherwise, It is not difficult to show that the short exact sequence determined by $\cM'\hookrightarrow \cM$ or $\cM''\hookrightarrow\cM$ is non-split. Note that the first sequence is non-split only for 
    $v_p(\Theta_0)=-1\not\in\{0,v_p(x_1)+2\}$, and so it can be happen if $t_0\leq -2$, in order to have $v_p(x_0)+1\leq v_p(\Theta_0)$. Hence, we conclude that $\rhobar^{ss}|_\IQptwo\cong\omega_2^{4p}\oplus\omega_2^{1+p}$, by Lemma~\ref{lemm: Tst^r for irred BrMod}~(i). In addition, the monodromy type is given as follow.
    \begin{enumerate}[leftmargin=*]
      \item If $v_p(\Theta_0)<0$, then $t_1<-2$, and the monodromy type is $(0,0)$.
      \item If $v_p(\Theta_0)=0$, then $t_1\geq -2$, and the monodromy type is $(1,0)$.
    \end{enumerate}
  \item $t_0=-1$, $t_1>-2$: We have $v_p(\Theta_0)=0=v_p(x_0)+1\neq v_p(x_1)+2$, so $\frac{1}{\Theta_0},\frac{\Theta_0}{px_0}\in\F^\times$ and $\frac{x_1}{\Theta_1}=p\Theta_0=0$ in $\F$. In particular, the monodromy type is $(1,0)$. Put $\cM'=\barS_{\F}(\barE_2^{(0)},\frac{\alpha_0}{\beta_0}\frac{\Theta_0}{px_0}\barE_1^{(1)}+\barE_2^{(1)})$. It is easy to see that $\cM'\cong\cM(4,4;\beta_0,-3\beta_1)$ and $\cM/\cM'=\barS_{\F}(\barE_1+\cM')\cong\cM(5,1;-\alpha_0,\frac{\alpha_1}{3})$, and so we conclude that $\rhobar^{ss}|_\IQptwo\cong\omega_2^{1+p}\oplus\omega_2^{4p}$, by Lemma~\ref{lemm: Tst^r for irred BrMod}~(i), and $\rhobar$ is non-split as the monodromy type is nonzero.
  \item $t_0=-1$, $t_1=-2$: We have $v_p(\Theta_0)=0=v_p(x_0)+1=v_p(x_1)+2$, so $\frac{1}{\Theta_0},\frac{x_1}{\Theta_1},\frac{\Theta_0}{px_0}\in\F^\times$ and $p\Theta_0=0$ in $\F$. In particular, the monodromy type is $(1,0)$. We note that $\frac{\beta_0}{\alpha_0}=\frac{\Theta_0}{p^2\Theta_1}$ from \eqref{eq: alpha_j/beta_j}.
      \begin{enumerate}[leftmargin=*]
        \item Assume $\frac{px_1}{x_0}=1$ in $\F$, and put $\cM'=\barS_{\F}(\barE_2^{(0)},\barE_1^{(1)}+\frac{x_1}{\Theta_1}\barE_2^{(1)})$. It is easy to see that $\cM'\cong\cM(4,2;-\frac{\beta_0\Theta_1}{px_0}, -\frac{\alpha_1}{6\Theta_0})$ and $\cM/\cM'=\barS_{\F}(\barE_1^{(0)}+\cM',\barE_2^{(1)}+\cM')\cong \cM(5,3;-\frac{\alpha_0x_1}{\Theta_1},-6\beta_1\Theta_0)$, and so we conclude that $\rhobar^{ss}|_\IQptwo\cong\omega_2^{1+3p}\oplus\omega_2^{2p}$, by Lemma~\ref{lemm: Tst^r for irred BrMod}~(i), and $\rhobar$ is non-split as the monodromy type is nonzero.
        \item Assume $\frac{px_1}{x_0}\neq 1$ in $\F$, and put $\cM'=\barS_{\F}(\barE_2^{(0)},\frac{\Theta_1}{px_0}\barE_1^{(1)}+\barE_2^{(1)})$. It is easy to see that $\cM'\cong\cM(4,4;\beta_0,\frac{-3\beta_1(px_0-x_1)}{px_0})$ and $\cM/\cM'=\barS_{\F}(\barE_1+\cM')\cong\cM(5,1;-\alpha_0,\frac{p\alpha_1x_0}{3(px_0-x_1)})$, and so we conclude that $\rhobar^{ss}|_\IQptwo\cong\omega_2^{1+p}\oplus\omega_2^{4p}$, by Lemma~\ref{lemm: Tst^r for irred BrMod}~(i), and $\rhobar$ is non-split as the monodromy type is nonzero.
      \end{enumerate}
  \item $t_0-t_1=1$, $-2<t_0<-1$: We have $v_p(\Theta_0)=v_p(x_0)+1=v_p(x_1)+2\neq -1,0$, and so $\frac{x_1}{\Theta_1},\frac{\Theta_0}{px_0}\in\F^\times$ and $\frac{1}{\Theta_0}=p\Theta_0=0$ in $\F$. In particular, the monodromy type is $(0,0)$. We note that $\frac{\beta_0}{\alpha_0}=\frac{\Theta_0}{p^2\Theta_1}$ from \ref{eq: alpha_j/beta_j}.
      \begin{enumerate}[leftmargin=*]
        \item Assume $\frac{px_1}{x_0}=1$ in $\F$. Choosing basis $\baseE'^{(1)}=(\barE_1^{(1)}, \frac{\Theta_1}{x_1}\barE_1^{(1)}+\barE_2^{(1)})$ and $\baseF'^{(1)}=(u^3\barF_1^{(1)}-\frac{x_1}{\Theta_1}\barF_2^{(1)}, \frac{\Theta_1}{x_1}\barF_1^{(1)})$, it is easy to see that $\cM$ is isomorphic to $\cM_1((5,4),(4,1);(-\alpha_0,\frac{\alpha_1\Theta_1}{3x_1}),(\beta_0,\frac{3\beta_1x_1}{\Theta_1}))$, and so we conclude that $\rhobar|_\IQptwo\cong\omega_4^{4p+p^2+p^3}\oplus\omega_4^{1+p+4p^3}$, by Lemma~\ref{lemm: Tst^r for irred BrMod}~(ii).
        \item Assume $\frac{px_1}{x_0}\neq 1$ in $\F$. Put $\cM'$ as in (iii)(b). By the same isomorphisms for $\cM'$ and $\cM/\cM'$ we have $\rhobar^{ss}|_\IQptwo\cong\omega_2^{1+p}\oplus\omega_2^{4p}$, by Lemma~\ref{lemm: Tst^r for irred BrMod}~(i). It is not difficult to show that the short exact sequence determined by $\cM'\hookrightarrow \cM$ is non-split.
      \end{enumerate}
  \item $t_1=-3$, $t_0\leq -2$: We have $v_p(\Theta_0)=-1=v_p(x_1)+2$, so $\frac{x_1}{\Theta_1},p\Theta_0\in\F^\times$ and $\frac{1}{\Theta_0}=0$ in $\F$. In particular, the monodromy type is $(0,0)$. Put $\cM'=\barS_{\F}(\frac{4p\Theta_0}{3}\barE_1^{(0)}+\barE_2^{(0)},\barE_1^{(1)})$. It is easy to see that $\cM'\cong\cM(5,4;-\frac{3\beta_1x_1}{\Theta_1},-\frac{4p\alpha_0\Theta_0}{3})$ and $\cM/\cM'=\barS_{\F}(\barE_2+\cM')\cong\cM(4,1;\beta_0,-\frac{\alpha_1\Theta_1}{4p\Theta_0x_1})$, and so we conclude that $\rhobar^{ss}\cong\omega_2^p\oplus\omega_2^{1+4p}$, by Lemma~\ref{lemm: Tst^r for irred BrMod} (i). It is not difficult to show that the short exact sequence determined by $\cM'\hookrightarrow \cM$ is non-split.
\end{enumerate}
We summarize the mod-$p$ reduction for $\vk'=(\frac{3}{2},\frac{3}{2})$ in the following table, Table~\ref{tab: (3/2,3/2)}. Note that by $\ast$ in the following table we mean a non-split extension.
\begin{small}
\begin{table}[htbp]
  \centering
  \begin{tabular}{|c|c||c|c|}
    \hline
    \multicolumn{2}{|c||}{$(t_0,t_1)$} & $\rhobar|_{I_{\Q_{p^2}}}$ & $\vec{\mathrm{MT}}$ \\\hline\hline
    \multirow{5}{*}{$R_{int}(\vr;(\frac{3}{2},\frac{3}{2}))$}
        & $v_p(\Theta_0)\not\in\{-1,0,v_p(x_1)+2\}$
        & $
        \begin{bmatrix}
            \omega_2^{1+p} & 0 \\ 0 & \omega_2^{4p}
        \end{bmatrix}$
        & $(0,0)$ \\\cline{2-4}
        & $t_0\leq -2$, $v_p(\Theta_0)=-1$
        & $
        \begin{bmatrix}
            \omega_2^{1+p} & * \\ 0 & \omega_2^{4p}
        \end{bmatrix}$
        & $(0,0)$ \\\cline{2-4}
        & $v_p(\Theta_0)\in\{0,v_p(x_1)+2\}$
        & $
        \begin{bmatrix}
            \omega_2^{4p} & * \\ 0 & \omega_2^{1+p}
        \end{bmatrix}$
        & $
        \begin{cases}
            (0,0) &\mbox{if } t_1<-2;\\
            (1,0) &\mbox{if } t_1\geq -2
        \end{cases}$ \\\hline
    \multicolumn{2}{|c||}{$t_0=-1$, $t_1>-2$}
        & $
        \begin{bmatrix}
          \omega_2^{4p} & * \\
          0 & \omega_2^{1+p}
        \end{bmatrix}$ & $(1,0)$ \\\hline
    \multirow{2}{*}[-0.5em]{
    \begin{tabular}{c}
      $t_0=-1$, \\ $t_1=-2$
    \end{tabular}}
      & $v_p(1-\frac{px_1}{x_0})>0$
        & $
        \begin{bmatrix}
          \omega_2^{2p} & * \\
          0 & \omega_2^{1+3p}
        \end{bmatrix}$ & $(1,0)$ \\\cline{2-4}
      & $v_p(1-\frac{px_1}{x_0})=0$
        & $
        \begin{bmatrix}
          \omega_2^{4p} & * \\
          0 & \omega_2^{1+p}
        \end{bmatrix}$ & $(1,0)$ \\\hline
    \multirow{2}{*}[-0.5em]{
    \begin{tabular}{c}
      $t_0-t_1=1$, \\ $-2<t_0<-1$
    \end{tabular}}
      & $v_p(1-\frac{px_1}{x_0})>0$
        & $\omega_4^{4p+p^2+p^3}\oplus\omega_4^{1+p+4p^3}$ & $(0,0)$ \\\cline{2-4}
      & $v_p(1-\frac{px_1}{x_0})=0$
        & $
        \begin{bmatrix}
          \omega_2^{4p} & * \\
          0 & \omega_2^{1+p}
        \end{bmatrix}$ & $(0,0)$ \\\hline
    \multicolumn{2}{|c||}{$t_1=-3$, $t_0\leq -2$}
        & $
        \begin{bmatrix}
          \omega_2^{1+4p} & * \\
          0 & \omega_2^p
        \end{bmatrix}$ & $(0,0)$ \\\hline
  \end{tabular}
  \caption{$\rhobar$ for $\vk'=(\frac{3}{2},\frac{3}{2})$}\label{tab: (3/2,3/2)}
\end{table}
\end{small}

\subsubsection{\textbf{For $\vk'=(\frac{3}{2},5)$}}\label{subsubsec: vr=(1,5), vk=(3/2,5)}
In this section of paragraph, by $\vk'$ we always mean $(\frac{3}{2},5)$. In this case, we have $\vx=\vx^{(0,1)}$, and we will write $\cM$ for the Breuil modules. From the results in \S\ref{subsec: mod p reduction}, the Breuil modules $\cM$ is determined by the following data:
\begin{itemize}[leftmargin=*]
  \item $M_{N,0}=
  \begin{bmatrix}
    0 & 0 \\
    \frac{1-4u^{p-1}}{\Theta_0} & 0
  \end{bmatrix}\quad\&\quad
  M_{N,1}=
  \begin{bmatrix}
    0 & 0 \\
    \frac{1}{\Theta_1} & 0
  \end{bmatrix}$;
  \item $M_{\Fil^5,0}=
  \begin{bmatrix}
    u^5 & 0 \\
    0 & u^4
  \end{bmatrix}\quad\&\quad
  M_{\Fil^5,1}=
  \begin{bmatrix}
    1 & \frac{\Theta_1}{x_1} \\
    0 & 1
  \end{bmatrix}
  \begin{bmatrix}
    u^5 & \frac{\alpha_1u^4}{4\beta_1\Theta_0} \\
    0 & 1
  \end{bmatrix}$;
  \item $M_{\phi_5,0}=
  \begin{bmatrix}
    1 & \frac{\alpha_0}{\beta_0}\frac{\Theta_0}{px_0}\\
    0 & 1
  \end{bmatrix}
  \begin{bmatrix}
    -\alpha_0 & 0\\
    0 & \beta_0
  \end{bmatrix}\quad\&\quad
  M_{\phi_5,1}=
  \begin{bmatrix}
    1 & -\frac{\alpha_1}{\beta_1}\frac{p^5x_1}{\Theta_1} \\
    \frac{4u^{p-1}}{\Theta_0} & 1
  \end{bmatrix}
  \begin{bmatrix}
    0 & \alpha_1 \\
    \beta_1 & 0
  \end{bmatrix}$.
\end{itemize}
Note that we have $v_p(\Theta_0)=3+v_p(x_1)$ and $v_p(\Theta_1)=1+v_p(x_1)$, and so $\frac{1}{\Theta_1}=\frac{\Theta_1}{x_1}=\frac{p^5x_1}{\Theta_1}=0$ in $\F$. We will divide $R(\vr;\vk')$ into following four areas, and in each area we will determine the mod-$p$ reduction:
\begin{enumerate}[leftmargin=*]
  \item $\vec{t}\in R_{int}(\vr;\vk)$: It is easy to see that $\cM\cong\cM_1((5,5),(4,0);(-\alpha_0,\alpha_1),(\beta_0,\beta_1))$, and so we conclude that $\rhobar|_\IQptwo\cong\omega_4^{5p+p^2}\oplus\omega_4^{1+5p^3}$ by Lemma~\ref{lemm: Tst^r for irred BrMod}~(ii).
  \item $t_1=-3$, $t_0<-1$: We have $\frac{1}{\Theta_0}\in\F^\times$ and $\frac{\Theta_0}{px_0}=0$ in $\F$. In particular, the monodromy type is $(1,0)$. Put $\cM'=\barS_{\F}(\barE_2)\subset\cM$. It is easy to see that $\cM'\cong \cM(4,1;\beta_0,-\frac{\alpha_1}{4\Theta_0})$ and $\cM/\cM'\cong \cM(5,4;-\alpha_0,4\beta_1\Theta_0)$, and so we conclude that $\rhobar^{ss}|_\IQptwo\cong \omega_2^{1+4p}\oplus\omega_2^p$, by Lemma~\ref{lemm: Tst^r for irred BrMod}~(i), and $\rhobar$ is non-split as the monodromy type is nonzero.
  \item $t_0=-1$, $t_1=-3$: We have $\frac{1}{\Theta_0},\frac{\Theta_0}{px_0}\in \F^\times$. In particular, the monodromy type is $(1,0)$. Put $\cM'=\barS_{\F}(\barE_2^{(0)}, \frac{\alpha_0}{\beta_0}\frac{\Theta_0}{px_0}\barE_1^{(1)}+\barE_2^{(1)})$. It is easy to see that $\cM'\cong\cM(4,5;\beta_0,\frac{\alpha_0\beta_1}{\beta_0}\frac{\Theta_0}{px_0})$ and $\cM/\cM'=\barS_{\F}(\barE_1^{(0)}+\cM',\barE_2^{(1)}+\cM')\cong \cM(5,0;-\frac{px_0}{\beta_0\Theta_0},\alpha_1)$, and so we conclude that $\rhobar^{ss}|_\IQptwo\cong\omega_2\oplus\omega_2^{5p}$, by Lemma~\ref{lemm: Tst^r for irred BrMod}~(i), and $\rhobar$ is non-split as the monodromy type is nonzero.
  \item $t_0-t_1=2$, $t_0\leq -1$: We have $\frac{\Theta_0}{px_0}\in\F^\times$ and $\frac{1}{\Theta_0}=0$ in $\F$. In particular, the monodromy type is $(0,0)$. Put $\cM'$ as in (iii). By the same isomorphisms for $\cM'$ and $\cM/\cM'$ we have $\rhobar^{ss}|_\IQptwo\cong\omega_2\oplus\omega_2^{5p}$, by Lemma~\ref{lemm: Tst^r for irred BrMod}~(i). It is not difficult to show that the short exact sequence determined by $\cM'\hookrightarrow \cM$ is non-split.
\end{enumerate}
We summarize the mod-$p$ reduction for $\vk'=(\frac{3}{2},5)$ in the following table, Table~\ref{tab: (3/2,5)}. Note that by $\ast$ in the following table we mean a non-split extension.
\begin{table}[htbp]
  \centering
  \begin{tabular}{|c||c|c|}
    \hline
    $(t_0,t_1)$ & $\rhobar|_{I_{\Q_{p^2}}}$ & $\vec{\mathrm{MT}}$ \\\hline\hline
    $R_{int}(\vr;(\frac{3}{2},5))$
      & $\omega_4^{5p+p^2}\oplus\omega_4^{1+5p^3}$ & $(0,0)$ \\\hline
    \begin{tabular}{c}
      $t_1=-3$, \\ $t_0<-1$
    \end{tabular}
      & $
      \begin{bmatrix}
        \omega_2^p & * \\
        0 & \omega_2^{1+4p}
      \end{bmatrix}$ & $(1,0)$ \\\hline
    \begin{tabular}{c}
      $t_0=-1$, \\ $t_1=-3$
    \end{tabular}
      & $
      \begin{bmatrix}
        \omega_2^{5p} & * \\
        0 & \omega_2
      \end{bmatrix}$ & $(1,0)$ \\\hline
    \begin{tabular}{c}
      $t_0-t_1=2$, \\ $t_0<-1$
    \end{tabular}
      & $
      \begin{bmatrix}
        \omega_2^{5p} & * \\
        0 & \omega_2
      \end{bmatrix}$ & $(0,0)$ \\\hline
  \end{tabular}
  \caption{$\rhobar$ for $\vk'=(\frac{3}{2},5)$}\label{tab: (3/2,5)}
\end{table}

\subsubsection{\textbf{For $\vk'=(\infty,\frac{3}{2})$}}\label{subsubsec: vr=(1,5), vk=(infty,3/2)}
In this section of paragraph, by $\vk'$ we always mean $(\infty,\frac{3}{2})$. In this case, we have $\vx=\vx^{(0,2)}$, and we will write $\cM$ for the Breuil modules. From the results in \S\ref{subsec: mod p reduction}, the Breuil modules $\cM$ is determined by the following data:
\begin{itemize}[leftmargin=*]
  \item $M_{N,0}=
  \begin{bmatrix}
    0 & 0 \\
    \frac{1-10u^{p-1}}{\Theta_0} & 0
  \end{bmatrix}\quad\&\quad
  M_{N,1}=
  \begin{bmatrix}
    0 & 0 \\
    \frac{1}{\Theta_1} & 0
  \end{bmatrix}$;
  \item $M_{\Fil^5,0}=
  \begin{bmatrix}
    u^5 & 0 \\
    0 & u^4
  \end{bmatrix}\quad\&\quad
  M_{\Fil^5,1}=
  \begin{bmatrix}
    1 & 0 \\
    \frac{x_1}{\Theta_1} & 1
  \end{bmatrix}
  \begin{bmatrix}
    u & 0 \\
    -\frac{\alpha_1u^3}{18\beta_1\Theta_0} & u^4
  \end{bmatrix}$;
  \item $M_{\phi_5,0}=
  \begin{bmatrix}
    -\alpha_0 & 0\\
    0 & \beta_0
  \end{bmatrix}\quad\&\quad
  M_{\phi_5,1}=
  \begin{bmatrix}
    1 & \frac{4p\Theta_0}{3} \\
    \frac{10u^{p-1}}{\Theta_0} & 1
  \end{bmatrix}
  \begin{bmatrix}
    \frac{\alpha_1}{3} & 0\\
    0 & -3\beta_1
  \end{bmatrix}$.
\end{itemize}
Note that one can choose any $\vTh$ satisfying
$$-1\leq v_p(\Theta_0)=v_p(\Theta_1)+2\leq \min\{0,v_p(x_1)+2\},$$
and in particular $\frac{1}{\Theta_1}=0$ in $\F$. We will divide $R(\vr;\vk')$ into following two areas, and in each area we will determine the mod-$p$ reduction:
\begin{enumerate}[leftmargin=*]
  \item $\vec{t}\in R_{int}(\vr;\vk')$: Note that we have either $v_p(\Theta_0)\not\in\{0,v_p(x_1)+2\}$ or $v_p(\Theta_0)\neq -1$. Consider two $\barS_{\F}$-submodules $\cM'=\barS_{\F}(\barE_1)$ and $\cM''=\barS_{\F}(\barE_2)$ of $\cM$. Then one can observe that
    \begin{itemize}[leftmargin=*]
      \item if $v_p(\Theta_0)\not\in\{0,v_p(x_1)+2\}$, then $\cM'$ is a Breuil submodule of $\cM$ satisfying $\cM'\cong\cM(5,1;-\alpha_0,\frac{\alpha_1}{3})$ and $\cM/\cM'\cong\cM(4,4;\beta_0,-3\beta_1)$;
      \item if $v_p(\Theta_0)\neq -1$, then $\cM''$ is a Breuil submodule of $\cM$ and it satisfies $\cM''\cong\cM(4,4;\beta_0,-3\beta_1)$ and $\cM/\cM''\cong\cM(5,1;-\alpha_0,\frac{\alpha_1}{3})$.
    \end{itemize}
    In particular, if $v_p(\Theta_0)\not\in\{-1,0,v_p(x_1)+2\}$, then $\cM=\cM'\oplus\cM''$ as Breuil modules. Otherwise, It is not difficult to show that the short exact sequence determined by $\cM'\hookrightarrow \cM$ or $\cM''\hookrightarrow\cM$ is non-split. Hence, we conclude that $\rhobar|_\IQptwo\cong\omega_2^{4p}\oplus\omega_2^{1+p}$, by Lemma~\ref{lemm: Tst^r for irred BrMod}~(i). In addition, the monodromy type is given as follow.
    \begin{enumerate}[leftmargin=*]
      \item If $v_p(\Theta_0)<0$, then $t_1<-2$, and the monodromy type is $(0,0)$.
      \item If $v_p(\Theta_0)=0$, then $t_1\geq -2$, and the monodromy type is $(1,0)$.
    \end{enumerate}
  \item $t_1=-3$: We have $v_p(\Theta_0)=-1=v_p(x_1)+2$, so $\frac{x_1}{\Theta_1},p\Theta_0\in\F^\times$ and $\frac{1}{\Theta_0}=0$ in $\F$. In particular, the monodromy type is $(0,0)$. Put $\cM'=\barS_{\F}(\frac{4p\Theta_0}{3}\barE_1^{(0)}+\barE_2^{(0)},\barE_1^{(1)})$. It is easy to see that $\cM'\cong\cM(5,4;-\frac{3\beta_1x_1}{\Theta_1},-\frac{4p\alpha_0\Theta_0}{3})$ and $\cM/\cM'=(\barS_{\F}(\barE_2)+\cM')/\cM'\cong\cM(4,1;\beta_0,-\frac{\alpha_1\Theta_1}{4p\Theta_0x_1})$, and so we conclude that $\rhobar^{ss}\cong\omega_2^p\oplus\omega_2^{1+4p}$, by Lemma~\ref{lemm: Tst^r for irred BrMod}~(i). It is not difficult to show that the short exact sequence determined by $\cM'\hookrightarrow \cM$ is non-split.
\end{enumerate}
We summarize the mod-$p$ reduction for $\vk'=(\infty,\frac{3}{2})$ in the following table, Table~\ref{tab: (infty,3/2)}. Note that by $\ast$ in the following table we mean a non-split extension.
\begin{table}[htbp]
  \centering
  \begin{tabular}{|c|c||c|c|}
    \hline
    \multicolumn{2}{|c||}{$(t_0,t_1)$} & $\rhobar|_{I_{\Q_{p^2}}}$ & $\vec{\mathrm{MT}}$ \\\hline\hline
    \multirow{5}{*}{$R_{int}(\vr;(\infty,\frac{3}{2}))$}
        & $v_p(\Theta_0)\not\in\{-1,0,v_p(x_1)+2\}$
        & $
        \begin{bmatrix}
          \omega_2^{1+p} & 0 \\
          0 & \omega_2^{4p}
        \end{bmatrix}$
        & $(0,0)$ \\\cline{2-4}
        & $v_p(\Theta_0)=-1$
        & $
        \begin{bmatrix}
          \omega_2^{1+p} & * \\
          0 & \omega_2^{4p}
        \end{bmatrix}$ & $(0,0)$ \\\cline{2-4}
        & $v_p(\Theta_0)\in\{0,v_p(x_1)+2\}$
        & $
        \begin{bmatrix}
          \omega_2^{4p} & * \\
          0 & \omega_2^{1+p}
        \end{bmatrix}$
        & $
        \begin{cases}
          (0,0) & \mbox{if } t_1<-2; \\
          (1,0) & \mbox{if } t_1\geq -2
        \end{cases}$ \\\hline
    \multicolumn{2}{|c||}{$t_1=-3$}
        & $
        \begin{bmatrix}
          \omega_2^{1+4p} & * \\
          0 & \omega_2^p
        \end{bmatrix}$ & $(0,0)$ \\\hline
  \end{tabular}
  \caption{$\rhobar$ for $\vk'=(\infty,\frac{3}{2})$}\label{tab: (infty,3/2)}
\end{table}

\subsubsection{\textbf{For $\vk'=(\infty,5)$}}\label{subsubsec: vr=(1,5), vk=(infty,5)}
In this section of paragraph, by $\vk'$ we always mean $(\infty,5)$. In this case, we have $\vx=\vx^{(0,1)}$, and we will write $\cM$ for the Breuil modules. From the results in \S\ref{subsec: mod p reduction}, the Breuil modules $\cM$ is determined by the following data:
\begin{itemize}[leftmargin=*]
  \item $M_{N,0}=
  \begin{bmatrix}
    0 & 0 \\
    \frac{1-4u^{p-1}}{\Theta_0} & 0
  \end{bmatrix}\quad\&\quad
  M_{N,1}=
  \begin{bmatrix}
    0 & 0 \\
    \frac{1}{\Theta_1} & 0
  \end{bmatrix}$;
  \item $M_{\Fil^5,0}=
  \begin{bmatrix}
    u^5 & 0 \\
    0 & u^4
  \end{bmatrix}\quad\&\quad
  M_{\Fil^5,1}=
  \begin{bmatrix}
    1 & \frac{\Theta_1}{x_1} \\
    0 & 1
  \end{bmatrix}
  \begin{bmatrix}
    u^5 & \frac{\alpha_1u^4}{4\beta_1\Theta_0} \\
    0 & 1
  \end{bmatrix}$;
  \item $M_{\phi_5,0}=
  \begin{bmatrix}
    -\alpha_0 & 0\\
    0 & \beta_0
  \end{bmatrix}\quad\&\quad
  M_{\phi_5,1}=
  \begin{bmatrix}
    1 & -\frac{\alpha_1}{\beta_1}\frac{p^5x_1}{\Theta_1} \\
    \frac{4u^{p-1}}{\Theta_0} & 1
  \end{bmatrix}
  \begin{bmatrix}
    0 & \alpha_1 \\
    \beta_1 & 0
  \end{bmatrix}$.
\end{itemize}
Note that we have $v_p(\Theta_0)=3+v_p(x_1)$ and $v_p(\Theta_1)=1+v_p(x_1)$, and so $\frac{1}{\Theta_1}=\frac{\Theta_1}{x_1}=\frac{p^5x_1}{\Theta_1}=0$ in $\F$. We will divide $R(\vr;\vk')$ into following two areas, and in each area we will determine the mod-$p$ reduction:
\begin{enumerate}[leftmargin=*]
  \item $t_1\in R_{int}(\vr;\vk')$: It is easy to see that $\cM\cong\cM_1((5,5),(4,0);(-\alpha_0,\alpha_1),(\beta_0,\beta_1))$, and so we conclude that $\rhobar|_\IQptwo\cong\omega_4^{5p+p^2}\oplus\omega_4^{1+5p^3}$ by Lemma~\ref{lemm: Tst^r for irred BrMod}~(ii).
  \item $t_1=-3$: We have $\frac{1}{\Theta_0}\in\F^\times$. In particular, the monodromy type is $(1,0)$. Put $\cM'=\barS_{\F}(\barE_2)\subset\cM$. It is easy to see that $\cM'\cong \cM(4,1;\beta_0,-\frac{\alpha_1}{4\Theta_0})$ and $\cM/\cM'\cong \cM(5,4;-\alpha_0,4\beta_1\Theta_0)$, and so we conclude that $\rhobar^{ss}|_\IQptwo\cong \omega_2^{1+4p}\oplus\omega_2^p$, by Lemma~\ref{lemm: Tst^r for irred BrMod}~(i), and $\rhobar$ is non-split as the monodromy type is nonzero.
\end{enumerate}
We summarize the mod-$p$ reduction for $\vk'=(\infty,5)$ in the following table, Table~\ref{tab: (infty,5)}. Note that by $\ast$ in the following table we mean a non-split extension.
\begin{table}[htbp]
  \centering
  \begin{tabular}{|c||c|c|}
    \hline
    $t_1$ & $\rhobar|_{I_{\Q_{p^2}}}$ & $\vec{\mathrm{MT}}$ \\\hline\hline
    $R_{int}(\vr;(\infty,5))$
      & $\omega_4^{5p+p^2}\oplus\omega_4^{1+5p^3}$ & $(0,0)$ \\\hline
    $t_1=-3$
      & $
      \begin{bmatrix}
        \omega_2^p & * \\
        0 & \omega_2^{1+4p}
      \end{bmatrix}$ & $(1,0)$ \\\hline
  \end{tabular}
  \caption{$\rhobar$ for $\vk'=(\infty,5)$}\label{tab: (infty,5)}
\end{table}

\subsubsection{\textbf{Summary on mod-$p$ reduction when $\vr=(1,5)$}}\label{subsubsec: summary for 1,5}
In this section of paragraph, we summarize the results on mod-$p$ reduction when $\vr=(1,5)$.

We first consider the case $\cJ_0=\emptyset$, i.e., those $\vk'\in J(\vr;\emptyset)$. We set $\vec{\xi}:=(\xi_0,\xi_1)\in E^2(=E^2_\emptyset)$ where
$$\xi_0=\frac{1}{p}\left(p\fL_0-\fL_1+\frac{17}{6}\right) \qquad\mbox{and}\qquad \xi_1=\frac{1}{p}(p\fL_1-\fL_0).$$
Note that we have $\vx^{(1,2)}=\vec{\xi}$, and that we further have
\begin{align*}
\vx^{(1,1)}&=\left(\xi_0-\tfrac{3}{4p},\xi_1\right);\\
\vx^{(1,3)}&=\left(\xi_0+\tfrac{1}{6p},\xi_1\right);\\ 
\vx^{(0,1)}&=\left(\xi_0-\tfrac{3}{4p},\xi_1+\tfrac{1}{p}\xi_0-\tfrac{3}{4p^2}\right);\\
\vx^{(0,2)}&=\left(\xi_0,\xi_1+\tfrac{1}{p}\xi_0\right).
\end{align*}

We partition $E^2(=E^2_\emptyset)$, such that within each subset the semi-simplification of $\rhobar|_{I_{\Q_{p^2}}}$ remains constant, as follow:
\begin{itemize}[leftmargin=*]
\item $S_1=\{\vL\in E^2\mid v_p(\xi_0)>-1,\ v_p(\xi_1)>-1\}$;
\item $S_2=\{\vL\in E^2\mid v_p(\xi_0)=-1,\ v_p(\xi_1)>-2\}$;
\item $S_3=\{\vL\in E^2\mid v_p(\xi_0)>-1,\ v_p(\xi_1)=-1\}$;
\item $S_4=\{\vL\in E^2\mid  v_p(\xi_0)-v_p(\xi_1)<1,\ v_p(\xi_0)>-1,\ v_p(\xi_1)<-1\}$;
\item $S_5=\{\vL\in E^2\mid v_p(\xi_0)-v_p(\xi_1)>1,\ -2\leq v_p(\xi_1)<-1\}$;
\item $S_6=\{\vL\in E^2\mid v_p(\xi_0)-v_p(\xi_1)=1,\ -2<v_p(\xi_0)<0\}$;
\item $S_7=\{\vL\in E^2\mid v_p(\xi_0)-v_p(\xi_1)<1,\ -2<v_p(\xi_0)<-1\}$;
\item $S_8=\{\vL\in E^2\mid 1<v_p(\xi_0-\tfrac{3}{4p})-v_p(\xi_1)<2,\ -3<v_p(\xi_1)<-2\}$;
\item $S_9=\{\vL\in E^2\mid v_p(\xi_0-\tfrac{3}{4p})-v_p(\xi_1)=2,\ v_p(\xi_0-\tfrac{3}{4p})<0\}$;
\item $S_{10}=\{\vL\in E^2\mid v_p(\xi_0-\tfrac{3}{4p})-v_p(\xi_1)>2,\ v_p(\xi_1)<-2\}$;
\item $S_{11}=\{\vL\in E^2\mid v_p(\xi_0)\leq -2,\ v_p(\xi_1+\tfrac{1}{p}\xi_0)>-3\}$;
\item $S_{12}=\{\vL\in E^2\mid v_p(\xi_0)<-1,\ v_p(\xi_1+\tfrac{1}{p}\xi_0)=-3\}$;
\item $S_{13}=\{\vL\in E^2\mid v_p(\xi_0)-v_p(\xi_1+\tfrac{1}{p}\xi_0)<2,\ v_p(\xi_1+\tfrac{1}{p}\xi_0)<-3\}$.
\end{itemize}
Note that $S_1,\cdots,S_7$ are defined in terms of $v_p(\vec{\xi})$, while $S_8$, $S_9$, $S_{10}$ (resp. $S_{11}$, $S_{12}$, $S_{13}$) are in terms of $(v_p(\xi_0-\frac{3}{4p}),v_p(\xi_1))$ (resp. of $(v_p(\xi_0), v_p(\xi_1+\frac{1}{p}\xi_0))$). But it is easy to see that the unions $S_8\cup S_9\cup S_{10}$ and $S_{11}\cup S_{12}\cup S_{13}$ can be described in terms of $v_p(\vec{\xi})$ as follows:
\begin{align*}
S_8\cup S_9\cup S_{10}
  =& \{\vL\in E^2\mid  v_p(\xi_0)-v_p(\xi_1)>1,\ -3<v_p(\xi_1)<-2\} \\
  &\,\,\, \cup\{\vL\in E^2\mid  v_p(\xi_0)-v_p(\xi_1)\geq 2,\ v_p(\xi_1)<-2\};\\
S_{11}\cup S_{12}\cup S_{13}
  =& \{\vL\in E^2\mid  v_p(\xi_0)\leq -2,\ v_p(\xi_1)\geq -3\} \\
  &\,\,\, \cup\{\vL\in E^2\mid  v_p(\xi_0)-v_p(\xi_1)<2,\ v_p(\xi_1)\leq -3\}.
\end{align*}

It is not difficult to show that $S_i\cap S_j=\emptyset$ if $i\neq j$ and that $\cup_{i=1}^{13}S_i=E^2$. The following picture, Figure~\ref{fig: partition of E^2 for vr=(1,5)}, describes those sets $S_1,\cdots,S_7, S_8\cup S_9\cup S_{10}, S_{11}\cup S_{12}\cup S_{13}$ on $(v_p(\xi_0),v_p(\xi_1))$-plane, emphasizing the disjointness.
\begin{figure}[htbp]
  \centering
  \includegraphics[scale=0.5]{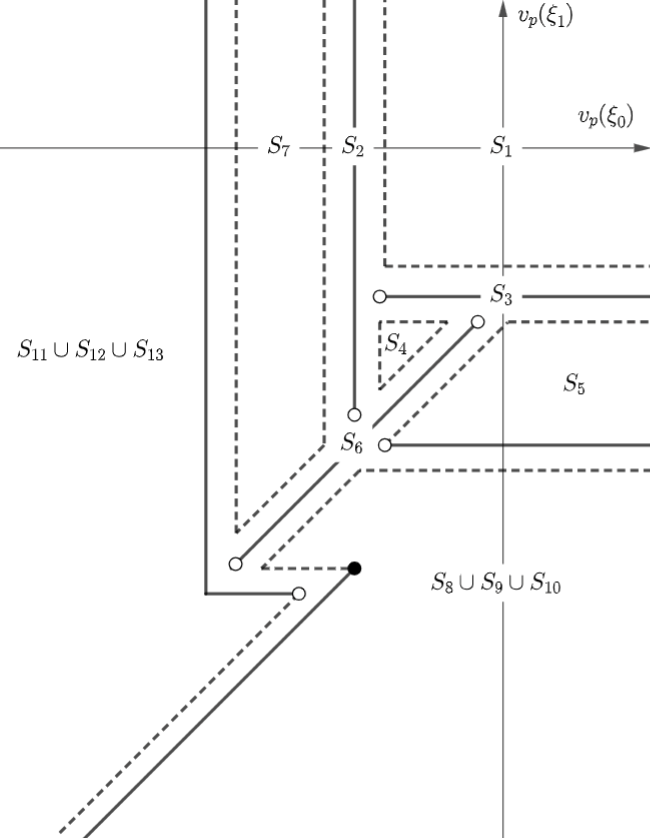}
  \caption{Partition of $E^2$ by $S_1,\cdots,S_7,S_8\cup S_9\cup S_{10}, S_{11}\cup S_{12}\cup S_{13}$.}\label{fig: partition of E^2 for vr=(1,5)}
\end{figure}
We further describe the sets $S_8$, $S_9$, $S_{10}$ (resp. $S_{11}$, $S_{12}$, $S_{13}$) on $(v_p(\xi_0-\frac{3}{4p}),v_p(\xi_1))$-plane (resp. on  $(v_p(\xi_0), v_p(\xi_1+\frac{1}{p}\xi_0))$-plane) in Figure~\ref{fig: vr=(1,5), S8 to S13}.
\begin{figure}[htbp]
  \centering
  \includegraphics[scale=0.5]{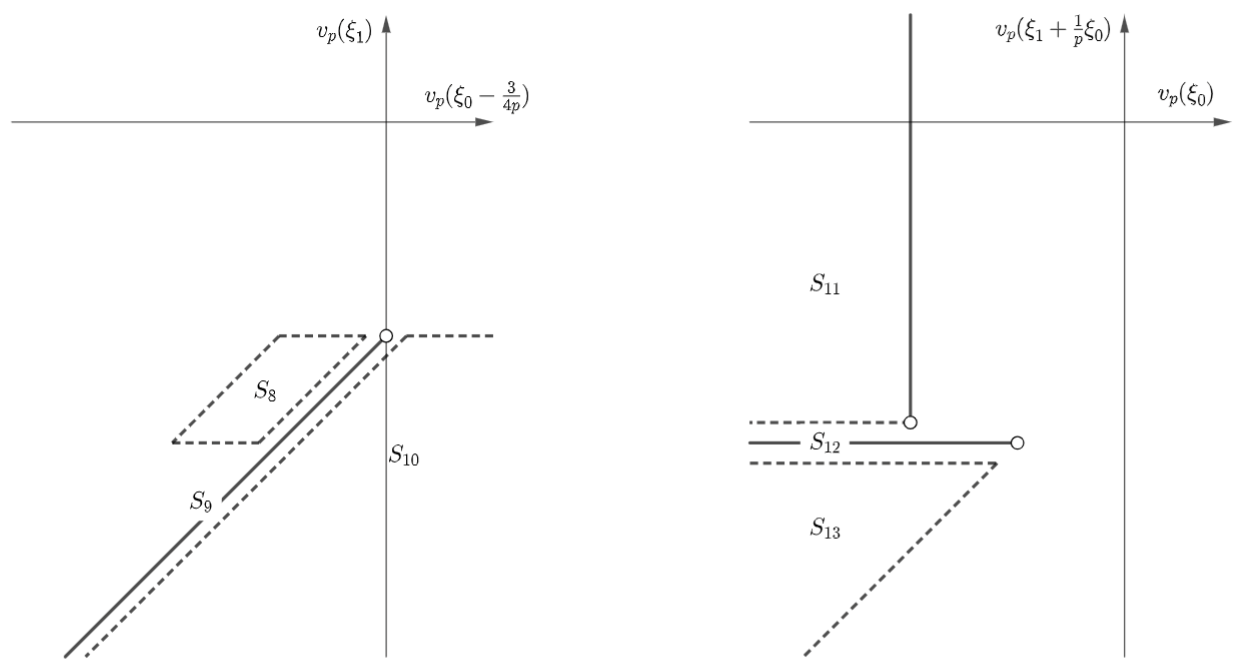}
  \caption{Partitions of $S_8\cup S_9\cup S_{10}$ (left) and of $S_{11}\cup S_{12}\cup S_{13}$ (right)}\label{fig: vr=(1,5), S8 to S13}
\end{figure}

We summarize the non-split mod-$p$ reduction when $\vr=(1,5)$ and $\cJ_0=\emptyset$ in Table~\ref{tab:summary for vr=(1,5), cJ=empty}. One can easily induce the summary in Remark~\ref{rema: 1,5} from Table~\ref{tab:summary for vr=(1,5), cJ=empty}.
\begin{table}[htbp]
    \centering
    \begin{tabular}{|c||c|c|}
    \hline
        $\vL$ & $\rhobar|_\IQptwo$ & $\vec{\mathrm{MT}}$ \\\hline\hline
        $S_1$
            & $\omega_4^{3p+p^2+2p^3}\oplus\omega_4^{1+2p+3p^3}$
            & $(0,0)$ \\\hline
        \multirow{2}{*}[-1.5ex]{$S_2$}
            & $
            \begin{bmatrix}
                \omega_2^{2p} & * \\ 0 & \omega_2^{1+3p}
            \end{bmatrix}$
            & $(0,0)$\\\cline{2-3}
            & $
            \begin{bmatrix}
                \omega_2^{1+3p} & * \\ 0 & \omega_2^{2p}
            \end{bmatrix}$
            & $
            \begin{cases}
                (1,1) &\mbox{if } v_p(\xi_0+\frac{1}{6p})\geq 0 \mbox{ and }v_p(\xi_1)\geq 0; \\
                (0,0) &\mbox{otherwise}
            \end{cases}$ \\\hline
        \multirow{2}{*}[-1.5ex]{$S_3$}
            & $
            \begin{bmatrix}
                \omega_2^{3p} & * \\ 0 & \omega_2^{1+2p}
            \end{bmatrix}$
            & $(0,0)$\\\cline{2-3}
            & $
            \begin{bmatrix}
                \omega_2^{1+2p} & * \\ 0 & \omega_2^{3p}
            \end{bmatrix}$
            & $
            \begin{cases}
                (0,0) &\mbox{if } -1<v_p(\xi_0)<0; \\
                (1,1) &\mbox{if } v_p(\xi_0)\geq 0
            \end{cases}$  \\\hline
        $S_4$
            & $\omega_4^{2p+p^2+3p^3}\oplus\omega_4^{1+3p+2p^3}$
            & $(0,0)$ \\\hline
        $S_5$
            & $\omega_4^{1+4p+p^3}\oplus\omega_4^{p+p^2+4p^3}$
            & $(0,0)$ \\\hline
        \multirow{2}{*}[-2ex]{$S_6$}
            & $
            \begin{bmatrix}
                \omega_2^{4p} & * \\ 0 & \omega_2^{1+p}
            \end{bmatrix}$
            & $
            \begin{cases}
                (0,0) &\mbox{if } v_p(\xi_1+\frac{1}{p}\xi_0)<-2 \\&\mbox{ and } -2<v_p(\xi_0)<-1;\\
                (1,0) &\mbox{otherwise}
            \end{cases}$ \\\cline{2-3}
            & $
            \begin{bmatrix}
                \omega_2^{1+p} & * \\ 0 & \omega_2^{4p}
            \end{bmatrix}$
            & $
            \begin{cases}
                (1,1) &\mbox{if } v_p(\xi_0-\frac{3}{4p})\geq 0;\\
                (0,0) &\mbox{otherwise}
            \end{cases}$ \\\hline
        $S_7$
            & $\omega_4^{4p+p^2+p^3}\oplus\omega_4^{1+p+4p^3}$
            & $(0,0)$ \\\hline
        $S_8$
            & $\omega_4^{p+p^2+4p^3}\oplus\omega_4^{1+4p+p^3}$
            & $(0,0)$ \\\hline
        \multirow{2}{*}[-1.5ex]{$S_9$}
            & $
            \begin{bmatrix}
                \omega_2^{5p} & * \\ 0 & \omega_2
            \end{bmatrix}$
            & $
            \begin{cases}
                (0,0) &\mbox{if } v_p(\xi_0-\frac{3}{4p})<-1;\\
                (1,0) &\mbox{if } -1\leq v_p(\xi_0-\frac{3}{4p})<0
            \end{cases}$ \\\cline{2-3}
            & $
            \begin{bmatrix}
                \omega_2 & * \\ 0 & \omega_2^{5p}
            \end{bmatrix}$
            & $(0,0)$ \\\hline
        $S_{10}$
            & $\omega_4^{1+5p}\oplus\omega_4^{p^2+5p^3}$
            & $(0,0)$ \\\hline
        \multirow{2}{*}[-1.5ex]{$S_{11}$}
            & $
            \begin{bmatrix}
                \omega_2^{4p} & * \\ 0 & \omega_2^{1+p}
            \end{bmatrix}$
            & $
            \begin{cases}
                (0,0) &\mbox{if } -3<v_p(\xi_1+\frac{1}{p}\xi_0)<-2;\\
                (1,0) &\mbox{if }v_p(\xi_1+\frac{1}{p}\xi_0)\geq -2 
            \end{cases}$\\\cline{2-3}
            & $
            \begin{bmatrix}
                \omega_2^{1+p} & * \\ 0 & \omega_2^{4p}
            \end{bmatrix}$
            & $(0,0)$ \\\hline
        \multirow{2}{*}[-0.5em]{$S_{12}$}
            & $
            \begin{bmatrix}
                \omega_2^p & * \\ 0 & \omega_2^{1+4p}
            \end{bmatrix}$
            & $(1,0)$ \\\cline{2-3}
            & $
            \begin{bmatrix}
                \omega_2^{1+4p} & * \\ 0 & \omega_2^p
            \end{bmatrix}$
            & $(0,0)$ \\\hline
        $S_{13}$
            & $\omega_4^{5p+p^2}\oplus\omega_4^{1+5p^3}$
            & $(0,0)$ \\
        \hline
    \end{tabular}
    \caption{Non-split mod-$p$ reduction for $\vr=(1,5)$, $\cJ_0=\emptyset$}
    \label{tab:summary for vr=(1,5), cJ=empty}
\end{table}
We now explain how one can induce Table~\ref{tab:summary for vr=(1,5), cJ=empty} from the results in the previous sections of paragraph.
\begin{itemize}[leftmargin=*]
\item Assume $\vL\in S_1$. Equivalently, we have $v_p(x_1^{(1,3)})=-1$, $v_p(x_0^{(1,3)})=-1$ and $v_p(px_0^{(1,3)}-\frac{1}{6})>0$. Then we have $\rhobar|_\IQptwo\cong\omega_4^{3p+p^2+2p^3}\oplus\omega_4^{1+2p+3p^3}$, from $\mathbf{Case}_\phi(\vr;(\frac{1}{2},\frac{5}{2}))$ (see \S\ref{subsubsec: vr=(1,5), vk=(1/2,5/2)}~(iv)(a)).

\item Assume $\vL\in S_2$, i.e.,  $v_p(\xi_0)=-1$ and $v_p(\xi_1)>-2$. Equivalently we have $v_p(x_0^{(0,2)})=-1$, $v_p(x_1^{(0,2)})=-2$ and $v_p(1-\frac{px_1^{(0,1)}}{x_0^{(0,1)}})>0$. Also, for $v_p(\xi_1)\geq -1$, we have another equivalent formulation given by $v_p(x_0^{(1,3)})\geq -1$, $v_p(x_1^{(1,3)})\geq -1$ and $v_p(px_0^{(1,3)}-\frac{1}{6})=0$. In this case there are two non-homothetic lattices whose mod-$p$ reductions are non-split and reducible.
\begin{itemize}[leftmargin=*]
    \item We first consider $\mathbf{Case}_\phi(\vr;(\frac{3}{2},\frac{3}{2}))$ (see \S\ref{subsubsec: vr=(1,5), vk=(3/2,3/2)}~(iii)(a)). In this case we have $\rhobar|_\IQptwo\cong
    \begin{bmatrix}
        \omega_2^{2p} & * \\ 0 & \omega_2^{1+3p}
    \end{bmatrix}$ with $\vec{\mathrm{MT}}=(1,0)$.
    \item We also consider $\mathbf{Case}_\phi(\vr;(\frac{1}{2},\frac{5}{2}))$ or $\mathbf{Case}_\phi(\vr;(1,\frac{7}{2}))$, which coincide if $v_p(\xi_1)=-1$ (see \S\ref{subsubsec: vr=(1,5), vk=(1/2,5/2)}~(iii)(b),(iv)(b) if $v_p(\xi_1)\geq -1$ and \S\ref{subsubsec: vr=(1,5), vk=(1,7/2)}~(v)  if $-2<v_p(\xi_1)\leq -1$). In this case we have $\rhobar|_\IQptwo\cong
    \begin{bmatrix}
        \omega_2^{1+3p} & * \\ 0 & \omega_2^{2p}
    \end{bmatrix}$, with $\vec{\mathrm{MT}}=(1,1)$ if $v_p(\xi_0+\frac{1}{6p})\geq 0$ and $v_p(\xi_1)\geq 0$, and $\vec{\mathrm{MT}}=(0,0)$ otherwise.
\end{itemize}

\item Assume $\vL\in S_3$, i.e., $v_p(\xi_0)>-1$ and $v_p(\xi_1)=-1$. Equivalently, we have $v_p(x_0^{(1,3)})=-1$, $v_p(x_1^{(1,3)})=-1$ and $v_p(px_0^{(1,3)}-\frac{1}{6})>0$. In this case there are two non-homothetic lattices whose mod-$p$ reductions are non-split and reducible.
\begin{itemize}[leftmargin=*]
    \item We first consider $\mathbf{Case}_\phi(\vr;(\frac{1}{2},\frac{5}{2}))$ (see \S\ref{subsubsec: vr=(1,5), vk=(1/2,5/2)}~(iii)(a)). In this case, we have $\rhobar|_\IQptwo\cong
    \begin{bmatrix}
        \omega_2^{3p} & * \\ 0 & \omega_2^{1+2p}
    \end{bmatrix}$, with $\vec{\mathrm{MT}}=(0,0)$.
    \item We also consider $\mathbf{Case}_\phi(\vr;(1,\frac{7}{2}))$ or $\mathbf{Case}_\phi(\vr;(\frac{1}{2},4))$, which coincide if $v_p(\xi_0)=0$ (see \S\ref{subsubsec: vr=(1,5), vk=(1,7/2)}~(ii),(iii) if $-1<v_p(\xi_0)\leq 0$ and \S\ref{subsubsec: vr=(1,5), vk=(1/2,4)}~(ii) if $v_p(\xi_0)\geq 0$). In this case, we have $\rhobar|_\IQptwo\cong
    \begin{bmatrix}
        \omega_2^{1+2p} & * \\ 0 & \omega_2^{3p}
    \end{bmatrix}$, with $\vec{\mathrm{MT}}=(0,0)$ if $-1<v_p(\xi_0)<0$ and $\vec{\mathrm{MT}}=(1,1)$ if $v_p(\xi_0)\geq 0$.
\end{itemize}

\item Assume $\vL\in S_4$, i.e., $v_p(x_0^{(1,2)})-v_p(x_1^{(1,2)})<-1$, $v_p(x_0^{(1,2)})<-1$ and $v_p(x_1^{(1,2)})>-1$. Then $\rhobar|_\IQptwo\cong\omega_4^{2p+p^2+3p^3}\oplus\omega_4^{1+3p+2p^3}$, from $\mathbf{Case}_\phi(\vr;(1,\frac{7}{2}))$ (see \S\ref{subsubsec: vr=(1,5), vk=(1,7/2)}~(i)).

\item Assume $\vL\in S_5$, i.e., $v_p(x_0^{(1,2)})-v_p(x_1^{(1,2)})>1$ and $-2\leq v_p(x_1^{(1,2)})<-1$. Then $\rhobar|_\IQptwo\cong\omega_4^{1+4p+p^3}\oplus\omega_4^{p+p^2+4p^3}$, from $\mathbf{Case}_\phi(\vr;(\frac{1}{2},4))$ (see \S\ref{subsubsec: vr=(1,5), vk=(1/2,4)}~(i),(v)). Note that if $v_p(\xi_1)=-2$, then it is compatible with the result from $\mathbf{Case}_\phi(\vr;(1,\frac{9}{2})$ (see \S\ref{subsubsec: vr=(1,5), vk=(1,9/2)}~(viii)(a)).

\item Assume $\vL\in S_6$, i.e., $v_p(\xi_0)-v_p(\xi_1)=1$ and $-2<v_p(\xi_1)<0$. In this case there are two non-homothetic lattices whose mod-$p$ reductions are non-split and reducible.
\begin{itemize}[leftmargin=*]
    \item We first consider $\mathbf{Case}_\phi(\vr;(\frac{3}{2},\frac{3}{2}))$ or $\mathbf{Case}_\phi(\vr;(1,\frac{7}{2}))$, which coincide if $v_p(\xi_0)=-1$ (see \S\ref{subsubsec: vr=(1,5), vk=(3/2,3/2)}~(i),(iv)(b) if $-2<v_p(\xi_0)\leq -1$ and from \S\ref{subsubsec: vr=(1,5), vk=(1,7/2)}~(iv) if $-1\leq v_p(\xi_0)<0$). In this case we have $\rhobar|_\IQptwo\cong
    \begin{bmatrix}
        \omega_2^{4p} & * \\ 0 & \omega_2^{1+p}
    \end{bmatrix}$ with $\vec{\mathrm{MT}}=(0,0)$ if $v_p(\xi_1+\frac{1}{p}\xi_0)<-2$ and $-2<v_p(\xi_0)<-1$, and $\vec{\mathrm{MT}}=(1,0)$ otherwise.
    \item We also consider
    \begin{itemize}[leftmargin=*]
        \item $\mathbf{Case}_\phi(\vr;(1,\frac{9}{2}))$ if $-2<v_p(\xi_0)<-1$ (see \S\ref{subsubsec: vr=(1,5), vk=(1,9/2)}~(vii);
        \item $\mathbf{Case}_\phi(\vr;(1,\frac{9}{2}))$ if $v_p(\xi_0)=-1$ and $-1\leq v_p(\xi_0-\frac{3}{4p})\leq 0$ (see \S\ref{subsubsec: vr=(1,5), vk=(1,9/2)}~(viii)(b));
        \item $\mathbf{Case}_\phi(\vr;(\frac{1}{2},5))$ if $v_p(\xi_0)=-1$ and $v_p(\xi_0-\frac{3}{4p})\geq 0$ (see \S\ref{subsubsec: vr=(1,5), vk=(1/2,5)}~(ii));
        \item $\mathbf{Case}_\phi(\vr;(\frac{1}{2},4))$ if $-1<v_p(\xi_0)<0$ (see \S\ref{subsubsec: vr=(1,5), vk=(1/2,4)}~(iii)),
    \end{itemize}
    where $\mathbf{Case}_\phi(\vr;(1,\frac{9}{2}))$ and $\mathbf{Case}_\phi(\vr;(\frac{1}{2},5))$ coincide if $v_p(\xi_0-\frac{3}{4p})=0$. In this case, we have $\rhobar|_\IQptwo\cong
    \begin{bmatrix}
        \omega_2^{1+p} & * \\ 0 & \omega_2^{4p}
    \end{bmatrix}$ with $\vec{\mathrm{MT}}=(1,1)$ if $v_p(\xi_0-\frac{3}{4p})\geq 0$, and $\vec{\mathrm{MT}}=(0,0)$ otherwise.
\end{itemize}

\item Assume $\vL\in S_7$. Equivalently, we have $v_p(x_0^{(0,2)})-v_p(x_1^{(0,2)})=1$, $-2<v_p(x_0^{(0,2)})<-1$ and $v_p(1-\frac{px_1^{(0,2)}}{x_0^{(0,2)}})>0$. Then $\rhobar|_\IQptwo\cong\omega_4^{4p+p^2+p^3}\oplus\omega_4^{1+p+4p^3}$, from $\mathbf{Case}_\phi(\vr;(\frac{3}{2},\frac{3}{2}))$ (see \S\ref{subsubsec: vr=(1,5), vk=(3/2,3/2)}~(iv)(a)).

\item Assume $\vL\in S_8$, i.e., $1<v_p(x_0^{(1,1)})-v_p(x_1^{(1,1)})<2$ and $-3<v_p(x_1^{(1,1)})<-2$. Then $\rhobar|_\IQptwo\cong\omega_4^{p+p^2+4p^3}\oplus\omega_4^{1+4p+p^3}$, from $\mathbf{Case}_\phi(\vr;(1,\frac{9}{2}))$ (see \S\ref{subsubsec: vr=(1,5), vk=(1,9/2)}~(i)).

\item Assume $\vL\in S_9$, i.e., $v_p(x_0^{(1,1)})-v_p(x_1^{(1,1)})=2$ and $v_p(x_0^{(1,1)})<0$. If $v_p(\xi_0-\frac{3}{4p}\leq 1$, we have an equivalent formulation given by $v_p(x_0^{(0,1)})-v_p(x_1^{(0,1)})=2$ and $v_p(x_0^{(0,1)})\leq 1$. In this case there are two non-homothetic lattices whose mod-$p$ reductions are non-split and reducible.
\begin{itemize}[leftmargin=*]
    \item We first consider $\mathbf{Case}_\phi(\vr;(\frac{3}{2},5))$ or $\mathbf{Case}_\phi(\vr;(1,\frac{9}{2}))$, which coincide if $v_p(\xi_0-\frac{3}{4p})=-1$ (see \S\ref{subsubsec: vr=(1,5), vk=(3/2,5)}~(iii),(iv) if $v_p(\xi_0-\frac{3}{4p})\leq -1$ and \S\ref{subsubsec: vr=(1,5), vk=(1,9/2)}~(iv) if $-1\leq v_p(\xi_0-\frac{3}{4p})<0$). In this case, we have $\rhobar_1|_\IQptwo\cong
    \begin{bmatrix}
        \omega_2^{5p} & * \\ 0 & \omega_2
    \end{bmatrix}$, with $\vec{\mathrm{MT}}=(0,0)$ if $v_p(\xi_0-\frac{3}{4p})<-1$ and $\vec{\mathrm{MT}}=(1,0)$ if $-1\leq v_p(\xi_0-\frac{3}{4p})<0$.
    \item We also consider $\mathbf{Case}_\phi(\frac{1}{2},5)$ (see \S\ref{subsubsec: vr=(1,5), vk=(1/2,5)}~(iii)). In this case, we have $\rhobar|_\IQptwo\cong
    \begin{bmatrix}
        \omega_2 & * \\ 0 & \omega_2^{5p}
    \end{bmatrix}$ with $\vec{\mathrm{MT}}=(0,0)$.
\end{itemize}

\item Assume $\vL\in S_{10}$, i.e., $v_p(x_0^{(1,1)})-v_p(x_1^{(!,1)})>2$ and $v_p(x_1^{(1,1)})<-2$. Then $\rhobar|_\IQptwo\cong\omega_4^{1+5p}\oplus\omega_4^{p^2+5p^3}$, from $\mathbf{Case}_\phi(\vr;(\frac{1}{2},5))$ (see \S\ref{subsubsec: vr=(1,5), vk=(1/2,5)}~(i)).

\item Assume $\vL\in S_{11}$, i.e., $v_p(x_0^{(0,2)})\leq -2$ and $v_p(x_1^{(0,2)})>-3$. In this case there are two non-homothetic lattices whose mod-$p$ reductions are non-split and reducible.
\begin{itemize}[leftmargin=*]
    \item We first consider the lattice from $v_p(\Theta_0)\in\{0,v_p(x_1^{(0,2)})+2\}$ of $\mathbf{Case}_\phi(\vr;(\frac{3}{2},\frac{3}{2}))$ (see \S\ref{subsubsec: vr=(1,5), vk=(3/2,3/2)}~(i)). In this case, we have $\rhobar|_\IQptwo\cong
    \begin{bmatrix}
        \omega_2^{4p} & * \\ 0 & \omega_2^{1+p}
    \end{bmatrix}$ with $\vec{\mathrm{MT}}=(0,0)$ if $-3<v_p(\xi_1+\frac{1}{p}\xi_0)<-2$ and $\vec{\mathrm{MT}}=(1,0)$ if $v_p(\xi_1+\frac{1}{p}\xi_0)\geq -2$.
    \item We also consider the lattice from $v_p(\Theta_0)=-1$ of $\mathbf{Case}_\phi(\vr;(\frac{3}{2},\frac{3}{2}))$ (see \S\ref{subsubsec: vr=(1,5), vk=(3/2,3/2)}~(i)). In this case, we have $\rhobar|_\IQptwo\cong
    \begin{bmatrix}
        \omega_2^{1+p} & * \\ 0 & \omega_2^{4p}
    \end{bmatrix}$ with $\vec{\mathrm{MT}}=(0,0)$.
\end{itemize}

\item Assume $\vL\in S_{12}$, i.e., $v_p(x_0^{(0,2)})<-1$ and $v_p(x_1^{(0,2)})=-3$. For each $\xi_0$, we have equivalent formulations by
\begin{itemize}[leftmargin=*]
    \item $v_p(x_0^{(0,1)})<-2$ and $v_p(x_1^{(0,1)})=-3$, if $v_p(\xi_0)\leq -2$;
    \item $v_p(x_0^{(1,1)})=v_p(px_1^{(1,1)}+x_0^{(1,1)})=-2$ and $v_p(x_1^{(1,1)})=-3$, if $v_p(\xi_0)=-2$;
    \item $-2<v_p(x_0^{(1,1)})<-1$ and $v_p(x_1^{(1,1)})=-3$, if $-2<v_p(\xi_0)<-1$.
\end{itemize}
In this case there are two non-homothetic lattices whose mod-$p$ reductions are non-split and reducible.
\begin{itemize}[leftmargin=*]
    \item We first consider $\rhobar$ from $\mathbf{Case}_\phi(\vr;(\frac{3}{2},5))$ (see \S\ref{subsubsec: vr=(1,5), vk=(3/2,5)}~(ii)). In this case, we have $\rhobar|_\IQptwo\cong
    \begin{bmatrix}
        \omega_2^p & * \\ 0 & \omega_2^{1+4p}
    \end{bmatrix}$ with $\vec{\mathrm{MT}}=(1,0)$.
    \item We also consider $\rhobar$ from $\mathbf{Case}_\phi(\vr;(\frac{3}{2},\frac{3}{2}))$ or from $\mathbf{Case}_\phi(\vr;(1,\frac{9}{2}))$, which coincide if $v_p(x_0^{(0,1)})=-2$ (see \S\ref{subsubsec: vr=(1,5), vk=(3/2,3/2)}~(v) if $v_p(\xi_0)\leq -2$ and \S\ref{subsubsec: vr=(1,5), vk=(1,9/2)}~(v),(vi)(b) if $-2\leq v_p(\xi_0)<-1$). In this case, we have $\rhobar|_\IQptwo\cong
    \begin{bmatrix}
        \omega_2^{1+4p} & * \\ 0 & \omega_2^p
    \end{bmatrix}$ with $\vec{\mathrm{MT}}=(0,0)$.
\end{itemize}

\item Assume $\vL\in S_{13}$. Equivalently, we have $v_p(x_0^{(0,1)})-v_p(x_1^{(0,1)})<-2$ and $v_p(x_1^{(0,1)})<-3$. Then $\rhobar|_\IQptwo\cong\omega_4^{5p+p^2}\oplus\omega_4^{1+5p^3}$, from $\mathbf{Case}_\phi(\vr;(\frac{3}{2},5))$ (see \S\ref{subsubsec: vr=(1,5), vk=(3/2,5)}~(i)).
\end{itemize}

\smallskip
We now treat the case $\cJ_0=\{0\}$, i.e., those $\vk'\in J(\vr;\{0\})$. We set $$\xi'=\frac{1}{p^2}\left(p^2\fL_1-\fL_1+\frac{17}{6}\right)\in E(= E^2_{\{0\}}).$$
Here, we identify $E$ with $E^2_{\{0\}}$ in the obvious manner. One can readily induce $\xi'$ by solving \eqref{eq: definition of vx} for $x^{(0,2)}_1$, keeping in mind that $\Delta_0(-1)=-x_0^{(0,2)}$ and $\Delta_1(-1)=-\frac{17}{6}$. It is easy to see that
$$\vx^{(0,1)}_1=\xi'-\tfrac{3}{4p^2}
\quad\mbox{and}\quad
\vx^{(0,2)}_1=\xi'.$$
Note that $\vx^{(0,1)}_0$ and $\vx^{(0,2)}_0$ do not affect the mod-$p$ reduction, as $\cJ_0=\{0\}$ (see \eqref{eq: R vs R'}). 

We summarize the non-split mod-$p$ reduction when $\vr=(1,5)$ and $J_0=\{0\}$ in Table~\ref{tab:summary for vr=(1,5), cJ={0}}. One can easily induce the summary in Remark~\ref{rema: 1,5} from Table~\ref{tab:summary for vr=(1,5), cJ={0}}.
\begin{table}[htbp]
    \centering
    \begin{tabular}{|c||c|c|}
    \hline
        $\xi'$ & $\rhobar|_\IQptwo$ & $\vec{\mathrm{MT}}$ \\\hline\hline
        \multirow{2}{*}[-1.5ex]{$v_p(\xi')>-3$}
            & $
            \begin{bmatrix}
                \omega_2^{4p} & * \\ 0 & \omega_2^{1+p}
            \end{bmatrix}$
            & $(0,0)$\\\cline{2-3}
            & $
            \begin{bmatrix}
                \omega_2^{1+p} & * \\ 0 & \omega_2^{4p}
            \end{bmatrix}$
            & $
            \begin{cases}
                (0,0) &\mbox{if }-3<v_p(\xi')<-2; \\
                (1,1) &\mbox{if }v_p(\xi')\geq -2 
            \end{cases}$  \\\hline
        \multirow{2}{*}[-1.5ex]{$v_p(\xi')=-3$}
            & $
            \begin{bmatrix}
                \omega_2^p & * \\ 0 & \omega_2^{1+4p}
            \end{bmatrix}$
            & $(1,0)$\\\cline{2-3}
            & $
            \begin{bmatrix}
                \omega_2^{1+4p} & * \\ 0 & \omega_2^p
            \end{bmatrix}$
            & $(0,0)$ \\\hline
        $v_p(\xi')<-3$
            & $\omega_4^{5p+p^2}\oplus\omega_4^{1+5p^3}$
            & $(0,0)$ \\
        \hline
    \end{tabular}
    \caption{Non-split mod-$p$ reduction for $\vr=(1,5)$, $\cJ_0=\{0\}$}
    \label{tab:summary for vr=(1,5), cJ={0}}
\end{table}
We now explain how one can induce Table~\ref{tab:summary for vr=(1,5), cJ={0}} from the results in the previous sections of paragraph.
\begin{itemize}[leftmargin=*]
\item Assume $v_p(\xi')>-3$. In this case, there are two non-homothetic lattices whose mod-$p$ reductions are non-split and reducible.
\begin{itemize}[leftmargin=*]
    \item We first consider the lattice from $v_p(\Theta_0)=-1$ of $\mathbf{Case}_\phi(\vr;(\infty,\frac{3}{2}))$ (see \S\ref{subsubsec: vr=(1,5), vk=(infty,3/2)}~(i)). In this case, we have $\rhobar_1|_\IQptwo\cong
    \begin{bmatrix}
        \omega_2^{4p} & * \\ 0 & \omega_2^{1+p}
    \end{bmatrix}$ with $\vec{\mathrm{MT}}=(0,0)$.
    \item We also consider the lattice from $v_p(\Theta_0)\in\{0,v_p(\xi')+2\}$ $\mathbf{Case}_\phi(\vr;(\infty,\frac{3}{2}))$ (see \S\ref{subsubsec: vr=(1,5), vk=(infty,3/2)}~(i)). In this case, we have $\rhobar|_\IQptwo\cong
    \begin{bmatrix}
        \omega_2^{1+p} & * \\ 0 & \omega_2^{4p}
    \end{bmatrix}$, with $\vec{\mathrm{MT}}=(0,0)$ if $-3<v_p(\xi')<-2$ and $\vec{\mathrm{MT}}=(1,1)$ if $v_p(\xi')\geq -2$.
\end{itemize}

\item Assume $v_p(\xi')=-3$. Equivalently, we have $v_p(x_1^{(0,1)})=-3$. In this case, there are two non-homothetic lattices whose mod-$p$ reductions are non-split and reducible. 
\begin{itemize}[leftmargin=*]
    \item We first consider $\mathbf{Case}_\phi(\vr;(\infty,5))$ (see \S\ref{subsubsec: vr=(1,5), vk=(infty,5)}~(ii)). In this case, we have $\rhobar|_\IQptwo\cong
    \begin{bmatrix}
        \omega_2^p & * \\ 0 & \omega_2^{1+4p}
    \end{bmatrix}$ with $\vec{\mathrm{MT}}=(1,0)$.
    \item We also consider $\mathbf{Case}_\phi(\vr;(\infty,\frac{3}{2}))$ (see \S\ref{subsubsec: vr=(1,5), vk=(infty,3/2)}~(ii)). In this case, we have $\rhobar|_\IQptwo\cong
    \begin{bmatrix}
        \omega_2^{1+4p} & * \\ 0 & \omega_2^p
    \end{bmatrix}$ with $\vec{\mathrm{MT}}=(0,0)$.
\end{itemize}

\item Assume $v_p(\xi')<-3$. Equivalently, we have $v_p(x_1^{(0,1)})<-3$. Then $\rhobar|_\IQptwo\cong\omega_4^{5p+p^2}\oplus\omega_4^{1+5p^3}$, from $\mathbf{Case}_\phi(\vr;(\infty,5))$ (see \S\ref{subsubsec: vr=(1,5), vk=(infty,5)}~(i)).
\end{itemize}

\begin{rema}\label{rema: 1,5} 
One can extract the following results from Table~\ref{tab:summary for vr=(1,5), cJ=empty} and Table~\ref{tab:summary for vr=(1,5), cJ={0}}:
\begin{itemize}[leftmargin=*]
\item our method provides with at least one Galois stable lattice in each $2$-dimensional semi-stable non-crystalline representation $V$ of $G_{\Q_{p^2}}$ with Hodge--Tate weights $\mathrm{HT}(V)_0=(0,1)$ and $\mathrm{HT}(V)_1=(0,5)$;
\item if the mod-$p$ reduction is an extension of two distinct characters then our method provides the two non-homothetic lattices.
\end{itemize}
\end{rema}

\smallskip

\bibliographystyle{alpha}

\end{document}